\documentclass[a4paper,
pagesize,
oneside, 				
numbers=noenddot,
DIV=10, 
headinclude=false, 
bibtotoc, 
fontsize=11pt, 
abstracton,
cleardoublepage=empty, 
english
]{scrartcl}

\usepackage{setspace}
\usepackage{standalone}

\usepackage{shuffle}

\usepackage{simplewick}

\usepackage{authblk}

\usepackage{mathtools}

\usepackage{stmaryrd}

\usepackage{pdflscape}

\usepackage{pst-pdf}

\usepackage{url}

\usepackage{upgreek}

\usepackage{sidecap}

\usepackage{xcolor}
\definecolor{darkred}{rgb}{0.9,0.1,0.1}
\definecolor{darkblue}{rgb}{0,0,0.7}
\definecolor{darkgreen}{rgb}{0,0.5,0}
\definecolor{bluegray}{rgb}{0.4, 0.6, 0.8}
\definecolor{cadmiumorange}{rgb}{0.93, 0.53, 0.18}
\definecolor{darkcerulean}{rgb}{0.03, 0.27, 0.49}
\definecolor{pagebackground}{rgb}{.15,.21,.18}
\definecolor{pageforeground}{rgb}{.84,.84,.85}
\usepackage[pagebackref, colorlinks=false, linkcolor=darkblue, linktocpage=true, linkbordercolor=red, citecolor=darkblue]{hyperref} % 

\usepackage[english]{babel}
\usepackage[utf8]{inputenc}
\usepackage{lmodern} 
\usepackage[T1]{fontenc} 
\usepackage{amsmath, amsfonts, amssymb}
\usepackage{mathtools}
\usepackage{mathrsfs}
\usepackage{verbatim}
\usepackage{stmaryrd}
\usepackage[framed,amsmath,thmmarks,hyperref]{ntheorem}
\usepackage{commath}

\usepackage[labelfont=bf]{caption}

\usepackage{longtable}

\allowdisplaybreaks

\usepackage[english=american]{csquotes}

\usepackage{graphicx}

\usepackage{comment}

\usepackage{framed}

\usepackage{pdflscape}

\usepackage[]{listofsymbols}

\setkomafont{sectioning}{\bfseries} 

\setkomafont{pagenumber}{\bfseries}
\setkomafont{pageheadfoot}{\normalfont}
\setkomafont{pagenumber}{\normalfont}

\usepackage{enumitem} 

\makeatletter
\newcommand{\myitem}[1]{%
\item[#1]\protected@edef\@currentlabel{#1}%
}
\makeatother

\usepackage{ifthen}

\usepackage{tikz,pgfplots}
\pgfplotsset{compat=newest}%{width=5cm,height=4cm}
\usetikzlibrary{cd}
\usetikzlibrary{calc}
\usetikzlibrary{arrows, scopes}
\usetikzlibrary{decorations.pathmorphing}
\usetikzlibrary{positioning}
\usetikzlibrary{shapes.geometric}
\usetikzlibrary{backgrounds}
\usetikzlibrary{arrows,chains,matrix,positioning,scopes}
\usepackage[multiple,hang,flushmargin]{footmisc}

\usepackage{graphicx}
\usepackage[colorinlistoftodos, disable]{todonotes}  %"disable" to turn it off
\setuptodonotes{inline}
\usepackage{marginnote}

\usepackage[framemethod=TikZ]{mdframed}

\mdfdefinestyle{temporary}{linewidth=3pt, linecolor=red}
\mdfdefinestyle{onlyframe}{middlelinecolor=black!100,roundcorner=1ex, skipbelow=15pt}
\mdfdefinestyle{theorem}{middlelinecolor=black!25, backgroundcolor=black!15,roundcorner=1ex, skipbelow=15pt}
\mdfdefinestyle{onlyback}{middlelinecolor=white!100, backgroundcolor=black!10,roundcorner=1ex, skipbelow=15pt}
\mdfdefinestyle{main}{middlelinecolor=darkblue!25, backgroundcolor=darkblue!15,roundcorner=1ex, skipbelow=15pt}

\usepackage{commath}

\usepackage{enumitem}

\usepackage{nicefrac}

\usepackage{rotating}

\usepackage{relsize}

\usepackage{ctable}
\usepackage{multicol}
\usepackage{multirow}

\usepackage{accents}

\usepackage{ucs}
\usepackage{fixltx2e} % Schickere Ausgabe
\usepackage{lmodern} % ordentliche Schriften
\usepackage{setspace,graphicx,tikz,tabularx} % für Elemente der Titelseite
\usepackage[draft=false,babel,tracking=true,kerning=true,spacing=true]{microtype} % optischer Randausgleich etc.

\usepackage{private}

\DeclareSymbolFont{extraup}{U}{zavm}{m}{n}
\DeclareMathSymbol{\varclub}{\mathalpha}{extraup}{84} 
\DeclareMathSymbol{\varspade}{\mathalpha}{extraup}{85}

\usepackage{tikz,pgfplots}
\pgfplotsset{width=\textwidth,height=5cm}
\usepackage{mhequ}
\usepackage{mathrsfs}

\usepackage{multirow}

\usepackage[framemethod=TikZ]{mdframed}

\mdfdefinestyle{temporary}{linewidth=3pt, linecolor=red}
\mdfdefinestyle{theorem}{middlelinecolor=black!25, backgroundcolor=black!15,roundcorner=1ex, skipbelow=15pt}
\mdfdefinestyle{main}{middlelinecolor=darkblue!25, backgroundcolor=darkblue!15,roundcorner=1ex, skipbelow=15pt}

\mdfdefinestyle{temporary}{linewidth=3pt, linecolor=red}
\mdfdefinestyle{onlyframe}{middlelinecolor=black!100,roundcorner=1ex, skipbelow=15pt}
\mdfdefinestyle{theorem}{middlelinecolor=black!25, backgroundcolor=black!15,roundcorner=1ex, skipbelow=15pt}
\mdfdefinestyle{onlyback}{
	middlelinecolor=white!100, 
	backgroundcolor=black!10,
	roundcorner=10pt, 
	skipbelow=25pt,
	linewidth=0pt,
	topline = true,    
	bottomline = true,
	leftline = true,
	rightline = true}
\mdfdefinestyle{main}{middlelinecolor=darkblue!25, backgroundcolor=darkblue!15,roundcorner=1ex, skipbelow=15pt}

\tikzset{
	dectriangle/.style 2 args={
		triangle,
		alias=sourcenode,
		append after command={\decorate{#1}{#2}}
	},
	dectriangle/.default={0}{0},
}

\tikzset{
	cross/.style={path picture={ 
			\draw[symbols]
			(path picture bounding box.south east) -- (path picture bounding box.north west) (path picture bounding box.south west) -- (path picture bounding box.north east);
	}},
	root/.style={circle,fill=green!50!black,inner sep=0pt, minimum size=1.2mm},
	dot/.style={circle,fill=pageforeground,inner sep=0pt, minimum size=1mm},
	bigdot/.style={circle,fill=black,inner sep=0pt,minimum size=4mm},
	dotred/.style={circle,fill=pageforeground!50!pagebackground,inner sep=0pt, minimum size=2mm},
	var/.style={circle,fill=pageforeground!10!pagebackground,draw=pageforeground,inner sep=0pt, minimum size=3mm},
	kernel/.style={semithick,shorten >=2pt,shorten <=2pt},
	kernels/.style={snake=zigzag,shorten >=2pt,shorten <=2pt,segment amplitude=1pt,segment length=4pt,line before snake=2pt,line after snake=5pt,},
	rho/.style={densely dashed,semithick,shorten >=2pt,shorten <=2pt},
	testfcn/.style={dotted,semithick,shorten >=2pt,shorten <=2pt},
	renorm/.style={shape=circle,fill=pagebackground,inner sep=1pt},
	labl/.style={shape=rectangle,fill=pagebackground,inner sep=1pt},
	xic/.style={very thin,circle,draw=symbols,fill=symbols,inner sep=0pt,minimum size=1.2mm},
	g/.style={very thin,rectangle,draw=symbols,fill=symbols!10!pagebackground,inner sep=0pt,minimum width=2.5mm,minimum height=1.2mm},
	xi/.style={very thin,circle,draw=symbols,fill=symbols!10!pagebackground,inner sep=0pt,minimum size=1.2mm},
	xicams/.style={very thin,circle,draw=cameron,fill=cameron!10!pagebackground,inner sep=0pt,minimum size=1.2mm},
	xicams2/.style={very thin,circle,draw=cameron2,fill=cameron2!10!pagebackground,inner sep=0pt,minimum size=1.2mm},
	cmn/.style={circle,fill=cameron!50,draw=cameron,inner sep=0pt,minimum size=1.2mm},
	gxi/.style={very thin,circle,draw=lightgray,fill=lightgray!10!pagebackground,inner sep=0pt,minimum size=1.2mm},
	xies/.style={very thin,rectangle,fill=green!50!black!25,draw=symbols,inner sep=0pt,minimum size=1.1mm},
	xiesf/.style={very thin,rectangle,fill=green!50!black,draw=symbols,inner sep=0pt,minimum size=1.1mm},
	xix/.style={very thin,crosscircle,fill=symbols!10!pagebackground,draw=symbols,inner sep=0pt,minimum size=1.2mm},
	X/.style={very thin,cross,rectangle,fill=pagebackground,draw=symbols,inner sep=0pt,minimum size=1.2mm},
	xib/.style={thin,circle,fill=symbols!10!pagebackground,draw=symbols,inner sep=0pt,minimum size=1.6mm},
	xigreen/.style={circle,fill=darkgreen!30,draw=darkgreen,inner sep=0pt,minimum size=2.2mm},
	xidarkcerulean/.style={circle,fill=darkcerulean!30,draw=darkcerulean,inner sep=0pt,minimum size=2.2mm},
	xivilet/.style={circle,fill=violet!30,draw=violet,inner sep=0pt,minimum size=2.2mm},
	xired/.style={circle,fill=darkred!30,draw=darkred,inner sep=0pt,minimum size=2.2mm},
	xifill/.style={circle,fill=white, draw=black,inner sep=0pt,minimum size=2.2mm},
	xibig/.style={circle,fill=black!30,draw=black,inner sep=0pt,minimum size=2.2mm},
	xie/.style={thin,circle,fill=green!50!black,draw=symbols,inner sep=0pt,minimum size=1.6mm},
	xid/.style={thin,circle,fill=symbols,draw=symbols,inner sep=0pt,minimum size=1.6mm},
	xibx/.style={thin,crosscircle,fill=symbols!10!pagebackground,draw=symbols,inner sep=0pt,minimum size=1.6mm},
	kernels2/.style={very thick,draw=connection,segment length=12pt},
	gkernels2/.style={very thick,draw=lightgray,segment length=12pt},
	keps/.style={thin,draw=symbols,->},
	kepspr/.style={thick,draw=connection,->},
	krho/.style={thin,draw=symbols,superdense,->},
	krhopr/.style={thick,draw=connection,superdense,->},
	triangle/.style = { regular polygon, regular polygon sides=3},
	notnorm/.style={thin,circle,draw=symbols,fill=symbols,inner sep=0pt,minimum size=0.5mm},
	not/.style={thin,circle,draw=connection,fill=connection,inner sep=0pt,minimum size=0.5mm},
	notcam/.style={thin,circle,draw=cameron,fill=cameron,inner sep=0pt,minimum size=0.5mm},
	gnot/.style={thin,circle,draw=lightgray,fill=lightgray,inner sep=0pt,minimum size=0.5mm},
	diff/.style = {very thin,draw=symbols,triangle,fill=red!50!black,inner sep=0pt,minimum size=1.6mm},
	diff1/.style = {very thin,dectriangle={1}{0},fill=red!50!black,draw=symbols,inner sep=0pt,minimum size=1.6mm},
	diff2/.style = {very thin,dectriangle={1}{1},fill=red!50!black,draw=symbols,inner sep=0pt,minimum size=1.6mm},
	diffmini/.style = {very thin,rectangle,fill=black,draw=black,inner sep=0pt,minimum size=0.75mm},
	kernelsmod/.style={very thick,draw=connection,segment length=12pt},
	rec/.style = {very thin,rectangle,fill=black,draw=black,inner sep=0pt,minimum size=2mm},
	cerc/.style={very thin,circle,draw=black,fill=symbols,inner sep=0pt,minimum size=2mm},
	trup/.style={very thin,regular polygon,regular polygon sides=3,shape border rotate=180,draw=black,fill=red,inner sep=0pt,minimum size=3mm},
	stars/.style={very thin,star,star points=6,star point ratio=0.5, draw=black,fill=red,inner sep=0pt,minimum size=0.7mm},
	>=stealth,
}

\makeatletter
\def\DeclareSymbol#1#2#3{%
	\expandafter\gdef\csname MH@symb@#1\endcsname{\tikzsetnextfilename{symbol#1}%
		\tikz[baseline=#2,scale=0.15,draw=symbols,line join=round]{#3}}%
	\expandafter\gdef\csname MH@symb@#1s\endcsname{\scalebox{0.75}{\tikzsetnextfilename{symbol#1}%
			\tikz[baseline=#2,scale=0.15,draw=symbols,line join=round]{#3}}}%
	\expandafter\gdef\csname MH@symb@#1ss\endcsname{\scalebox{0.65}{\tikzsetnextfilename{symbol#1}%
			\tikz[baseline=#2,scale=0.15,draw=symbols,line join=round]{#3}}}%
}
\def\<#1>{\ifthenelse{\boolean{mmode}}{\mathchoice{\csname MH@symb@#1\endcsname}{\csname MH@symb@#1\endcsname}{\csname MH@symb@#1s\endcsname}{\csname MH@symb@#1ss\endcsname}}{\csname MH@symb@#1\endcsname}}
\makeatother

\DeclareSymbol{greennode}{-2.8}{\node[xigreen] {};}
\DeclareSymbol{rednode}{-2.8}{\node[xired] {};}

\def\CA{\mathcal{A}}
\def\CB{\mathcal{B}}
\def\CC{\mathcal{C}}
\def\CD{\mathcal{D}}

\def\CF{\mathcal{F}}
\def\CG{\mathcal{G}}

\def\CI{\mathcal{I}}

\def\CL{\mathcal{L}}
\def\CM{\mathcal{M}}
\def\CN{\mathcal{N}}

\def\CR{\mathcal{R}}
\def\CS{\mathcal{S}}

\def\CW{\mathcal{W}}

\renewcommand{\S}{\mathbb{S}}
\renewcommand{\Y}{\mathbb{Y}}
\newcommand{\YY}{\boldsymbol{Y}}

\newcommand{\BBbar}{\boldsymbol{\bar{B}}}
\newcommand{\ESig}{\operatorname{ESig}}

\def\fC{\mathfrak{C}}

\def\fw{\mathfrak{w}}

\def\w{\mathtt{w}}
\def\v{\mathtt{v}}

\newcommand{\bi}{\boldsymbol{i}}

\newcommand{\bq}{\boldsymbol{q}}

\def\tte{\mathtt{e}}

\newcommand{\RN}[1]{%
\textup{\uppercase\expandafter{\romannumeral#1}}%
}

\newcommand{\scal}[1]{\langle #1 \rangle}

\colorlet{testcolor}{green!60!black}
\colorlet{testcolor2}{blue!100!black}
\colorlet{lightblue}{darkblue!80}
\colorlet{grayblue}{bluegray!100}
\colorlet{orange}{cadmiumorange!100}

\usetikzlibrary{calc}
\usetikzlibrary{shapes.misc}
\usetikzlibrary{shapes.symbols}
\usetikzlibrary{shapes.geometric}
\usetikzlibrary{snakes}
\usetikzlibrary{decorations}
\usetikzlibrary{decorations.markings}
\usetikzlibrary{patterns}

\makeatletter
\pgfdeclareshape{crosscircle}
{
\inheritsavedanchors[from=circle] % this is nearly a circle
\inheritanchorborder[from=circle]
\inheritanchor[from=circle]{north}
\inheritanchor[from=circle]{north west}
\inheritanchor[from=circle]{north east}
\inheritanchor[from=circle]{center}
\inheritanchor[from=circle]{west}
\inheritanchor[from=circle]{east}
\inheritanchor[from=circle]{mid}
\inheritanchor[from=circle]{mid west}
\inheritanchor[from=circle]{mid east}
\inheritanchor[from=circle]{base}
\inheritanchor[from=circle]{base west}
\inheritanchor[from=circle]{base east}
\inheritanchor[from=circle]{south}
\inheritanchor[from=circle]{south west}
\inheritanchor[from=circle]{south east}
\inheritbackgroundpath[from=circle]
\foregroundpath{
\centerpoint%
\pgf@xc=\pgf@x%
\pgf@yc=\pgf@y%
\pgfutil@tempdima=\radius%
\pgfmathsetlength{\pgf@xb}{\pgfkeysvalueof{/pgf/outer xsep}}%  
\pgfmathsetlength{\pgf@yb}{\pgfkeysvalueof{/pgf/outer ysep}}%  
\ifdim\pgf@xb<\pgf@yb%
\advance\pgfutil@tempdima by-\pgf@yb%
\else%
\advance\pgfutil@tempdima by-\pgf@xb%
\fi%
\pgfpathmoveto{\pgfpointadd{\pgfqpoint{\pgf@xc}{\pgf@yc}}{\pgfqpoint{-0.707107\pgfutil@tempdima}{-0.707107\pgfutil@tempdima}}}
\pgfpathlineto{\pgfpointadd{\pgfqpoint{\pgf@xc}{\pgf@yc}}{\pgfqpoint{0.707107\pgfutil@tempdima}{0.707107\pgfutil@tempdima}}}
\pgfpathmoveto{\pgfpointadd{\pgfqpoint{\pgf@xc}{\pgf@yc}}{\pgfqpoint{-0.707107\pgfutil@tempdima}{0.707107\pgfutil@tempdima}}}
\pgfpathlineto{\pgfpointadd{\pgfqpoint{\pgf@xc}{\pgf@yc}}{\pgfqpoint{0.707107\pgfutil@tempdima}{-0.707107\pgfutil@tempdima}}}
}
}
\makeatother

\colorlet{symbols}{blue!90!black}
\colorlet{trees}{black!90!black}
\colorlet{cameron}{darkgreen!90!black}
\colorlet{testcolor}{green!60!black}
\colorlet{darkblue}{blue!60!black}
\colorlet{darkgreen}{green!60!black}

\def\symbol#1{\textcolor{symbols}{#1}}
\def\1{\mathbf{1}}
\def\X{\symbol{X}}

\usetikzlibrary{calc}
\usetikzlibrary{shapes.misc}
\usetikzlibrary{shapes.symbols}
\usetikzlibrary{shapes.geometric}
\usetikzlibrary{snakes}
\usetikzlibrary{decorations}
\usetikzlibrary{decorations.markings}

\def\drawx{\draw[-,solid] (-3pt,-3pt) -- (3pt,3pt);\draw[-,solid] (-3pt,3pt) -- (3pt,-3pt);}
\tikzset{
root/.style={draw=symbols,circle,inner sep=0pt,minimum size=0.5mm,fill=white},	
smalldot/.style={circle,fill=symbols,draw=symbols,inner sep=0pt,minimum size=0.5mm},
dot/.style={circle,fill=black,inner sep=0pt,minimum size=1mm},
bigdot/.style={circle,fill=black,inner sep=0pt,minimum size=4mm},
noiseedge/.style={black,semithick,decorate, decoration={snake,segment length=4pt,amplitude=1pt}},
smallestdot1/.style={circle,fill=symbols,draw=symbols,inner sep=0pt,minimum size=0.2mm},
smallestdot2/.style={circle,fill=cameron,draw=cameron,inner sep=0pt,minimum size=0.2mm},
smallnoiseedge1/.style={draw=symbols,decorate, decoration={snake,segment length=1.75pt,amplitude=0.55pt}},
smallnoiseedge2/.style={draw=cameron,decorate, decoration={snake,segment length=1.75pt,amplitude=0.55pt}},
noiseedge1/.style={draw=symbols,decorate, decoration={snake,segment length=0.9pt,amplitude=0.55pt}},
noiseedge2/.style={draw=cameron,decorate, decoration={snake,segment length=0.9pt,amplitude=0.55pt}},
smallnoisenode1/.style={draw=symbols,circle,inner sep=0pt,minimum size=0.5mm,fill=white},	
smallnoisenode2/.style={draw=cameron,circle,inner sep=0pt,minimum size=0.5mm,fill=white},
poly/.style={draw=symbols,circle,inner sep=0pt,minimum size=0.5mm,fill=black},
dot/.style={circle,fill=black,inner sep=0pt,minimum size=1mm},
int/.style={circle,fill=black,draw=black,inner sep=0pt,minimum size=0.7mm},
circ/.style={circle,draw=black,inner sep=0pt, minimum size=1mm},
var/.style={circle,fill=black!10,draw=black,inner sep=0pt, minimum size=2mm},
dotred/.style={circle,fill=black!50,inner sep=0pt, minimum size=2mm},
generic/.style={semithick,shorten >=1pt,shorten <=1pt},
oddfunc/.style={generic, dotted},
dist/.style={ultra thick,draw=testcolor,shorten >=1pt,shorten <=1pt},
testfcn/.style={ultra thick,testcolor,shorten >=1pt,shorten <=1pt,<-},
testfunction/.style={ultra thick,testcolor,shorten >=1pt,shorten <=1pt},
testfcnx/.style={ultra thick,testcolor,shorten >=1pt,shorten <=1pt,<-,
postaction={decorate,decoration={markings,mark=at position 0.6 with {\drawx}}}},
kprime/.style={semithick,shorten >=1pt,shorten <=1pt,densely dashed,->},
kprimex/.style={semithick,shorten >=1pt,shorten <=1pt,densely dashed,->,
postaction={decorate,decoration={markings,mark=at position 0.4 with {\drawx}}}},
kernel/.style={semithick,shorten >=1pt,shorten <=1pt,->,draw=black},
Pkernel/.style={ultra thick,shorten >=1pt,shorten <=1pt,->,draw=blue},
PkernelBig/.style={very thick,shorten >=1pt,shorten <=1pt,decorate, draw=blue, decoration={zigzag,amplitude=1.5pt,segment length = 3pt,pre length=2pt,post length=2pt}},
multx/.style={shorten >=1pt,shorten <=1pt,
postaction={decorate,decoration={markings,mark=at position 0.5 with {\drawx}}}},
kernelx/.style={semithick,shorten >=1pt,shorten <=1pt,->,
postaction={decorate,decoration={markings,mark=at position 0.4 with {\drawx}}}},
kernel1/.style={->,semithick,shorten >=1pt,shorten <=1pt,postaction={decorate,decoration={markings,mark=at position 0.45 with {\draw[-] (0,-0.1) -- (0,0.1);}}}},
kernel2/.style={->,semithick,shorten >=1pt,shorten <=1pt,postaction={decorate,decoration={markings,mark=at position 0.45 with {\draw[-] (0.05,-0.1) -- (0.05,0.1);\draw[-] (-0.05,-0.1) -- (-0.05,0.1);}}}},
kernelBig/.style={semithick,shorten >=1pt,shorten <=1pt,decorate, decoration={zigzag,amplitude=1.5pt,segment length = 3pt,pre length=2pt,post length=2pt}},
kernelBigg/.style={thick,shorten >=1pt,shorten <=1pt,decorate, decoration={zigzag,amplitude=3.5pt,segment length = 7pt,pre length=2pt,post length=2pt}},
kernelBigg1/.style={thick,shorten >=1pt,shorten <=1pt,decorate, decoration={saw,amplitude=3.5pt,segment length = 7pt,pre length=2pt,post length=2pt}},
kernelBigg2/.style={thick,shorten >=1pt,shorten <=1pt,decorate, decoration={bumps,amplitude=3.5pt,segment length = 7pt,pre length=2pt,post length=2pt}},
rho/.style={dotted,semithick,shorten >=1pt,shorten <=1pt},
k/.style={color=trees, snake=zigzag, segment length=1pt, segment amplitude=0.3pt},
r/.style={color=trees, double, double distance=0.1pt},
renorm/.style={shape=circle,fill=white,inner sep=1pt},
labl/.style={shape=rectangle,fill=white,inner sep=1pt},
xi/.style={circle,fill=symbols!30,draw=symbols,inner sep=0pt,minimum size=0.8mm},
cmn/.style={circle,fill=cameron!50,draw=symbols,inner sep=0pt,minimum size=1.2mm},
xix/.style={crosscircle,fill=symbols!10,draw=symbols,inner sep=0pt,minimum size=1.2mm},
xib/.style={circle,fill=symbols!30,draw=symbols,inner sep=0pt,minimum size=1.4mm},
xiblack/.style={circle,fill=trees,draw=trees,inner sep=0pt,minimum size=0.7mm},
xigray/.style={circle,fill=symbols!30,draw=symbols,inner sep=0pt,minimum size=0.8mm},
xibx/.style={crosscircle,fill=symbols!10,draw=symbols,inner sep=0pt,minimum size=1.6mm},
not/.style={circle,fill=symbols,draw=symbols,inner sep=0pt,minimum size=0.5mm},
>=stealth,
}
\makeatletter
\def\DeclareSymbol#1#2#3{\expandafter\gdef\csname MH@symb@#1\endcsname{\tikz[baseline=#2,scale=0.13,draw=symbols]{#3}}\expandafter\gdef\csname MH@symb@#1s\endcsname{\scalebox{0.6}{\tikz[baseline=#2,scale=0.15,draw=symbols]{#3}}}}
\def\<#1>{\csname MH@symb@#1\endcsname}
\makeatother

\DeclareSymbol{xi}{-3}{\node[xibig] {};}
\DeclareSymbol{wn}{-3}{\node[xibig] {};}
\DeclareSymbol{xigray}{-3}{\node[xigray] {};}
\DeclareSymbol{X}{-2.4}{\node[dot] {};}
\DeclareSymbol{1}{0}{\draw[white] (-.4,0) -- (.4,0); \draw (0,0)  -- (0,1.2) node[xi] {};}
\DeclareSymbol{2}{0}{\draw (-0.5,1.2) node[xi] {} -- (0,0) -- (0.5,1.2) node[xi] {};}
\DeclareSymbol{3}{0}{\draw (0,0) -- (0,1.2) node[xi] {}; \draw (-.7,1) node[xi] {} -- (0,0) -- (.7,1) node[xi] {};}
\DeclareSymbol{31}{-3}{\draw (0,0) -- (0,-1) -- (1,0) node[xi] {}; \draw (0,0) -- (0,1.2) node[xi] {}; \draw (-.7,1) node[xi] {} -- (0,0) -- (.7,1) node[xi] {};}
\DeclareSymbol{30}{-3}{\draw (0,0) -- (0,-1); \draw (0,0) -- (0,1.2) node[xi] {}; \draw (-.7,1) node[xi] {} -- (0,0) -- (.7,1) node[xi] {};}
\DeclareSymbol{10}{-3}{\draw (0,0) -- (0,-1); \draw (-.8,1) node[xi] {} -- (0,0);}
\DeclareSymbol{32}{-3}{\draw (0,0) -- (0,-1) -- (1,0) node[xi] {}; \draw (0,0) -- (0,-1) -- (-1,0) node[xi] {}; \draw (0,0) -- (0,1.2) node[xi] {}; \draw (-.7,1) node[xi] {} -- (0,0) -- (.7,1) node[xi] {};}
\DeclareSymbol{22}{-3}{\draw (0,0.3) -- (0,-1) -- (1,0) node[xi] {}; \draw (0,0.3) -- (0,-1) -- (-1,0) node[xi] {};\draw (-.7,1) node[xi] {} -- (0,0.3) -- (.7,1) node[xi] {};}
\DeclareSymbol{12}{-3}{\draw (0,0) -- (0,-1) -- (1,0) node[xi] {}; \draw (0,0) -- (0,-1) -- (-1,0) node[xi] {}; \draw (-.8,1) node[xi] {} -- (0,0);}
\DeclareSymbol{20}{-3}{\draw (0,0) -- (0,-1);\draw (-.7,1) node[xi] {} -- (0,0) -- (.7,1) node[xi] {};}

\DeclareSymbol{xib}{-3}{\node[xiblack] {};}
\DeclareSymbol{X}{-2.4}{\node[dot] {};}
\DeclareSymbol{1b}{0}{\draw[white] (-.4,0) -- (.4,0); \draw[trees] (0,0)  -- (0,1.2) node[xiblack] {};}
\DeclareSymbol{2b}{0}{\draw[trees] (-0.5,1.2) node[xiblack] {} -- (0,0) -- (0.5,1.2) node[xiblack] {};}
\DeclareSymbol{3b}{0}{\draw[trees] (0,0) -- (0,1.2) node[xiblack] {}; \draw[trees] (-.7,1) node[xiblack] {} -- (0,0) -- (.7,1) node[xiblack] {};}
\DeclareSymbol{31b}{-3}{\draw[trees] (0,0) -- (0,-1) -- (1,0) node[xiblack] {}; \draw[trees] (0,0) -- (0,1.2) node[xiblack] {}; \draw[trees] (-.7,1) node[xiblack] {} -- (0,0) -- (.7,1) node[xiblack] {};}
\DeclareSymbol{30b}{-3}{\draw[trees] (0,0) -- (0,-1); \draw[trees] (0,0) -- (0,1.2) node[xiblack] {}; \draw[trees] (-.7,1) node[xiblack] {} -- (0,0) -- (.7,1) node[xiblack] {};}
\DeclareSymbol{10b}{-3}{\draw[trees] (0,0) -- (0,-1); \draw[trees] (-.8,1) node[xiblack] {} -- (0,0);}
\DeclareSymbol{32b}{-3}{\draw[trees] (0,0) -- (0,-1) -- (1,0) node[xiblack] {}; \draw[trees] (0,0) -- (0,-1) -- (-1,0) node[xiblack] {}; \draw[trees] (0,0) -- (0,1.2) node[xiblack] {}; \draw[trees] (-.7,1) node[xiblack] {} -- (0,0) -- (.7,1) node[xiblack] {};}
\DeclareSymbol{22b}{-3}{\draw[trees] (0,0.3) -- (0,-1) -- (1,0) node[xiblack] {}; \draw[trees] (0,0.3) -- (0,-1) -- (-1,0) node[xiblack] {};\draw[trees] (-.7,1) node[xiblack] {} -- (0,0.3) -- (.7,1) node[xiblack] {};}
\DeclareSymbol{12b}{-3}{\draw[trees] (0,0) -- (0,-1) -- (1,0) node[xi] {}; \draw[trees] (0,0) -- (0,-1) -- (-1,0) node[xiblack] {}; \draw[trees] (-.8,1) node[xiblack] {} -- (0,0);}
\DeclareSymbol{20b}{-3}{\draw[trees] (0,0) -- (0,-1);\draw[trees] (-.7,1) node[xiblack] {} -- (0,0) -- (.7,1) node[xiblack] {};}

\DeclareSymbol{G}{1}{\draw[white] (-.4,0) -- (.4,0); \draw[trees, semithick] (0,0)  -- (0,2.5) {};}
\DeclareSymbol{K}{1}{\draw[white] (-.4,0) -- (.4,0); \draw[trees, semithick, snake=zigzag, segment length=2pt, segment amplitude=0.5pt] (0,0)  -- (0,2.5)  {};}
\DeclareSymbol{R}{1}{\draw[white] (-.4,0) -- (.4,0); \draw[color=trees, double, double distance=0.4pt] (0,0)  -- (0,2.5) {};}
\DeclareSymbol{I}{1}{\draw[white] (-.4,0) -- (.4,0); \draw[symbols, semithick] (0,0)  -- (0,2.5) {};}

\DeclareSymbol{1r}{0}{\draw[white] (-.4,0) -- (.4,0); \draw[r] (0,0)  -- (0,1.2) node[xiblack] {};}
\DeclareSymbol{1k}{0}{\draw[white] (-.4,0) -- (.4,0); \draw[k] (0,0)  -- (0,1.2) node[xiblack] {};}
\DeclareSymbol{2kr}{0}{\draw[k] (-0.5,1.2) node[xiblack] {} -- (0,0); \draw[r] (0,0) -- (0.5,1.2) node[xiblack] {};}
\DeclareSymbol{2rk}{0}{\draw[k] (-0.5,1.2) node[xiblack] {} -- (0,0); \draw[r] (0,0) -- (0.5,1.2) node[xiblack] {};}
\DeclareSymbol{2kk}{0}{\draw[k] (-0.5,1.2) node[xiblack] {} -- (0,0); \draw[k] (0,0) -- (0.5,1.2) node[xiblack] {};}
\DeclareSymbol{2rr}{0}{\draw[r] (-0.5,1.2) node[xiblack] {} -- (0,0); \draw[r] (0,0) -- (0.5,1.2) node[xiblack] {};}
\DeclareSymbol{3rrr}{0}{\draw[r] (0,0) -- (0,1.2) node[xiblack] {}; \draw[r] (-.7,1) node[xiblack] {} -- (0,0) -- (.7,1) node[xiblack] {};}
\DeclareSymbol{3kkk}{0}{\draw[k] (0,0) -- (0,1.2) node[xiblack] {}; \draw[k] (-.7,1) node[xiblack] {} -- (0,0) -- (.7,1) node[xiblack] {};}
\DeclareSymbol{3kkr}{0}{\draw[k] (0,0) -- (0,1.2) node[xiblack] {}; \draw[k] (-.7,1) node[xiblack] {} -- (0,0); \draw[r] (0,0) -- (.7,1) node[xiblack] {};}
\DeclareSymbol{3krr}{0}{\draw[r] (0,0) -- (0,1.2) node[xiblack] {}; \draw[k] (-.7,1) node[xiblack] {} -- (0,0); \draw[r] (0,0) -- (.7,1) node[xiblack] {};}
\DeclareSymbol{3rrr}{0}{\draw[r] (0,0) -- (0,1.2) node[xiblack] {}; \draw[r] (-.7,1) node[xiblack] {} -- (0,0) -- (.7,1) node[xiblack] {};}
\DeclareSymbol{3kkk}{0}{\draw[k] (0,0) -- (0,1.2) node[xiblack] {}; \draw[k] (-.7,1) node[xiblack] {} -- (0,0) -- (.7,1) node[xiblack] {};}
\DeclareSymbol{3kkr}{0}{\draw[k] (0,0) -- (0,1.2) node[xiblack] {}; \draw[k] (-.7,1) node[xiblack] {} -- (0,0); \draw[r] (0,0) -- (.7,1) node[xiblack] {};}
\DeclareSymbol{3krr}{0}{\draw[r] (0,0) -- (0,1.2) node[xiblack] {}; \draw[k] (-.7,1) node[xiblack] {} -- (0,0); \draw[r] (0,0) -- (.7,1) node[xiblack] {};}
\DeclareSymbol{30k}{-3}{\draw[k] (0,0) -- (0,-1); \draw[trees] (0,0) -- (0,1.2) node[xiblack] {}; \draw[trees] (-.7,1) node[xiblack] {} -- (0,0) -- (.7,1) node[xiblack] {};}
\DeclareSymbol{30r}{-3}{\draw[r] (0,0) -- (0,-1); \draw[trees] (0,0) -- (0,1.2) node[xiblack] {}; \draw[trees] (-.7,1) node[xiblack] {} -- (0,0) -- (.7,1) node[xiblack] {};}
\DeclareSymbol{3rrr0k}{-3}{\draw[k] (0,0) -- (0,-1); \draw[r] (0,0) -- (0,1.2) node[xiblack] {}; \draw[r] (-.7,1) node[xiblack] {} -- (0,0) -- (.7,1) node[xiblack] {};}
\DeclareSymbol{3kkk0k}{-3}{\draw[k] (0,0) -- (0,-1);\draw[k] (0,0) -- (0,1.2) node[xiblack] {}; \draw[k] (-.7,1) node[xiblack] {} -- (0,0) -- (.7,1) node[xiblack] {};}
\DeclareSymbol{3kkr0k}{-3}{\draw[k] (0,0) -- (0,-1);\draw[k] (0,0) -- (0,1.2) node[xiblack] {}; \draw[k] (-.7,1) node[xiblack] {} -- (0,0); \draw[r] (0,0) -- (.7,1) node[xiblack] {};}
\DeclareSymbol{3krr0k}{-3}{\draw[k] (0,0) -- (0,-1);\draw[r] (0,0) -- (0,1.2) node[xiblack] {}; \draw[k] (-.7,1) node[xiblack] {} -- (0,0); \draw[r] (0,0) -- (.7,1) node[xiblack] {};}
\DeclareSymbol{3rrr0r}{-3}{\draw[r] (0,0) -- (0,-1); \draw[r] (0,0) -- (0,1.2) node[xiblack] {}; \draw[r] (-.7,1) node[xiblack] {} -- (0,0) -- (.7,1) node[xiblack] {};}
\DeclareSymbol{3kkk0r}{-3}{\draw[r] (0,0) -- (0,-1);\draw[k] (0,0) -- (0,1.2) node[xiblack] {}; \draw[k] (-.7,1) node[xiblack] {} -- (0,0) -- (.7,1) node[xiblack] {};}
\DeclareSymbol{3kkr0r}{-3}{\draw[r] (0,0) -- (0,-1);\draw[k] (0,0) -- (0,1.2) node[xiblack] {}; \draw[k] (-.7,1) node[xiblack] {} -- (0,0); \draw[r] (0,0) -- (.7,1) node[xiblack] {};}
\DeclareSymbol{3krr0r}{-3}{\draw[r] (0,0) -- (0,-1);\draw[r] (0,0) -- (0,1.2) node[xiblack] {}; \draw[k] (-.7,1) node[xiblack] {} -- (0,0); \draw[r] (0,0) -- (.7,1) node[xiblack] {};}
\DeclareSymbol{30b}{-3}{\draw[trees] (0,0) -- (0,-1); \draw[trees] (0,0) -- (0,1.2) node[xiblack] {}; \draw[trees] (-.7,1) node[xiblack] {} -- (0,0) -- (.7,1) node[xiblack] {};}

\DeclareSymbol{xib}{-3}{\node[xiblack] {};}
\DeclareSymbol{X}{-2.4}{\node[dot] {};}
\DeclareSymbol{1b}{0}{\draw[white] (-.4,0) -- (.4,0); \draw[trees] (0,0)  -- (0,1.2) node[xiblack] {};}
\DeclareSymbol{2b}{0}{\draw[trees] (-0.5,1.2) node[xiblack] {} -- (0,0) -- (0.5,1.2) node[xiblack] {};}
\DeclareSymbol{3b}{0}{\draw[trees] (0,0) -- (0,1.2) node[xiblack] {}; \draw[trees] (-.7,1) node[xiblack] {} -- (0,0) -- (.7,1) node[xiblack] {};}
\DeclareSymbol{31b}{-3}{\draw[trees] (0,0) -- (0,-1) -- (1,0) node[xiblack] {}; \draw[trees] (0,0) -- (0,1.2) node[xiblack] {}; \draw[trees] (-.7,1) node[xiblack] {} -- (0,0) -- (.7,1) node[xiblack] {};}
\DeclareSymbol{30b}{-3}{\draw[trees] (0,0) -- (0,-1); \draw[trees] (0,0) -- (0,1.2) node[xiblack] {}; \draw[trees] (-.7,1) node[xiblack] {} -- (0,0) -- (.7,1) node[xiblack] {};}
\DeclareSymbol{10b}{-3}{\draw[trees] (0,0) -- (0,-1); \draw[trees] (-.8,1) node[xiblack] {} -- (0,0);}
\DeclareSymbol{32b}{-3}{\draw[trees] (0,0) -- (0,-1) -- (1,0) node[xiblack] {}; \draw[trees] (0,0) -- (0,-1) -- (-1,0) node[xiblack] {}; \draw[trees] (0,0) -- (0,1.2) node[xiblack] {}; \draw[trees] (-.7,1) node[xiblack] {} -- (0,0) -- (.7,1) node[xiblack] {};}
\DeclareSymbol{22b}{-3}{\draw[trees] (0,0.3) -- (0,-1) -- (1,0) node[xiblack] {}; \draw[trees] (0,0.3) -- (0,-1) -- (-1,0) node[xiblack] {};\draw[trees] (-.7,1) node[xiblack] {} -- (0,0.3) -- (.7,1) node[xiblack] {};}
\DeclareSymbol{12b}{-3}{\draw[trees] (0,0) -- (0,-1) -- (1,0) node[xi] {}; \draw[trees] (0,0) -- (0,-1) -- (-1,0) node[xiblack] {}; \draw[trees] (-.8,1) node[xiblack] {} -- (0,0);}
\DeclareSymbol{20b}{-3}{\draw[trees] (0,0) -- (0,-1);\draw[trees] (-.7,1) node[xiblack] {} -- (0,0) -- (.7,1) node[xiblack] {};}

\usepackage{caption}
\usepackage{subcaption}

\usepackage{lipsum}

\usepackage{cite}

\setkomafont{dictumauthor}{\normalfont}
\setkomafont{dictumtext}{\normalfont}

\usepackage{colortbl}

\allowdisplaybreaks[1]

\usepackage{scalerel}

\renewcommand{\emptyset}{\varnothing}

\newcommand{\Gammareg}{\Gamma_{\text{\tiny R}}}
\newcommand{\Gammairreg}{\Gamma_{\text{\tiny IR}}}
\newcommand{\DD}{\mathbb{D}}
\renewcommand{\AA}{\mathbb{A}}
\newcommand{\ext}{\mathrm{ext}}
\newcommand{\var}{\mathrm{var}}
\newcommand{\internal}{\mathrm{int}}
\newcommand{\symotimes}{\tilde{\otimes}}
\newcommand{\Dom}{\operatorname{Dom}}
\newcommand{\Sym}{\operatorname{Sym}}
\newcommand{\ASym}{\operatorname{ASym}}
\newcommand{\biota}{\boldsymbol{\underline{\iota}}}

\newcommand{\cbox}[2]{\tikz[baseline=(char.base)]{\node[draw=black,fill=#1,fill opacity=0.4,text opacity=1,inner sep=2pt,minimum height=1em] (char) {#2};}}

\definecolor{mutedorange}{RGB}{255,165,0}
\definecolor{mutedgreen}{RGB}{0,128,0}
\definecolor{mutedcyan}{RGB}{70,130,180}
\definecolor{mutedmagenta}{RGB}{229,43,80}
\definecolor{mutedpurple}{RGB}{130,100,150}

\makeatletter
\newcommand\bigwp{\mathop{\mathpalette\bigDi@mond\relax}}
\newcommand\bigDi@mond[2]{%
\vcenter{\hbox{\m@th
\scalebox{\ifx#1\displaystyle 2\else1.2\fi}{$#1\diamond$}%
}}%
}
\makeatother

\author{Henri Elad Altman, Tom Klose, Nicolas Perkowski}

\date{\today}

\title{A Rough Functional \\ Breuer--Major Theorem}

\numberwithin{equation}{section}

\begin{document}

\maketitle	

\begin{abstract}
	We extend the functional Breuer--Major theorem by Nourdin and Nualart~(2020) to the space of rough paths.
	The proof of tightness combines the multiplication formula for iterated Malliavin divergences, due to Furlan and Gubinelli~(2019), with Meyer's inequality and a Kolmogorov-type criterion for the $r$-variation of càdlàg rough paths, due to Chevyrev et al.~(2022).
	Since martingale techniques do not apply, we obtain the convergence of the finite-dimensional distributions through a bespoke version of Slutsky's lemma: 
	First, we overcome the lack of hypercontractivity by an iterated integration-by-parts scheme which reduces the remaining 
	analysis to finite Wiener chaos; crucially, this argument relies on Malliavin differentiability of the nonlinearity but not on chaos decay and, as a consequence, encompasses the centred absolute value function.    
	Second, in the spirit of the law of large numbers, 
	we show that the diagonal of the second-order process converges to an explicit symmetric correction term. 
	Finally, we compute all the moments of the remaining process and, through a fine combinatorial analysis, show that they converge to those of the Stratonovich Brownian rough path perturbed by an antisymmetric area correction, as computed by a suitable amendment of Fawcett's theorem. 
	All of these steps benefit from a major combinatorial reduction that is implied by the original argument of Breuer and Major~(1983).
\end{abstract}

\noindent
\emph{Key words and phrases.} \\
\noindent
Breuer--Major theorem, Rough paths, Rough invariance principles, Malliavin calculus, Meyer's inequality, Hermite shift	

\vspace{0.5em}
\noindent
\emph{2020 Mathematics Subject Classification.} \\
\noindent
60F17, 60L20, 60H07

\setcounter{tocdepth}{3} 			

\tableofcontents

\section{Introduction}

The \emph{Central Limit Theorem}~(CLT) is at the heart of probability and statistics. We can state it in the following way: for a sequence of centred i.i.d. random variables $Z=(Z_i)_{i \in \Z}$ with variance $\sigma_Z^2 < \infty$, there exists an i.i.d. sequence of standard normal Gaussians~$X = (X_i)_{i \in \Z}$ and a 
real-valued function\footnote{For example, one can take~$f = F_Z^{-1} \circ \Phi$ where~$F_Z^{-1}$ is the quantile function of~$Z$ and~$\Phi$ the standard normal cumulative distribution function.} $f$ such that~$Z_i \overset{d}{=} f(X_i)$ and
\begin{equation*}
	\frac{1}{N^{\nicefrac{1}{2}}} \sum_{i=0}^{N-1} f(X_i) \longrightarrow \CN(0,\sigma_Z^2) \quad \text{in law as} \ N \to \infty \,.
	%\tag{CLT}
\end{equation*}
In many applications of interest, however, the data exhibits non-trivial \emph{correlations} (see, for example,~\cite{beran_92} for a variety of examples) and thus violates the crucial independence assumption. 
In this context, the celebrated Breuer--Major theorem~\cite{breuer_major_83} provides a sufficient criterion on the function~$f$ and the decay of correlations under which one can still observe CLT behaviour:
\begin{theorem}[Breuer--Major] \label{thm:breuer_major}
	Consider a stationary sequence~$X = (X_i)_{i \in \Z}$ of centred, one-dimensional Gaussian random variables with covariance function~$\rho(i) = \E\sbr[0]{X_0 X_i}$ such that~$\rho(0) = 1$.
	Further, let~$\gamma \coloneqq \CN(0,1)$ and $f \in L^2(\gamma)$ with Hermite rank $d \geq 1$, i.e. its chaos decomposition is given by 
	\begin{equation*}
		f(x) = \sum_{q=d}^{\infty} c_q \, H_q(x)	
	\end{equation*}
	where~$H_q$ denotes the $q$-th Hermite polynomial.
	Then, if~$\rho \in \ell^d(\Z)$ and~$S_N$ is given by 
	\begin{equation*} 
		S_N(t) := \frac{1}{N^{\nicefrac{1}{2}}} \sum_{i=0}^{\lfloor N t \rfloor-1} f(X_i), \quad t \in [0,1] \,,
	\end{equation*}
	we have 
	\begin{equation} \label{e:fdd_convergence}
		S_N \overset{\text{\upshape{f.d.d.}}} \longrightarrow \sigma B \quad \text{as} \quad N \to \infty 
	\end{equation}
	where $B$ is a standard Brownian motion and the variance~$\sigma^2$ is given by 
	\begin{equation} \label{e:sigma}
		\sigma^2 
		\coloneqq \sum_{q=d}^{\infty} q! \, c_q^2 \sum_{k \in \mathbb{Z}} \rho(k)^q < \infty \,.
	\end{equation}
\end{theorem}

\smallskip
The previous theorem immediately begs the question as to whether the convergence in~\eqref{e:fdd_convergence} can be updated to \emph{functional convergence} in the Skorokhod space~$\mathbf{D}(0,1)$. 
As Chambers and Slud~\cite{chambers_slud_89} have shown by an explicit counterexample, this is not the case under the sole assumptions that~$f$ has Hermite rank~$d$ and~$\rho \in \ell^d(\Z)$; instead, they provide a sufficient condition, later slightly improved by Ben Hariz~\cite{ben_hariz_02}, which requires explicit information on the decay of the Hermite coefficients~$(c_q)_{q \geq d}$.
Their \emph{fast chaos decay assumption}, however, is difficult to verify in practice and has been replaced by Nourdin and Nualart~\cite{nourdin_nualart_20} (see also their joint work with Campese~\cite{campese_nourdin_nualart_20}) by the less restrictive, more easily checkable assumption that~$f \in L^{p_1}(\gamma)$ for some~$p_1 > 2$.   

\paragraph{Main result.}
The purpose of the present work is to lift the functional Breuer--Major theorem to the space of \emph{rough paths} which, in order to be non-trivial, requires to look at~$\R^m$-valued functions~$\vec{f}$ for~$m \geq 2$. 
More precisely, we consider the vector of functions
\begin{equation*}
	\vec{f}
	= \del[0]{f_1,\ldots,f_m}, \quad f_k: \R \to \R \,, 
\end{equation*}
where each~$f_k \in L^2(\gamma)$ is of the \emph{same} finite Hermite rank~$d \geq 1$ with corresponding Hermite decomposition
\begin{equation} \label{e:hermite_exp_fk}
	f_k(x) = \sum_{q \geq d} c_q^{(k)} H_q(x), \quad k \in \llbracket 1,m \rrbracket \coloneqq \cbr[0]{1,\ldots,m} \, .
\end{equation}
In addition, from here onwards, we let
\begin{equation} \label{e:vector_X}
	\del[0]{X_i = (X^{(1)}_i,\ldots,X^{(m)}_i)}_{i \in \Z}
\end{equation}
be an~$\R^m$-valued stationary centred Gaussian sequence, that is, a multivariate Gaussian process indexed by~$\Z$. 
In addition, for~$k, \ell \in \llbracket 1,m \rrbracket$, we let\footnote{Observe that, \emph{by convention}, we always write~$\rho_{k,\ell}(n)$ to denote the time difference from the first to the second variable encoded by the subindices, and not the other way round.}
\begin{equation} \label{e:multivariate_covariance}
	\rho_{k,\ell}(j-i) \coloneqq \E\sbr[0]{X_i^{(k)} X_j^{(\ell)}}
\end{equation}
and assume that~$\rho_{k,k}(0) = 1$. Note that~$\rho_{k,\ell}(u) = \rho_{\ell,k}(-u)$ for~$u \in \Z$.
For~$t \in [0,1]$, we then define~$\boldsymbol{S}_N(t) = (S_N(t), \S_N(t))$ where the \emph{first} and \emph{second-order processes}~$S_N$ and~$\S_N$ are, respectively, given by
\begin{equs}[][e:1st_2nd_order_p]
	S_N(t) & \coloneqq \frac{1}{N^{\nicefrac{1}{2}}} \sum_{i = 0}^{\lfloor N t \rfloor - 1} \vec{f}(X_i), \quad 
	\S_N(t) \coloneqq \frac{1}{N} \sum_{0 \leq i < j \leq \lfloor N t \rfloor - 1} \vec{f}(X_i) \otimes \vec{f}(X_j) 
\end{equs}
where
\begin{equation} \label{e:vector_f}
	\vec{f}(X_i) \coloneqq \del[0]{f_1(X^{(1)}_i), \ldots, f_m(X^{(m)}_i)} \,.
\end{equation}
We refer to Remark~\ref{rmk:general}\ref{rmk:general:3} below for further comments on our setting.
Finally, we write~$\DD^{k,p}(\gamma)$ for the Malliavin--Sobolev space w.r.t.~$\gamma = \CN(0,1)$ (see Definition~\ref{d:malliavin_derivative} and Remark~\ref{rmk:definitions_1D} below) as well as~$\mathcal{D}^{r-\var}([0,1];\R^m)$ for the space of Skorokhod-type $r$-variation rough paths (see Definition~\ref{d:skorokhod_variation_metric} below).

\medskip 
The following is our main result: 

\begin{theorem}[Rough Functional Breuer--Major Theorem] \label{thm:main}
	Let~$m \in \N$ and consider an $m$-dimensional vector $\vec{f} = (f_1,\ldots,f_m)$ whose components~$f_k \in L^2(\gamma)$ each are of Hermite rank \underline{at least}\footnote{See Remark~\ref{rmk:main} below for further comments.} $d \geq 1$ with Hermite expansion as in~\eqref{e:hermite_exp_fk}.
	We also impose the following conditions:
	\begin{enumerate}[label=(\arabic*)]
		\item \label{thm:main:1}
		For any~$k \in \llbracket 1,m \rrbracket$, we have\footnote{We introduce the parameter~$p$ here, so we can later refer to it.}~$f_k \in \mathbb{D}^{d,2p}(\gamma)$ for $p = 2$, i.e.~$f_k \in \mathbb{D}^{d,4}(\gamma)$.
		\item \label{thm:main:2}
		For each~$k, \ell \in \llbracket 1,m \rrbracket$, we have~$\sum_{i \in \Z} \abs[0]{\rho_{k,\ell}(i)}^{d} < \infty$, i.e.~$\rho_{k,\ell} \in \ell^{d}(\Z)$.
	\end{enumerate}
	Then, for any~$r > 2$ we have
	\begin{equation} \label{thm:main:conv}
		\lim_{N \to \infty} \boldsymbol{S}_N  = \boldsymbol{B} \quad \text{in law in} \ \mathcal{D}^{r-\var}([0,1];\R^m)
	\end{equation}
	where~$\boldsymbol{B} = (B,\mathbb{B})$ is a Brownian rough path with characteristics~$(\Sigma,\Gamma)$, that is: 
	\begin{enumerate}[label=(\roman*)]
		\item The first component $B$ is an $m$-dimensional Brownian motion with covariance matrix~$\Sigma$ given by
		\begin{equation} \label{e:Sigma}
			\Sigma 
			=
			\DD + 2 \Sym(\Gamma) 
		\end{equation}
		where
		\begin{equation} \label{e:Delta}
			\Delta(u) 
			\coloneqq 
			\E\sbr[1]{\vec{f}(X_1) \otimes \vec{f}(X_{u+1})}, \quad \text{i.e.} \quad 
			\Delta_{k,\ell}(u) 
			= \sum_{q \geq d} q! \thinspace c_q^{(k)} c^{(\ell)}_q \thinspace \rho_{k,\ell}(u)^q 
		\end{equation}
		and~$\DD$ as well as~$\Gamma$ are given by 
		\begin{equation} \label{e:Gamma}
			\DD \coloneqq \Delta(0), \quad \Gamma \coloneqq \sum_{u=1}^\infty \Delta(u) \,.
		\end{equation}
		\item The second component $\B$ is of the form
			\begin{equation} \label{e:rp_limit}
				\B(t) = \int_0^t B(s) \otimes \dif B(s) +  \, t\Gamma, 
			\end{equation}
		where the integration denotes \emph{It\^{o} integration}, cf.~ Remark~\ref{rmk:general}, Point~\ref{rmk:general:5} below.
	\end{enumerate}
\end{theorem}

We record the following corollary which is an immediate consequence of the continuity of the It\^{o}--Lyons map w.r.t. rough paths metric in~$\mathcal{D}^{r-\var}([0,1];\R^m)$, see~\cite[Coro.~4.4]{chevyrev_et_al_homogenisation_part_2}.
Recall that~$\vec{f}(X_i)$ has been introduced in~\eqref{e:vector_f}.
\begin{corollary} \label{coro:main}
	Let~$n, N  \geq 1$ and~$r > 2$. For $b \in \C_b^3(\R^n;\R^n)$ and $V \in \C_b^3(\R^n; \R^{n \times m})$, let~$Y_N$ be defined through the Euler-type recurrence relation 
	\begin{equation*} 
		Y_N\del[1]{\frac{i+1}{N}} 
		= Y_N\del[1]{\frac{i}{N}} + \frac{1}{\sqrt{N}} V\del[1]{Y_N\del[1]{\frac{i}{N}}} \vec{f}(X_i) + \frac{1}{N} b\del[1]{Y_N\del[1]{\frac{i}{N}}}, \quad Y_N(0) = y_N \in \R^n,
	\end{equation*}
	whenever~$i \in \llbracket 0, N-1 \rrbracket$, which is extended to~$t \in [0,1]$ by setting~$Y_N(t) \coloneqq Y_N\del[1]{\frac{i}{N}}$ whenever~$i/N \leq t < (i+1)/N$.
	For~$\Gamma \in \R^{m \times m}$ given in~\eqref{e:Gamma}, define~$c: \R^n \to \R^n$ through
	\begin{equs}
		c(z) & 
		\coloneqq \sum_{k,\ell=1}^m \Gamma_{k,\ell} \sbr[0]{(V_k \cdot \nabla) V_\ell}(z), \quad   %\\
		\sbr[0]{(V_k \cdot \nabla) V_\ell}_j(z)
		\coloneqq \sum_{u=1}^n \frac{\partial V_{j,\ell}(z)}{\partial z_u} V_{u,k}(z), \quad j \in \llbracket 1,n \rrbracket \,.
	\end{equs}
	If~$y_N \to y$ as~$N\to\infty$, then~$Y_N$ converges in law in $r$-variation topology to the unique strong solution of the SDE
	\begin{equation*}
		\dif Y(t) = V(Y(t)) \dif B(t) + \sbr[0]{b(Y(t)) + c(Y(t))} \dif t, \quad Y(0) = y\,,
	\end{equation*}
	where~$B$ is an $\R^m$-valued Brownian motion with covariance matrix~$\Sigma$ (at time~$1$) given in~\eqref{e:Sigma}. 
\end{corollary}

Before we discuss how the previous theorem relates to the existing literature, a few remarks are in order.
\medskip
\begin{remarkb}[General comments] \label{rmk:general}
	\begin{enumerate}[label=(\roman*)]
		\item The previous result also holds for~$p > 1$ and~$d \in \cbr[0]{1,2}$, i.e. under less stringent integrability, but more stringent summability assumptions. See Remark~\ref{rmk:main}, Point~\ref{rmk:main:integrability} and, in particular, Remark~\ref{e:L1_norm_reduction_strategy} for more detailed comments in this direction.
		\item The previous theorem \emph{cannot} be deduced from any mixing-type conditions. 
		In fact, Breuer and Major~\cite{breuer_major_83} have already provided two explicit sequences of random variables which generate the same $\sigma$-algebra, but one of them obeys the CLT while the other satisfies a \emph{non-central limit theorem} in the spirit of Dobrushin and Major~\cite{dobrushin_major_1979,major_1981}, Rosenblatt~\cite{rosenblatt_1981}, and Taqqu~\cite{taqqu_1975,taqqu_1977,taqqu_1979}.  
		\item The assumption that all functions~$f_k$ have the \emph{same} Hermite rank at least~$d$ can easily be dropped: If we denote by~$d_k \geq 1$ the lower bound for the Hermite rank of~$f_k$, then all the arguments in this article work for~$d \coloneqq \min_{k=1,\ldots,m} d_k$. 
		\item \label{rmk:general:5} 
		The fact that~$\boldsymbol{B}$ can only have characteristics~$(\Sigma,\Gamma)$ given in~\eqref{e:Sigma} and~\eqref{e:Gamma}, respectively, is well-understoof by now, cf.~\cite[Thm.~1]{engel_friz_orenshtein_24}; the difficulty lies in the verification of its prerequisities which are a special case of our computations in Section~\ref{s:convergence_fdd}.
		
		Note that, since we can decompose~$\Gamma$ into symmetric and antisymmetric part, 
		\begin{equation} \label{e:AA}
			\Gamma = \Sym(\Gamma) + \AA, \quad \AA \coloneqq \ASym(\Gamma) \,
		\end{equation}
		the formula for~$\Sigma$ in~\eqref{e:Sigma} combined with It\^{o}--Stratonovich correction implies that
			\begin{equation} \label{e:Stratonovich_correction}
				\B(t) = \int_0^t B(s) \otimes \circ \dif B(s) -  \, \frac{t}{2}\DD + t \AA \,, 
			\end{equation}
		where $\circ$ denotes \emph{Stratonovich integration}.
		While the natural expected limit is given by the It\^{o} lift of~$B$,  re\-pre\-sen\-ting $\B$ as in~\eqref{e:Stratonovich_correction} is more aligned with our proof, see the strategy and organisation paragraphs below.	
		\item \label{rmk:general:3} 
		Let us comment on the structural assumption in~\eqref{e:vector_f}. 
		In line with the work by Arcones~\cite{arcones_94}, the most general assumption would be for the process~$X$ in~\eqref{e:vector_X} to be an $\R^n$-valued, centred, and stationary Gaussian process for some~$n \in \N$ such that each function $f_k: \R^n \to \R$ has~$X$ as its input, i.e.
		\begin{equation*}
			f_k(X_i) = f_k(X^{(1)}_i, \ldots, X^{(n)}_i) \,.
		\end{equation*}  
		However, this would lead to a significant increase in notation complexity (for example due to multivariate Hermite expansions) without adding much conceptual value.
		In contrast, our~\enquote{diagonal case} in~\eqref{e:vector_f} gives non-trivial symmetric and \emph{area corrections}~$\DD$ and~$\AA$ without being overly notationally involved;
		we expect our arguments for that case to generalise without much difficulty.
		\item \label{rmk:general:4} 
		Observe that the \emph{univariate case}, i.e. if~$(X_i)_{i \in Z}$ is a one-dimensional stationary Gaussian sequence as in Theorem~\ref{thm:breuer_major}, is a special (degenerate) case in which
		\begin{equation*}
			(X^{(1)}_i,\ldots,X^{(m)}_i) = (X_i,\ldots,X_i) \,.
		\end{equation*} 
		In this scenario, we will simply write~$\rho$ instead of~$\rho_{1,1}$ and note that~$\rho(u) = \rho(-u)$ which, in turn, implies that~$\AA \equiv 0$ and~$\Gamma = \Sym(\Gamma)$. 
	\end{enumerate}
\end{remarkb}

\begin{remarkb}[Differentiability, integrability, and summability assumptions] \label{rmk:main}
	\begin{enumerate}[label=(\roman*)]
		\item \textbf{Differentiability.} \label{rmk:main:malliavin_sobolev} The Malliavin--Sobolev regularity assumption~\ref{thm:main:1} is crucial in two places, namely in the proof of the \emph{tightness estimate} (Theorem~\ref{thm:tightness}) and in the \emph{reduction to finite chaos} (Proposition~\ref{p:reduction}) for the f.d.d. convergence.
		In both cases, the reason for the increased regularity (as compared to the first-level case by Nourdin and Nualart~\cite{nourdin_nualart_20}) is the same: While the $d$-th Hermite shift (see Definition~\ref{d:hermite_shift}) increases regularity by~$d$, by the Leibniz rule, the $2d$ Malliavin derivatives encoded in the product might all hit the same of the two factors, thus leading to~$d$ required derivatives for each of the factors.
		We do not know whether the resulting assumption on the Malliavin differentiability is optimal or an artefact of our proof.
		\item \textbf{Summability.} 
		For illustrational purposes, suppose for a moment that we are in the univerate setting of Remark~\ref{rmk:general}, Point~\ref{rmk:general:4}.
		
		As we have just discussed in Point~\ref{rmk:main:malliavin_sobolev}, our argument requires $d$ Malliavin derivatives to work.
		Note that there is a fundamental tension between this (encoded by assumption~\ref{thm:main:1} in Theorem~\ref{thm:main}) and assumption~\ref{thm:main:2}: 
		While the latter becomes a \emph{stronger} requirement when~$d$ shrinks, the former becomes \emph{weaker}, and vice versa.
		In fact, this observation is the reason why Theorem~\ref{thm:main} is phrased so as to require Hermite rank \enquote{at least}~$d \geq 1$ because the Hermite rank is not straightforwardly connected to the order of Malliavin differentiability.
		A paradigmatic example is the centred absolute value function 
		\begin{equation*}
			f(x) =\abs[0]{x} - \sqrt{\frac{2}{\pi}}
		\end{equation*}	
		which has Hermite rank~$2$ but $f \notin \mathbb{D}^{2,q}(\gamma)$ for any~$q$. 
		However, choosing~$d = 1$, we do know that~$f \in \mathbb{D}^{1,q}(\gamma)$ for all~$q$---it is, therefore, covered by our assumptions provided we have~$\rho \in \ell^1(\Z)$ (instead of just~$\rho \in \ell^2(\Z)$).
		
		In addition, note that imposing that~$f = f_k$ has \enquote{at least} Hermite rank~$d$ is consistent with Theorem~\ref{thm:breuer_major}. 
		Indeed: If $f$ had Hermite rank~$d_\star \geq d$ and we did impose~$\rho \in \ell^d(\Z)$, then~$\rho$ would also be in~$\ell^{d_\star}(\Z)$ and the conclusion of the theorem holds.
		In that (simpler) setting, there is just nothing to be gained from this trade-off. 
		\item \textbf{Integrability.} 
		\label{rmk:main:integrability} %\blue{
		The parameter~$p \geq 2$ in the Malliavin--Sobolev regularity assumption~\ref{thm:main:1} corresponds to~$p_1$ in the work by Nourdin and Nualart~\cite{nourdin_nualart_20}, i.e.~$p_1 = 2p$. 
		We recall that they only deal with the first-order process~$S_N$ in~\eqref{e:1st_2nd_order_p} and require~$p_1 > 2$ to obtain functional convergence. 
		Therefore, the natural assumption for us should be~$p > 1$ (for tightness) and~$p = 1$ (for the f.d.d. convergence).
		
		For \emph{tightness}, assuming~$p > 1$ seems indeed to be enough: We provide detailed comments in Remark~\ref{rmk:tightness_p}.

		For the \emph{f.d.d. convergence} (Theorem~\ref{thm:conv_fdd}), in contrast, $p = 2$ is crucial. 
		This is deeply connected with the breakdown of Meyer's inequality for~$p = 1$ and the fact that the Poincaré inequality can only compensate for it when~$d \in \cbr[0]{1,2}$ in~\ref{thm:main:1}; this has already been observed by Addona, Muratori, and Rossi~\cite{addona_muratori_rossi_22}. 
		We provide detailed comments in Remarks~\ref{rmk:L1_vs_L2} and~\ref{rmk:failure_tightness_strategy} below.  
		As a consequence, we need to control the $L^2(\P)$-norm (and not just that in~$L^1(\P)$) of~$\S_N$ in Proposition~\ref{p:reduction}.   
		Let us also mention that this is structurally similar to a problem in quantitative Breuer--Major theorems: While Nourdin, Nualart, and Peccati~\cite{nourdin_nualart_peccati_21} (whose strategy informs parts of our argument) require~$f_k \in \DD^{1,4}(\gamma)$ and~$d = 2$, Angst, Dalmao, and Poly~\cite{angst_dalmao_poly_24} have recently shown that~$f_k \in \DD^{1,2}(\gamma)$ is enough by means of the so-called sharp operator.
		It would be very interesting to see if those techniques would also apply in our context.  
		\item \textbf{Projection.} After projection to finite chaoses of order less than or equal to~$M$ in Section~\ref{s:reduction}, the functions~$f_k^M$ are Malliavin smooth and have moments of all orders; accordingly, Sections~\ref{s:LLN_diagonal} and~\ref{s:moment_computations} only require the assumption~$\rho_{k,\ell} \in \ell^d(\Z)$ via the summability $\Delta^M \in \ell^1(\Z)$ (for~$\Delta$ as in~\eqref{e:Delta} and~$\Delta^M$ its counterpart with all functions~$f_k$ projected onto chaos components~$M$ and lower).  
	\end{enumerate}
\end{remarkb}

\paragraph{Discussion and relation to the literature.}
Theorem~\ref{thm:main} fits squarely into an ever growing body of literature on \emph{rough invariance principles}, comprising the rough Donsker's theorem~\cite{breuillard_friz_huesmann_09}, magnetic Brownian motion~\cite{friz_gassiat_lyons_15}, slow-fast systems~\cite{kelly_melbourne_16, kelly_melbourne_17, chevyrev_et_al_19, chevyrev_et_al_homogenisation_part_2}, random walks in random environment~\cite{deuschel_orenshtein_perkowski_21, orenshtein_21,lopusanschi_orenshtein_21}, as well as averaging~\cite{friz_kifer_24} and rough (fractional) homogenisation~\cite{Gehringer_Li_JTP, gehringer_li_20, Gehringer_Li_Sieber_2022, hairer_li_20, hairer_li_23, djurdjevac_kremp_perkowski_25}; we refer to the recent article~\cite{engel_friz_orenshtein_24} for a systematic and more complete overview based on the Green--Kubo formula.

Fundamentally, the difficulty in establishing Theorem~\ref{thm:main} is two-fold:
\begin{enumerate}[label=(\arabic*)]
	\item Each of the functions~$f_k$ in~\eqref{e:hermite_exp_fk} is allowed to have an \emph{infinite chaos decomposition}, which rules out arguments based on Gaussian hypercontractivity.
	This is particularly relevant since computing~$p$-th moments of the process~$\mathbb{S}_N$ for~$p \geq 2$ is similar to computing $(2p)$-th moments of the functions~$f_k$ but it is known that chaos tails do not generally decay in any other space than~$L^2$.
	\item The sequence~$(X_i)_{i \in \Z}$ has \emph{non-trivial correlations}, which prohibits the use of martingale techniques to show convergence in finite-dimensional distributions to the limiting rough path.  
\end{enumerate}
Although viewed from a homogenisation perspective slightly different from ours, similar problems have appeared in the work of Gehringer and Li~\cite{Gehringer_Li_JTP, gehringer_li_20} and their joint work with Sieber~\cite{Gehringer_Li_Sieber_2022}, so we will now discuss in which way our work differs from theirs.
First, let us mention that they deal with the \emph{mixed case} where some of the components~$f_k$ enjoy the CLT scaling, while others obey a non-central CLT scaling that leads to Hermite processes in the limit; crucially, the latter are not moment-determined in general, which complicates the convergence analysis. 
However, while the mixed case is very interesting from a modelling perspective, the iterated integrals become \emph{Young integrals} if at least one of the involved components is of Hermite type.
From now onwards, we will therefore focus on the \emph{purely Gaussian} part of their work (with truly rough limits), specifically w.r.t. the two issues raised above and phrased in the terminology of the article at hand.

\begin{enumerate}[label=(\arabic*)]
	\item In case~$f_k$ has components in infinitely many chaoses, \cite{Gehringer_Li_JTP, gehringer_li_20, Gehringer_Li_Sieber_2022} impose a \emph{fast chaos decay assumption} akin to that of Chambers and Slud mentioned above.
	In essence, this assumption guarantees that one can obtain the required $L^{2p}$-estimates for $p > 2$, which are necessary for tightness, by Gaussian hypercontractivity. 
	However, that assumption is \emph{systematically stronger} than our assumption that $f_k \in \DD^{d,2p}(\gamma)$.
	In particular, as pointed out by Nourdin and Nualart \cite{nourdin_nualart_20}, it does not cover the benchmark case~$f_k(x) = \abs{x} - \sqrt{\nicefrac{2}{\pi}}$ while ours does, albeit at the price of the stronger assumption that~$\rho \in \ell^1(\Z)$, see Remark~\ref{rmk:main} for details. 
	\item \label{intro:conditional_decay}In order to establish the convergence in finite-dimensional distributions, \cite{Gehringer_Li_JTP, gehringer_li_20, Gehringer_Li_Sieber_2022} impose a \emph{conditional decay condition} on the functions~$f_k$ w.r.t. the Gaussian sequence~$X = (X_i)_{i \in \Z}$.
	Given this assumption, they can again resort to martingale techniques to show the convergence of the finite-dimensional distributions.
	However, we will see in Appendix~\ref{a:counterexample} that, already in the univariate setting of Remark~\ref{rmk:general}~\ref{rmk:general:4},  there is a rich family of \emph{long-range dependent} stationary Gaussian sequences~$X$ (i.e. $\rho \notin \ell^1(\Z)$) such that the function~$f_k = H_d$ for some~$d \geq 2$ violates the conditional decay assumption w.r.t.~$X$.  
	In contrast, we will show that~$\rho \in \ell^d(\Z)$ such that the counterexample is, indeed, covered by our main result.
\end{enumerate}
On the one hand, our tightness argument uses more flexible Malliavin calculus tools (rather than moment-computations combined with the fast chaos decay condition). 
On the other hand, in our case, the limit \emph{is} moment-determined and we develop a novel reduction argument for the f.d.d. convergence, based on integration-by-parts, that compensates the lack of martingale methods.
We will provide further details in the strategy paragraph below.

Before doing so, we want to highlight that the problem of a nonlinearity with components in infinitely many chaoses is also an active research direction in the study of (singular) stochastic partial differential equations, see~\cite{Gubinelli_Perkowski_16,furlan_gubinelli_19, Hairer_Xu_19} in the Gaussian and \cite{kong_wang_xu_24} in the Poisson case.
While these works deal with so-called subcritical equations and in the weak coupling regime, the recent work of Cannizzaro, the second named author, and Moulard~\cite{cannizzaro_klose_moulard_26} has, for the first time, tackled this problem in case of the \emph{critical} Stochastic Burgers Equation with general nonlinearity, both on the full space and at strong coupling.

Finally, let us mention that the techniques developed in the above-mentioned works have recently been employed by Kong and Wang~\cite{kong_wang_25} to prove the functional Breuer--Major theorem in the Poisson case, in which Meyer's inequality is not available. 

\paragraph{Strategy of proof, key novelties, and insights.}
In order to establish Theorem~\ref{thm:main}, we need to show tightness and convergence of the finite-dimensional distributions. 
\begin{enumerate}[label=(\arabic*)]
	\item \emph{Tightness}: As for Nourdin and Nualart, the key tools for establishing tightness will come from Malliavin calculus, namely the regularisation properties of the Hermite shift operator alongside Meyer's inequalities for iterated divergences.
	In contrast to their work, however, the second-order process~$\mathbb{S}^{k,\ell}_N$ involves the \emph{product} of~$f_k(X_i^{(k)})$ and~$f_\ell(X_j^{(j)})$.
	Representing both as iterated Malliavin divergences with input the respective Hermite shifts, the key novelty is our use of \emph{multiplication formula for iterated divergences} due to Furlan and Gubinelli~\cite{furlan_gubinelli_19}, see Proposition~\ref{p:divergence_multiplication} below.
	As we will see, the multiplication formula includes up to $2d$ Malliavin derivatives of the $d$-th Hermite shift of~$f_k$ and~$f_\ell$, i.e. $d$ derivatives on~$f_k$ and~$f_\ell$ itself, cf. Remark~\ref{rmk:main}, Point~\ref{rmk:main:malliavin_sobolev}.   
	\item \emph{Convergence of finite-dimensional distributions}: 
	Modern proofs of the Breuer--Major theorem, for example in~\cite[Chap.~7]{nourdin_peccati_book}, make use of the Malliavin--Stein method; in turn, however, this requires a good enough understanding of the limiting distribution. 
	In contrast to the Gaussian case (related to the first rough path level~$B$), this is not available for the second-order process~$\mathbb{B}$.  
	In the absence of martingale arguments, the saving grace is that~$(B(s,t),\mathbb{B}(s,t)) \in \R^{\otimes m} \oplus (\R^{m})^{\otimes 2}$ is in the second (inhomogeneous) Wiener chaos and, thereby, its law is moment-determined.
	In computations structurally similar to those of Hairer~\cite{hairer_variance_blowup}, we therefore show convergence of all the moments. 
	The argument is quite intricate, so we will present a detailed outline of the strategy in Section~\ref{s:strategy} below and just summarise the key difficulties and insights here:
	\begin{itemize}[leftmargin=*, label=$\star$]
		\item When computing expectations of products of terms~$f_{\w_j}(X^{(\w_j)}_{i_j})$, $\w_j \in \llbracket 1,m \rrbracket$ for $j \in \llbracket 1,l \rrbracket$ and~$l \geq 2$, the Breuer--Major argument (see Proposition~\ref{p:bm83_irregular} below for a recap) still applies in the multivariate setting and if there are constraints between the indices~$i_1, \ldots, i_{l}$: It states that, \emph{asymptotically}, Wick's theorem applies, leading to a major computational simplification.
		The reason is that the argument by Breuer and Major does not use the sign of the correlations~$\rho_{k,\ell}$.
		\item Chaos tails do not decay in~$L^p$ for~$p > 2$. 
		As a consequence, looking at~$\mathbb{S}^{k,\ell}_N$ as defined in~\eqref{e:1st_2nd_order_p} but with at least one of the functions~$f_k$ and~$f_\ell$ projected onto chaoses~$M$ and higher, it is not at all clear how to show that its~$L^2$-norm (i.e., morally, the~$L^4$-norm of~$f_k$ and~$f_\ell$!) goes to zero as~$M$ tends to infinity.
		
		Instead of hypercontractivity, we will use an iterated Gaussian integration-by-parts procedure, the computational complexity of which we tame by a graphical notation similar to Feynman diagrams (in the sense of pairwise matchings); 
		this allows to split the expression for~$\norm[0]{\mathbb{S}_N^{k,\ell}}_{L^2}$ into a \emph{deterministic} and (an expectation of) a \emph{probabilistic} part.
		Roughly speaking, in most cases, the deterministic part will vanish as~$N \to \infty$ while the probabilistic part is independent of~$N$ and uniformly bounded in the multi-index for each fixed~$M$, thanks to the differentiability assumption~\ref{thm:main:1} in Theorem~\ref{thm:main}. 
		In the few other cases, the deterministic part is uniformly bounded in~$N$ while the probabilistic part vanishes as~$M \to \infty$.  
		We refer to the paragraph on p.~\pageref{strategy_reduction} for details.
		\item It is precisely the \enquote{diagonal terms}, corresponding to~$i=j$ in~\eqref{e:1st_2nd_order_p}, that asymp\-to\-ti\-cal\-ly give the~$\DD$ in~\eqref{e:Stratonovich_correction}, i.e. the symmetric part of correction to the \emph{Stratonovich Brownian rough path}. We will see that this is a relatively straightforward consequence of the original Breuer--Major argument.
		\item Adding the diagonals (times a factor~$(1/2)$) to~$\mathbb{S}_N$ in~\eqref{e:1st_2nd_order_p}, we can temporarily convert from the piecewise constant to the piecewise linear interpolation. 
		In this way, we can use iterated \emph{integral}-signatures and bypass the need to work with ite\-ra\-ted sum-signatures as developed and investigated by~\cite{kiraly_oberhauser_2019, Diehl_EF_Tapia_20, Diehl_EF_Tapia_20_proceedings, Diehl_EF_Tapia_23} in the context of machine learning and time series analysis (see Remark~\ref{rmk:iterated_sum_signature} for further comments in this direction).
		As indicated in~\eqref{e:Stratonovich_correction}, our moment computations then need to reproduce those of the Stratonovich rough path, shifted by the antisymmetric matrix~$\AA$, in the limit. 
		\item In the detailed convergence analysis for the moments of~$\S_N$ with the dia\-gonals added in, we will identify the \emph{exact} mechanism that gives rise to the $\AA$-translated Stratonovich signature of Brownian motion. 
		More precisely: (1) Odd tensor components vanish because, in the language of Breuer and Major, there are no regular diagrams with an odd number of levels. 
		(2) In the \enquote{asymptotic Wick theorem} implied by their argument, at tensor level~$2n$, only the ladder pairing~$P_\star = \cbr[0]{\cbr[0]{1,2}, \cbr[0]{3,4}, \ldots, \cbr[0]{2n-1,2n}}$ survives in the limit.
		(3)~Not all multi-indices~$0 \leq i_1 \leq \ldots \leq i_{2n} \leq N-1$ contribute to that limit. In fact, \enquote{diagonals of order~$3$ and higher} (i.e. when three or more consecutive indices coincide) vanish asymptotically.
		(4) Further, equality between two consecutive indices is only asymptotically relevant if it is in accordance with the pairing~$P_\star$, i.e.~$0 \leq i_1 \leq i_2 < i_3 \leq i_4 < \ldots < i_{2n-1} \leq i_{2n}$.    
		(5) Sums over equal indices give rise to~$\frac{1}{2}\Delta(0) = \frac{1}{2}\DD$ and, otherwise, to~$\Gamma = \sum_{k \in \N} \Delta(k)$. 
		This combines to~$\frac{1}{2} \DD + \Gamma = \frac{1}{2}(\Sigma + 2 \AA)$, as implied by a suitable amendment of Fawcett's theorem.   
	\end{itemize}
\end{enumerate}

\paragraph{Organisation of the article.}
The remainder of the article is organised as follows. 
In Section~\ref{s:preliminaries}, we introduce some notation and recall basic concepts of Gaussian analysis and Malliavin calculus. 
In particular, we recall the original argument by Breuer and Major and precisely state the form in which we use it throughout the article~(Corollary~\ref{coro:BM83}).
In Section~\ref{s:tightness}, we briefly recall background material on rough path spaces equipped with $r$-variation topologies (Section~\ref{s:rough_paths_primer}) and, in~Section~\ref{s:tightness_result}, then establish tightness in such spaces~(Corollary~\ref{coro:tighness}); the key $L^p$-estimates are presented in Theorem~\ref{thm:tightness} and proved in Section~\ref{subsec:proof_tightness_thm}.
Section~\ref{s:convergence_fdd} establishes the convergence of the finite-dimensional distributions (Theorem~\ref{thm:conv_fdd}).
We begin with a high-level overview of the strategy in Section~\ref{s:strategy} which leads to the tailor-made Slutsky-type statement~(Proposition~\ref{prop:abstract_convergence_statement}) that the ensuing analysis is based upon.
We then implement the outlined agenda by verifying the assumptions for the Slutsky-type result in the sections that follow.
Section~\ref{s:reduction} contains the reduction step which shows that chaos tails above a fixed order vanish asymptotically in a suitable double limit.
As a consequence, we can always assume that we are working in a fixed (inhomogeneous) chaos. 
In Section~\ref{s:LLN_diagonal}, we then show that, in the spirit of the law of large numbers, the diagonals of order two give rise to the term~$\DD$ in~\eqref{e:Stratonovich_correction}.
Section~\ref{s:moment_computations} contains a detailed combinatorial analysis, as outlined in the strategy paragraph, culminating in the proof of Theorem~\ref{thm:conv_fdd_stratonovich} (via that of Theorem~\ref{t:product_case}), which shows that the moments of the discrete approximation converge to those of the Stratonovich rough path corrected by~$\AA$.
In Section~\ref{s:removing_chaos_cutoff}, we then remove the chaos cut-off introduced in Section~\ref{s:reduction} from the limiting rough path, which only enters via the covariance matrix and the corrections.
Finally, in Section~\ref{s:proof_fdd_convergence}, we bring all the previous effort to fruition and present the proof of Theorem~\ref{thm:conv_fdd} which establishes the f.d.d. convergence.

\paragraph{Acknowledgements.}

We are grateful to Ilya Chevyrev, Emilio Ferrucci, Fanhao Kong, Ivan Nourdin,  Nikolas Tapia, and Weijun Xu for insightful discussions that have benefited this work.

Work on this project started as HEA was being employed as Dirichlet Postdoctoral Fellow in Mathematics at Freie Universit\"{a}t Berlin, with funding provided by MATH+, in the framework of the ``MATH+ EXC 2046'' research project. 
TK is supported by a UKRI Horizon Europe Guarantee MSCA Postdoctoral Fellowship (UKRI, SPDEQFT, grant reference EP/Y028090/1). Views and opinions expressed are however those of the authors only and do not necessarily reflect those of UKRI. In particular, UKRI cannot be held responsible for them. 
NP gratefully acknowledges funding by the Deutsche Forschungsgemeinschaft (DFG, German Research Foundation) -- CRC/TRR 388 ``Rough Analysis, Stochastic Dynamics and Related Fields'' -- Project ID 516748464

\section{Preliminaries and auxiliary results} \label{s:preliminaries}

In this section, we collect preliminaries and auxiliary results that will be used throughout the article.

\subsection{Notation} \label{s:notation}

\begin{itemize}
	\item For two expressions~$A$ and~$B$, we write~$A \lesssim B$ if there exists a constant~$c > 0$ such that~$A \leq c B$.
	If the constant~$c$ depends on a parameter~$\alpha$, we will write~$A \lesssim_\alpha B$.
	\item We write~$\N$ for the natural numbers starting from~$1$ and~$\N_0 \coloneqq \N \cup \cbr[0]{0}$.
	For~$a,b \in \R$ such that~$a < b$, we write~$\llbracket a,b \rrbracket \coloneqq [a,b] \cap \Z$.
	\item We write~$A_N \asymp B_N$ if~$A_N = B_N + o(1)$ as~$N \to \infty$, i.e.~$\lim_{N \to \infty} A_N = \lim_{N \to \infty} B_N$.
	Unless otherwise stated, we will exclusively use this notation for the limit~$N \to \infty$ which is particularly relevant for quantities depending, additionally, on another parameters besides~$N$.
	\item For~$k \in \N$, we use the multi-index notation~$i_{1:k} \coloneqq (i_1,\ldots,i_k)$ and~$i_{[1:k]} \coloneqq \sum_{j=1}^k i_j$.
	Furthermore, for~$a,b \in \R$ such that~$a < b$, we write~$a \leq i_{1:k} \leq b$ to mean~$a \leq i_\ell \leq b$ for all~$\ell \in \llbracket 1,k \rrbracket$. 
	\item For~$l \in \N$, we let~$\square_{l,N} \coloneqq \llbracket 0,N-1 \rrbracket^{l}$ denote the $l$-dim. box of side length~$N-1$. 
	\item For~$l \in \N$ and~$a,b \in \Z$ with~$a < b$, we let
	\begin{equation*}
		\triangle^{(l)}_{a,b} \coloneqq  \cbr[1]{ r_{1:l} \in [a,b]^{l}: \ a \leq r_1 < r_2 < \ldots < r_l \leq b-1}
	\end{equation*}
	denote the $l$-dim. simplex over the interval~$[a,b]$ as well as its closure
	\begin{equation*}
		\bar{\triangle}_{a,b}^{(l)} 
		=  
		\cbr[1]{ r_{1:l} \in [a,b]^{l}: \ a \leq r_1 \leq r_2 \leq \ldots \leq r_l \leq b-1} 
	\end{equation*}
	We will not notationally distinguish the \emph{discrete simplex}, i.e. the case when~$a,b, r_{1:\ell}$ in the previous definition are constrained to be integers.
\end{itemize}

\subsection{Malliavin calculus}

In this section, we introduce the key concepts and results from Malliavin calculus which we will use throughout this article. 
For a detailed introduction to these topics, we refer the reader to the books by Nualart~\cite{nualart} and Nourdin and Peccati~\cite{nourdin_peccati_book}.

We need the following auxiliary definition:
\begin{definition}[Isonormal Gaussian process] \label{d:isonormal_gp}
	Let~$H$ be a real separable Hilbert space. A stochastic process~$W = \cbr[0]{W(h): \ h \in H}$, defined on a complete probability space~$(\Omega, \CF,\P)$, is an \emph{isonormal Gaussian process} over~$H$ if~$W$ is a centred Gaussian family of random variables such that~$\E\sbr[0]{W(h)W(g)} = \scal{h,g}$ for all~$g,h \in H$. 
\end{definition}

Recall that~$\cbr[0]{X_i}_{i \in \Z} \subseteq L^2(\Omega)$ is a stationary sequence of centred, $\R^m$-valued Gaussian random variables with correlation function~$\rho_{k,\ell}(u) = \E\sbr[0]{X_0^{(k)} X^{(\ell)}_u}$ and~$\rho_{k,k}(0) = 1$ for all~$k, \ell \in \llbracket 1,m \rrbracket$.
In particular, for any~$i \in \Z$ and~$k \in \llbracket 1,m \rrbracket$, we have~$X_i^{(k)} \sim \gamma$ for~$\gamma = \CN(0,1)$.
The following statement is a straightforward amendment of~\cite[Prop.~7.2.3]{nourdin_peccati_book}.
\begin{lemma} \label{l:Gaussian_repr}
	There exists a real separable Hilbert space~$H$, an isonormal Gaussian process~$\cbr[0]{W(h): h \in H}$ over~$H$, and a set~$E = \cbr[0]{e_{k,i}: k \in \llbracket 1,m \rrbracket, \ i \in \Z} \subseteq H$ such that
	\begin{enumerate}[label=(\roman*)]
		\item $E$ generates~$H$,
		\item $\rho_{k,\ell}(j-i) = \E\sbr[0]{X^{(k)}_i X^{(\ell)}_j} = \scal{e_{k,i},e_{\ell,j}}_{H} $ for any~$k,l \in \llbracket 1,m \rrbracket$ and~$i,j \in \Z$, and
		\item $X^{(k)}_i = W(e_{k,i})$ for every~$k \in \llbracket 1,m \rrbracket$ and~$i \in \Z$. 
	\end{enumerate}
\end{lemma}

\begin{definition}[Hermite polynomials and Wiener chaos] \label{d:wiener_chaos}
	We let~$H_0(x) \coloneqq 1$ and, for~$q \in \N$, introduce the \emph{$q$-th Hermite polynomial}~$H_q$ by 
	\begin{equation} \label{e:hermite_poly}
		H_q(x) \coloneqq (-1)^q e^{\frac{x^2}{2}}\frac{\dif^{\, q}}{\dif x^q}e^{-\frac{x^2}{2}},\quad 
	\end{equation}
	For~$q \in \N_0$, the \emph{$q$-th Wiener chaos}~$\CW_q$ is defined as 
	\begin{equation*}
		\CW_q \coloneqq \operatorname{cl}_{L^2(\Omega)} \, \operatorname{span}\cbr[1]{H_q(W(h)):\, h\in H,\,\norm[0]{h}_{H}=1} \,.
	\end{equation*}
	where~$\operatorname{cl}_{L^2(\Omega)}$ denotes the closure w.r.t.~$L^2(\Omega)$.
\end{definition}

For the next definition, we denote by~$\symotimes$ the \emph{symmetric} tensor product on the Hilbert space~$H$, that is: 
\begin{equation*}
	h_1 \symotimes \ldots \symotimes h_q 
	\coloneqq 
	\frac{1}{q!}
	\sum_{\pi \in \mathfrak{S}(q)} \bigotimes_{\ell=1}^q h_{\pi(\ell)} \,.
\end{equation*}	
\begin{definition}[Wiener--It\^{o} isometry] \label{d:wiener_ito_isometry}
	We let~$I_0 \coloneqq \operatorname{id}_{\R}$ and, for $q\geq 1$, define 
	\begin{equation} \label{e:wiener_ito_isometry}
		I_q: H^{\symotimes q} \to \CW_q, 
		\quad 
		I_q(h^{\otimes q}) \coloneqq H_q(W(h)) \quad \text{for} \quad h \in H, \quad \norm[0]{h}_{H}=1
	\end{equation} 
	which is extended to a linear isometry (w.r.t. the scaled norm $\sqrt{q!} \norm[0]{\cdot}_{H^{\symotimes q}}$).
\end{definition}

It can be shown that any $G\in L^2(\Omega)$ admits the \emph{Wiener chaos expansion}
	\begin{equation}\label{e:wiener_chaos_expansion}
		G = \E[G] +\sum_{q=1}^\infty I_q(g_q) \,,
	\end{equation}
where the symmetric kernels $g_q \in H^{\symotimes q}$ are uniquely determined by $G$.

\begin{definition}[Malliavin derivative] \label{d:malliavin_derivative}
	For smooth cylindrical random variables of the form $G= g(W(h_1), \dots , W(h_n))$, where $h_i \in  H$ and $g \in C_b^{\infty}(\R^n)$, the \emph{Malliavin derivative} $D$ is defined as the $H$-valued random variable
	\begin{equation*}
		DG \coloneqq \sum_{i=1}^n \frac{\partial g}{\partial x_i} (W(h_1), \dots, W(h_n))h_i, \quad \text{i.e.} \quad DG \in L^2(\Omega;H) \,.		
	\end{equation*}
	By iteration, higher-order derivatives $D^k G$ are defined as elements of $L^2(\Omega; H^{\otimes k})$ and we will also write $D^0 \coloneqq \operatorname{id}$. 
\end{definition}

\begin{definition}[Malliavin--Sobolev spaces] \label{d:malliavin_derivative}
	For any $k \in \N_0$ and $p \geq 1$, the \emph{Malliavin--Sobolev space} $\DD^{k,p}$ is obtained as the completion of smooth cylindrical random variables under the norm
	\begin{equation*}
		\norm{G}^p_{k,p} \coloneqq \E\sbr[0]{\abs[0]{G}^p} + \sum_{l=1}^k \E\sbr[1]{\thinspace \norm[0]{D^l G}_{H^{\symotimes l}}^p} \,.	
	\end{equation*}
\end{definition}

Next, we introduce the \emph{divergence operator} $\delta$ as the adjoint of the Malliavin derivative~$D$.
\begin{definition}[(Iterated) Malliavin divergence] \label{d:malliavin_divergence}
We define
	\begin{equation*}
		\Dom \delta \coloneqq
		\cbr[1]{
			u \in L^2(\Omega; H): \exists c_u \in \R, \ \forall G \in \DD^{1,2}: \,
			\abs[0]{\E\sbr[0]{\scal{DG, u}_H}} \leq c_u \norm{G}_{L^2(\Omega)}
		} \,.
	\end{equation*}
	For $u \in \Dom \delta$, the random variable $\delta(u)$ is defined via the duality relation
	\begin{equation*} 
		\forall G \in \DD^{1,2}: \ \E\sbr[0]{G \delta(u)} = \E\sbr[0]{\scal{DG, u}_{H}} \,.
	\end{equation*}
	It is called the \emph{(Malliavin) divergence} operator.
	For~$k \geq 2$, we can analogously define the \emph{iterated divergence operators}~$\delta^k$ via the duality statement 
	\begin{equation*} 
		\forall G \in \DD^{k,2}: \ \E\sbr[0]{G \delta^k(u)} = \E\sbr[0]{\scal{D^k G, u}_{H^{\symotimes k}}} \,,
	\end{equation*}
	where $u$ belongs to $\Dom \delta^k \subseteq L^2(\Omega; H^{\symotimes k})$.
\end{definition}

\begin{remark}
	When $u\in H^{\symotimes k}$ is deterministic, we have~$\delta^k(u)= I_k(u)$.
	In particular, this applies to~$u = h^{\otimes k}$ for~$h \in H$.
\end{remark}

\begin{remark}
	The Malliavin derivative~$D$ and its adjoint~$\delta$ have obvious analogues for random variables that take values in another separable Hilbert space~$V$, see~\cite[Sec.~2.4 and 2.6]{nourdin_peccati_book}.
	If we wish to highlight this fact, we will for example write $\DD^{k,p}(V)$ etc.
	The corresponding results in this subsection then apply mutatis mutandis.
\end{remark}

Essentially, the following proposition is the content of~\cite[Lem.~B.7]{furlan_gubinelli_19}, see also Remark~\ref{rmk:divergence_multiplication} below. 
\begin{proposition}[Divergence multiplication formula] \label{p:divergence_multiplication}
	For~$n_1,n_2 \in \N$,~$i,j \in \Z$, $k,\ell \in \llbracket 1,m \rrbracket$, and $p \geq 2$, let 
	\begin{equation*}
		u = f(W(e_{k,i})) e_{k,i}^{\otimes n_1}, \quad v = g(W(e_{\ell,j})) e_{\ell,j}^{\otimes n_2}, \quad f, g \in \DD^{n_1+n_2,2p}(\gamma) \,. 
	\end{equation*}
	Then, the following \emph{multiplication formula}
	\begin{align*}
		\thinspace &
		\delta^{n_1}(u) \delta^{n_2}(v) \\[0.5em]
		%= 
		= \ & 
		\sum_{(q,r,l) \in \CI_{n_1,n_2}} C_{n_1,n_2,q,r,l} \delta^{n_1+n_2-q-r}\del[1]{f^{(r-l)}(X^{(k)}_i) g^{(q-l)}(X^{(\ell)}_j) e_{k,i}^{\symotimes (n_1-q)} \symotimes e_{\ell, j}^{\symotimes (n_2-r)}} \times \\
		%%%
		& \qquad \qquad \times 
		\rho_{k,\ell}(j-i)^{q+r-l} 
	\end{align*}
	holds with
	\begin{align*}
		\CI_{n_1,n_2} 
		& \coloneqq \cbr[1]{(q,r,l) \in \N_0^3: \ q \in \llbracket 0,n_1 \rrbracket, \ r \in \llbracket 0,n_2 \rrbracket, \ l \in \llbracket 0,q \wedge r \rrbracket}, \\[0.5em]
		%%%
		C_{n_1,n_2,q,r,l} & \coloneqq 
		\binom{n_1}{q} \binom{n_2}{r} \binom{q}{l} \binom{r}{l} l! \,.
	\end{align*}
	Additionally, the product~$\delta^{n_1}(u)\delta^{n_2}(v)$ is in~$L^p(\Omega)$.
\end{proposition}

\begin{remark} \label{rmk:divergence_multiplication}
	In fact, we stated the previous proposition in the special case detailed in~\cite[Rem.~B.8]{furlan_gubinelli_19} and, additionally, made use of the statements in Lemma~\ref{l:Gaussian_repr}.
	
	Besides, we have added the assumption that~$f, g \in \DD^{n_1+n_2,2p}(\gamma)$ (rather than $C^{n_1+n_2}(\R;\R)$) which implies the integrability~$\delta^{n_1}(u) \delta^{n_2}(v) \in L^p(\Omega)$; one can easily check that the proof provided by Furlan and Gubinelli still applies.
\end{remark}

For the following definition, see~\cite[Sec.~2.8.1 and~2.8.2]{nourdin_peccati_book}.

\begin{definition}[OU semigroup and generator] \label{d:OU_operator}
	For~$G \in L^2(\Omega)$ with Wiener chaos decomposition given in~\eqref	{e:wiener_chaos_expansion}, we define the \emph{Ornstein--Uhlenbeck (OU) semigroup}~$(P_t)_{t \geq 0}$ and its infinitesimal \emph{generator}~$L$ as follows:
	\begin{equation*}
		P_t G
		=
		\sum_{q=0}^\infty e^{-qt} I_q(g_q), 
		\quad
		(-L)^r G
		=
		\sum_{q=1}^\infty q^r I_q(g_q), \quad t \geq 0, \ r \in \R \,.
	\end{equation*}
	Sometimes, we will also refer to~$(-L)$ as the~\emph{number operator}.
	For~$r = -1$, the operator~$L^{-1}$ is called the \emph{pseudo-inverse} of~$L$ and satisfies~$LL^{-1} G = G - \E\sbr[0]{G}$.
\end{definition}

\begin{remark} \label{rmk:definitions_1D}
	For the isonormal Gaussian process~$W(h) = h$ over~$H = \R$, defined on the complete probability space~$(\R,\CB(\R),\gamma)$ with~$\gamma = \CN(0,1)$, we write~$\DD^{k,p}(\gamma)$ for the Malliavin--Sobolev space introduced in Definition~\ref{d:malliavin_derivative} above.
	In that case, we have~$Dg = g'$, $(\delta g)(x) = xg(x) - g'(x)$, and~$(Lg)(x) = g''(x) - xg'(x)$, see~\cite[Chap.~1]{nourdin_peccati_book} for details.
\end{remark}

The following proposition presents \emph{Meyer's inequality} (see \cite[Thm.~1.5.1]{nualart}). 
\begin{proposition}[Meyer's inequality] \label{p:meyer_inequality}
	 For $p>1$ and~$k \in \N$, we have 
		\begin{equation} \label{e:meyer}
		  \norm[0]{D^k G}_{L^p(\Omega, H^{\symotimes k})} 
		  \lesssim_{k,p} 
		  \norm[0]{(- L)^{k/2} G}_{L^p(\Omega)} 
		  \lesssim_{k,p} 
		  \del[2]{\norm[0]{D^k G}_{L^p(\Omega, H^{\symotimes k})}
			+\norm[0]{G}_{L^p(\Omega)}}
	\end{equation}
	uniformly over for all $G \in \DD^{k,p}$.
\end{proposition}

We record the following corollary, cf.~\cite[Prop.~1.5.4]{nualart} in case~$k=1$; the case of general~$k \geq 1$ is well-known.
\begin{corollary}[Continuity of iterated divergences] \label{coro:continuity_div}
	For any~$q > 1$ and~$k \in \N$, the operator $\delta^k$ maps $\DD^{k,p} (H^{\symotimes k})$ continuously into $L^p(\Omega)$, that is:
		\begin{equation} \label{e:continuity_div}
			\norm[0]{\delta ^k(v)}_{L^p(\Omega)} 
			\lesssim_p
			\sum_{j=0}^k  \norm[0]{D^jv}_{L^p(\Omega; H^{\symotimes j})} \,.
		\end{equation}
\end{corollary}

The following statement is the content of~\cite[Lem.~B.13]{furlan_gubinelli_19}, see also~\cite[Identity~(7c)]{nualart_zakai_89}.
\begin{lemma}\label{lem:commutation_D_delta}
	Let~$j,k \in \N$ and assume that~$u \in \mathbb{D}^{k+j,2}(H^{\otimes j})$ symmetric and such that all of its derivatives are symmetric. 
	Then, the following identity holds:
	\begin{equation*}
		D^k \delta^j(u) = \sum_{\ell = 0}^{k \wedge j} \binom{k}{\ell} \binom{j}{\ell} \ell! \, \delta^{j-\ell} D^{k-\ell} u \,.
	\end{equation*}
	In the previous formula, the iterated divergence~$\delta^{j-\ell}$ acts on the original variables of~$u$, not on those created by the derivatives~$D^{k-\ell}$.
\end{lemma}

\subsection{The Hermite shift operator} \label{s:hermite_shift}

Let~$\gamma \coloneqq \CN(0,1)$ and recall that the Hermite polynomials~$(H_k)_{k \in \N_0}$ form a complete orthonormal basis of~$L^2(\gamma) \coloneqq L^2(\R,\gamma)$.

\begin{definition}[Hermite rank] \label{d:hermite_rank}
	Let~$g \in L^2(\gamma)$ with Hermite decomposition
	\begin{equation} \label{e:hermite_decomposition}
		g = \sum_{q = d}^\infty c_q H_q
	\end{equation}
	The parameter
	\begin{equation*}
		d \coloneqq \inf\cbr[0]{q \in \N_0: \ c_q \neq 0} \in \N 
	\end{equation*} 
	is called the \emph{Hermite rank}~$d \in \N$ of~$g$, that is~$c_0 = \ldots = c_{d-1} = 0$.
\end{definition}

\begin{definition}[Hermite shift] \label{d:hermite_shift}
	Let~$g \in L^2(\gamma)$ with Hermite decomposition~\eqref{e:hermite_decomposition}. 
	For any~$n \geq 1$, we introduce the \emph{$n$-th Hermite shift}~$\CS_n$ as follows:
	\begin{equation*}
		\CS_n: L^2(\gamma) \to L^2(\gamma), \quad 
		\CS_n g 
		\coloneqq 
		\sum_{q \geq n} c_q H_{q-n} \,.
	\end{equation*}
\end{definition}

The following lemma provides a representation of the Hermite shift in terms of the Mallavin derivative~$D$ and the number opeator~$(-L)$.
Its proof can be found in~\cite[Lem.~2.1]{nourdin_nualart_20} when the Hermite rank is \emph{exactly}~$d$, but the adaptation is straightforward.
\begin{lemma} \label{l:hermite_shift_representation}
	Let $g \in L^{2}(\gamma)$ with Hermite expansion~\eqref{e:hermite_decomposition} and Hermite rank at least $d \geq 1$. 
	Then, for any $h \in H$ with~$\norm[0]{h}_{H} = 1$, the following identities hold:
	\begin{align*} 
		g(W(h)) & = \delta^{d}\del[1]{\CS_{d} g(W(h)) h^{\otimes d}} \\[0.5em]
		\sbr[0]{\CS_{d} g}(W(h)) h^{\otimes d} & = (D (-L)^{-1})^{d} (g(W(h))) \\[0.5em]
		\sbr[0]{\CS_{d} g}(W(h)) & = \scal{(D (-L)^{-1})^{d} (g(W(h))),h^{\otimes d}}_{H^{\symotimes d}}
	\end{align*}
\end{lemma}

The following \emph{regularisation property} of the Hermite shift is crucial for our analysis. 
\begin{proposition} \label{p:estimate_hermite_shift}
	Let~$j \in \N_0$ and~$p > 1$. Then, for any~$n \in \N$, the operator~$\CS_n$ is continuous from~$\DD^{j,p}(\gamma)$ to~$\DD^{j+n,p}(\gamma)$, that is:
	\begin{equation} \label{e:estimate_shift:0}
		\norm[0]{\CS_n g}_{j+n,p} \lesssim_{j,n,p} \norm[0]{g}_{j,p} \,.
	\end{equation}
	In particular, for any~$k,n \in \N_0$, we have the estimate
	\begin{equation} \label{e:estimate_shift}
		\norm[0]{\CS_n g}_{k,p} \lesssim_{k,n,p} \norm[0]{g}_{(k-n) \vee 0,p} \,.
	\end{equation}
\end{proposition}
\begin{proof}
	The estimate in~\eqref{e:estimate_shift:0} is shown in~\cite[Lem.~2.3]{nualart_zhou_21}, see also~\cite[bottom of p.~6]{nourdin_nualart_peccati_21}.
	If~$k \geq n$, the estimate in~\eqref{e:estimate_shift} follows immediately from~\eqref{e:estimate_shift:0} upon choosing~$j = k - n$.
	The case~$k < n$ is a consequence of the estimate in~\cite[Eq.~(2.7)]{campese_nourdin_nualart_20}.
\end{proof}

The following lemma is inspired by the strategy presented in~\cite{nourdin_nualart_peccati_21}. 

\begin{lemma} \label{lem:OU_regularisation}
	Let $p > 1$, consider~$h \in L^p(\gamma)$, and let~$h^{[r]} \coloneqq P_{1/r} h$ where~$P_t$ denotes the OU semigroup, see Definition~\ref{d:OU_operator}.
	Then, for~$n \in \N$, we have:
	\begin{enumerate}[label=(\roman*)]
		\item \label{lem:OU_regularisation:i} $h^{[r]} \in \mathbb{D}^{k,p}(\gamma)$ for any~$k \geq 1$.
		\item \label{lem:OU_regularisation:ii} If~$h \in \mathbb{D}^{n,p}(\gamma)$, then~$h^{[r]} \to h$ in~$\mathbb{D}^{n,p}(\gamma)$ as~$r \to \infty$. 
		\item \label{lem:OU_regularisation:iii} 
		If~$h \in \mathbb{D}^{n,p}(\gamma)$, then we have~$ \CS_n h^{[r]} \to \CS_n h$ in~$\mathbb{D}^{2n,p}(\gamma)$ as~$r \to \infty$.
	\end{enumerate}
\end{lemma}

\begin{proof}
	By~\cite[Prop.~3.8]{Nualart_CBMS}, we have 
	\begin{equation*}
		\norm[0]{h^{[r]}}_{k,p} 
		= \norm[0]{P_{\nicefrac{1}{r}} h}_{k,p} 
		\lesssim_{k,p} r^k \norm[0]{h}_p 
	\end{equation*}
	which establishes~\ref{lem:OU_regularisation:i}. 
	For~\ref{lem:OU_regularisation:ii},~\cite[Prop.~3.7]{Nualart_CBMS} states that 
	\begin{equation*}
		\norm[0]{h^{[r]} - h}_{n,p} = \norm[0]{P_{\nicefrac{1}{r}} h - h}_{n,p} \to 0 \quad \text{as} \quad r \to \infty \,.
	\end{equation*}
	For~\ref{lem:OU_regularisation:iii}, we use Proposition~\ref{p:estimate_hermite_shift} to find 
	\begin{equation*}
		\norm[0]{\CS_n h^{[r]} - \CS_n h}_{2n,p} = \norm[0]{\CS_n\del[0]{h^{[r]} - h}}_{2n,p} 
		\lesssim \norm[0]{h^{[r]} - h}_{n,p}
	\end{equation*}
	and, since~$\norm[0]{\cdot}_{k,p} \leq \norm[0]{\cdot}_{2n,p}$ for any~$k \in \llbracket 0,2n \rrbracket$, the claim follows from~\ref{lem:OU_regularisation:ii}. 
\end{proof}

\subsection{Diagram formula, regular, and irregular diagrams} \label{s:diagram_formulas}

The following definition is a slightly amended version of~\cite[pp.~431-432]{breuer_major_83}.
\begin{definition}[Diagrams] \label{d:diagrams}
	Let~$l \in \N$, ~$q_{1:l} \in \N^l$, and~$n \coloneqq q_{[1:l]}$.
	A (complete Feynman) \emph{diagram} of order~$q_{1:l}$ is an undirected graph~$G = (V(G),E(G))$ with vertex set~$V(G)$ of cardinality~$n$ % = \llbracket 1, n \rrbracket
	and edge set~$E(G)$ satisfying the following properties:
	\begin{enumerate}[label=(\roman*)]
		\item The set~$V(G)$ has the form 
			\begin{equation*}
				V(G) = \bigsqcup_{\ell=1}^l L_\ell, \quad L_\ell = \cbr[0]{(\ell,u): \ u \in \llbracket 1, q_\ell \rrbracket}, \quad \ell \in \llbracket 1,l \rrbracket
			\end{equation*}
		where~$L_j$ is called the $j$-th \emph{level} of the graph~$G$ and~$q_j$ its \emph{size}.
		\item Each vertex~$w \in V(G)$ is of degree~$1$, i.e. each vertex is connected to \emph{exactly} one other vertex.  
		\item Each edge~$w = \del[0]{(\ell_1,j_1), (\ell_2,j_2)}$ only passes between different levels, i.e.~$\ell_1 \neq \ell_2$
	\end{enumerate}
	One can summarise the properties of~$G$ by saying that the vertices of~$G$ are matched in pairs without self-intersections within the same level, see Figure~\ref{fig:diagrams} for a visual representation.
	
	We let~$\Gamma(q_{1:l})$ denote the set of all such graphs and, for~$G \in \Gamma(q_{1:l})$ and~$w \in E(V)$ with~$w = ((\ell_1,j_1),(\ell_2,j_2))$ such that~$\ell_1 < \ell_2$, we define the functions~$d_1(w)\coloneqq \ell_1$ and~$d_2(w)\coloneqq \ell_2$.  
\end{definition}

\begin{remark}
	Note that, if~$\Gamma(q_{1:l}) \neq \emptyset$, Definition~\ref{d:diagrams} automatically implies that $n = q_{[1:l]}$ is even.
\end{remark}

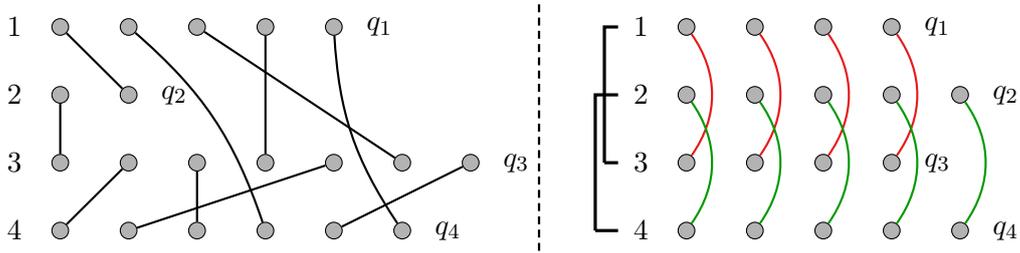
\begin{figure}[h]
	\centering
		\begin{subfigure}{.45\textwidth}
			%\begin{equs}
			\centering
			\begin{tikzpicture}[scale=0.3,baseline=0cm]
				%%%%%%%%
				\node at (4,3)  [] [label={[label distance=0.05cm]90:}] (1) {$1$};
				\node at (4,0)  [] [label={[label distance=0.05cm]90:}] 	(2) {$2$};
				\node at (4,-3)  [] [label={[label distance=0.05cm]90:}] 	(3) {$3$};
				\node at (4,-6)  [] [label={[label distance=0.05cm]90:}] 	(4) {$4$};
				%%%%%%%%
				\node at (6,3)  [xibig] [label={[label distance=0.05cm]90:}] (11) {};
				\node at (9,3)  [xibig] [label={[label distance=0.05cm]90:}] (12) {};
				\node at (12,3)  [xibig] [label={[label distance=0.05cm]90:}] (13) {};
				\node at (15,3)  [xibig] [label={[label distance=0.05cm]90:}] (14) {};
				\node at (18,3)  [xibig] [label={[label distance=0.05cm]90:}] (15) {};
				\node at (20,3)  [] [label={[label distance=0.05cm]90:}] (q1) {$q_1$};
				%%%%%
				\node at (6,0)  [xibig] [label={[label distance=0.05cm]-90:}] (21) {};
				\node at (9,0)  [xibig] [label={[label distance=0.05cm]-90:}] (22) {};
				\node at (11,0)  [] [label={[label distance=0.05cm]-90:}] (q2) {$q_2$};
				%%%%%
				\node at (6,-3)  [xibig] [label={[label distance=0.05cm]90:}] (31) {};
				\node at (9,-3)  [xibig] [label={[label distance=0.05cm]90:}] (32) {};
				\node at (12,-3)  [xibig] [label={[label distance=0.05cm]90:}] (33) {};
				\node at (15,-3)  [xibig] [label={[label distance=0.05cm]90:}] (34) {};
				\node at (18,-3)  [xibig] [label={[label distance=0.05cm]90:}] (35) {};
				\node at (21,-3)  [xibig] [label={[label distance=0.05cm]90:}] (36) {};
				\node at (24,-3)  [xibig] [label={[label distance=0.05cm]90:}] (37) {};
				\node at (26,-3)  [] [label={[label distance=0.05cm]90:}] (q3) {$q_3$};
				%%%%%
				\node at (6,-6)  [xibig] [label={[label distance=0.05cm]-90:}] (41) {};
				\node at (9,-6)  [xibig] [label={[label distance=0.05cm]-90:}] (42) {};
				\node at (12,-6)  [xibig] [label={[label distance=0.05cm]-90:}] (43) {};
				\node at (15,-6)  [xibig] [label={[label distance=0.05cm]-90:}] (44) {};
				\node at (18,-6)  [xibig] [label={[label distance=0.05cm]-90:}] (45) {};
				\node at (21,-6)  [xibig] [label={[label distance=0.05cm]-90:}] (46) {};
				\node at (23,-6)  [] [label={[label distance=0.05cm]-90:}] (q4) {$q_4$};
				%%%%%
				\draw[thick] (11) to (22);
				\draw[thick] (13) to (36);
				\draw[thick] (12) to [bend left=15] (44);
				\draw[thick] (21) to (31);
				\draw[thick] (32) to (41);
				\draw[thick] (14) to (34);
				\draw[thick] (15) to [bend right=15] (46);
				\draw[thick] (33) to (43);
				\draw[thick] (35) to (42);
				\draw[thick] (37) to (45);
				\draw[line width=0.3mm, densely dashed] (27,4) to (27,-7);
			\end{tikzpicture}\; 
			\subcaption{\emph{Irregular} diagram for~$q_{1:4} = (5,2,7,6)$ \phantom{together with the pairing of its levels.}}
			\label{subfig:irregular_diagram}
	\end{subfigure}
	\hspace{1.5em}
	\begin{subfigure}{.45\textwidth}
		%\begin{equs}
		\centering
		\begin{tikzpicture}[scale=0.3,baseline=0cm]
			%%%%%%%%
				\node at (4,3)  [] [label={[label distance=0.05cm]90:}] (1) {$1$};
				\node at (3,3) (L1) {};
				\node at (4,0)  [] [label={[label distance=0.05cm]90:}] 	(2) {$2$};
				\node at (4,-3)  [] [label={[label distance=0.05cm]90:}] 	(3) {$3$};
				\node at (3,-3) (L3) {};
				\node at (4,-6)  [] [label={[label distance=0.05cm]90:}] 	(4) {$4$};
				%%%%%%%%
				\node at (6,3)  [xibig] [label={[label distance=0.05cm]90:}] (11) {};
				\node at (9,3)  [xibig] [label={[label distance=0.05cm]90:}] (12) {};
				\node at (12,3)  [xibig] [label={[label distance=0.05cm]90:}] (13) {};
				\node at (15,3)  [xibig] [label={[label distance=0.05cm]90:}] (14) {};
				\node at (17,3)  [] [label={[label distance=0.05cm]90:}] (q1) {$q_1$};
				%%%%%
				\node at (6,0)  [xibig] [label={[label distance=0.05cm]-90:}] (21) {};
				\node at (9,0)  [xibig] [label={[label distance=0.05cm]-90:}] (22) {};
				\node at (12,0)  [xibig] [label={[label distance=0.05cm]-90:}] (23) {};
				\node at (15,0)  [xibig] [label={[label distance=0.05cm]-90:}] (24) {};
				\node at (18,0)  [xibig] [label={[label distance=0.05cm]-90:}] (25) {};
				\node at (20,0)  [] [label={[label distance=0.05cm]-90:}] (q2) {$q_2$};
				%%%%%
				\node at (6,-3)  [xibig] [label={[label distance=0.05cm]90:}] (31) {};
				\node at (9,-3)  [xibig] [label={[label distance=0.05cm]90:}] (32) {};
				\node at (12,-3)  [xibig] [label={[label distance=0.05cm]90:}] (33) {};
				\node at (15,-3)  [xibig] [label={[label distance=0.05cm]90:}] (34) {};
				\node at (17,-3)  [] [label={[label distance=0.05cm]90:}] (q3) {$q_3$};
				%%%%%
				\node at (6,-6)  [xibig] [label={[label distance=0.05cm]-90:}] (41) {};
				\node at (9,-6)  [xibig] [label={[label distance=0.05cm]-90:}] (42) {};
				\node at (12,-6)  [xibig] [label={[label distance=0.05cm]-90:}] (43) {};
				\node at (15,-6)  [xibig] [label={[label distance=0.05cm]-90:}] (44) {};
				\node at (18,-6)  [xibig] [label={[label distance=0.05cm]-90:}] (45) {};
				\node at (20,-6)  [] [label={[label distance=0.05cm]-90:}] (q4) {$q_4$};
			%%%%%
			\draw[thick, color=darkred] (11) to [bend left=35] (31);
			\draw[thick, color=darkred] (12) to [bend left=35] (32);
			\draw[thick, color=darkred] (13) to [bend left=35] (33);
			\draw[thick, color=darkred] (14) to [bend left=35] (34);
			%%%%%
			\draw[thick, color=darkgreen] (21) to [bend left=35] (41);
			\draw[thick, color=darkgreen] (22) to [bend left=35] (42);
			\draw[thick, color=darkgreen] (23) to [bend left=35] (43);
			\draw[thick, color=darkgreen] (24) to [bend left=35] (44);
			\draw[thick, color=darkgreen] (25) to [bend left=35] (45);
			%%%%%
			%pairing 1
			\draw[very thick] (3,3) -- (2.4,3) -- (2.4,-3);
  			\draw[very thick] (2.4,-3) -- (3,-3);
			%%%%%
			%pairing 2
			\draw[very thick] (3,0) -- (2,0) -- (2,-6);
  			\draw[very thick] (2,-6) -- (3,-6);
			%%%%%
		\end{tikzpicture}\; 
		\subcaption{\emph{Regular} diagram for~$q_{1:4} = (4,5,4,5)$ together with the pairing of its levels.} 
		\label{subfig:regular_diagram}
\end{subfigure}
\caption{\textbf{Visualisation of diagrams}. Note that, by relabelling the vertices within levels, a regular diagram~$G$ can always be represented as in Subfigure~\ref{subfig:regular_diagram} because this does not change~$\fC_G$ as introduced in~\eqref{e:diagram_formula:CG} below.}
\label{fig:diagrams}
\end{figure}

The previous definition is important because it allows us to compute the expectation of products of Hermite polynomials with Gaussian input via the so-called \emph{diagram formula}.

\begin{proposition}[Diagram formula] \label{p:diagram_formula}
	Let~$l \in \N_{\geq 2}$,~$q_{1:l} \in \N^l$, $\w = \w_{1:l} \in \llbracket 1,m \rrbracket^l$, and~$i_{1:l} \in \N^{l}$. For a centered, stationary Gaussian vector~$(X^{(\w_1)}_{i_1},\ldots,X^{(\w_l)}_{i_l})$ with covariance~$\E\sbr[0]{X^{(\w_k)}_{i_k} X^{(\w_\ell)}_{i_\ell}} = \rho_{\w_k,\w_\ell}(i_\ell - i_k)$ and~$\rho_{\w_k,\w_k}(0) = 1$ for any~$k,\ell \in \llbracket 1,l \rrbracket$, we have 
	\begin{equation} \label{e:diagram_formula}
		\E\sbr[4]{\prod_{k=1}^l H_{q_k}(X^{(\w_k)}_{i_k})}  
		= \sum_{G \in \Gamma(q_{1:l})} \fC_G(\w_{1:l}, q_{1:l}, i_{1:l}) 
	\end{equation}
	where\footnote{Note that the dependence on~$q_{1:l}$ in~$\fC_G$ is encoded via~$G \in \Gamma(q_{1:l})$. We made this dependence explicit because it will become relevant in Section~\ref{s:reduction}.}
	\begin{equation} \label{e:diagram_formula:CG}
		\fC_G(\w_{1:l}, q_{1:l}, i_{1:l}) \coloneqq \prod_{w \in E(G)} \rho_{\w_k,\w_\ell}(i_\ell - i_k) \,.
	\end{equation}
\end{proposition}

\begin{remark}[Level labels] \label{rmk:level_labels}
	In Figure~\ref{fig:diagrams}, we have added labels to the levels, in this case~$1$, $2$, $3$, and~$4$.
	These labels serve another important purpose beyond just numbering the levels: 
	They encode the fact that all nodes in the~$\ell$-th level correspond to the random variable~$X^{(\w_\ell)}_{i_\ell}$. 
	In particular, any edge between levels~$\ell_1$ and~$\ell_2$ corresponds to an instance of~$\rho_{\w_{\ell_1}, \w_{\ell_2}}(i_{\ell_2} - i_{\ell_1})$.
	Note that this agrees with our covariance convention in~\eqref{e:multivariate_covariance}.
\end{remark}

Let us now introduce \emph{regular} and~\emph{irregular} diagrams as on~\cite[p.~432]{breuer_major_83}.
\begin{definition}[Regular and irregular diagrams] \label{d:regular_graph}
	Let~$l \in \N$ and~$q_{1:l} \in \N^l$.  
	A diagram~$G \in \Gamma[q_{1:l}]$ is called \emph{regular} if its levels can be paired in such a way that no edge passes between levels in different pairs, and \emph{irregular} otherwise, see Figure~\ref{subfig:irregular_diagram} for a visualisation.
	
	We then have $\Gamma(q_{1:l}) = \Gammareg(q_{1:l}) \sqcup \Gammairreg(q_{1:l})$ where~$\Gammareg(q_{1:l})$ denotes the set of regular and $\Gammairreg(q_{1:l})$ the set of irregular diagrams. 
\end{definition}

Before we comment on the previous definition, we introduce the notion of (perfect) \emph{pairwise matchings}. 
\begin{definition}[Matchings] \label{d:matchings}
	For~$k \in \N$, we denote the set~$\CM(l)$ of perfect \emph{pairwise matchings} between vertices labelled by~$1, \ldots, l$, i.e.~$\CM(l) = \emptyset$ if~$l$ is odd. 
	For~$l$ even, i.e.~$l = 2n$ for some~$n \in \N$, we represent a matching~$P \in \CM(2n)$ by 
	\begin{equation*}
		P = \cbr[1]{\cbr[0]{P(2j-1),P(2j)}: \ j \in \llbracket 1,n \rrbracket} 
	\end{equation*}
	which consists of \emph{unordered} tuples.
	Sometimes we will just write~$p_{\ell} \coloneqq P(\ell)$ for~$\ell \in \llbracket 1,2n \rrbracket$. 
\end{definition}

\begin{remark} \label{rmk:regular_diagrams}
	Observe that~$l$ being odd automatically implies~$\Gamma_{\text{\tiny R}}(q_{1:l})= \emptyset$. 
	Further, in case~$l$ is even, i.e. $l = 2n$ for some~$n \in \N$, observe that any regular diagram~$\Gamma \in \Gamma_{\text{\tiny R}}(q_{1:2n})$ is  uniquely characterised (up to relabelling the vertices) by a pairing~$P \in \CM(2n)$, see Figure~\ref{subfig:regular_diagram} for a visualisation of this fact.
	Furthermore, in that case we know that there exists a permutation~$\pi \in \mathfrak{G}(2n)$ and~$\tilde{q}_{1:n} \in \N^n$ such that~$q_{1:2n} = \pi(\tilde q_1,\tilde q_1,\ldots,\tilde q_n,\tilde q_n)$, possibly with~$\tilde q_j = \tilde q_k$ even if~$j \neq k$. 
\end{remark}

The following result, essentially shown in~\cite[Proposition on~p.~433]{breuer_major_83}, will be central for several of our arguments.
See, however, Remark~\ref{rmk:breuer_major} below for the key modification to their argument which renders it applicable in our situation. 

\medskip
Recall that $\square_{l,N} \coloneqq \llbracket 0,N-1 \rrbracket^{l}$ denotes the discrete cube of dimension~$l \in \N$ and side length~$N$.
\begin{proposition}[Breuer--Major] \label{p:bm83_irregular}
	Let~$d,l\in \N$, ~$\bq = q_{1:l} \geq d$, $\w = \w_{1:l} \in \llbracket 1,m \rrbracket^l$, and~$\bi = i_{1:l} \geq 1$. 
	For any $G \in \Gamma(\bq)$, we let 
	\begin{equation} \label{e:def_TG}
		\mathring{T}_G(\w,\bq,N) 
		\coloneqq 
		\frac{1}{N^{\nicefrac{l}{2}}} \sum_{\bi \in \square_{l,N}} \abs[0]{\fC_G(\w, \bq, \bi)} 
	\end{equation}
	where~$\fC_G(\w, \bq, \bi)$ has been introduced in~\eqref{e:diagram_formula:CG}.
	Then, for any \emph{irregular} diagram~$G$, we have 
	\begin{equation} \label{p:bm83_irregular:conv}
		\lim_{N \to \infty} \mathring{T}_G(\w,\bq,N) = 0 \,.
	\end{equation}
\end{proposition}

We also define~$T_G(\w,\bq,N)$ like~$\mathring{T}_G(\w,\bq,N)$ in~\eqref{e:def_TG}, but with~$\abs[0]{\fC_G(\w,\bq,\bi)}$ replaced by $\fC_G(\w,\bq,\bi)$, i.e. without absolute values. 
Note that the former only contains factors~$\abs[0]{\rho_{\w_k,\w_\ell}(i_\ell - i_k)}$ in~\eqref{e:diagram_formula:CG}, while the latter only has factors~$\rho_{\w_k,\w_\ell}(i_\ell - i_k)$.

\begin{remark} \label{rmk:breuer_major}
	The previous proposition is \emph{almost} identical to that in~\cite[Proposition on~p.~433]{breuer_major_83}. 
	There are, however, two small differences that are crucial for several of our arguments in Section~\ref{s:convergence_fdd}:
	\begin{enumerate}[label=(\arabic*)]
		\item Breuer and Major consider the univariate case, i.e.~$\w = \w_{1:l}$ for~$\w_1 = \ldots = \w_l = 1$.
		\item The convergence in~\eqref{p:bm83_irregular:conv} does not only hold for~$T_G(\w,\bq,N)$, but also for~$\mathring{T}_G(\w,\bq,N)$ defined in terms of the \emph{absolute value}~$\abs[0]{\rho}$ (instead of just~$\rho$ itself). 
	\end{enumerate}
	Accounting for both of these changes only requires small modifications in the original proof by Breuer and Major: We will present them in Appendix~\ref{a:proof_BM83}.

	One immediate consequence is that, if one replaces the summation condition~$\bi \in \square_{l,N}$ in the definition of~$\mathring{T}_G(\bq,N)$ by~$\bi \in A_{l,N}$ (which represents some \emph{index summation constraints}) where~$A_{l,N} \subseteq \square_{l,N}$, then the corresponding quantity still goes to zero because the sum is simply over a smaller set of indices. 
	Note that, by a simple scaling argument, the previous proposition still holds if the box~$\square_{l,N}$ depends on~$s,t \in [0,1]^2$ such that~$s < t$, i.e. if it is replaced by  
	\begin{equation*}
		\square_{l,N}(s,t) \coloneqq \llbracket \lfloor Ns \rfloor , \lfloor Nt \rfloor -1 \rrbracket^{l} \,.
	\end{equation*}
	This, then, also allows for the subset~$A_{l,N}$ to depend on~$s$ and~$t$ in an analogous way, of course.
\end{remark}

The previous proposition itself will be used in the proof of Proposition~\ref{p:reduction}, the reduction to finite chaos. 
Mostly, however, it will be used via the following corollary.
To this end, recall that for~$f \in L^2(\gamma)$, we let~$f^M \coloneqq \pi_{\leq M} f$ be the projection onto Wiener chaoses smaller than or equal to~$M$.

\begin{corollary} \label{coro:BM83}
	Let~$l \in \N_0$,~$\w_j \in \llbracket 1,m \rrbracket$ for~$j \in \llbracket 1,l \rrbracket$, and choose some set~$A_{l,N} \subseteq \square_{l,N}$. Then, there exists some~$n \in \N$ such that, as~$N \to \infty$, 
	\begin{equs}[][coro:BM83:eq]
		\thinspace &
		\frac{1}{N^{\nicefrac{l}{2}}} \sum_{\bi \in A_{l,N}} \E\sbr[4]{\prod_{j=1}^{l} f^{M}_{\w_{j}}(X^{(\w_{j})}_{i_j})} \\  
		\asymp \ & 
		\frac{1}{N^{n}} \sum_{\bi \in A_{2n,N}} \sum_{P \in \CM(2n)} 
		\prod_{j=1}^{n} \E\sbr[1]{f^{M}_{\w_{P(2j-1)}}\del[1]{X^{(\w_{P(2j-1)})}_{i_{P(2j-1)}}} f^{M}_{\w_{P(2j)}}\del[1]{X^{(\w_{P(2j)})}_{i_{P(2j)}}}} \1_{l=2n} \,.
	\end{equs} 
	In other words: Asymptotically, Wick's formula holds.
\end{corollary}

\begin{proof}
	By assumption, we have
	\begin{equation*}
		f_{k}^M(x) = \sum_{q \geq d}^M c_q^{(k)} H_q(x) \,.
	\end{equation*}
	We now set~$\bq \coloneqq q_{1:l} \in \llbracket d,M \rrbracket^{l}$,~$\w \coloneqq \w_{1:l} \in \llbracket 1,m \rrbracket^{l}$, as well as~$c_{\bq}^{(\w)} \coloneqq \prod_{j=1}^{l} c_{q_j}^{(\w_j)}$.
	By the diagram formula~(Proposition~\ref{p:diagram_formula}), we can now rewrite the LHS of~\eqref{coro:BM83:eq} as follows:
	\begin{equs}[][e:BMcoro:e1]
		\frac{1}{N^{\nicefrac{l}{2}}} \sum_{\bi \in A_{l,N}} \E\sbr[4]{\prod_{j=1}^{l} f^{M}_{\w_{j}}(X^{(\w_{j})}_{i_j})} %\\[0.5em]
		& = 
		\frac{1}{N^{\nicefrac{l}{2}}} \sum_{\bi \in A_{l,N}} \sum_{d \leq \bq \leq M}
		c_{\bq}^{(\w)} \, \E\sbr[4]{\prod_{j=1}^{l} H_{q_j}(X^{(\w_{j})}_{i_j})} \\[0.5em]
		%%%
		& =  
		\frac{1}{N^{\nicefrac{l}{2}}} \sum_{\bi \in A_{l,N}} \sum_{d \leq \bq \leq M}
		c_{\bq}^{(\w)} \sum_{G \in \Gamma(\bq)} \fC_G(\w_{1:l}, q_{1:l}, i_{1:l})  \,.
	\end{equs}
	Since~$\Gamma(\bq) = \Gammareg(\bq) \sqcup \Gammairreg(\bq)$ (see Definition~\ref{d:regular_graph}), the claim follows once we show that the sum over \emph{irregular} graphs~$G \in \Gammairreg(\bq)$ in~\eqref{e:BMcoro:e1} vanishes as~$N \to \infty$. 
	For the sum over those graphs, using that~$A_{l,N} \subseteq \square_{l,N}$ as well as the definition of~$\mathring{T}_G(\w,\bq,N)$ in~\eqref{e:def_TG} we find that
	\begin{equs} 
		\thinspace &
		\frac{1}{N^{\nicefrac{l}{2}}} \sum_{\bi \in A_{l,N}} \sum_{d \leq \bq \leq M}
		\abs[0]{c_{\bq}^{(\w)}} \sum_{G \in \Gammairreg(\bq)} \abs[0]{\fC_G(\w_{1:l}, q_{1:l}, i_{1:l})}
		\leq %\ &  
		\sum_{d \leq \bq \leq M}
		\abs[0]{c_{\bq}^{(\w)}} \sum_{G \in \Gammairreg(\bq)} \mathring{T}_G(\w,\bq,N)
	\end{equs}
	and the claim now follows from Proposition~\ref{p:bm83_irregular} because the sums over~$\bq$ and~$G \in \Gamma(\bq)$ are each over \emph{finitely many} terms. 
	Note that if~$l$ is odd, by Remark~\ref{rmk:regular_diagrams}, we have~$\Gamma(\bq) = \Gammairreg(\bq)$ which gives rise to the indicator function~$\1_{l=2n}$ in~\eqref{coro:BM83:eq}. 
\end{proof}

The previous corollary will be used in various places of the article:
\begin{itemize}
	\item In Section~\ref{s:LLN_diagonal}, specifically in the proof of Lemma~\ref{l:diagonal_variance}.
	\item In Section~\ref{s:higher_diagonals}, specifically in the proof of Proposition~\ref{p:higher_order_diagonals}.
	\item In Section~\ref{s:conv_moments}, specifically in the proofs of Proposition~\ref{p:exp_sig} and Theorem~\ref{thm:conv_fdd_stratonovich}.
\end{itemize}
In each case, the index constraint and, therefore, the set~$A_{2n,N}$ is different; we will define it locally when it is required.

\subsection{Convergence of normalised sums of covariances} 

The following two lemmas collect some elementary convergence results.

\begin{lemma} \label{lem:elemantary_sum_covariances}
	For any function~$\rho: \Z \to \R$ such that~$\abs[0]{\rho(k)} \to 0$ as~$\abs[0]{k} \to \infty$, we have
	\begin{equation*}
		\lim_{N \to \infty} \frac{1}{N^2} \sum_{1 \leq i,j \leq N} \abs[0]{\rho(i-j)} = 0 \,.
	\end{equation*}
\end{lemma}

\begin{proof}
	Let~$\eps, \delta > 0$. Then, there exists some~$N_0 = N_0(\eps,\delta) \in \N$ such that, for all~$N \geq N_0$, we know that~$\abs[0]{k} > \delta N$ implies that~$\abs[0]{\rho(k)} < \eps$.
	Then, we estimate
	\begin{equs}
		\frac{1}{N^2} \sum_{1 \leq i, j \leq N} \abs[0]{\rho(i-j)}
		& \leq
		\frac{1}{N^2} \sum_{\abs[0]{i-j} \leq \delta N} 1 + \frac{1}{N^2} \sum_{\abs[0]{i-j} > \delta N} \abs[0]{\rho(i-j)} \\
		& \leq
		\delta^2 + \frac{1}{N^2} \sum_{1 \leq i,j \leq N} \eps = \delta^2 + \eps \,.
	\end{equs}
	Since~$\eps > 0$ and~$\delta > 0$ were arbitrary, the claim follows. 
\end{proof}

\begin{lemma} \label{lem:path_sum_covariances}
	Let~$d \geq 1$,~$n \in \llbracket 1,d \rrbracket$, and assume that~$\rho_i \in \ell^n(\Z)$ for~$i = 1,2,3$.
	Further, let~$j,k,l,m \in \llbracket 1,4 \rrbracket$ be pairwise different, i.e.~$\cbr[0]{j,k,l,m} = \cbr[0]{1,2,3,4}$. Then,
	\begin{equation*}
		\lim_{N \to \infty}
		\frac{1}{N^2} \sum_{1 \leq i_{1:4} \leq N} \abs[0]{\rho_1(i_j - i_k)}^n \abs[0]{\rho_2(i_k - i_\ell)} \abs[0]{\rho_3(i_\ell - i_m)}^n = 0 \,.
	\end{equation*}
\end{lemma}

\begin{proof}
	By possibly relabelling the indices, we may w.l.o.g. assume that~$(j,k,l,m) = (1,2,3,4)$. We first observe that
	\begin{equation*}
		\sum_{i_1 = 1}^N  \abs[0]{\rho_1(i_1 - i_2)}^n \leq \norm[0]{\rho_1}_{\ell^n(\Z)}^n, \quad 
		\sum_{i_4 = 1}^N  \abs[0]{\rho_3(i_3 - i_4)}^n \leq \norm[0]{\rho_3}_{\ell^n(\Z)}^n 
	\end{equation*}
	and, therefore,
	\begin{equs}
		\thinspace &
		\frac{1}{N^2} \sum_{1 \leq i_{1:4} \leq N} \abs[0]{\rho_1(i_1 - i_2)}^n \abs[0]{\rho_2(i_2 - i_3)} \abs[0]{\rho_3(i_3 - i_4)}^n \\
		\leq \ &  
		\frac{1}{N^2} \sum_{1 \leq i_2, i_3 \leq N}  \abs[0]{\rho_2(i_2 - i_3)} \norm[0]{\rho_1}_{\ell^n(\Z)}^n \norm[0]{\rho_3}_{\ell^n(\Z)}^n\,.
	\end{equs}
	The claim now follows from Lemma~\ref{lem:elemantary_sum_covariances}.
\end{proof}

\section{Tightness} \label{s:tightness}

In this section we prove that the sequence of discrete processes $\boldsymbol{S}_N = (S_N,\mathbb{S}_N)$, $N \geq 1$, as defined in \eqref{e:1st_2nd_order_p}, is tight in an appropriate topology on the space of rough paths.

\subsection{Variation topologies on rough paths} \label{s:rough_paths_primer}

Let us first introduce the rough path setting we are going to consider.
We will closely follow the definitions and notations of
\cite{chevyrev_et_al_homogenisation_part_2}; for a general introduction to rough paths theory, we refer to~\cite{lyons_caruana_levy,friz_victoir,friz_hairer_book}.

We introduce the truncated tensor algebra 
\begin{equation} \label{e:truncated_tensor_algebra}
	\tilde{T}^{(2)}(\mathbb{R}^m) := \mathbb{R}^m \oplus (\mathbb{R}^m \otimes \mathbb{R}^m)\,,	
\end{equation}
which we will denote by $\tilde{T}^{(2)}$ for short.
We will also equip it with the following multiplication operation: 
For two elements $\boldsymbol{x}=(x^{(1)},x^{(2)})$ and $\boldsymbol{y}=(y^{(1)},y^{(2)})$ in $\tilde{T}^{(2)}$, we write
\begin{equation*}
	\boldsymbol{xy} \coloneqq (x^{(1)}+y^{(1)}, x^{(2)} + x^{(1)} \otimes y^{(1)} + y^{(2)}).	
\end{equation*}
Recall that this defines a group structure on $\tilde{T}^{(2)}$ with $\boldsymbol{0} := (0,0)$ as the identity and the inverse of every $\boldsymbol{x}=(x^{(1)},x^{(2)}) \in \tilde{T}^{(2)}$ being given by $\boldsymbol{x}^{-1} := (-x^{(1)},-x^{(2)} + x^{(1)}\otimes x^{(1)})$. We further equip $\tilde{T}^{(2)}$ with the pseudo-norm given by
\begin{equation*}
	\|\boldsymbol{x}\| := |x^{(1)}| + \sqrt{|x^{(2)}|}, \qquad \boldsymbol{x} \in \tilde{T}^{(2)} 	
\end{equation*}
where $|\cdot|$ denotes the Euclidean norms on $\mathbb{R}^m$ and $(\mathbb{R}^m)^{\otimes 2}$, respectively, and with the metric $d$ given by
\begin{equation*}
	d(\boldsymbol{x},\boldsymbol{y}) \coloneqq \|\boldsymbol{x}^{-1} \boldsymbol{y}\|, \qquad \boldsymbol{x},\boldsymbol{y} \in \tilde{T}^{(2)}.	
\end{equation*}
For a fixed time horizon $T>0$ (eventually, we shall take $T=1$) and a path $\boldsymbol{X}:[0,T] \to \tilde{T}^{(2)}$, we define its increments as 
\begin{equation*}
	\boldsymbol{X}(s,t) := \boldsymbol{X}(s)^{-1} \boldsymbol{X}(t), \qquad 0 \leq s \leq t \leq T \,.	
\end{equation*}
If we write $\boldsymbol{X}(t) = (X(t),\mathbb{X}(t))$ for all $t \geq 0$, then the above expression can be equivalently written $\boldsymbol{X}(s,t) = (X(s,t),\mathbb{X}(s,t))$ for $0 \leq s \leq t \leq T$, where
\begin{equs}[][eq:p_rp_increments]
	X(s,t) & \coloneqq X(t)-X(s) \,, \\
	%%%
	\mathbb{X}(s,t) & \coloneqq \mathbb{X}(t) - \mathbb{X}(s) - (X(s)-X(0)) \otimes (X(t)-X(s)) \,. 
\end{equs}
Note that, with the above definition, the following Chen relation is plainly satisfied
\begin{equation*}
\boldsymbol{X}(s,u) \boldsymbol{X}(u,t) = \boldsymbol{X}(s,t), \qquad 0 \leq s \leq u \leq t.  
\end{equation*}
We recall the following definition from \cite{chevyrev_et_al_homogenisation_part_2}.

\begin{definition}
\label{def:p_rp}
Let $r \in (2,3)$. An $r$-rough path over $\mathbb{R}^m$ is a c\`{a}dl\`{a}g process $\boldsymbol{X} :[0,T] \to \tilde{T}^{(2)}$ such that $\boldsymbol{X}(0)= \boldsymbol{0}$ and $\|\boldsymbol{X}\|_{r-\var} := \|X\|_{r-\var} + \|\mathbb{X}\|_{r/2-\var} < \infty$, where 
\begin{equation*}
	\|X\|_{r-\var} := \sup_{\mathcal{P}} \left(\sum_{[s,t] \in \mathcal{P}} |X(s,t)|^r \right)^{1/r}, \quad \|\mathbb{X}\|_{r/2-\var} := \sup_{\mathcal{P}} \left(\sum_{[s,t] \in \mathcal{P}} |\mathbb{X}(s,t)|^{r/2} \right)^{2/r},
\end{equation*}
where the suprema run over all partitions $\mathcal{P}$ of the interval $[0,T]$.
\end{definition}

\begin{remark}
Note that we are working with $r$-variation (rather than $\alpha$-Hölder) rough paths so as to deal with the c\`{a}dl\`{a}g processes~$\boldsymbol{S}_N$.  
\end{remark}

We finally recall the (Skorokhod-type) $r$-variation metric on the space of $r$-rough paths.

\begin{definition} \label{d:skorokhod_variation_metric}
For $r$-rough paths $\boldsymbol{X} = (X,\mathbb{X})$ and $\boldsymbol{Y} = (Y,\mathbb{Y})$, the $r$-variation metric between $\boldsymbol{X}$ and $\boldsymbol{Y}$ is defined as
\begin{equation*}
	\sigma_{r-\var}(\boldsymbol{X},\boldsymbol{Y}) 
	\coloneqq 
	\inf_{\omega \in \Omega} 
	\cbr[1]{\abs[0]{\omega} + \norm[0]{\boldsymbol{X}; \boldsymbol{Y} \circ \omega}_{r-\var}} 	
\end{equation*}
where $\Omega$ denotes the set of all continuous increasing bijections from $[0,T]$ to itself. 
Above, we used the notation $|\omega| := \sup_{t \in [0,T]} |\omega(t) -t|$ for all $\omega \in \Omega$, and  
\begin{equation*}
	\norm{\boldsymbol{X};\boldsymbol{Y}}_{r-\var} \coloneqq \norm{X -Y}_{r-\var} + \norm{\mathbb{X} - \mathbb{Y}}_{r/2-\var} \,.
\end{equation*}
We denote the space of all $r$-rough paths, equipped with the metric~$\sigma_{r-\var}$, by $\mathcal{D}^{r-\var}(\mathbb{R}^m)$.
\end{definition}

We will briefly discuss some algebraic aspects of rough paths beyond level two when they are needed in Section~\ref{s:moment_computations}. 

\subsection{Tightness result} \label{s:tightness_result}

Recall that, for all~$N \geq 1$, the processes $S_N$ and ~$\mathbb{S}_N$ introduced in \eqref{e:1st_2nd_order_p} can be written in components as follows:  
\begin{equation*}
	S^k_N(t) := \frac{1}{\sqrt{N}} \sum_{0 \leq i < \lfloor N t \rfloor} f_k(X^{(k)}_i), \qquad
     \S^{k,\ell}_N (t) := \frac{1}{N} \sum_{0 \leq i < j < \lfloor N t \rfloor} f_k(X^{(k)}_i) f_\ell(X^{(\ell)}_j),  
\end{equation*}
where~$t \geq 0$ and $k,\ell \in \llbracket 1, m \rrbracket$.

The process~$\boldsymbol{S}_N$ given by $\boldsymbol{S}_N(t) \coloneqq (S_N(t), \mathbb{S}_N(t)) \in \tilde{T}^{(2)}(\R^m)$ then becomes a $p$-rough path in the sense of Definition~\ref{def:p_rp} over $\mathbb{R}^m$ for all $p \in (2,3)$. 
Note that the $p$-variation is finite because this is a  piecewise constant process. Note also that, for all $s,t \geq 0$,  the increment $\mathbb{S}_N(s,t)$ as defined in \eqref{eq:p_rp_increments} is given by
\begin{equation}
\label{eq:def.2nd_order_increment_expr}
\mathbb{S}^{k,\ell}_N(s,t) = \frac{1}{N} \sum_{\lfloor Ns \rfloor \leq i < j < \lfloor N t \rfloor} f_k(X^{(k)}_i) f_\ell(X^{(\ell)}_j),  \qquad k, \ell \in \llbracket 1, m \rrbracket.
\end{equation}
In particular it holds that, for $0 \leq s \leq t$, we have
\begin{equation}
\label{eq:exp_dist_incr_rp}
d(\boldsymbol{S}_N(s),\boldsymbol{S}_N(t)) = |S_N(t)-S_N(s)| + \|\mathbb{S}_N(s,t)\|^{1/2}. 
\end{equation}

Thus, in order to prove tightness results for the sequence of processes $(\boldsymbol{S}_N)_{N \geq 1}$, the key ingredient will be to control the increments $S_N(s,t) \coloneqq S_N(t)-S_N(s)$ and $\mathbb{S}_N(s,t)$. 
While moment bounds on the first-order increments $S_N(s,t)$ were already obtained in \cite{nourdin_nualart_20}, additional work is required to get similar bounds on $\mathbb{S}_N(s,t)$: This is the content of the following theorem. 

\begin{theorem} \label{thm:tightness}
	Let\footnote{See Remark~\ref{rmk:tightness_p} below for a discussion on the assumptions on~$p$. In particular, note that Lemma~\ref{lem:tightness_skorokhod} and Corollary~\ref{coro:tighness} only require~$p > 1$.} $p \geq  2$ and~$d \geq 1$. 
	Let us assume the following conditions:
	\begin{enumerate}[label=(\arabic*)]
		\item \label{thm:tightness:1}
		For each~$k \in \llbracket 1,m \rrbracket$,~$f_k$ has Hermite rank at least~$d$ and satisfies~$f_k \in\mathbb{D}^{d,2p}(\gamma)$ 
        and where $\gamma = \CN(0,1)$ denotes the standard normal distribution.
		\item \label{thm:tightness:2}
		For each~$k, \ell \in \llbracket 1,m \rrbracket$, we have~$\sum_{i \in \Z} \abs[0]{\rho_{k,\ell}(i)}^{d} < \infty$, i.e.~$\rho_{k,\ell} \in \ell^{d}(\Z)$.
	\end{enumerate}
	Then, for $0 \leq s \leq t$ and every~$k,\ell \in \llbracket 1,m \rrbracket$, we have $\norm[0]{\S^{k,\ell}_N(s,t)}_{L^{p}(\Omega)} \lesssim_{d,p} \frac{(\lfloor Nt \rfloor - \lfloor N s \rfloor)}{N}$.
\end{theorem}

Before proving the above result in Subsection \ref{subsec:proof_tightness_thm} below, we first explain how it entails the claimed tightness property, namely that the sequence $(\mathbf{S}_N)_{N \geq 1}$ is tight in the space $\mathcal{D}^{r-\var}(\mathbb{R}^m)$ for all $r>2$.

By \cite[pp.~9-10]{nourdin_nualart_20} applied to the processes $S^k_N$, $k \in \llbracket 1, m \rrbracket$ (note that $2p>2$ for any~$p > 1$, so in particular for~$p \geq 2$) we obtain, for $0 \leq s \leq t$,
\begin{equation}
\label{eq:tightness_est_1st_order}
\|S^k_N(t)-S^k_N(s)\|_{L^{2p}(\Omega)} \lesssim \left( \frac{\lfloor N t \rfloor - \lfloor N s \rfloor}{N} \right)^{1/2} ,
\end{equation}
while, thanks to Theorem \ref{thm:tightness}, we have 
$$ \|\mathbb{S}^{k, \ell}_N(s,t)\|_{L^{p}(\Omega)} \lesssim \left( \frac{\lfloor N t \rfloor - \lfloor N s \rfloor}{N} \right)$$
for all $k, \ell \in \llbracket 1,m \rrbracket$.
In view of \eqref{eq:exp_dist_incr_rp},
we thereby obtain
\begin{equation}
\label{eq:tightness_est_ta_metric}
\mathbb{E}[d(\boldsymbol{S}_N(s),\boldsymbol{S}_N(t))^{2p}] \lesssim \left( \frac{\lfloor N t \rfloor - \lfloor N s \rfloor}{N} \right)^{p}
\end{equation}
which can be seen as a rough-path enhancement of the estimate \cite[eq. (3.2)]{nourdin_nualart_20}.

Thanks to these estimates we first  establish tightness of the sequence of processes $(\boldsymbol{S}_N)_{N \geq 1}$ in the Skorokhod topology:

\begin{lemma} \label{lem:tightness_skorokhod}
The sequence of processes $(\boldsymbol{S}_N)_{N \geq 1}$ is tight in the $(J_1)$ Skorokhod space $D([0,1],\tilde{T}^{(2)})$, where we recall that $\tilde{T}^{(2)} = \mathbb{R}^m \oplus (\mathbb{R}^m)^{\otimes 2}$. 
\end{lemma}

\begin{proof}
We show that conditions (a) and (b) of \cite[Thm.~7.2]{ethier2005markov} are satisfied. Regarding condition $(a)$, note that, thanks to \eqref{eq:tightness_est_ta_metric} applied with $s =0$, and recalling that $\boldsymbol{S}_N(0) = \boldsymbol{0}$ (the unit element in $\tilde{T}^{(2)}$), we deduce by Chebyshev's inequality that, for all $t>0$ and all $M >0$, $$\mathbb{P}(d(\boldsymbol{0},\boldsymbol{S}_N(t))>M) \leq \frac{t^{p}}{M^{2p}}.$$ 
Since the ball $\{\boldsymbol{x} \in \tilde{T}^{(2)}:d(\boldsymbol{0},\boldsymbol{x}) \leq M\}$ is compact in $\tilde{T}^{(2)}$, we deduce that the sequence of $\tilde{T}^{(2)}$-valued random variables $(\boldsymbol{S}_N(t))_{N \geq 1}$ is tight, i.e.   condition $(a)$ of \cite[Thm.~7.2]{ethier2005markov} holds. For condition $(b)$, we apply \cite[Prop.~3.9]{chevyrev_et_al_homogenisation_part_2} to the process $\boldsymbol{S}_N$. The latter is a piecewise constant process with values in $\tilde{T}^{(2)}$ and with jump times given by $t_j = \frac{j}{N}$, $0 \leq j \leq N$. Furthermore, for $0\leq i<j \leq N$,  thanks to \eqref{eq:tightness_est_ta_metric} we have, 
$$ \|S_N(t_j) - S_N(t_i)\|_{L^{2p}(\Omega)} \lesssim \left( \frac{\lfloor N t_j \rfloor - \lfloor N t_i \rfloor}{N} \right)^{1/2} = (t_j-t_i)^{1/2},$$
as well as
$$ \|\mathbb{S}_N(t_i,t_j)\|_{L^{p}(\Omega)} \lesssim \left( \frac{\lfloor N t_j \rfloor - \lfloor N t_i \rfloor}{N} \right) = (t_j-t_i).$$
By \cite[Prop.~3.9]{chevyrev_et_al_homogenisation_part_2} with $q=p$ and $\beta = 1/2$, we obtain, for all $\alpha \in (\frac{1}{2p}, \frac{1}{2})$ 
$$\mathbb{E}\left[\left| \sup_{i \neq j} \frac{d(\boldsymbol{S}_N(t_i),\boldsymbol{S}_N(t_j))}{|t_i-t_j|^{\alpha-\frac{1}{2p}}} \right|^{2p}\right] \leq C $$
for some constant $C$ which is independent of $N$. Condition $(b)$ of \cite[Prop.~3.9]{chevyrev_et_al_homogenisation_part_2} therefrom follows.
\end{proof}

Finally, we obtain the claimed tightness property:

\begin{corollary} \label{coro:tighness}
	Let $r >2$.
	Then, under the assumptions of Theorem~\ref{thm:tightness}, the sequence of processes $(\boldsymbol{S}_N(t))_{t \in \sbr[0]{0,1}}$ is tight in $\mathcal{D}^{r-\var}(\mathbb{R}^m)$.
\end{corollary}

\begin{proof}
We proceed as in the proof of \cite[Lem.~4.7]{chevyrev_et_al_homogenisation_part_2}. Let $\eps >0$. Since the sequence of processes $(\boldsymbol{S}_N)_{N \geq 1}$ is tight in the $(J_1)$ Skorohod space $D([0,1],\tilde{T}^{(2)})$, by Prohorov's theorem there exists a compact subset $K$ of that space such that 
$$\forall N \geq 1, \qquad \mathbb{P}(\boldsymbol{S}_N \in K) \geq 1 - \eps/2. $$
Now, for all $N \geq 1$, we again apply \cite[Prop.~3.9]{chevyrev_et_al_homogenisation_part_2} to the process $\boldsymbol{S}_N$, with $q=p$ and $\beta = 1/2$ to obtain, 
for any $\alpha \in \del[1]{\frac{1}{2p},\frac{1}{2}}$, 
$$\E\sbr[1]{\|\boldsymbol{S}_N\|_{1/\alpha-\hbox{var}}^{2p}} \leq C $$
for some constant $C>0$ which is independent of $N$. Thanks to Markov's inequality we deduce that, for $R>0$ sufficiently large, for all $N \geq 1$, with probability at least $1-\eps$, $\boldsymbol{S}_N$ belongs to the set
$$\cbr[1]{\boldsymbol{X} \in K: \, \|\boldsymbol{X}\|_{1/\alpha-\var} \leq R}$$
which is a compact subset of $\mathcal{D}^{r-\var}(\mathbb{R}^m)$ for any $r > 1/\alpha$. Since $1/\alpha > 2$ can be chosen arbitrarily close to $2$, the claimed tightness follows.
\end{proof}

\subsection{Proof of moment estimates for tightness}
\label{subsec:proof_tightness_thm}

We will need the following lemma which is a simple application of Hölder's inequality.
\begin{lemma}\label{lem:sum_hoelder}
	Assume that~$\rho \in \ell^n(\Z)$ and that~$a \in \sbr[0]{0, n}$. For~$\alpha \coloneqq \alpha(a) = \frac{n-a}{n}$, we then have
	\begin{equation}
\label{eq:estimate_sum_sho_exp_alpha_var}
\frac{1}{N^\alpha} \sum_{\lfloor N s \rfloor \leq u < \lfloor Nt \rfloor } \abs[0]{\rho(u)}^a \leq \|\rho\|_{\ell^n}^a \, \left(\frac{\lfloor Nt \rfloor - \lfloor Ns \rfloor}{N} \right)^\alpha, 
\end{equation}
for all $t,s \geq 0$ such that $0\leq s<t$.
\end{lemma}

\begin{proof}
	For~$a \in \cbr[0]{0,n}$, the claim is trivial, so we assume that~$a \in \del[0]{0,n}$.
	For~$p \coloneqq \nicefrac{n}{a} > 1$, the conjugate Hölder exponent is
	\begin{equation*}
		p' = \frac{p}{p-1} = \frac{n}{n-a} > 1 \,.		
	\end{equation*}
	Hölder's inequality now implies that
	\begin{equs}
		\frac{1}{N^\alpha} \sum_{\lfloor N s \rfloor \leq u < \lfloor Nt \rfloor} \abs[0]{\rho(u)}^a
		& = 
		\frac{1}{N^\alpha} \norm[1]{\abs[0]{\rho(\cdot)}^a \1_{\{\lfloor N s \rfloor \leq \, \cdot \, < \lfloor Nt \rfloor \}}}_{\ell^1} \\[0.5em]
		& \leq
		\frac{1}{N^\alpha} \norm[1]{\abs[0]{\rho(\cdot)}^a}_{\ell^p} \norm[0]{\1_{\{\lfloor N s \rfloor \leq \, \cdot \, < \lfloor Nt \rfloor \}}}_{\ell^{p'}} \\[0.5em]
		& = 
		\norm[0]{\rho}_{\ell^n}^a \left(\frac{\lfloor Nt \rfloor - \lfloor Ns \rfloor}{N} \right)^{\alpha},
	\end{equs}
	where we have used that~$1/p' = \alpha$.
\end{proof}

We finally proceed to the proof of Theorem~\ref{thm:tightness}.

\begin{proof}[of Theorem~\ref{thm:tightness}]
	
	In order to alleviate the notation and computations, we first assume that $s=0$, $t=1$, and we will show eventually how the computations are modified in the case of general $s,t \geq 0$ with $s \leq t$.
	
	\smallskip
	\noindent
	\emph{Conventions in this proof.}
    Let~$\v, \w \in \llbracket 1,m \rrbracket$; we will use~$\v$ in place of~$k$ and~$\w$ in place of~$\ell$ to free up $k$ and~$\ell$ as parameters for other use. 
	Furthermore, we will write
	\begin{equation*}
		g \coloneqq f_{\v}, \quad 
		h \coloneqq f_{\w}, \quad 
		g_{d}^{(r)} \coloneqq \sbr[0]{\CS_{d} g}^{(r)}, \quad 
		h_{d}^{(r)} \coloneqq \sbr[0]{\CS_{d} h}^{(r)}.
	\end{equation*}
	where~$\CS_{d}$ denotes the~$d$-th Hermite shift operator, see Definition~\ref{d:hermite_shift}.
	
	At first, we use Lemma~\ref{l:hermite_shift_representation} to represent~$g(X^{(\v)}_i)$ and~$h(X^{(\w)}_j)$ as iterated Malliavin divergences and then apply the multiplication formula, Proposition~\ref{p:divergence_multiplication}, to compute the product: 
	\begin{align}
		\thinspace &
		g(X^{(\v)}_i)h(X^{(\w)}_j)
		= 
		\delta^{d}\del[1]{g_{d}(W(e_{\v,i})) e_{\v,i}^{\otimes d}} \delta^{d}\del[1]{h_{d}(W(e_{\w,j})) e_{\w,j}^{\otimes d}} \label{e:product_formula} \\[0.5em]
		= \ & 
		\sum_{(q,r,l) \in \CI_{d}} C_{d,q,r,l} \delta^{2d-q-r}\del[1]{g_{d}^{(r-l)}(X^{(\v)}_i) h_{d}^{(q-l)}(X^{(\w)}_j) e_{\v,i}^{\tilde{\otimes} (d-q)} \tilde{\otimes} e_{\w,j}^{\tilde{\otimes} (d-r)}} \rho_{\v,\w}(j-i)^{q+r-l} \notag
	\end{align}
	where
	\begin{equs}
		\CI_{d} & \coloneqq \CI_{d, d} = \cbr[1]{(q,r,l) \in \N^3: \ q \in \llbracket 0,d \rrbracket, \ r \in \llbracket 0,d \rrbracket, \ l \in \llbracket 0,q \wedge r \rrbracket}, \\
		%%%
		C_{d,q,r,l} & \coloneqq C_{d,d,q,r,l} =
		\binom{d}{q} \binom{d}{r} \binom{q}{l} \binom{r}{l} l!\,.
	\end{equs}
	Note that in~\eqref{e:product_formula}, we have that $r-l, q-l \in \llbracket 0,d \rrbracket$; in particular, the expression only contains instances of~$g_{d}^{(a)}$ for~$a \in \llbracket 0,d \rrbracket$ and likewise for $h_{d}^{(a)}$.
	Therefore, for fixed~$(q,r,l) \in \CI_{d}$, we set ~$K = K(d,q,r) \coloneqq 2d-q-r$ and apply Meyer's inequality, Proposition~\ref{p:meyer_inequality}, to get
	\begin{equs}[][e:tightness_meyer]
		\thinspace & 
		\norm[2]{\delta^{K}\del[2]{\frac{1}{N} \sum_{i < j} g_{d}^{(r-l)}(X^{(\v)}_i) h_{d}^{(q-l)}(X^{(\w)}_j) e_{\v,i}^{\tilde{\otimes} (d-q)} \tilde{\otimes} e_{\w,j}^{\tilde{\otimes} (d-r)}} \rho_{\v,\w}(j-i)^{q+r-l}}_{L^p(\Omega)} \\[0.5em]
		\lesssim  \  &
		\sum_{k = 0}^{K} \norm[2]{D^k\del[2]{\frac{1}{N} \sum_{i < j} g_{d}^{(r-l)}(X^{(\v)}_i) h_{d}^{(q-l)}(X^{(\w)}_j) e_{\v,i}^{\tilde{\otimes} (d-q)} \tilde{\otimes} e_{\w,j}^{\tilde{\otimes} (d-r)}} \rho_{\v,\w}(j-i)^{q+r-l}}_{L^p}
	\end{equs}
	where the last~$L^p$-norm is the one in~$L^p(\Omega; H^{\tilde{\otimes} (K+k)})$.
	
	Next, we observe that by the Leibniz formula (see~\cite[Exercise~2.3.10]{nourdin_peccati_book}), we have
	\begin{equs}[][e:tightness_leibniz]
		D^k \sbr[1]{g_{d}^{(r-l)}(X^{(\v)}_i) h_{d}^{(q-l)}(X^{(\w)}_j)}
		& =
		\sum_{\ell = 0}^k \binom{k}{\ell} \del[1]{D^{\ell} g_{d}^{(r-l)}(X^{(\v)}_i)} \tilde{\otimes} \del[1]{D^{k - \ell} h_{d}^{(q-l)}(X^{(\w)}_j)} \\
		& =
		\sum_{\ell = 0}^k \binom{k}{\ell} g_{d}^{(r-l + \ell)}(X^{(\v)}_i) h_{d}^{(q-l + k- \ell)}(X^{(\w)}_j) \, e_{\v,i}^{\tilde{\otimes} \ell} \tilde{\otimes} e_{\w,j}^{\tilde{\otimes} (k-\ell)}
	\end{equs}
	Since~$k \leq K = 2d-q-r$ and~$\ell \leq k$, as well as~$l \leq q \wedge r$, we emphasise that the derivatives in~\eqref{e:tightness_leibniz} satisfy
	\begin{equs}[][e:parameters_tightness]
		r-l+\ell & \leq r - l + (2d - q - r) = 2d - q - l \leq 2d \, , \\[0.5em]
		q-l + k- \ell & \leq q - l + (2d - q - r) = 2d - r - l \leq 2d \,.
	\end{equs}
	Hence, we require only~$2d$ derivatives for~$g_{d}$ and~$h_{d}$, that is, only~$d$ derivatives for~$g$ and~$h$ by Proposition~\ref{p:estimate_hermite_shift}, as claimed.
	We now plug~\eqref{e:tightness_leibniz} into~\eqref{e:tightness_meyer} and bound the resulting expression by
	\begin{align}
		\thinspace &
		\sum_{k = 0}^{K} \sum_{\ell = 0}^k \binom{k}{\ell} \times \label{e:tightness:bound}\\ 
		%%%
		& \qquad \times 
		\norm[2]{\frac{1}{N} \sum_{i < j} g_{d}^{(r-l + \ell)}(X^{(\v)}_i) h_{d}^{(q-l + k- \ell)}(X^{(\w)}_j) e_{\v,i}^{\tilde{\otimes} (d-q + \ell)} \tilde{\otimes} e_{\w,j}^{\tilde{\otimes} (d-r + k - \ell)} \rho_{\v,\w}(j-i)^{q+r-l}}_{L^p} \notag
	\end{align}
	where the~$L^p$-norm in the previous expression is a shorthand for the norm in~$L^p(\Omega; H^{\tilde{\otimes} (K+k)})$.
	Now we fix~$k \in \llbracket 0,K \rrbracket$ and~$\ell \in \llbracket 0,k \rrbracket$ and use the elementary identity
	\begin{equation*}
		\norm[0]{u}_{L^p(\Omega;H^{\tilde{\otimes} A})}^2
		= 
		\norm[1]{ \norm[0]{u}_{H^{\tilde{\otimes} A}}^2}_{L^{\frac{p}{2}}(\Omega)}
	\end{equation*}
	to bound the square of the norm in~\eqref{e:tightness:bound} as follows:
	\begin{align}
		\thinspace & 
		\norm[2]{\frac{1}{N} \sum_{i < j} g_{d}^{(r-l + \ell)}(X^{(\v)}_i) h_{d}^{(q-l + k- \ell)}(X^{(\w)}_j) e_{\v,i}^{\tilde{\otimes} (d-q + \ell)} \tilde{\otimes} e_{\w,j}^{\tilde{\otimes} (d-r + k - \ell)} \rho_{\v,\w}(j-i)^{q+r-l}}_{L^p}^2 \notag \\[0.5em]
		\leq 
		\ & \frac{1}{N^2} \sum_{i < j} \sum_{\tilde{i} < \tilde{j}} \, \abs[0]{\rho_{\v,\w}(j-i)^{q+r-l} \rho_{\v,\w}(\tilde{j}-\tilde{i})^{q+r-l}} \times \notag\\[0.5em]
		\qquad & \times
		\abs[0]{\scal{e_{\v,i}^{\tilde{\otimes} (d-q + \ell)} \tilde{\otimes} e_{\w,j}^{\tilde{\otimes} (d-r + k - \ell)},e_{\v,\tilde{i}}^{\tilde{\otimes} (d-q + \ell)} \tilde{\otimes} e_{\w,\tilde{j}}^{\tilde{\otimes} (d-r + k - \ell)}}} \times \notag\\[0.5em]
		\qquad & \times \norm[1]{
			g_{d}^{(r-l + \ell)}(X^{(\v)}_i) h_{d}^{(q-l + k- \ell)}(X^{(\w)}_j)
			g_{d}^{(r-l + \ell)}(X^{(\v)}_{\tilde{i}}) h_{d}^{(q-l + k- \ell)}(X^{(\w)}_{\tilde{j}})
		}_{L^{\frac{p}{2}}(\Omega)}
		\label{e:tightness:bound:2}
	\end{align}
	where the scalar product is taken in~$H^{\tilde{\otimes} (K+k)}$.
	By Hölder's inequality with parameters
	\begin{equation*}
		p_i = 2p \quad \text{for} \quad i \in \llbracket 1,4 \rrbracket, \quad \frac{1}{p/2} = \frac{2}{p} = \frac{4}{2p} = \frac{1}{2p} + \frac{1}{2p} + \frac{1}{2p} + \frac{1}{2p}\,,
	\end{equation*}
	we can bound the~$L^{\frac{p}{2}}(\Omega)$-norm in the previous expression uniformly in~$i,j,\tilde{i}, \tilde{j}$ as follows:
	\begin{equs}[][e:est_product_g_h]
		\thinspace &
		\norm[1]{
			g_{d}^{(r-l + \ell)}(X^{(\v)}_i) h_{d}^{(q-l + k- \ell)}(X^{(\w)}_j)
			g_{d}^{(r-l + \ell)}(X^{(\v)}_{\tilde{i}}) h_{d}^{(q-l + k- \ell)}(X^{(\w)}_{\tilde{j}})
		}_{L^{\frac{p}{2}}(\Omega)} \\[0.5em]
		\leq \  &
		\norm[1]{g_{d}^{(r-l + \ell)}}_{L^{2p}(\gamma)} \norm[1]{h_{d}^{(q-l + k- \ell)}}_{L^{2p}(\gamma)}
		\norm[1]{g_{d}^{(r-l + \ell)}}_{L^{2p}(\gamma)} \norm[1]{h_{d}^{(q-l + k- \ell)}}_{L^{2p}(\gamma)} \\[0.5em]
		\lesssim \  &
		\norm[1]{g^{(r-l + \ell-d)}}_{L^{2p}(\gamma)} \norm[1]{h^{(q-l + k- \ell-d)}}_{L^{2p}(\gamma)}
		\norm[1]{g^{(r-l + \ell-d)}}_{L^{2p}(\gamma)} \norm[1]{h^{(q-l + k- \ell-d)}}_{L^{2p}(\gamma)}
	\end{equs}
	By Proposition~\ref{p:estimate_hermite_shift}, the expression on the right hand side is finite by assumption~\ref{thm:tightness:1} in the formulation of the theorem.
	
	It remains to show that the double sum in~\eqref{e:tightness:bound:2} is bounded. 
	We now set~$a \coloneqq d-q+\ell$ and~$b \coloneqq d-r+k-\ell$ to find
	\begin{equs}
		\thinspace &
		\scal{e_{\v,i}^{\tilde{\otimes} (d-q + \ell)} \tilde{\otimes} e_{\w,j}^{\tilde{\otimes} (d-r + k - \ell)},e_{\v,\tilde{i}}^{\tilde{\otimes} (d-q + \ell)} \tilde{\otimes} e_{\w,\tilde{j}}^{\tilde{\otimes} (d-r + k - \ell)}}_{H^{\tilde{\otimes} (K+k)}} \\[0.5em]
		\ = \ &
		\sum_{u = 0}^{a \wedge b}  
		\frac{a! b!}{u! (a+b)!}
		\rho_{\v,\w}(\tilde{j}-i)^u \rho_{\v,\v}(\tilde{i}-i)^{a-u} 
		\rho_{\w,\v}(\tilde{i}-j)^{u} \rho_{\w,\w}(\tilde{j}-j)^{b-u} \,.
	\end{equs}
	We will henceforth ignore the combinatorial prefactor and bound the corresponding expression in~\eqref{e:tightness:bound:2} as follows:
	\begin{equs}[][e:tightness:deterministic]
		\sum_{u = 0}^{a \wedge b} \frac{1}{N^2} \sum_{i < j} \sum_{\tilde{i} < \tilde{j}} \, & \abs[0]{\rho_{\v,\w}(j-i)}^{q+r-l} \abs[0]{\rho_{\v,\w}(\tilde{j}-\tilde{i})}^{q+r-l} 
		\abs[0]{\rho_{\v,\w}(\tilde{j}-i)}^u \times \\
		%%%
		\times \ & 
		\abs[0]{\rho_{\v,\v}(\tilde{i}-i)}^{a-u} 
		\abs[0]{\rho_{\w,\v}(\tilde{i}-j)}^{u} \abs[0]{\rho_{\w,\w}(\tilde{j}-j)}^{b-u}
	\end{equs}
	Before we start to bound the expression in~\eqref{e:tightness:deterministic}, let us observe the following: 
	\begin{itemize}
		\item 
		Since~$\ell \geq 0$ and~$k - \ell \geq 0$, it suffices to consider the respective worst case scenarios~$\ell = 0$ and $k =\ell$. 
		In this case, the corresponding expressions for~$a$ and~$b$ read: $a = d - q$, $b = d - r$.
		
		If~$\ell > 0$ or~$k-\ell > 0$, for~$\mathtt{c},\mathtt{d} \in \cbr[0]{\v,\w}$ one can bound the corresponding power of~$\abs[0]{\rho_{\mathtt{c},\mathtt{d}}(\ldots)}$ by~$1$. (Recall that we have assumed~$\rho_{\mathtt{c},\mathtt{c}}(0)=1$, so this bound follows from Cauchy--Schwarz.)
		\item The resulting expression is symmetric in~$q$ and~$r$. W.l.o.g., we assume that~$r \leq q$ which enforces the constraint~$l \leq q \wedge r = r$.
		
		We will focus on the case where~$l = r$.
		This is the hardest because the exponent $q+r-l = q$ is the smallest in that situation.
		\item
		As $q \geq r$,  considering the worst cases $\ell = 0$ and $k - \ell = 0$ we may assume that $a = d_\star -q \leq d_\star - r = b$.
	\end{itemize}
	In view of these bullet points, the expression in~\eqref{e:tightness:deterministic} can be upper bounded as follows:
	\begin{equs}[][e:tightness:deterministic_2]
		\sum_{u = 0}^{d-q} \frac{1}{N^2} \sum_{i < j} \sum_{\tilde{i} < \tilde{j}} \, & \abs[0]{\rho_{\v,\w}(j-i)}^{q} \abs[0]{\rho_{\v,\w}(\tilde{j}-\tilde{i})}^{q} 
		\abs[0]{\rho_{\v,\w}(\tilde{j}-i)}^u \times \\
		%%%%
		\times & \abs[0]{\rho_{\v,\w}(\tilde{i}-i)}^{d-q-u} 
		\abs[0]{\rho_{\w,\v}(\tilde{i}-j)}^{u} \abs[0]{\rho_{\w,\w}(\tilde{j}-j)}^{d-q-u}
	\end{equs}
	In fact, this expression almost equals the one in~\eqref{e:tightness:deterministic} subject to the assumptions in the bullet points; we have only upper bounded 
	\begin{equation*}
		\abs[0]{\rho_{\w,\w}(\tilde{j}-j)}^{b-u} = \abs[0]{\rho_{\w,\w}(\tilde{j}-j)}^{d-r-u} \leq \abs[0]{\rho_{\w,\w}(\tilde{j}-j)}^{d-q-u} \,.
	\end{equation*}
	\smallskip
	\noindent
	\emph{Convention.}
	From here onwards, we will not distinguish the covariance functions~$\rho_{\v,\v}$, $\rho_{\v,\w}$, $\rho_{\w,\v}$, and~$\rho_{\w,\w}$ any more and simply write~$\rho$; otherwise, the resulting expressions become completely unwieldy. 
	One can check that the remainder of this proof only ever uses the integrability assumption~$\rho \in \ell^{d}$, and nothing else, so this abuse of notation does not cause any problems.
	
	We will now forget the summation constraints~$i < j$ and~$\tilde{i} < \tilde{j}$ in~\eqref{e:tightness:deterministic_2} and re-order the summation.
	For fixed~$u \in \llbracket 0,d-q\rrbracket$, we bound the inner sum in ~\eqref{e:tightness:deterministic_2} from above by the following expression:
	\begin{equation} \label{e:tightness:deterministic_3}
		\frac{1}{N^2}
		\sum_{j} 
		\underbrace{\sum_{\tilde{j}}
			\abs[0]{\rho_j(\tilde{j})}^{d-q-u}}_{\eqqcolon C}
		\underbrace{\sum_{\tilde{i}} \abs[0]{\rho_{\tilde{j}}(\tilde{i})}^{q} 
			\abs[0]{\rho_j(\tilde{i})}^{u}}_{\eqqcolon B}
		\underbrace{\sum_{i}\abs[0]{\rho_j(i)}^{q} \abs[0]{\rho_{\tilde{i}}(i)}^{d-q-u} 
			\abs[0]{\rho_{\tilde{j}}(i)}^u}_{\eqqcolon A}
	\end{equation}
	where we have used the notation~$\rho_{v} \coloneqq \rho(\cdot - v)$. Recall that the summations in~$i$, $\tilde{i}$, $j$, and~$\tilde{j}$ are over the set~$\llbracket -N, N \rrbracket$.
	We now set 
	\begin{equation*}
		\beta \coloneqq \frac{d-(q+u)}{d}, \quad \gamma \coloneqq \frac{d-(d-q-u)}{d} = \frac{q+u}{d}, \quad \beta + \gamma = 1 
	\end{equation*}
	and apply Lemma~\ref{lem:sum_hoelder} to show that
	\begin{enumerate}[label=(\roman*)]
		\item \label{e:tightness:i} $A < \infty$ uniformly in $j$,$\tilde{i}$, and~$\tilde{j}$, i.e. there exists a constant~$c > 0$ s.t. $A \leq c$ for all~$j$,$\tilde{i}$, and~$\tilde{j}$, and
		\item \label{e:tightness:ii} $\frac{1}{N^\beta} B < \infty$ uniformly in $j$ and~$\tilde{j}$, and
		\item \label{e:tightness:iii} $\frac{1}{N^\gamma} C < \infty$ uniformly in $j$.
	\end{enumerate}
	Once we have verified these assertions, the claim then follows because~$\beta + \gamma = 1$, i.e. we have one inverse power of~$N$ left to bound~$\frac{1}{N} \sum_{j} 1 \lesssim 1$.
	
	\vspace{0.5em}
	\noindent
	$\triangleright$ For the assertion in~\ref{e:tightness:i}, we set
	\begin{equation*}
		p_1 \coloneqq \frac{d}{q}, \quad p_2 \coloneqq \frac{d}{u}, \quad p_3 \coloneqq \frac{d}{d-q-u}, \quad \frac{1}{p_1} + \frac{1}{p_2} + \frac{1}{p_3} = 1 \,.
	\end{equation*}
	with the understanding that~\enquote{$d/0 \coloneqq \infty$} (i.e., if~$q$, $u$, or $d-q-u$ is zero, the corresponding~$p_i$ equals infinity).  
	By Hölder's inequality for multiple products, we then find
	\begin{equation*}
		A 
		\leq \norm[0]{\rho_j}_{\ell^{d}}^q \norm[0]{\rho_{\tilde{j}}}^u_{\ell^{d}} \norm[0]{\rho_{\tilde{i}}}^{d-q-u}_{\ell^{d}}
	\end{equation*}
	and the claim follows because~$\norm[0]{\rho_v}_{\ell^{d}} \leq \norm[0]{\rho}_{\ell^{d}}$ for each fixed~$v \in \Z$.
	
	\vspace{0.5em}
	\noindent
	$\triangleright$ For the assertion in~\ref{e:tightness:ii}, we set 
	\begin{equation*}
		p \coloneqq \frac{d}{q}, \quad p' \coloneqq \frac{p}{p-1} = \frac{d}{d-q}
	\end{equation*}
	and then, by Hölder's inequality, obtain the following estimate:
	\begin{equs}
		\frac{1}{N^\beta} B 
		& =  
		\frac{1}{N^\beta} \norm[1]{\abs[0]{\rho_{\tilde{j}}}^q \abs[0]{\rho_{j}}^u}_{\ell^1}
		\leq
		\norm[1]{\abs[0]{\rho_{\tilde{j}}}^q}_{\ell^{p}} \frac{1}{N^\beta}  \norm[1]{\abs[0]{\rho_{j}}^u}_{\ell^{p'}}
		=
		\norm[1]{\rho_{\tilde{j}}}_{\ell^{d}}^q \frac{1}{N^\beta} \del[3]{\sum_{k} \abs[0]{\rho_{j}(k)}^{\frac{u d}{d-q}}}^{\frac{d-q}{d}} \\[0.5em]
		& = 
		\norm[1]{\rho_{\tilde{j}}}_{\ell^{d}}^q \del[3]{\frac{1}{N^{\bar{\beta}}} \sum_{k} \abs[0]{\rho_{j}(k)}^{\frac{u d}{d-q}}}^{\frac{d-q}{d}}
	\end{equs}	
	where
	\begin{equation*}
		\bar{\beta} \coloneqq \beta \cdot \frac{d}{d-q} = \frac{d-(q+u)}{d} \cdot \frac{d}{d-q} 
		= \frac{d - \frac{u d}{d-q}}{d} \,. 
	\end{equation*}
	The assertion in~\ref{e:tightness:ii} now follows by Lemma~\ref{lem:sum_hoelder} with~$a \coloneqq \frac{u d}{d-q}$ since~$\beta = \frac{d-a}{d} = \alpha(a)$. 
	
	\vspace{0.5em}
	\noindent
	$\triangleright$ The assertion in~\ref{e:tightness:iii} is a direct consequence of Lemma~\ref{lem:sum_hoelder} with~$a \coloneqq d-q-u$ since~$\gamma = \alpha(a)$. 
	
	\vspace{0.5em}
	The claim of the theorem now follows by the combination of~\eqref{e:tightness_meyer} to~\eqref{e:tightness:deterministic_3} because
	\begin{equation*}
		\abs[0]{\CI_{d}} < \infty, \quad \ell \leq k \leq K \leq 2d, \quad a \vee b \leq 3d 
	\end{equation*}
	and all of those bounds are independent of~$N$; we therefore get $\|\mathbb{S}^{k,\ell}_N(0,1)\|_{L^p(\Omega)} \lesssim_{d,p} 1$ as requested. 
	
    The above covers the case where $s=0$ and $t=1$. For general $s,t \geq 0$ with $0 \leq s \leq t$, we use exactly the same bounds as above, only now the bound for the contribution from the correlations is as in \eqref{e:tightness:deterministic_3}, but with indices $j,\tilde{j},i,\tilde{i}$ running through the set $\llbracket \lfloor Ns \rfloor, \lfloor Nt \rfloor-1 \rrbracket$. 
    The factors $A$,$B$,
    $C$ are then bounded as follows: 
    \begin{enumerate}[label=(\roman*)]
		\item \label{e:tightness:i_bis} $A < \infty$ uniformly in $j$, $\tilde{i}$, $\tilde{j}$, $N$, $s$, and $t$, i.e. there exists a constant~$c > 0$ s.t. $A \leq c$ for all $j$, $\tilde{i}$, $\tilde{j}$, $N$, $s$, and $t$.
		\item \label{e:tightness:ii_bis} $\frac{1}{N^\beta} B \lesssim  \left(\frac{\lfloor Nt \rfloor - \lfloor Ns \rfloor}{N} \right)^\beta$ uniformly in $j$ and~$\tilde{j}$, and
		\item \label{e:tightness:iii_bis} $\frac{1}{N^\gamma} C  \lesssim  \left(\frac{\lfloor Nt \rfloor - \lfloor Ns \rfloor}{N} \right)^{\gamma}$ uniformly in $j$.
	\end{enumerate}
    Since $\beta + \gamma =1$, we finally obtain the requested bound.
\end{proof}

\begin{remark}[On the parameter~$p$] \label{rmk:tightness_p}
	The attentive reader might wonder whether the previous proof truly requires the suboptimal assumption~$p \geq 2$, cf. Remark~\ref{rmk:main}, Point~\ref{rmk:main:integrability}. 
	The answer to that question is quite subtle. 
	First, note that the assumption~$p > 1$ is crucial in Lemma~\ref{lem:tightness_skorokhod} and Corollary~\ref{coro:tighness}: It is required for the Kolmogorov-type theorem in~\cite[Prop.~3.9]{chevyrev_et_al_homogenisation_part_2}. 
	This dovetails nicely with the fact that Meyer's inequality, as applied in~\eqref{e:tightness_meyer}, is valid for~$p > 1$, but fails for~$p=1$.
	However, the previous step---the application of the product formula (see Proposition~\ref{p:divergence_multiplication}) in~\eqref{e:product_formula}---already requires~$p \geq 2$. 
	It seems plausible to expect that the assumption in Proposition~\ref{p:divergence_multiplication}, i.e. that~$f,g \in \DD^{n_1+n_2,2p}$ for~$p \geq 2$, can be relaxed to just require~$p > 1$---but not to~$p \geq 1$, because its proof, again, requires Meyer's inequality. 
	Since we require the assumption~$p \geq 2$ for the f.d.d. convergence anyway, cf. Remark~\ref{rmk:failure_tightness_strategy}, we will not follow this line of thought further.
\end{remark}

\section{Convergence of finite-dimensional distributions} \label{s:convergence_fdd}

After we have established tightness in the previous section, it remains to show the convergence of the finite-dimensional distributions.
The main result in this section is the following theorem:

\begin{theorem}[Convergence of f.d.d.] \label{thm:conv_fdd}
	Let~$d \geq 1$ and assume the following conditions:
	\begin{enumerate}[label=(\arabic*)]
		\item \label{thm:conv_fdd:1}
		For any~$k \in \llbracket 1,m \rrbracket$, $f_k$ has Hermite rank at least~$d$ and satisfies~$f_k \in \mathbb{D}^{d,4}(\gamma)$.
		\item \label{thm:conv_fdd:2}
		For all~$k, \ell \in \llbracket 1,m \rrbracket$, we have~$\sum_{i \in \Z} \abs[0]{\rho_{k,\ell}(i)}^{d} < \infty$, i.e.~$\rho_{k,\ell} \in \ell^{d}(\Z)$.
	\end{enumerate}
	Then, the finite-dimensional distributions of~$\boldsymbol{S}_N = (S_N, \S_N)$ converge to those of a Brownian rough path~$\boldsymbol{B} = (B,\mathbb{B})$ with characteristics~$(\Sigma,\Gamma)$ given in~\eqref{e:Sigma} and~\eqref{e:Gamma}, respectively.
\end{theorem}

\subsection{Overview and strategy} \label{s:strategy}

Recall from~\eqref{e:truncated_tensor_algebra} that~$\tilde{T}^{(2)}(\R^m) = \R^m \oplus (\R^m \otimes \R^m)$.
We have to show that the finite-dimensional distributions of~$\boldsymbol{S}_N$ converge to those of~$\boldsymbol{B}$, the latter characterised by~$\Sigma$ and~$\Gamma$.
In other words, we have to show that for any~$l \in \N$ and any sequence of intervals~$(s_i,t_i)_{i = 1}^l \subseteq [0,1]^2$ with~$s_i < t_i$, the convergence 
\begin{equation} \label{e:goal_fdd_conv}
	\CA_N \to \CL \quad \text{in law in} \quad \del[1]{\tilde{T}^{(2)}(\R^m)}^{l} \quad \text{as} \quad N \to \infty
\end{equation}
holds with 
\begin{align}
	\CA_N & \coloneqq (\boldsymbol{S}_N(s_1,t_1), \ldots, \boldsymbol{S}_N(s_l,t_l)) = \del[1]{S_N(s_1,t_1), \mathbb{S}_N(s_1,t_1), \ldots, S_N(s_l,t_l), \mathbb{S}_N(s_l,t_l)} \label{def:AN} \\[0.5em]
	\CL & \coloneqq (\boldsymbol{B}(s_1,t_1), \ldots, \boldsymbol{B}(s_l,t_l)) = \del[1]{B(s_1,t_1), \mathbb{B}(s_1,t_1), \ldots, B(s_l,t_l), \mathbb{B}(s_l,t_l)} \label{def:L}
\end{align}
Let us now briefly explain our strategy which, in full detail, will be implemented in Sections~\ref{s:reduction} to~\ref{s:removing_chaos_cutoff} below. 
The starred parts are those where most of the work is required.  
\begin{enumerate}
	\setcounter{enumi}{-1} 
	\item \label{strategy:0} The vector~$\CL$ is composed of elements in the first and second (inhomogeneous) Wiener--It\^{o} chaos. As such, it is \emph{moment-determined.}
	The problem, however, is that the functions~$f_k$ defining~$S_N$ and~$\S_N$, a priori, have components \emph{in every chaos}.
	Therefore, we cannot rely on hypercontractivity to infer that~$S_N(s,t)$ and~$\S_N(s,t)$ actually have moments of all orders. 
	In particular, the assumption~$f_k \in L^4(\gamma)$ (as implied by~$f_k \in \DD^{d,4}(\gamma)$) is non-trivial.   
	\item \label{strategy:1} \textbf{$^\star$Reduction}. In a first step, we introduce a \emph{chaos truncation} parameter~$M$ and, thereby, reduce the problem to a situation which only involves finite chaos expansions.
	More precisely, we decompose 
	\begin{equation*}
		\CA_N = \CA_N^{M} + \CR_N^M 
	\end{equation*}
	where 
	\begin{align}
		\CA_N^M
		& \coloneqq 	
		\del[1]{S^{M}_N(s_1,t_1), \mathbb{S}^{M}_N(s_1,t_1), \ldots, S^{M}_N(s_l,t_l), \mathbb{S}^{M}_N(s_l,t_l)} \label{e:def_ANM}\\[0.5em]
		%%%
		\CR_N^M 
		& \coloneqq	
		\del[1]{S^{> M}_N(s_1,t_1), \mathbb{S}^{> M}_N(s_1,t_1), \ldots, S^{> M}_N(s_l,t_l), \mathbb{S}^{> M}_N(s_l,t_l)} \label{e:def_RNM}
	\end{align}
Roughly speaking, the vector~$\CA_N^M$ contains all chaos components up to~$M$ and the remainder term~$\CR_N^M$ all those above~$M$, see Section~\ref{s:reduction} for their precise definition.  
Most importantly, it will be shown in Proposition~\ref{p:reduction} below that~$\CR_N^M \to 0$ in~$L^2(\P)$ in a suitable joint limit of taking, first,~$N \to \infty$, and then~$M \to \infty$. ($\triangleright$ Section~\ref{s:reduction})
\item \label{strategy:2} \textbf{LLN on the diagonal.} After the reduction step, the analysis proceeds at a \emph{fixed truncation level}~$M$ with the term~$\CA_N^M$.
With  Remark \ref{rmk:general}~\ref{rmk:general:5} in mind, we only expect corrections to the Stratonovich rough path coming from the \emph{symmetric part}: 
In order to see them, we~\enquote{add in the diagonal} which is, itself, symmetric and write 
\begin{equation*}
	\CA_N^M = \bar{\CA}_N^M - \frac{1}{2} \CD_N^M
\end{equation*}
where
	\begin{align}
		\bar{\CA}_N^M
		& \coloneqq 	
		\del[1]{S^{M}_N(s_1,t_1), \mathbb{Y}^{M}_N(s_1,t_1), \ldots, S^{M}_N(s_l,t_l), \mathbb{Y}^{M}_N(s_l,t_l)} \label{e:def_barANM}\\[0.5em]
		%%%
		\CD_N^M
		& \coloneqq  
		\del[1]{0, \mathbb{D}^{M}_N(s_1,t_1), \ldots, 0, \mathbb{D}^{M}_N(s_l,t_l)}
		\notag
	\end{align}
for
\begin{equation} \label{e:def_DNM}
	\mathbb{D}_N^M(s,t) \coloneqq  \frac{1}{N} \sum_{\lfloor Ns \rfloor \leq i < \lfloor N t \rfloor} \vec{f}^M(X_i) \otimes  \vec{f}^M(X_i), \quad
	%%%
	\mathbb{Y}_N^M(s,t) \coloneqq \mathbb{S}_N^M(s,t) + \frac{1}{2} \DD_N^M(s,t) \,.
\end{equation}	
In Proposition~\ref{p:LLN_diagonal} below, we will show that 
\begin{equation*}
	\DD_N^M(s,t) \to (t-s)\Delta^M(0) \quad \text{in} \quad  L^2(\P) \quad \text{as} \quad N \to \infty \,,
\end{equation*}
where~$\Delta^M$ is defined like~$\Delta$ in~\eqref{e:Delta}, but with~$\vec{f}$ replaced by~$\vec{f}^M$.
The previous convergence can be seen as a \emph{law of large numbers on the diagonal}. ($\triangleright$ Section~\ref{s:LLN_diagonal})
\item \label{strategy:3} \textbf{Relation to moment convergence.} 
In the previous step, we have correctly identified the correction term---but there is another benefit: Instead of looking at the piecewise constant interpolation~$S_N^M$, we may now look at its piecewise linear counterpart~$Y_N^M$ to write 
	\begin{equation} \label{e:strategy_pcw_linear}
		\mathbb{Y}_N^M(s,t) = \int_{\lfloor Ns \rfloor}^{\lfloor Nt \rfloor} \dif Y_N^M(r) \otimes \dif Y_N^M(r) \,.
	\end{equation}
This leads to the decomposition
\begin{equation*}
	\bar{\CA}_N^M = \CB_N^M + \CC_N^M 
\end{equation*}
with 
	\begin{align}
		\CB_N^M
		& \coloneqq 	
		\del[1]{Y^{M}_N(s_1,t_1), \mathbb{Y}^{M}_N(s_1,t_1), \ldots, Y^{M}_N(s_l,t_l), \mathbb{Y}^{M}_N(s_l,t_l)} \,,\label{e:def_BNM}\\[0.5em]
		%%%
		\CC_N^M 
		& \coloneqq 
		\del[1]{S^{M}_N(s_1,t_1) - Y^{M}_N(s_1,t_1), 0, \ldots, S^{M}_N(s_l,t_l) - Y^{M}_N(s_l,t_l), 0} \,. \label{e:def_CNM}
		%%%
	\end{align}	
	We will show in Lemma~\ref{l:ptwise_L2_conv} below that 
	\begin{equation*}
		\CC_N^M \to 0 \quad \text{in} \quad L^2(\P) \quad \text{as} \quad N \to \infty \,.
	\end{equation*}
	The remaining goal is now to show that, for fixed~$M$ and as~$N \to \infty$, that   
	\begin{equation} \label{e:conv_BNM}
		\CB_N^M \Rightarrow \CB^M \coloneqq (B^{M}(s_1,t_1), \bar{\mathbb{B}}^{M}(s_1,t_1), \ldots, B^{M}(s_l,t_l), \bar{\mathbb{B}}^{M}(s_l,t_l))
	\end{equation}
	where\footnote{The matrices~$\Sigma^M$ and~$\AA^M$ are defined like their counterparts in~\eqref{e:Sigma} and~\eqref{e:AA}, respectively, but with~$\vec{f}$ replaced by~$\vec{f}^M$ in~\eqref{e:Delta} and~\eqref{e:Gamma}, respectively. Recall that~$\vec{f}^M$ means that all component functions~$f_k$ of~$\vec{f}$ have been projected onto chaoses of order~$M$ and lower, i.e. we replace~$f_k \rightsquigarrow f_k^M$.}	
	\begin{itemize}
		\item $B^M$ is an $m$-dimensional Brownian motion with \emph{truncated} covariance matrix~$\Sigma^M$, and
		\item $\bar{\mathbb{B}}^M$ its Stratonovich lift corrected by the antisymmetric matrix~$\AA^M$, i.e.~$\bar{\mathbb{B}}^M = \mathbb{B}^{M,\text{\tiny Strat}} + \AA^M$. 
	\end{itemize}
	The benefit of the representation~\eqref{e:strategy_pcw_linear} is that it immediately allows to build the complete \emph{rough path} by setting
	\begin{equs}[][e:strategy_higher_levels]
		\boldsymbol{Y}^{M;(k)}_N(s,t)
		& \coloneqq \int_{\Delta_k(\lfloor Ns \rfloor,\lfloor Nt \rfloor)} \dif Y^M_N(r_1) \otimes \ldots \otimes \dif Y^M_N(r_k), \quad k \geq 1 \,, 
	\end{equs}
	such that 
	\begin{equation*}
		\YY_N^{M;(2)}(s,t) = \mathbb{Y}_N^M(s,t), 
		\quad 
		\YY_N^{M;(1)}(s,t) = Y_N^M(s,t) \,.
	\end{equation*}
	The point is this: 
	Via the shuffle product relation as well as Chen's identity, we can then relate the convergence in~\eqref{e:conv_BNM} to that of certain moments of~$\YY_N^M$.
	More precisely,~\eqref{e:conv_BNM} follows once we show that, for 
		\begin{equation*}
			0 \leq s_1 < t_1 < s_2 < t_2 \ldots < s_l < t_l \leq 1\,,
		\end{equation*}
		and words~$\w_1, \ldots, \w_\ell$ of arbitrary length, we have
		\begin{equation} \label{e:strategy:moment_convergence}
			\lim_{N \to \infty} \E\sbr{\prod_{\ell=1}^{l}  \scal{\boldsymbol{Y}^M_N(s_\ell,t_\ell),\w_\ell}}
			=
			\prod_{\ell=1}^{l}  \E\sbr[1]{\scal{\boldsymbol{\bar{B}}^M(s_\ell,t_\ell),\w_\ell}} \,.
		\end{equation}
		where~$\boldsymbol{\bar{B}}^M$ is the full rough path lift of~$(B^M,\mathbb{\bar B}^M)$ defined similarly to~\eqref{e:strategy_higher_levels}.
		This is the contents of Theorems~\ref{thm:conv_fdd_stratonovich} and~\ref{t:product_case} below ($\triangleright$ Section~\ref{s:limit_exp_signature})
	\item \label{strategy:4} \textbf{$^\star$Convergence of the moments of~$\boldsymbol{B_N^M}$.}
	Most of the remaining work lies in sho\-wing~\eqref{e:strategy:moment_convergence}. 
	We first deal with the case~$l = 1$.
	In that case, for a word~$\v = \v_{1:k}$, the right hand side of~\eqref{e:strategy:moment_convergence} can be computed via a suitable amendment of Fawcett's theorem (see Lemma~\ref{s:limit_exp_signature}), namely:
	\begin{equation} \label{e:strategy_fawcett}
	\E\sbr[1]{\scal{\boldsymbol{\bar{B}}^M(s,t),\v}}
	=
	\1_{k = 2n}
	\frac{(t-s)^{2n}}{2^n n!} \prod_{j=1}^n \del[0]{\Sigma^M_{\v_{2j-1},\v_{2j}} + 2 \AA^M_{\v_{2j-1},\v_{2j}}}
	\end{equation}
	In order to show that the left hand side of~\eqref{e:strategy:moment_convergence} for~$l = 1$ converges to~\eqref{e:strategy_fawcett}, the first step is to \enquote{un-do} the conversion to the continuous case and explicitly compute	
	\begin{equation}
		\boldsymbol{Y}^{M;(k)}_N(s_,t) 
		=
		\frac{1}{N^{k/2}}
		\sum_{i_{1:k} \in \bar{\triangle}^{(k)}_{\lfloor Ns \rfloor, \lfloor Nt \rfloor}}
		\fw(i_{1:k})
		\bigotimes_{\ell=1}^k \vec{f}^M(X_{i_\ell})
	\end{equation}
	where~$\fw(i_{1:k})$ is a certain weight factor and 
	\begin{equation*}
		\bar{\triangle}^{(k)}_{\lfloor Ns \rfloor, \lfloor Nt \rfloor}
		=
		\cbr[0]{i_{1:k}: \lfloor Ns \rfloor \leq i_1 \leq \ldots \leq i_k \leq \lfloor Nt \rfloor - 1} \,.
	\end{equation*}
	This is necessary so that we can make rigorous use of the Breuer--Major argument, Corollary~\ref{coro:BM83} (with~$A_{k,N}(s,t) = \bar{\triangle}^{(k)}_{\lfloor Ns \rfloor, \lfloor Nt \rfloor}$, see the comments in the end of Remark~\ref{rmk:breuer_major}), which leads to the following asymptotic relation as~$N \to \infty$ for the word~$\v = \v_{1:k}$:
	\begin{align}
		\thinspace & 
		\E\sbr[1]{\scal{\YY_N^{M; (k)}(s,t),\v}} \label{e:strategy_breuer_major}\\
		\asymp & \ 
		\frac{\1_{k=2n}}{N^{n}}
		\sum_{p=1}^{2n}
		\sum_{i_{1:2n} \in \bar{\triangle}^{(2n)}_{\lfloor Ns \rfloor, \lfloor Nt \rfloor}} \fw(i_{1:2n})
		\sum_{P\in \CM(2n)}
		\prod_{j=1}^{n}
		\Delta^M_{\v_{P_{2j-1}},\v_{P_{2j}}}\!\del[1]{i_{P_{2j}}-i_{P_{2j-1}}}  \notag \,.
	\end{align}
	The analysis is then three-fold:
	\begin{itemize}
		\item In Proposition~\ref{p:higher_order_diagonals}, we show that any multi-index~$i_{1:2n}$ in~\eqref{e:strategy_breuer_major} which forms blocks of sizes~$3$ and higher~(i.e.~$i_{\ell_1} = \ldots = i_{\ell_r}$ for~$r \geq 3$ and~$i_{\ell_j} \neq i_{\ell_n}$ when~$j \neq n$), gives rise to an asymptotically negligible contribution. 
		In other words: These \emph{higher-order diagonals} asymptotically vanish. ($\triangleright$ Section~\ref{s:higher_diagonals})
		\item We perform a detailed analysis of the pairings~$P \in \CM(2n)$. 
		First, denoting by~$P_\star \coloneqq \cbr[0]{\cbr[0]{1,2}, \cbr[0]{3,4}, \ldots, \cbr[0]{2n-1,2n}}$ the \emph{ladder pairing}, we show that any other pairing~$P \neq P_\star$ in~\eqref{e:strategy_breuer_major} gives rise to an asymptotically vanishing contribution, see Proposition~\ref{p:non_ladder_vanishes}. ($\triangleright$ Section~\ref{s:analysis_pairings})
		\item The sum that remains to analyse is the one in~\eqref{e:strategy_breuer_major}, restricted to~$P = P_\star$ and all multi-indices that form blocks of sizes~$1$ or~$2$: 
		In Proposition~\ref{p:limit_ladder}, we show that its limit is given precisely by the term on the right hand side of~\eqref{e:strategy_fawcett}.
		On a technical level, we will split the analysis into~\emph{bridging} and~\emph{non-bridging} multi-indices, see~\eqref{e:non_bridging_multiindex} for a visualisation; in particular, we will see that only non-bridging multi-indices contribute non-trivially to the limit. ($\triangleright$~Section~\ref{s:analysis_pairings})
	\end{itemize}
	We will then reduce the general case,~$l \geq 1$, to the case~$l=1$ by showing that pairings across simplices (where each simplex is associated with an interval~$[s_\ell,t_\ell]$), vanish asymptotically (Lemma~\ref{l:cross_simplex_pairings}). 
	Finally, all the previous arguments are combined in the proof of~\eqref{e:strategy:moment_convergence} in Section~\ref{s:conv_moments}. 
		\item \label{strategy:5} \textbf{Removing the chaos cutoff.} In the final step, we remove the chaos cutoff and show that, as~$M \to \infty$, we have
		\begin{equation*}
			\Delta^M(0) \to \Delta(0), 
			\quad 
			B^M(s,t) \to B(s,t), 
			\quad 
			\bar{\mathbb{B}}^M(s,t) \to \bar{\mathbb{B}}(s,t) \quad \text{in} \quad L^2(\P) \,
		\end{equation*}
	which finishes the analysis. ($\triangleright$ Section~\ref{s:removing_chaos_cutoff}) 
\end{enumerate}
In summary, the previous strategy leads to the following decomposition of the vector~$\CA_N^M$ as introduced in~\eqref{e:def_ANM}:
\begin{equation} \label{e:decomp_AN}
	\CA_N = \del[1]{S_N(s_1,t_1), \mathbb{S}_N(s_1,t_1), \ldots, S_N(s_l,t_l), \mathbb{S}_N(s_l,t_l)} = \CR_N^M + \CB_N^M + \CC_N^M - \frac{1}{2} \CD_N^M 
\end{equation}
where
\begin{align}
		\CR_N^M 
		& \coloneqq	
		\del[1]{S^{> M}_N(s_1,t_1), \mathbb{S}^{> M}_N(s_1,t_1), \ldots, S^{> M}_N(s_l,t_l), \mathbb{S}^{> M}_N(s_l,t_l)} \label{e:def_RNM}\\[0.5em] 
		%%%
		\CB_N^M
		& \coloneqq 	
		\del[1]{Y^{M}_N(s_1,t_1), \mathbb{Y}^{M}_N(s_1,t_1), \ldots, Y^{M}_N(s_l,t_l), \mathbb{Y}^{M}_N(s_l,t_l)} \label{e:def_BNM}\\[0.5em]
		%%%
		\CC_N^M 
		& \coloneqq 
		\del[1]{S^{M}_N(s_1,t_1) - Y^{M}_N(s_1,t_1), 0, \ldots, S^{M}_N(s_l,t_l) - Y^{M}_N(s_l,t_l), 0} \label{e:def_CNM} \\[0.5em]
		%%%
		\CD_N^M
		& \coloneqq  
		\del[1]{0, \mathbb{D}^{M}_N(s_1,t_1), \ldots, 0, \mathbb{D}^{M}_N(s_l,t_l)} \label{e:def_DNM}
\end{align}

We will rely on the following tailor-made statement, of Slutsky type, about convergence in law:
\begin{proposition} \label{prop:abstract_convergence_statement}
	Let~$\CA_N = \CR_N^M + \CB_N^M + \CC_N^M + \CD_N^M$ where all the sequences of random variables take values in a separable Banach space.
	Further, assume that, for all fixed~$M \geq d$, the following conditions hold:
	\begin{enumerate}
		\myitem{(R1)} \label{ass:R1} The limes superior~$r^M \coloneqq \limsup_{N \to \infty} \norm[0]{\CR_N^M}^2_{L^2(\P)}$ exists.
		\myitem{(B1)} \label{ass:B1} The sequence~$(\CB_N^M)_N$ has a limit \emph{in law} to~$B^M$ as~$N \to \infty$.
		\myitem{(C1)} \label{ass:C1} The sequence~$(\CC_N^M)_N$ converges to~$0$ \emph{in~$L^2(\P)$} as~$N \to \infty$.
		\myitem{(D1)} \label{ass:D1} The sequence~$(\CD_N^M)_N$ has a \emph{deterministic} limit~$\CD^M$ \emph{in probability} as~$N \to \infty$. 
	\end{enumerate}
	Assume further that, as~$M \to \infty$: 
	\begin{enumerate}
		\myitem{(R2)} \label{ass:R2} The deterministic sequence~$r^M$ converges to~$0$.
		\myitem{(B2)} \label{ass:B2} The sequence $\CB^M$ converges to~$\CB$ \emph{in law}. 
		\myitem{(D2)} \label{ass:D2} The deterministic sequence~$\CD^M$ converges to~$\CD$. 
	\end{enumerate}
	Then, the sequence~$(\CA_N)_N$ converges in law to~$\CB + \CD$ as~$N \to \infty$.
\end{proposition}

\begin{remark}
	In our situation of interest, the separable Banach space is given by the finite-dimensional space~$\del[0]{\R^m \oplus (\R^{m})^{\otimes 2}}^{l}$ for~$l \in \N$ and the convergence statements in~\ref{ass:D1} and~\ref{ass:B2} actually hold \emph{in~$L^2(\P)$} rather than in probability resp. in law only.   
\end{remark}
We will see that one can reduce the proof of the previous proposition to the following lemma:
\begin{lemma} \label{lem:conv_in_law}
	Let the random variable~$\chi_N$ have the decomposition~$\chi_N = \theta_N^M + \eta_N^M$ where all the random variables take values in a separable Banach space and the following properties hold:
	\begin{enumerate}[label=(\arabic*)]
		\item \label{lem:conv_in_law_1} For each~$M \geq d$, we have~$\lim_{N \to \infty} \theta_N^M = \theta^M$ in law.
		\item \label{lem:conv_in_law_2} We have~$\lim_{M \to \infty} \theta^M = \theta$ in law.
		\item \label{lem:conv_in_law_3} For each~$M \geq d$, the limes superior $b_M := \limsup_{N \to \infty} \norm[0]{\eta_N^M}_{L^2(\P)}^2$ exists.
		\item \label{lem:conv_in_law_4} We have $\lim_{M \to \infty} b_M = 0$.
	\end{enumerate}
	Then, $\lim_{N \to \infty} \chi_N = \theta$ in law.
\end{lemma}	

\begin{proof}
	Denote by~$d$ the bounded Lipschitz metric which is known to metrise convergence in law in general metric spaces (see, for example, \cite[Thm.~11.3.3]{dudley_book}).
	Recall that
	\begin{equation*} \label{e:bdd_lip_metric}
		d(U,V) = \sup_{\norm[0]{f}_{\text{\tiny Lip}} \leq 1} \abs[0]{\E[f(U)] - E[f(V)]} \leq \norm[0]{U-V}_{L^1(\P)}
	\end{equation*}
	
	We then have\footnote{See Remark~\ref{rmk:L1_vs_L2} below for a comment on for the last (trivial) estimate.}
	\begin{equation*}
		d(\chi_N,\theta_N^M) = d(\theta_N^M + \eta_N^M,\theta_N^M) \leq \norm[0]{\eta_N^M}_{L^1(\P)} \leq \norm[0]{\eta_N^M}_{L^2(\P)}.  
	\end{equation*}
	By the triangle equality, we have
	\begin{equation*}
		d(\chi_N,\theta) \leq d(\chi_N,\theta_N^M) + d(\theta_N^M,\theta^M) + d(\theta^M,\theta)
	\end{equation*}
	and, therefore, using all the assumptions,
	\begin{equation} \label{e:conv_law_aux}
		\limsup_{N \to \infty} d(\chi_N,\theta) \leq b_M^{\nicefrac{1}{2}} + d(\theta^M,\theta) \,. 
	\end{equation}
	Note that we have used that
	\begin{equation*}
		\limsup_{N \to \infty} \norm[0]{\eta_N^M}_{L^2(\P)} = \del[1]{\limsup_{N \to \infty} \norm[0]{\eta_N^M}_{L^2(\P)}^2}^{\nicefrac{1}{2}} = b_M^{\nicefrac{1}{2}}
	\end{equation*}
	because the square root function is continuous and increasing on~$[0,\infty)$.
	Taking~$M \to \infty$ in~\eqref{e:conv_law_aux} establishes the claim.
\end{proof}

\begin{proof}[of Proposition~\ref{prop:abstract_convergence_statement}]
	Set~$\theta_N^M \coloneqq \CB_N^M + \CD_N^M$ and~$\eta_N^M \coloneqq \CR_N^M + \CC_N^M$.
	By Slutsky's lemma, we know that: 
	\begin{enumerate}[label=(\arabic*)]
		\item For each fixed~$M \geq d$, due to assumptions~\ref{ass:B1} and~\ref{ass:D1}, we have~$\theta_N^M \to \theta^M \coloneqq \CB^M + \CD^M$ in law as~$N \to \infty$
		\item By assumptions~\ref{ass:B2} and~\ref{ass:D2}, we know that~$\theta^M = \CB^M + \CD^M \to \CB + \CD \eqqcolon \theta$ in law as~$M \to \infty$.
	\end{enumerate}
	Further, by assumptions~\ref{ass:R1} and~\ref{ass:C1} we know that for fixed~$M \geq d$, we have 
	\begin{equs}
		0 \leq b_M = 
		\limsup_{N \to \infty} \norm[0]{\eta_N^M}_{L^2(\P)}^2 \lesssim 
		\limsup_{N \to \infty}\norm[0]{\CR_N^M}_{L^2(\P)}^2 
		+ 
		\limsup_{N \to \infty}\norm[0]{\CC_N^M}^2_{L^2(\P)} 
		 = r^M + 0 < \infty
	\end{equs}
	and, by assumption~\ref{ass:R2}, we have~$r^M \to 0$ as~$M \to \infty$.
	The conclusion now follows from Lemma~\ref{lem:conv_in_law}.
\end{proof}

\begin{remark} \label{rmk:L1_vs_L2}
	The reader might wonder why we use the stronger $L^2(\P)$-norm here even though it seems to be enough to control the~$L^1(\P)$-norm. This is because, for the application we have in mind, there is no good way (we know of) to control the latter (essentially because Hermite expansions do not interact well with absolute values), but we can control the former, see the strategy paragraph on~p.~\pageref{strategy_reduction} and also Remark~\ref{rmk:failure_tightness_strategy}.
\end{remark}

\subsection{Reduction to finite chaos} \label{s:reduction}

The purpose of this subsection is to introduce the vector~$\CR_N^M$ given in~\eqref{e:def_RNM} by
\begin{equation*}
	\CR_N^M 
	\coloneqq	
	\del[1]{S^{> M}_N(s_1,t_1), \mathbb{S}^{> M}_N(s_1,t_1), \ldots, S^{> M}_N(s_l,t_l), \mathbb{S}^{> M}_N(s_l,t_l)} \in \del[1]{\tilde{T}^{(2)}(\R^m)}^l
\end{equation*}
and to show that it satisfies Assumptions~\ref{ass:R1} and~\ref{ass:R2} of Proposition~\ref{prop:abstract_convergence_statement}.
This will be done in Proposition~\ref{p:reduction} below which is the main result of this subsection.

\begin{notation} \label{notation:reduction}
Recall that the components of the~$m$-dim. vector~$S_N(s,t)$ read
\begin{equation*}
	S^k_N(s,t) = \frac{1}{\sqrt{N}} \sum_{i = \lfloor Ns \rfloor}^{\lfloor Nt \rfloor -1}
	f_k(X^{(k)}_i), \quad k \in \llbracket 1,m \rrbracket \,.
\end{equation*}
For~$M \in \N$, we write
\begin{equation*}
	f_k^{M} \coloneqq \sum_{q = d}^M c_q^{(k)} H_q, 
	\qquad 
	f_k^{> M} \coloneqq \sum_{q=M+1}^\infty c_q^{(k)} H_q \,.
\end{equation*}
Furthermore, we let 
\begin{equation} \label{e:1st_order_tail}
	S_N^{k; M}(t) 
	\coloneqq \frac{1}{\sqrt{N}} \sum_{i = \lfloor Ns \rfloor}^{\lfloor Nt \rfloor -1} f_k^{\leq M}(X^{(k)}_i), 
	\qquad 
	S_N^{k; > M}(s,t) \coloneqq \frac{1}{\sqrt{N}} \sum_{i = \lfloor Ns \rfloor}^{\lfloor Nt \rfloor -1} f_k^{> M}(X^{(k)}_i) 
\end{equation}
and, for~$k, \ell \in \llbracket 1,m \rrbracket$
\begin{equation*}
	\mathbb{S}_N^{k,\ell;M}(t) 
	\coloneqq 
	\frac{1}{N} \sum_{\lfloor Ns \rfloor \leq i < j \leq \lfloor Nt \rfloor -1} f_k^{M}(X^{(k)}_i) \, f_\ell^{M}(X^{(\ell)}_j), 
\end{equation*} 
as well as
\begin{align}
	\mathbb{S}_N^{k,\ell;> M}(s,t) 
	& \coloneqq 
	\frac{1}{N} \sum_{\lfloor Ns \rfloor \leq i < j \leq \lfloor Nt \rfloor -1}
	\left(f_{k}^{> M}(X^{(k)}_i) \, f_\ell^{M}(X^{(\ell)}_j) +
		f_k^{M}(X^{(k)}_i) \, f_\ell^{> M}(X^{(\ell)}_j) \right. 
		\label{e:2nd_order_tail:1}  \\
		& \qquad \qquad \qquad \left. + \ f_k^{> M}(X^{(k)}_i) \, f_\ell^{> M}(X^{(\ell)}_j)\right) \notag\\
	& \eqqcolon \mathbb{S}_N^{k,\ell;> M,[1]}(s,t) + \mathbb{S}_N^{k,\ell;> M,[2]}(s,t) + \mathbb{S}_N^{k,\ell;> M,[3]}(s,t) \,.
	\label{e:2nd_order_tail}
\end{align}
With this notation, we then have
\begin{equation*}
	S_N^{>M}(s,t) = \del[1]{S_N^{k;>M}(s,t)}_{k=1}^m, \quad
	\mathbb{S}_N^{>M}(s,t) = \del[1]{\mathbb{S}_N^{k,\ell;>M}(s,t)}_{k,\ell=1}^m
\end{equation*}
as a vector respectively a matrix, and analogously for the corresponding terms with~\enquote{$>M$} replaced by~\enquote{$M$} (for the projection onto chaoses~$\leq M$).
\end{notation} 

\paragraph{Notation.} \label{notation_reduction}
Throughout this section, we will consider~$\w = \w_{1:4}$ for~$\w_{1} = \w_{3} = k$ as well as $\w_{2} = \w_{4} = \ell$.
Furthermore, for~$i_{1:4} \in \Z^4$, we will write~$\biota_{j} \coloneqq (\w_j,i_j)$ such that, for~$j = \cbr[0]{1,2,3,4}$, we have 
\begin{equation} \label{e:assignment:h_strategy}
	X^{(\w_j)}_{i_j} = W(e_{\w_j,i_j})= W(e_{\biota_j}) \quad \text{and} 
	\quad 
	h_j(W(e_{\biota_j})) \coloneqq 
	f_{\w_j}^{>M}(W(e_{\biota_j})) \,.
\end{equation}
Finally, we will also write 
\begin{equation} \label{e:covariance_strategy}
	\rho(\biota_k - \biota_j) \coloneqq \rho_{\w_j,\w_k}(i_k - i_j) \,.
\end{equation}

\begin{remark}[Failure of the \enquote{tightness strategy}.] \label{rmk:failure_tightness_strategy}
This remark further elaborates on the issue touched upon in Remark~\ref{rmk:L1_vs_L2}.
For the moment, let us focus on the third term in~\eqref{e:2nd_order_tail}, i.e. $\mathbb{S}_N^{k,\ell;> M,[3]}(s,t)$; as we will argue in the proof of Proposition~\ref{p:reduction} below, this is actually sufficient.

As discussed in Remark~\ref{rmk:L1_vs_L2}, in principle, it would be sufficient to control its~$L^1(\P)$-norm after taking limits~$N \to \infty$ (first) and~$M \to \infty$ (second); the~$L^1(\P$)-norm (for $s=0$ and $t=1$) reads
\begin{equation} \label{e:L1_norm_reduction_strategy}
	\norm[1]{\mathbb{S}_N^{k,\ell;> M, [3]}}_{L^1(\P)}
	= 
	\frac{1}{N} \E\sbr[3]{\thinspace \abs[3]{\sum_{0 \leq i_1 < i_2 \leq N-1} h_1(W(e_{\biota_1})) h_2(W(e_{\biota_2}))}} \,.
\end{equation} 
Controlling this quantity \emph{precisely} corresponds to the case~$p = 1$ in the proof of Theorem~\ref{thm:tightness}: 
After naively repeating all the steps therein, for~$p=1$ the corresponding norms in~\eqref{e:est_product_g_h} are $L^2(\gamma)$-norms which do go to~$0$ as~$M \to \infty$! 
The proof, however, would then require Meyer's inequality for~$p=1$ to estimate~\eqref{e:tightness_meyer}---and, as we emphasised in Remark~\ref{rmk:tightness_p}, this is \emph{precisely} where it fails.

Compensating the lack of Meyer's inequality for~$p = 1$ is known to be difficult problem, see for example~\cite{addona_muratori_rossi_22}.   
In following their line of thought, one could replace the use of Meyer's inequality in~\eqref{e:tightness_meyer} by \emph{Poincaré's inequality}
\begin{equation} \label{e:poincare}
	\norm[0]{F}_{L^p(\Omega)} \leq \norm[0]{D^d F}_{L^p(\Omega;H^{\symotimes n})} + \sum_{r=0}^{d-1} \norm[0]{\E\sbr[0]{D^r F}}_{H^{\symotimes r}}	
\end{equation}
which works for any~$p \geq 1$. 
Indeed: One can then control the Malliavin derivative terms in~\eqref{e:poincare} like in the proof of Theorem~\ref{thm:tightness}, but controlling the expectation terms only works for~$d \in \cbr[0]{1,2}$ and the argument applies in exactly the same way to the reduction argument as well.
As discussed in Remark~\ref{rmk:main}, Point~\ref{rmk:main:integrability}, this is a structurally similar problem to the one in~\cite{nourdin_nualart_peccati_21} which has recently been overcome by~\cite{angst_dalmao_poly_24}.

Let us finally mention that the Poincaré inequality has also proved intrumental in the (singular) SPDE and rough paths literature~\cite{LOTT24, hairer_steele_bphz, gassiat_klose_2024, bailleul_bruned_26}.
\end{remark}

\paragraph{New strategy.} \label{strategy_reduction}
In light of the Remarks~\ref{rmk:L1_vs_L2} and~\ref{rmk:failure_tightness_strategy}, we now need to study the $L^2(\P)$-norm of $\mathbb{S}_N^{k,\ell;> M,[3]}(s,t)$ in the limits~$N \to \infty$ (first) and~$M \to \infty$ (second): 
\begin{equs}[][e:L2_norm_reduction_strategy]
	\norm[1]{\mathbb{S}_N^{k,\ell;> M, [3]}}_{L^2(\P)}^2
	& = 
	\frac{1}{N^2} \sum_{\substack{0 \leq i_1 < i_2 \leq N-1, \\ 0 \leq i_3 < i_4 \leq N-1}} \E\sbr[4]{\prod_{j=1}^4 h_j(W(e_{\biota_j}))}
\end{equs}
Roughly speaking, the idea is to write
\begin{equation} \label{e:idea_reduction}
	"
	\text{RHS of} \ \eqref{e:L2_norm_reduction_strategy}
	=
	\sum_{G} 
	(N\text{-dep. terms})_{G}
	\times
	(M\text{-dep. expectation terms})_{G}
	"
\end{equation}
where the sum is over \emph{finitely many} diagrams~$G$ in the sense of Definition~\ref{d:diagrams}.
We will then split the analysis into regular and irregular diagrams (in the sense of Definition~\ref{d:diagrams}):
\begin{enumerate}
	\item For each \emph{fixed}~$M$ and any diagram~$G$, the expectation term in~\eqref{e:idea_reduction} is uniformly bounded in~$i_{1:4}$.
	\item For \emph{irregular} diagrams~$G$, we will show that the $N$-dep. terms in~\eqref{e:idea_reduction} go to~$0$ as~$N \to \infty$, for each fixed~$M$.%which is a consequence of \ref{p:bm83_irregular} 
	\item For \emph{regular} diagrams~$G$, the~$N$-dep. terms in~\eqref{e:idea_reduction} have a finite limes superior as~$N \to \infty$. In that case, however the~$M$-dep. expectation terms in~\eqref{e:idea_reduction} go to~$0$ as~$M \to \infty$. 
\end{enumerate}

This strategy will be implemented in the proof of Proposition~\ref{p:reduction} below.
To this end, we will need the following two key lemmas, the first of which rigorously establishes the identity~\eqref{e:idea_reduction} and the second shows Point~$1$ in the previous list.

In the following lemma, one should think of~$h_j$ representing its regularised version $h_j^{[r]} = P_{\nicefrac{1}{r}} h_j$ as defined in Lemma~\ref{lem:OU_regularisation};
see the proof of Proposition~\ref{p:reduction} below for further details.
\begin{lemma}[Conversion to Feynman diagrams] \label{l:representation_feynman_diagrams}
	Let~$d \geq 1$ and, for each~$j=1,\ldots,4$, assume that~$h_j \in \mathbb{D}^{3d,4}(\gamma)$ and that~$h_j$ has Hermite rank at least~$d$.
	Then, we have 
	\begin{equation} \label{e:representation_feynman_diagrams}
		\frac{1}{N^2} \sum_{\substack{0 \leq i_1 < i_2 \leq N-1, \\ 0 \leq i_3 < i_4 \leq N-1}} 
		\E\sbr[4]{\prod_{j=1}^4 h_j(W(e_{\biota_j}))}
		=
		\sum_{G \in \mathring{\Gamma}} C_G \thinspace \sum_{\substack{0 \leq i_{1} < i_2 \leq N-1, \\ 0 \leq i_{3} < i_4 \leq N-1}} \frac{\fC_G(\w,\bq,\bi)}{N^2}  E_G(\w,\bi)
	\end{equation}
	where~$\w = \w_{1:4}$, $\bq = q_{1:4}$, and~$\bi = i_{1:4}$.
	Furthermore,~$C_G$ is a combinatorial factor and
	\begin{enumerate}[label=(\roman*)]
		\item \label{l:representation_feynman_diagrams:i} the sum is over some \emph{finite} set $\mathring{\Gamma}$ of Feynman diagrams~$G$, where~$ \mathring{\Gamma} \subseteq \bigcup_{q_{1:4} \geq d} \Gamma(q_{1:4})$, and the quantity~$\fC_G(\w,\bq,\bi)$ has been defined in~\eqref{e:diagram_formula:CG}, cf. also Proposition~\ref{p:bm83_irregular}.
		\item \label{l:representation_feynman_diagrams:ii} for each~$G \in \mathring{\Gamma}$ in the sum, there exists~$a_{1:4} \in \N_0^4$ satisfying~$0 \leq a_{[1:4]} \leq 4d$ and~$0 \leq a_j \leq 3d$ such that 
		\begin{equation} \label{e2:representation_feynman_diagrams}
			E_G(\w,\bi) = \E\sbr[4]{\prod_{j=1}^4[\CS_{d} h_j]^{(a_j)}(W(e_{\biota_j}))} \,.
		\end{equation}
	\end{enumerate}
\end{lemma}

\begin{proof}
	Using integration by parts and the Leibniz rule for the Malliavin derivative (see~\cite[Exercise~2.3.10]{nourdin_peccati_book}), we find 
	\begin{align}
		\thinspace & 
		\E\sbr[4]{\prod_{j=1}^4 h_j(W(e_{\biota_j}))}
		= 
		\E\sbr[4]{\delta^{d}\del[1]{[\CS_{d} h_1](W(e_{\biota_1}))e_{\biota_1}^{\otimes d}} \prod_{j=2}^4 h_j(W(e_{\biota_j}))} \notag \\[0.5em]%\label{e:first_step_reduction_vanishing}\\[0.5em]
		= & \ 
		\sum_{\substack{r_{12}, r_{13}, r_{14} \geq 0, \\ r_{12} + r_{13} + r_{14} = d}} \binom{d}{r_{12}, r_{13}, r_{14}}
		\E\sbr[2]{\sbr[1]{\CS_{d} h_{1}}(W(e_{\biota_1})) \big\langle e_{\biota_1}^{\otimes d}, \widetilde{\bigotimes}_{j=2}^4 D^{r_{1j}} h_j(W(e_{\biota_j})) \big\rangle_{H^{\tilde{\otimes} n}}} \,. \notag
	\end{align}
	The next step is to apply the commutation relation for iterated Malliavin derivatives and iterated divergences, Lemma~\ref{lem:commutation_D_delta}, to each of the terms\footnote{Note that, due to the specific form of~$u$ to which Lemma~\ref{lem:commutation_D_delta} is applied, we do not actually have to distinguish between the variables of~$u$ and those created by the Malliavin derivatives because they coincide.}
	\begin{equation} \label{e:factors_divergence}
		D^{r_{1j}} h_j(W(e_{\biota_j}))	
		=
		D^{r_{1j}} \delta^{d} \del[1]{\sbr[0]{\CS_{d} h_j}(W(e_{\biota_j})) e_{\biota_j}^{\otimes d}}, \quad j=2,3,4,
	\end{equation}
	which leads to the following expression:
	\begin{align}
		\thinspace & 
		\E\sbr[4]{\prod_{j=1}^4 h_j(W(e_{\biota_j}))} \notag \\
		= & \  
		\sum_{\substack{r_{12}, r_{13}, r_{14} \geq 0, \\ r_{12} + r_{13} + r_{14} = d}} 
		\binom{d}{r_{12}, r_{13}, r_{14}}
		\sum_{\ell_{12} = 0}^{d \wedge r_{12}} 
		\sum_{\ell_{13} = 0}^{d \wedge r_{13}}
		\sum_{\ell_{14} = 0}^{d \wedge r_{14}} 
		\del[4]{\prod_{j=2}^4 \binom{d}{\ell_{1j}} \binom{r_{1j}}{\ell_{1j}} \ell_{1j}!} \thinspace
		\times \notag \\
		& \ \times  
		\E\sbr[2]{\sbr[1]{\CS_{d} h_{1}}(W(e_{\biota_1})) \Big\langle e_{\biota_1}^{\otimes d}, \widetilde{\bigotimes}_{j=2}^4 
		\delta^{d-\ell_{1j}}\del[1]{D^{r_{1j} - \ell_{1j}} \sbr[1]{\CS_{d} h_j}(W(e_{\biota_j})) e_{\biota_j}^{\otimes d}} \Big\rangle_{H^{\tilde{\otimes} d}}}
		\label{e:commutation_first_iteration}
		\\
		= & \  
		\sum_{\substack{r_{12}, r_{13}, r_{14} \geq 0, \\ r_{12} + r_{13} + r_{14} = d}} 
		%\binom{d}{r_{12}, r_{13}, r_{14}}
		\binom{d}{r_{1,2:4}}
		\prod_{j=2}^4 \rho_{\w_1,\w_j}(i_j-i_1)^{r_{1j}} \sum_{\ell_{12} = 0}^{r_{12}} 
		\sum_{\ell_{13} = 0}^{r_{13}}
		\sum_{\ell_{14} = 0}^{r_{14}} 
		\del[4]{\prod_{j=2}^4 \binom{d}{\ell_{1j}} \binom{r_{1j}}{\ell_{1j}} \ell_{1j}!}
		\times \notag \\
		& \ 
		\times \E\sbr[2]{\sbr[1]{\CS_{d} h_{1}}(W(e_{\biota_1})) \prod_{j=2}^4 
		\delta^{d-\ell_{1j}}\del[1]{\sbr[1]{\CS_{d} h_j}^{(r_{1j} - \ell_{1j})}(W(e_{\biota_j})) e_{\biota_j}^{\otimes (d - \ell_{1j})}}} 
		\notag
	\end{align}
	Note that we have used~$r_{1,2:4}$ as a placeholder for~\enquote{$r_{12}, r_{13}, r_{14}$}.

	The proof now proceeds by applying the same three steps---integration by parts, Leibniz rule, and the commutation relation for~$D^k$ and~$\delta^j$---to each of the other three factors~$h_j(W(e_{\biota_j}))$ for~$j = 2,3,4$. 
	In doing so, the resulting expressions become completely unwieldy and intransparent very quickly, which is the reason for the \emph{qualitative} formulation of the lemma.\footnote{In the proof of Proposition~\ref{p:reduction} below, we will be more precise on the combinatorics when~$G \in \mathring{\Gamma}$ is \emph{regular} in the sense of Definition~\ref{d:regular_graph}; for \emph{irregular} diagrams, the qualitative statement in this proposition is enough as they correspond to vanishing contributions when~$N \to \infty$.} %that lends itself to a graphical argument. 
	
	That said, we point out that the statement in~\ref{l:representation_feynman_diagrams:ii} is straightforward:
	Since each of the four factors has (at most)~$n$ derivatives to distribute among the other factors and those derivatives cannot hit the term they originate from itself, we immediately find the claimed identity in~\eqref{e2:representation_feynman_diagrams} with~$0 \leq a_{[1:4]} \leq 4d$ and~$0 \leq a_j \leq 3d$. 
	In fact, the worst case situation where~$a_1 = d$,~$a_2 = 3d$, and~$a_3 = a_4 = 0$ can be visualised as follows:
	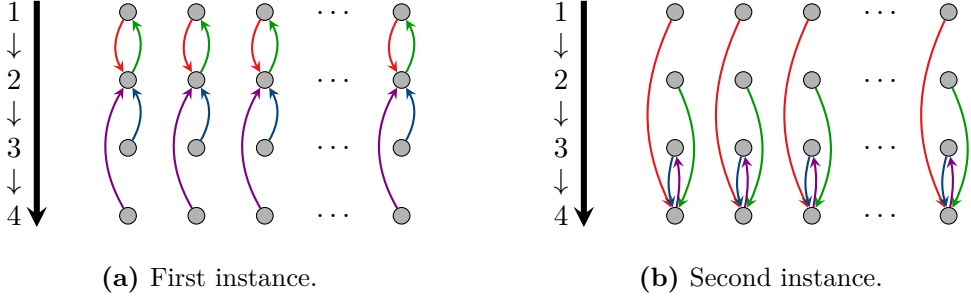
\begin{figure}[h]
		\centering
		\begin{subfigure}{.45\textwidth}
			%\begin{equs}
			\centering
			\begin{tikzpicture}[scale=0.3,baseline=0cm]
				%%%%%%%%
				\node at (1,3)  [] [label={[label distance=0.05cm]90:}] (1) {$1$};
				\node at (1,1.5)  [] [label={[label distance=0.05cm]90:}] (a) {$\downarrow$};
				\node at (1,0)  [] [label={[label distance=0.05cm]90:}] 	(2) {$2$};
				\node at (1,-1.5)  [] [label={[label distance=0.05cm]90:}] (b) {$\downarrow$};
				\node at (1,-3)  [] [label={[label distance=0.05cm]90:}] 	(3) {$3$};
				\node at (1,-4.5)  [] [label={[label distance=0.05cm]90:}] (c) {$\downarrow$};
				\node at (1,-6)  [] [label={[label distance=0.05cm]90:}] 	(4) {$4$};
				%%%%%%%%
				\node at (6,3)  [xibig] [label={[label distance=0.05cm]90:}] (11) {};
				\node at (9,3)  [xibig] [label={[label distance=0.05cm]90:}] (12) {};
				\node at (12,3)  [xibig] [label={[label distance=0.05cm]90:}] (13) {};
				\node at (15,3)   [label={[label distance=0.05cm]90:}] (14) {$\ldots$};
				\node at (18,3)  [xibig] [label={[label distance=0.05cm]90:}] (15) {};
				%%%%%
				\node at (6,0)  [xibig] [label={[label distance=0.05cm]-90:}] (21) {};
				\node at (9,0)  [xibig] [label={[label distance=0.05cm]-90:}] (22) {};
				\node at (12,0)  [xibig] [label={[label distance=0.05cm]-90:}] (23) {};
				\node at (15,0)   [label={[label distance=0.05cm]90:}] (24) {$\ldots$};
				\node at (18,0)  [xibig] [label={[label distance=0.05cm]90:}] (25) {};
				%%%%%
				\node at (6,-3)  [xibig] [label={[label distance=0.05cm]90:}] (31) {};
				\node at (9,-3)  [xibig] [label={[label distance=0.05cm]90:}] (32) {};
				\node at (12,-3)  [xibig] [label={[label distance=0.05cm]90:}] 
				(33) {};
				\node at (15,-3)   [label={[label distance=0.05cm]90:}] (34) {$\ldots$};
				\node at (18,-3)  [xibig] [label={[label distance=0.05cm]90:}] (35) {};
				%%%%%
				\node at (6,-6)  [xibig] [label={[label distance=0.05cm]-90:}] (41) {};
				\node at (9,-6)  [xibig] [label={[label distance=0.05cm]-90:}] (42) {};
				\node at (12,-6)  [xibig] [label={[label distance=0.05cm]-90:}] (43) {};
				\node at (15,-6)   [label={[label distance=0.05cm]90:}] (44) {$\ldots$};
				\node at (18,-6)  [xibig] [label={[label distance=0.05cm]90:}] (45) {};
				%%%%%
				\draw[thick, ->, color=darkred] (11) to [bend right=30] (21);
				\draw[thick, ->, color=darkred] (12) to [bend right=30] (22);
				\draw[thick, ->, color=darkred] (13) to [bend right=30] (23);
				\draw[thick, ->, color=darkred] (15) to [bend right=30] (25);
				%%%%
				%%%%%
				\draw[thick, ->, color=darkgreen] (21) to [bend right=30] (11);
				\draw[thick, ->, color=darkgreen] (22) to [bend right=30] (12);
				\draw[thick, ->, color=darkgreen] (23) to [bend right=30] (13);
				\draw[thick, ->, color=darkgreen] (25) to [bend right=30] (15);
				%%%%
				\draw[thick, ->, color=darkcerulean] (31) to [bend right=30] (21);
				\draw[thick, ->, color=darkcerulean] (32) to [bend right=30] (22);
				\draw[thick, ->, color=darkcerulean] (33) to [bend right=30] (23);
				\draw[thick, ->, color=darkcerulean] (35) to [bend right=30] (25);
				%%%%
				\draw[thick, ->, color=violet] (41) to [bend left=30] (21);
				\draw[thick, ->, color=violet] (42) to [bend left=30] (22);
				\draw[thick, ->, color=violet] (43) to [bend left=30] (23);
				\draw[thick, ->, color=violet] (45) to [bend left=30] (25);
				%%%%%%
				\draw[line width=0.8mm, ->] (2,3.5) to (2,-6.5);
			\end{tikzpicture}\; 
		%\end{equs}	
		%\hspace{1.5cm}
		\subcaption{First instance.}
		\label{subfig:worst_case_1}
		\end{subfigure}
		\hspace{1em}
		\begin{subfigure}{.45\textwidth}
			%\begin{equs}
			\centering
			\begin{tikzpicture}[scale=0.3,baseline=0cm]
				%%%%%%%%
				\node at (1,3)  [] [label={[label distance=0.05cm]90:}] (1) {$1$};
				\node at (1,1.5)  [] [label={[label distance=0.05cm]90:}] (a) {$\downarrow$};
				\node at (1,0)  [] [label={[label distance=0.05cm]90:}] 	(2) {$2$};
				\node at (1,-1.5)  [] [label={[label distance=0.05cm]90:}] (b) {$\downarrow$};
				\node at (1,-3)  [] [label={[label distance=0.05cm]90:}] 	(3) {$3$};
				\node at (1,-4.5)  [] [label={[label distance=0.05cm]90:}] (c) {$\downarrow$};
				\node at (1,-6)  [] [label={[label distance=0.05cm]90:}] 	(4) {$4$};
				%%%%%%%%
				\node at (6,3)  [xibig] [label={[label distance=0.05cm]90:}] (11) {};
				\node at (9,3)  [xibig] [label={[label distance=0.05cm]90:}] (12) {};
				\node at (12,3)  [xibig] [label={[label distance=0.05cm]90:}] (13) {};
				\node at (15,3)   [label={[label distance=0.05cm]90:}] (14) {$\ldots$};
				\node at (18,3)  [xibig] [label={[label distance=0.05cm]90:}] (15) {};
				%%%%%
				\node at (6,0)  [xibig] [label={[label distance=0.05cm]-90:}] (21) {};
				\node at (9,0)  [xibig] [label={[label distance=0.05cm]-90:}] (22) {};
				\node at (12,0)  [xibig] [label={[label distance=0.05cm]-90:}] (23) {};
				\node at (15,0)   [label={[label distance=0.05cm]90:}] (24) {$\ldots$};
				\node at (18,0)  [xibig] [label={[label distance=0.05cm]90:}] (25) {};
				%%%%%
				\node at (6,-3)  [xibig] [label={[label distance=0.05cm]90:}] (31) {};
				\node at (9,-3)  [xibig] [label={[label distance=0.05cm]90:}] (32) {};
				\node at (12,-3)  [xibig] [label={[label distance=0.05cm]90:}] 
				(33) {};
				\node at (15,-3)   [label={[label distance=0.05cm]90:}] (34) {$\ldots$};
				\node at (18,-3)  [xibig] [label={[label distance=0.05cm]90:}] (35) {};
				%%%%%
				\node at (6,-6)  [xibig] [label={[label distance=0.05cm]-90:}] (41) {};
				\node at (9,-6)  [xibig] [label={[label distance=0.05cm]-90:}] (42) {};
				\node at (12,-6)  [xibig] [label={[label distance=0.05cm]-90:}] (43) {};
				\node at (15,-6)   [label={[label distance=0.05cm]90:}] (44) {$\ldots$};
				\node at (18,-6)  [xibig] [label={[label distance=0.05cm]90:}] (45) {};
				%%%%%
				\draw[thick, ->, color=darkred] (11) to [bend right=25] (41);
				\draw[thick, ->, color=darkred] (12) to [bend right=25] (42);
				\draw[thick, ->, color=darkred] (13) to [bend right=25] (43);
				\draw[thick, ->, color=darkred] (15) to [bend right=25] (45);
				%%%%
				%%%%%
				\draw[thick, ->, color=darkgreen] (21) to [bend left=25] (41);
				\draw[thick, ->, color=darkgreen] (22) to [bend left=25] (42);
				\draw[thick, ->, color=darkgreen] (23) to [bend left=25] (43);
				\draw[thick, ->, color=darkgreen] (25) to [bend left=25] (45);
				%%%%
				\draw[thick, ->, color=darkcerulean] (31) to [bend right=15] (41);
				\draw[thick, ->, color=darkcerulean] (32) to [bend right=15] (42);
				\draw[thick, ->, color=darkcerulean] (33) to [bend right=15] (43);
				\draw[thick, ->, color=darkcerulean] (35) to [bend right=15] (45);
				%%%%
				\draw[thick, ->, color=violet] (41) to [bend right=10] (31);
				\draw[thick, ->, color=violet] (42) to [bend right=10] (32);
				\draw[thick, ->, color=violet] (43) to [bend right=10] (33);
				\draw[thick, ->, color=violet] (45) to [bend right=10] (35);
				%%%%%%
				\draw[line width=0.8mm, ->] (2,3.5) to (2,-6.5);
			\end{tikzpicture}\; 
			\subcaption{Second instance.}
			\label{subfig:worst_case_2}
		\end{subfigure}
		\caption{Two instances of the worst-case scenario in terms of \emph{actual} required derivatives, namely~$3d$. 
		In the right figure, however, the difference is that the the commutation relation between~$D$ and~$\delta$, Lemma~\ref{lem:commutation_D_delta}, requires~$4d$ derivatives in an intermediate step to deal with~$D^{3d} \delta^{d} \del[0]{\sbr[0]{\CS_{d}h_4}(W(e_{\biota_4}))e_{\biota_4}^{\otimes d}}$.}
		\label{fig:worst_case_derivatives}
	\end{figure}
	
	In this figure, 
	\begin{itemize}
		\item every row represents one factor~$h_j(W(e_{\biota_j}))$, $j = 1,2,3,4$, 
		\item there are~$d$ nodes in every row, each of which represents once instance of the divergence~$\delta$ in $h_j(W(e_{\biota_j}))	
		=
		\delta^{d} \del[0]{\sbr[0]{\CS_{d} h_j}(W(e_{\biota_j})) e_{\biota_j}^{\otimes d}}$,
		\item the number of incoming arrows per row denotes the number of derivatives which hit that row after originating from the outgoing row of the arrow. 
	\end{itemize}
	For example, the choice
	\begin{equation*}
		r_{12} = d, \quad \ell_{12} = 0, \quad r_{21} = r_{32} = r_{42} = d
	\end{equation*}
	with all other choosable parameters equal to~$0$ leads to this graph in Figure~\ref{subfig:worst_case_1}.
	The procedure that lead to~\eqref{e:commutation_first_iteration} can be iterated as follows:
	\begin{enumerate}[label=(\arabic*)]
		\item The number~$r_{1j}$ denotes the number of derivatives from row~$1$ that hit row~$j$.
		\item \label{item:graphical_formalism_free_legs} Among each of the rows~$j$ that have been hit by derivatives from row~$1$, i.e. if~$r_{1j} > 0$, one chooses a number~$\ell_{1j} \in \llbracket 0, d \wedge r_{1j} \rrbracket$, which then signifies that row~$j$ has~$d-\ell_{1j}$ \enquote{free legs} (in the form of~$\delta^{d-\ell_{1j}}$) left.
		\item We proceed row by row from row~$1$ on top to row~$4$ at the bottom, i.e.~$1 \to 2 \to 3 \to 4$, as indicated by the bold arrow on the left of the diagram.
		\item One chooses numbers~$r_{2j}$, $j \in \llbracket 1,4 \rrbracket\setminus \cbr[0]{2}$ that sum up to~$d-\ell_{12}$, the total number of free derivatives that row~$2$ has to distribute.
		\item We now iterate this procedure until we have reached row~$4$. 
		(Note that, at Step~$i$, one can only choose~$\ell_{ij}$ for~$j > i$ as the procedure has because all the rows~$j < i$ have already distributed all their derivatives.)
		\item Note that it is entirely possible that, in a given iteration step, a row does not distribute any derivatives: 
		For example, the sum in~\eqref{e:commutation_first_iteration} contains the case~$\ell_{12} = r_{12} = d$ which would correspond to the graph in Figure~\ref{subfig:worst_case_1} but with no green arrows.
		\item In general, at each iteration step, every node within a given row that has not yet been hit by a derivative, i.e. has no incoming arrow, must distribute its free derivatives among the \emph{other} rows (i.e. \emph{not} within its own row).
		In particular, a row that has not been hit by any derivative (i.e. has no incoming arrows at all) always distributes the maximal number of~$d$ derivatives among the other rows.
		\item \label{item:graphical_formalism_last}The previous point guarantees that each node has \emph{at least} one outgoing or incoming arrow: It cannot have more than one outgoing arrow, but it can have at most three incoming arrows, each of which originates from a different row.
		In particular, no node can have multiple incoming arrows from the same row.
	\end{enumerate}
	Let us record two observations related to the previous algorithm:
	\begin{itemize}
		\item First, the commutation relation from Lemma~\ref{lem:commutation_D_delta} is always applicable and all the involved objects have enough Malliavin--Sobolev regularity:
		Since we have assumed that~$h_j \in \DD^{3 d,4}(\gamma)$, by Proposition~\ref{p:estimate_hermite_shift}, we can infer that $\CS_{d} h_j \in \DD^{4 d,4}(\gamma)$. 
		Even though~$a_j \leq 3d$ as argued above, this is relevant because the scenario~$D^{3d} \delta^{d} \del[0]{\sbr[0]{\CS_{d}h_4}(W(e_{\biota_4}))e_{\biota_4}^{\otimes d}}$ might occur and requires $4 d$ derivatives, see Figure~\ref{subfig:worst_case_2} for a pictorial representation.
		\item While each choice of~$(r_{ij})_{i,j=1}^4$ and~$(\ell_{ij})_{i,j=1}^4$ with~$r_{ii} = \ell_{ii} = 0$ leads to a unique diagram such as the one in Figure~\ref{fig:worst_case_derivatives}, the reverse is not always true, essentially because of the freedom to \enquote{choose} the parameter~$\ell$ in the commutation relation in Lemma~\ref{lem:commutation_D_delta}.
		However, this is not a problem for us, as we only need one direction of the mapping.
	\end{itemize}

	Let us now use the above graphical formalism to explain why the claim in~\ref{l:representation_feynman_diagrams:i} is true. 
	In fact, point~\ref{item:graphical_formalism_last} in the previous list allows us to map each gram in Figure~\ref{fig:worst_case_derivatives} to a diagram (in the sense of Definition~\ref{d:diagrams}) that arises when computing the expectation of the product of Hermite polynomials with different Gaussians as their input via the \emph{diagram formula}, see Proposition~\ref{p:diagram_formula}.
	
	The key observation is the following: Beside modelling the structure of derivatives hitting each factor, each graph like the one in Figure~\ref{fig:worst_case_derivatives} also encodes a product of correlation functions: Each arrow between row~$j$ and row~$k$ gives rise to a factor~$\rho_{\w_j,\w_k}(i_k - i_j)$, regardless of its direction. 
	The crucial difference to a Feynman diagram from Definition~\ref{d:diagrams} is that, therein, each node is only paired \emph{exactly once} which we can achieve by the following procedure:
	Define
	\begin{equation*}
		q_k \coloneqq \#\cbr[0]{\text{\emph{outgoing} arrows in row}~k} + \#\cbr[0]{\text{\emph{incoming} arrows in row}~k}, \ k \in \llbracket 1,4 \rrbracket 
	\end{equation*}
	and draw a graph like in Figure~\ref{fig:worst_case_derivatives} where
	\begin{itemize}
		\item the~$k$-th row has~$q_k$ nodes, 
		\item each node is paired exactly once, but
		\item each row as a whole has the same number of incoming and outgoing arrows.
	\end{itemize}
	Pictorially, this procedure can be represented as in Figure~\ref{fig:conversion_to_feynman_diagram} where, for concreteness, we symbolise the procedure when~$d = 3$. 
	In this picture, we have~$q_1 = 4$, $q_2 = 6$, $q_3 = 5$, and~$q_4 = 3$ as well as, for $\w = \w_{1:4}$, $\bq = q_{1:4}$, and~$\bi = i_{1:4}$:
	\begin{equs}
		\fC_G(\w, \bq, \bi) 
		& = 
		\rho_{\w_1,\w_2}(i_2 - i_1)^2 \rho_{\w_1,\w_3}(i_3 - i_1) \rho_{\w_1,\w_4}(i_4 - i_1) \times \\[0.5em] 
		%%%
		& \times 
		\rho_{\w_2,\w_3}(i_3 - i_2)^3 \rho_{\w_2,\w_4}(i_4 - i_2) \rho_{\w_3,\w_4}(i_4-i_3) \,.
	\end{equs} 
	It is immediate to convince oneself that this procedure \emph{always} works and the set~$\mathring{\Gamma}$ is defined to be the \emph{finite} set of all Feynman diagrams that are produced in this way.
	Since giving precise formulas would only obfuscate, rather than clarify, our argument, we refrain from doing so.
	\begin{figure}[h]
		\centering
		\begin{subfigure}{.45\textwidth}
			%\begin{equs}
			\centering
			\begin{tikzpicture}[scale=0.3,baseline=0cm]
				%%%%%%%%
				\node at (1,3)  [] [label={[label distance=0.05cm]90:}] (1) {$1$};
				\node at (1,1.5)  [] [label={[label distance=0.05cm]90:}] (a) {$\downarrow$};
				\node at (1,0)  [] [label={[label distance=0.05cm]90:}] 	(2) {$2$};
				\node at (1,-1.5)  [] [label={[label distance=0.05cm]90:}] (b) {$\downarrow$};
				\node at (1,-3)  [] [label={[label distance=0.05cm]90:}] 	(3) {$3$};
				\node at (1,-4.5)  [] [label={[label distance=0.05cm]90:}] (c) {$\downarrow$};
				\node at (1,-6)  [] [label={[label distance=0.05cm]90:}] 	(4) {$4$};
				%%%%%%%%
				%%%%%%%%
				\node at (6,3)  [xibig] [label={[label distance=0.05cm]90:}] (11) {};
				\node at (9,3)  [xibig] [label={[label distance=0.05cm]90:}] (12) {};
				\node at (12,3)  [xibig] [label={[label distance=0.05cm]90:}] (13) {};
				%%%%%
				\node at (6,0)  [xibig] [label={[label distance=0.05cm]-90:}] (21) {};
				\node at (9,0)  [xibig] [label={[label distance=0.05cm]-90:}] (22) {};
				\node at (12,0)  [xibig] [label={[label distance=0.05cm]-90:}] (23) {};
				%%%%%
				\node at (6,-3)  [xibig] [label={[label distance=0.05cm]90:}] (31) {};
				\node at (9,-3)  [xibig] [label={[label distance=0.05cm]90:}] (32) {};
				\node at (12,-3)  [xibig] [label={[label distance=0.05cm]90:}] (33) {};
				%%%%%
				\node at (6,-6)  [xibig] [label={[label distance=0.05cm]-90:}] (41) {};
				\node at (9,-6)  [xibig] [label={[label distance=0.05cm]-90:}] (42) {};
				\node at (12,-6)  [xibig] [label={[label distance=0.05cm]-90:}] (43) {};
				%%%%%
				\draw[thick, ->, color=darkred] (11) to (21);
				\draw[thick, ->, color=darkred] (12) to [bend right=25] (32);
				\draw[thick, ->, color=darkred] (13) to [bend left=25] (43);
				\draw[thick, ->, color=darkgreen] (21) to [bend right=25] (41);
				\draw[thick, ->, color=darkgreen] (22) to [bend left=25] (31);
				\draw[thick, ->, color=darkgreen] (23) to [bend left=25] (13);
				\draw[thick, ->, color=darkcerulean] (32) to [bend right=25] (22);
				\draw[thick, ->, color=darkcerulean] (33) to [bend left=25] (23);
				\draw[thick, ->, color=violet] (42) to (33);
				%%%%%%
				%
				\draw[line width=0.8mm, ->] (2,3.5) to (2,-6.5);
			\end{tikzpicture}\; 
			\subcaption{A diagram which represents iterated integration by parts.}
		\end{subfigure}
		\hspace{1em}
		\begin{subfigure}{.45\textwidth}
			%\begin{equs}
			\centering
			\begin{tikzpicture}[scale=0.3,baseline=0cm]
				%%%%%%%%
				\node at (4,3)  [] [label={[label distance=0.05cm]90:}] (1) {$1$};
				\node at (4,0)  [] [label={[label distance=0.05cm]90:}] 	(2) {$2$};
				\node at (4,-3)  [] [label={[label distance=0.05cm]90:}] 	(3) {$3$};
				\node at (4,-6)  [] [label={[label distance=0.05cm]90:}] 	(4) {$4$};
				%%%%%%%%
				\node at (6,3)  [xibig] [label={[label distance=0.05cm]90:}] (11) {};
				\node at (9,3)  [xibig] [label={[label distance=0.05cm]90:}] (12) {};
				\node at (12,3)  [xibig] [label={[label distance=0.05cm]90:}] (13) {};
				\node at (15,3)  [xibig] [label={[label distance=0.05cm]-90:}] (14) {};
				%%%%%
				\node at (6,0)  [xibig] [label={[label distance=0.05cm]-90:}] (21) {};
				\node at (9,0)  [xibig] [label={[label distance=0.05cm]-90:}] (22) {};
				\node at (12,0)  [xibig] [label={[label distance=0.05cm]-90:}] (23) {};
				\node at (15,0)  [xibig] [label={[label distance=0.05cm]-90:}] (24) {};
				\node at (18,0)  [xibig] [label={[label distance=0.05cm]-90:}] (25) {};
				\node at (21,0)  [xibig] [label={[label distance=0.05cm]-90:}] (26) {};
				%%%%%
				\node at (6,-3)  [xibig] [label={[label distance=0.05cm]90:}] (31) {};
				\node at (9,-3)  [xibig] [label={[label distance=0.05cm]90:}] (32) {};
				\node at (12,-3)  [xibig] [label={[label distance=0.05cm]90:}] (33) {};
				\node at (15,-3)  [xibig] [label={[label distance=0.05cm]90:}] (34) {};
				\node at (18,-3)  [xibig] [label={[label distance=0.05cm]90:}] (35) {};
				%%%%%
				\node at (6,-6)  [xibig] [label={[label distance=0.05cm]-90:}] (41) {};
				\node at (9,-6)  [xibig] [label={[label distance=0.05cm]-90:}] (42) {};
				\node at (12,-6)  [xibig] [label={[label distance=0.05cm]-90:}] (43) {};
				%%%%%
				%%%%%
				\draw[thick, color=darkred] (11) to (21);
				\draw[thick, color=darkred] (12) to [bend right=30] (32);
				\draw[thick, color=darkred] (13) to [bend left=30] (43);
				\draw[thick, color=darkgreen] (24) to (41);
				\draw[thick, color=darkgreen] (22) to [bend left=25] (31);
				\draw[thick, color=darkgreen] (23) to [bend right=25] (14);
				\draw[thick, color=darkcerulean] (33) to (25);
				\draw[thick, color=darkcerulean] (34) to (26);
				\draw[thick, color=violet] (42) to (35);
				%%%%%
				%
			\end{tikzpicture}\; 
			\subcaption{The corresponding Feynman diagram. \phantom{Fill text to force a line break.}}
			\label{subfig:fd_reparametrisation}
		\end{subfigure}
		 \caption{
			Example of the conversion mechanism into Feynman diagrams.
		 }
		\label{fig:conversion_to_feynman_diagram}
	\end{figure}
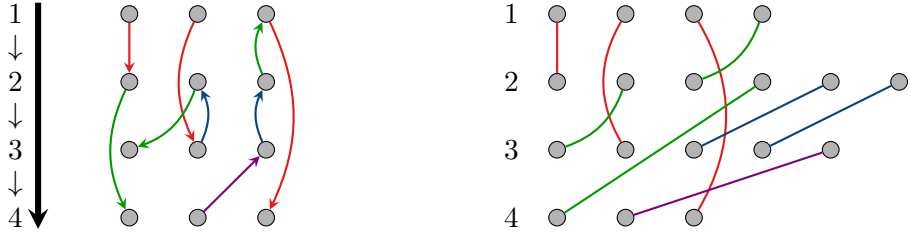

	Finally, note that the sums which originate from the Leibniz rule and the commutation of~$D^k$ and~$\delta^j$ give rise to long expressions with many sums---all of which are, however, \emph{finite}---and the corresponding combinatorial coefficients are collected in the factor~$C_G$, the precise form of which is irrelevant.
	The proof is complete.
\end{proof}

The following lemma provides a uniform-in-$i_{1:4}$ bound for the expectation term in~\eqref{e2:representation_feynman_diagrams}.
\begin{lemma} \label{lem:worst_case_estimate}
	Let~$d \in \N$ and consider~$h_1, \ldots, h_4 \in \mathbb{D}^{3d,4}(\gamma)$, each of which has Hermite rank at least~$d$.
	Consider~$a_{1:4} \in \N_0^4$ such that~$0 \leq a_{[1:4]} \leq 4d$ and~$0 \leq a_k \leq 3d$ for any~$k \in \llbracket 1,4 \rrbracket$.
	Then, for any $\bi = i_{1:4} \in \llbracket 0,N-1 \rrbracket^4$ and~$G \in \mathring{\Gamma}$ as in Lemma~\ref{l:representation_feynman_diagrams}, with~$E_G(\w,\bi)$ given in~\eqref{e2:representation_feynman_diagrams}, the following estimate holds:
	\begin{equation*}
		\abs[0]{E_G(\w,\bi)}
		\lesssim
		\prod_{j=1}^4
		\norm[0]{h_j}_{d,4} 
	\end{equation*}
	where the latter denotes the norm on~$\DD^{d,4}(\gamma)$ and the implicit constant is independent of~$\bi$.
\end{lemma}

\begin{proof}
	We can assume w.l.o.g. that~$a_1 = \max_{j = 1:4} a_j$, otherwise we relabel the indices. 
	
	If~$a_1 \leq 2d$, then we also have~$a_j \leq 2d$ for all~$j = 2,3,4$ since~$a_{[1:4]} \leq 4d$; in that case, there is nothing to prove. In this case, Hölder's inequality combined with Proposition~\ref{p:estimate_hermite_shift} (with~$n = j = d$) immediately implies the claim.
	
	In case~$a_1 \in \llbracket 2d+1, 3d \rrbracket$, we have~$a_j \in \llbracket 0, 4d-a_1 \rrbracket$ for each~$j=2,3,4$. 
	Using integration by parts, Hölder's, and Meyer's inequality (Proposition~\ref{p:meyer_inequality}), we then have that  
	\begin{equs}
		\thinspace &
		\abs[0]{E_G(\w,\bi)}
		=
		\abs[3]{\E\sbr[3]{\prod_{j=1}^4[\CS_{d} h_j]^{(a_j)}(W(e_{\biota_j}))}} \\
		= & \ 
		\abs[3]{\E\sbr[3]{
			\scal{D^{a_1 - 2d} \sbr[1]{\CS_{d} h_1}^{(2d)}(W(e_{\biota_1})), \prod_{j=2}^4[\CS_{d} h_j]^{(a_j)}(W(e_{\biota_j})) e_{\biota_1}^{\otimes (a_1 - 2d)}}_{H^{\otimes (a_1-2d)}}}} \\[0.5em]
		= & \
		\abs[3]{\E\sbr[3]{
			\sbr[1]{\CS_{d} h_1}^{(2d)}(W(e_{\biota_1})) \delta^{a_1 - 2d}\del[3]{\prod_{j=2}^4[\CS_{d} h_j]^{(a_j)}(W(e_{\biota_j})) e_{\biota_1}^{\otimes (a_1 - 2d)}}}} \\[0.5em]
		\leq & \
		\norm[1]{\sbr[1]{\CS_{d} h_1}^{(2d)}}_{L^4(\gamma)} 
		\norm[3]{\delta^{a_1 - 2d}\del[3]{\prod_{j=2}^4[\CS_{d} h_j]^{(a_j)}(W(e_{\biota_j})) e_{\biota_1}^{\otimes (a_1 - 2d)}}}_{L^{\nicefrac{4}{3}}(\Omega)} \\[0.5em]
		\leq & \
		\norm[1]{\sbr[1]{\CS_{d} h_1}^{(2d)}}_{L^4(\gamma)} 
		\sum_{\ell=0}^{a_1 - 2d} \norm[4]{D^\ell \del[4]{\prod_{j=2}^4[\CS_{d} h_j]^{(a_j)}(W(e_{\biota_j})) e_{\biota_1}^{\otimes (a_1 - 2d)}}}_{L^{\nicefrac{4}{3}}(\Omega;H^{\otimes a_1-2d+\ell})}  \,.
	\end{equs}
	We now note that, for each~$j = 2,3,4$, we have
	\begin{equation*}
		a_j + \ell \leq (4d-a_1) + (a_1 - 2d) = 2d, 	
	\end{equation*}
	i.e. the maximum possible number of derivatives~$a_j$ on each factor~$\CS_{d} h_j$ is bounded by~$2d$. 
	The claim follows by using the Leibniz rule for the Malliavin derivative (see~\cite[Exercise~2.3.10]{nourdin_peccati_book}) again to brutally estimate 
	\begin{equs}
		\thinspace &
		\norm[3]{D^\ell \del[3]{\prod_{j=2}^4[\CS_{d} h_j]^{(a_j)}(W(e_{\biota_j})) e_{\biota_1}^{\otimes (a_1 - 2d)}}}_{L^{\nicefrac{4}{3}}(\Omega;H^{\otimes a_1-2d+\ell})} \\
		\lesssim \ & 
		\sum_{\ell_{[2:4]} = \ell} \,
		\norm[3]{\prod_{j=2}^4[\CS_{d} h_j]^{(a_j+\ell_j)}(W(e_{\biota_j}))}_{L^{\nicefrac{4}{3}}(\Omega)} \,,
	\end{equs}
	followed by an application of (the multi-factor) Hölder's inequality (with~$p_j = 4$) and Proposition~\ref{p:estimate_hermite_shift} (with~$n = j = d$, as before). 
\end{proof}

We are now ready to state and prove the main result of this Section~\ref{s:reduction}. 
Recall the conventions introduced in Notation~\ref{notation:reduction}.
\begin{proposition}[Reduction to finite chaos] \label{p:reduction}
	Let~$d \geq 1$ and assume that the fol\-lo\-wing conditions hold:
	\begin{enumerate}[label=(\arabic*)]
		\item \label{thm:reduction:1}
		For any~$k \in \llbracket 1,m \rrbracket$, we have~$f_k \in \mathbb{D}^{d,4}(\gamma)$ and each~$f_k$ has Hermite rank at least~$d$.
		\item \label{thm:reduction:2}
		For all~$k, \ell \in \llbracket 1,m \rrbracket$, we have~$\sum_{i \in \Z} \abs[0]{\rho_{k,\ell}(i)}^{d} < \infty$, i.e.~$\rho_{k,\ell} \in \ell^{d}(\Z)$.
	\end{enumerate}
	Then, for any~$l \in \N$, the vector~$\CR_N^M \in \del[1]{\tilde{T}^{(2)}(\R^m)}^l$ given in~\eqref{e:def_RNM} satisfies Assumptions~\ref{ass:R1} and~\ref{ass:R2} of Proposition~\ref{prop:abstract_convergence_statement}.
\end{proposition}

\begin{proof}
	It is clear that, in order to verify Assumptions~\ref{ass:R1} and~\ref{ass:R2} in Proposition~\ref{prop:abstract_convergence_statement}, we can assume that~$l = 1$ and then proceed componentwise.	
	We start with the $m$-dimensional components of~$\CR_N^M$, in which case the argument is straightforward:
	\begin{equs}
	\norm[1]{S^{k;> M}_N(s,t)}_{L^2(\P)}^2
	& =
	\frac{1}{N} \sum_{i,j = \lfloor Ns \rfloor}^{\lfloor Nt \rfloor - 1} \E\sbr[1]{f^{> M}_k(X^{(k)}_i) f^{> M}_k(X^{(k)}_j)} \\[0.5em]
	& = 
	\frac{1}{N} \sum_{i,j = \lfloor Ns \rfloor}^{\lfloor Nt \rfloor - 1} \sum_{q > M}\sbr[1]{c_q^{(k)}}^2 q! \rho_{k,k}(j-i)^q
	\lesssim \norm[0]{\rho_{k,k}}_{\ell^{d}(\Z)} \norm[0]{f_k^{> M}}^2_{L^2(\gamma)}
\end{equs}
and the RHS (which is uniformly bounded in~$N \in \N$) goes to~$0$ as~$M \to \infty$ because chaos tails decay in~$L^2(\gamma)$.
Since this decay, in general, \emph{does not} hold in~$L^p(\gamma)$ for any~$p > 2$ (see~\cite[Thm.~5.18]{janson}), the argument for the~$(m \times m)$-dimensional components of~$\CR_N^M$ is substantially more involved. 

At first, we observe that we can deal with each matrix component in each of the $(m \times m)$-dimensional entries of~$\CR_N^M$ separately.
To this end, we fix~$k,\ell \in \llbracket 1,m \rrbracket$ and look at~$\mathbb{S}_N^{k,\ell;> M}(s,t)$ (recall the notation in~\eqref{e:2nd_order_tail:1}).
Since the interval boundaries~$s < t$ do not matter at all for our arguments, we will w.l.o.g. assume that~$s= 0$ and~$t=1$ and simply suppress the dependence on~$s$ and~$t$ in our arguments. 
By definition in~\eqref{e:2nd_order_tail},~$\mathbb{S}_N^{k,\ell;> M}$ can be further decomposed into~$\mathbb{S}_N^{k,\ell;> M, [a]}$ for~$a= 1,2,3$.
Since all of these summands can be analysed analogously, we will henceforth only deal with~$\mathbb{S}_N^{k,\ell;> M, [3]}$; 
the only information that is going to be important in our analysis is that at least one of the factors is projected onto chaos components bigger than~$M$, which is the case for all three summands.

Further, recall the notation that we have introduced in the paragraph on p.~\pageref{notation_reduction}. 
In particular, in eq.~\eqref{e:assignment:h_strategy} and~\eqref{e:covariance_strategy} we have set
\begin{equation} \label{e:assignment:h}
	h_j(W(e_{i_j})) = 
	f_{\w_j}^{>M}(W(e_{\biota_j})), \quad \rho(\biota_k - \biota_j) = \rho_{\w_j,\w_k}(i_k - i_j), \quad \biota_j = (\w_j,i_j)
\end{equation}
and, by doubling the components, we have obtained the following identity in~\eqref{e:L2_norm_reduction_strategy}:
\begin{equs}[][e:L2_norm_reduction]
	\norm[1]{\mathbb{S}_N^{k,\ell;> M, [1]}}_{L^2(\P)}^2
	& = 
	\frac{1}{N^2} \sum_{\substack{0 \leq i_1 < i_2 \leq N-1, \\ 0 \leq i_3 < i_4 \leq N-1}} \E\sbr[4]{\prod_{j=1}^4 h_j(W(e_{\biota_j}))}\,.
\end{equs}
Recall also that we write~$\w = \w_{1:4}$, $\bq = q_{1:4}$, and~$\bi = i_{1:4}$.
We now proceed in various steps. 

\vspace{0.5em}
\noindent
$\triangleright$ \textit{Step 1: Regularisation by the OU semigroup.} \label{triangle_OU_regularisation}
This approximation step is inspired by a similar strategy in the proof of~\cite[Thm.~1.2]{nourdin_nualart_peccati_21}, see their pp.~11-12.
Recall from Definition~\ref{d:OU_operator} that~$(P_r)_r$ denotes the OU semigroup and recall that we have assumed that each~$f_k \in \mathbb{D}^{d,4}(\gamma)$ which is inherited by~$f_k^{>M}$ and~$f_k^{> M}$ for each fixed~$M$. 
Thus,~$h_j \in \mathbb{D}^{d,4}(\gamma)$ for~$j \in \llbracket 1,4 \rrbracket$ and, by Lemma~\ref{lem:OU_regularisation} (with~$p = 4$), we know that the regularisations~$h_j^{[r]} = P_{1/r} h_j$ satisfy~$h_j^{[r]} \in \mathbb{D}^{3d,4}(\gamma)$ as well as 
\begin{equation*}
	h_j^{[r]} \overset{r \to \infty}{\longrightarrow} h_j \quad \text{in} \quad \mathbb{D}^{d,4}(\gamma) \,.
	%, \quad 
\end{equation*} 
Therefore, we can replace each~$h_j$ in~\eqref{e:L2_norm_reduction} by~$h_j^{[r]}$ and, if we can show the claimed convergences assuming~$h_j \in \mathbb{D}^{3d,4}(\gamma)$ with bounds that only depend on 
	$\norm[0]{h_j}_{d,4}$,  
then we are done.

\vspace{0.5em}
\noindent
$\triangleright$ \textit{Step 2: Splitting into regular and irregular diagrams.}
Note that, thanks to the previous approximation argument, we are now in the setting of Lemma~\ref{l:representation_feynman_diagrams} above, which implies that 
\begin{equs}[][e:L2_norm_reduction:2]
	\norm[1]{\mathbb{S}_N^{k,\ell;> M, [1]}}_{L^2(\P)}^2
	& = 
	\sum_{G \in \mathring{\Gamma}} C_G \thinspace \sum_{\substack{0 \leq i_{1} < i_2 \leq N-1, \\ 0 \leq i_{3} < i_4 \leq N-1}}  \frac{\fC_G(\w,\bq,\bi)}{N^2} E_G(\w,\bi) 
	\eqqcolon 
	S_{\text{\tiny IR}} + S_{\text{\tiny R}}
	\,.
\end{equs}
where~$S_{\text{\tiny IR}}$ denotes the sum over~\emph{irregular} diagrams, $G \in \mathring{\Gamma} \cap \Gammairreg$, and~$S_{\text{\tiny R}}$ denotes the sum over~\emph{regular} diagrams, $G \in \mathring{\Gamma} \cap \Gammareg$, both in the sense of Definition~\ref{d:regular_graph}.
Recall that both sums are over \emph{finitely many} diagrams since~$\mathring{\Gamma}$ has finite cardinality.

\vspace{0.5em}
\noindent
$\triangleright$ \textit{Step 3: Analysis of irregular diagrams.}
By Lemma~\ref{lem:worst_case_estimate}: 
	\begin{equation*}
		\abs[0]{E_G(\w,\bi)}
		\lesssim
		\prod_{j=1}^4
		\norm[0]{h_j}_{d,4}. 
	\end{equation*}
Recall that the~$h_j$ implicitly depend on~$M$ via~\eqref{e:assignment:h} and that, for each fixed~$M \in \N$, the quantity on the RHS of the previous estimate is finite by assumption and does not depend on~$i_{1:4}$.
Therefore, since the combinatorial factor~$C_G$ is positive as well, we find that 
\begin{equation}\label{e:bound_IR_terms}
	\abs[0]{S_{\text{\tiny IR}}}
	\leq 
	\sum_{G \in \mathring{\Gamma} \cap \Gammairreg} C_G \thinspace \sum_{\substack{0 \leq i_{1} < i_2 \leq N-1, \\ 0 \leq i_{3} < i_4 \leq N-1}} \frac{\abs[0]{\fC_G(\w,\bq,\bi)}}{N^2} \prod_{j=1}^4 \norm[0]{h_j}_{d,4} \,. %\\
\end{equation}
With the set~$A_{4,N}$ defined as
\begin{equation*}
	A_{4,N} \coloneqq \cbr[0]{i_{1:4} \in \Z^4: \ 0 \leq i_1 < i_2 \leq N-1, \ 0 \leq i_3 < i_4 \leq N-1} \,,
\end{equation*}
Proposition~\ref{p:bm83_irregular} now implies that 
\begin{equation*}
	0 \leq
	\lim_{N \to \infty} 
	\sum_{i_{1:4} \in A_{4,N}} \frac{\abs[0]{\fC_G(\w,\bq,\bi)}}{N^2} 
	\leq
	\lim_{N \to \infty} 
	\mathring{T}_G(\w,\bq,N) 
	= 0
\end{equation*}
for each~$G \in \mathring{\Gamma} \cap \Gammairreg$. 
Therefore, since we sum over~\emph{finitely many}~$G \in \mathring{\Gamma} \cap \Gammairreg$, for each fixed~$M \in \N$, we have 
\begin{equation} \label{e:convergence_IR_terms}
	\lim_{N \to \infty} S_{\text{\tiny IR}} = 0 \,.
\end{equation}

\vspace{0.5em}
\noindent
$\triangleright$ \textit{Step 4: Analysis of regular diagrams.}
In order to deal with the regular diagrams, we will use the graphical formalism introduced in Figures~\ref{fig:worst_case_derivatives} and~\ref{fig:conversion_to_feynman_diagram} above. 
A regular diagram looks as follows: 
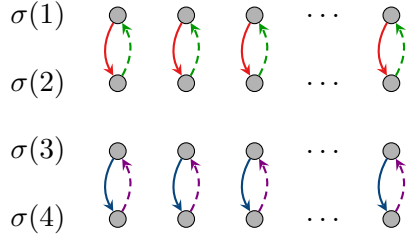
\begin{figure}[h]
		\centering
			\begin{tikzpicture}[scale=0.3,baseline=0cm]
				%%%%%%%%
				\node at (2.5,3)  [] [label={[label distance=0.05cm]90:}] (1) {$\sigma(1)$};
				\node at (2.5,0)  [] [label={[label distance=0.05cm]90:}] 	(2) {$\sigma(2)$};
				\node at (2.5,-3)  [] [label={[label distance=0.05cm]90:}] 	(3) {$\sigma(3)$};
				\node at (2.5,-6)  [] [label={[label distance=0.05cm]90:}] 	(4) {$\sigma(4)$};	
				%%%%%%%%
				%%%%%%%%
				\node at (6,3)  [xibig] [label={[label distance=0.05cm]90:}] (11) {};
				\node at (9,3)  [xibig] [label={[label distance=0.05cm]90:}] (12) {};
				\node at (12,3)  [xibig] [label={[label distance=0.05cm]90:}] (13) {};
				\node at (15,3)   [label={[label distance=0.05cm]90:}] (14) {$\ldots$};
				\node at (18,3)  [xibig] [label={[label distance=0.05cm]90:}] (15) {};
				%%%%%
				\node at (6,0)  [xibig] [label={[label distance=0.05cm]-90:}] (21) {};
				\node at (9,0)  [xibig] [label={[label distance=0.05cm]-90:}] (22) {};
				\node at (12,0)  [xibig] [label={[label distance=0.05cm]-90:}] (23) {};
				\node at (15,0)   [label={[label distance=0.05cm]90:}] (24) {$\ldots$};
				\node at (18,0)  [xibig] [label={[label distance=0.05cm]90:}] (25) {};
				%%%%%
				\node at (6,-3)  [xibig] [label={[label distance=0.05cm]90:}] (31) {};
				\node at (9,-3)  [xibig] [label={[label distance=0.05cm]90:}] (32) {};
				\node at (12,-3)  [xibig] [label={[label distance=0.05cm]90:}] 
				(33) {};
				\node at (15,-3)   [label={[label distance=0.05cm]90:}] (34) {$\ldots$};
				\node at (18,-3)  [xibig] [label={[label distance=0.05cm]90:}] (35) {};
				%%%%%
				\node at (6,-6)  [xibig] [label={[label distance=0.05cm]-90:}] (41) {};
				\node at (9,-6)  [xibig] [label={[label distance=0.05cm]-90:}] (42) {};
				\node at (12,-6)  [xibig] [label={[label distance=0.05cm]-90:}] (43) {};
				\node at (15,-6)   [label={[label distance=0.05cm]90:}] (44) {$\ldots$};
				\node at (18,-6)  [xibig] [label={[label distance=0.05cm]90:}] (45) {};
				%%%%%
				\draw[thick, ->, color=darkred] (11) to [bend right=30] (21);
				\draw[thick, ->, color=darkred] (12) to [bend right=30] (22);
				\draw[thick, ->, color=darkred] (13) to [bend right=30] (23);
				\draw[thick, ->, color=darkred] (15) to [bend right=30] (25);
				%%%%
				%%%%%
				\draw[thick, ->, densely dashed, color=darkgreen] (21) to [bend right=30] (11);
				\draw[thick, ->, densely dashed, color=darkgreen] (22) to [bend right=30] (12);
				\draw[thick, ->, densely dashed, color=darkgreen] (23) to [bend right=30] (13);
				\draw[thick, ->, densely dashed, color=darkgreen] (25) to [bend right=30] (15);
				%%%%
				\draw[thick, ->, color=darkcerulean] (31) to [bend right=30] (41);
				\draw[thick, ->, color=darkcerulean] (32) to [bend right=30] (42);
				\draw[thick, ->, color=darkcerulean] (33) to [bend right=30] (43);
				\draw[thick, ->, color=darkcerulean] (35) to [bend right=30] (45);
				%%%%
				\draw[thick, ->, densely dashed, color=violet] (41) to [bend right=30] (31);
				\draw[thick, ->, densely dashed, color=violet] (42) to [bend right=30] (32);
				\draw[thick, ->, densely dashed, color=violet] (43) to [bend right=30] (33);
				\draw[thick, ->, densely dashed, color=violet] (45) to [bend right=30] (35);
				%%%%%%
				%
			\end{tikzpicture}\; 
		\caption{Structure of regular diagrams. 
		}
		\label{fig:regular_diagrams_derivatives}
	\end{figure}

\vspace{1em}
Figure~\ref{fig:regular_diagrams_derivatives} has the following structural features:
\begin{enumerate}[label=(\arabic*)]
	\item For visual clarity, we have applied a permutation~$\sigma \in \mathfrak{S}(1,2,3,4)$ of the indices~$1$,$2$,$3$, and~$4$. 
	Any regular diagram in the sense of Definition~\ref{d:regular_graph} (see also Figure~\ref{subfig:regular_diagram}) can always be arranged in that way:
	This reflects the fact that we need to match the rows in pairs and then connect the nodes only \emph{within} those pairs.
	We refer to Remark~\ref{rmk:level_labels} for the interpretation of the level labels~$\sigma(1),\ldots,\sigma(4)$ on the left hand side.
	\item According to point~\ref{item:graphical_formalism_free_legs} of the algorithm on p.~\pageref{item:graphical_formalism_free_legs}, each node that has already been hit by a derivative can either contribute a derivative itself or not. This is symbolised by the green resp. violet \emph{dashed} arrows. Note, in particular, that a regular diagram could have any number~$k_g, k_v \in \llbracket 0,d \rrbracket$ of green resp. violet arrows.
\end{enumerate}
For concreteness, we will henceforth consider~$\sigma = \operatorname{Id}$, i.e.~$\sigma(j) = j$ for~$j \in \llbracket 1,4 \rrbracket$, corresponding to the pairing~$12 \sVert[0] 34$. One can deal with all other cases in exactly the same way.
The sum~$S_{\text{\tiny R}}$ then becomes a \emph{finite} sum of terms of the following type:
\begin{equs}
	\thinspace & 
	\CG_N^M(12 \sVert 34) \\
	\coloneqq \ & \frac{1}{N^2} 
	\sum_{\substack{0 \leq i_{1} < i_2 \leq N-1, \\ 0 \leq i_{3} < i_4 \leq N-1}} \rho(\biota_2 - \biota_1)^{d+k_g} \rho(\biota_4 - \biota_3)^{d+k_v} \times \\
	\qquad & \times  \E\sbr[1]{[\CS_{d} h_1]^{(k_g)}(W(e_{\biota_1})) [\CS_{d} h_2]^{(d)}(W(e_{\biota_2})) [\CS_{d} h_3]^{(k_v)}(W(e_{\biota_3}))[\CS_{d} h_4]^{(d)}(W(e_{\biota_4}))}
\end{equs}
The \emph{worst case scenario} in terms of required derivatives is~$k_g = k_v = d$ whereas the worst-case scenario in terms of the covariance powers is~$k_g = k_v = 0$:
Therefore, we will henceforth deal with the term
\begin{equ}
	\bar{\CG}_N^M(12 \sVert 34) 
	\coloneqq \frac{1}{N^2} 
	\sum_{\substack{0 \leq i_{1} < i_2 \leq N-1, \\ 0 \leq i_{3} < i_4 \leq N-1}} \rho(\biota_2 - \biota_1)^{d} \rho(\biota_4 - \biota_3)^{d} \E\sbr[0]{F_1 F_2 F_3 F_4}
\end{equ}
where we have incorporated \emph{both worst case scenarios} and, for notational simplicity, have set
\begin{equation} \label{e:assignment:F}
	F_j \coloneqq [\CS_{d} h_j]^{(d)}(W(e_{\biota_j})), \quad j \in \llbracket 1,4 \rrbracket \,.
\end{equation}

We will now proceed using integration by parts \emph{once} according to how we have paired the rows, i.e. in our case~$\sigma = \operatorname{Id}$, we have pairs~$1 \leftrightarrow 2$ and~$3 \leftrightarrow 4$:
	\begin{equation} \label{e:gaussian_IBP_applied}
		\E\sbr[1]{F_1F_2F_3F_4} 
		= \E\sbr[1]{F_1F_2}\E\sbr[1]{F_3F_4} + \E\sbr[1]{\scal{D(F_1F_2),-DL^{-1}(F_3F_4)}_H}\,.
	\end{equation}
We will now split up the remaining analysis into two parts and define 
\begin{equs}
	\bar{\CG}_N^{M;1}(12 \sVert 34) 
	& \coloneqq \frac{1}{N^2} 
	\sum_{\substack{0 \leq i_{1} < i_2 \leq N-1, \\ 0 \leq i_{3} < i_4 \leq N-1}} \rho(\biota_2 - \biota_1)^{d} \rho(\biota_4 - \biota_3)^{d} \times \\
	& \qquad \qquad \qquad \qquad \times \E\sbr[1]{\scal{D(F_1F_2),-DL^{-1}(F_3F_4)}_H} \,,\\
	\bar{\CG}_N^{M;2}(12 \sVert 34) 
	& \coloneqq \frac{1}{N^2} 
	\sum_{\substack{0 \leq i_{1} < i_2 \leq N-1, \\ 0 \leq i_{3} < i_4 \leq N-1}} \rho(\biota_2 - \biota_1)^{d} \rho(\biota_4 - \biota_3)^{d} \E\sbr[1]{F_1F_2}\E\sbr[1]{F_3F_4} \,.
\end{equs}

\vspace{0.5em}
\noindent
$\triangleright$ \textit{Step 4a: Analysis of the term~$\bar{\CG}_N^{M;1}(12 \sVert 34)$ involving derivatives.}
	Recall that~$L^{-1}$ is the \emph{pseudo}inverse of~$L$, which, by definition, means that~$L^{-1} F = L^{-1}(F - \E\sbr[0]{F})$; therefore, we can always commute~$-DL^{-1}F = (1-L)^{-1}DF$, even if~$\E\sbr[0]{F} \neq 0$.
	Performing this change in the second summand in~\eqref{e:gaussian_IBP_applied} and applying the Leibniz rule for the Malliavin derivative~(see~\cite[Exercise~2.3.10]{nourdin_peccati_book}), we find that
	\begin{equs}[][e:gaussian_IBP_applied:2]
		\thinspace &
		\E\sbr[1]{\scal{D(F_1F_2),-DL^{-1}(F_3F_4)}} = 
		\E\sbr[1]{\scal{D(F_1F_2),(1-L)^{-1}D(F_3F_4)}} \\[0.5em]
		& = \E\sbr[1]{F_1'F_2 (1-L)^{-1}(F_3' F_4)} \rho(\biota_3 - \biota_1)
		+
		\E\sbr[1]{F_1'F_2 (1-L)^{-1}(F_3 F_4')} \rho(\biota_4 - \biota_1) \\[0.5em]
		& \quad +
		\E\sbr[1]{F_1 F_2' (1-L)^{-1}(F_3' F_4)} \rho(\biota_3 - \biota_2)
		+
		\E\sbr[1]{F_1 F_2' (1-L)^{-1}(F_3 F_4')} \rho(\biota_4 - \biota_2) 
	\end{equs} 
	where we have written~$F_j' = \sbr[0]{\CS_{d} h_j}^{(d+1)}(W(e_{\biota_j}))$ for~$F_j$ as given in~\eqref{e:assignment:F} and~$h_j$ as given in~\eqref{e:assignment:h}.
	By our assumption that each component~$f_k$ is in $\mathbb{D}^{d,4}(\gamma)$, we can bound each of the summands in the previous display by Hölder's inequality, applied to the expectations, in a similar way. 
	For example, for the first term we find
	\begin{equs}[][e:gaussian_IBP_applied:3]
		\thinspace & 
		\abs[1]{\E\sbr[1]{F_1'F_2 (1-L)^{-1}(F_3' F_4)}} \abs[0]{\rho(\biota_3 - \biota_1)} \\
		\leq \ & 
		\norm[0]{F_1'}_{L^4(\P)} \norm[0]{F_2}_{L^4(\P)}\norm[0]{F_3'}_{L^4(\P)}\norm[0]{F_4}_{L^4(\P)} \abs[0]{\rho(\biota_3 - \biota_1)} \\
		= \ & 
		\norm[0]{[\CS_{d} f_k^{>M}]^{(d+1)}}_{L^4(\gamma)}^2 \norm[0]{\sbr[0]{\CS_{d} f_\ell^{M}}^{(d)}}_{L^4(\gamma)}^2 \abs[0]{\rho(\biota_3 - \biota_1)} \\
		\leq \ & 
		\norm[0]{f_k^{> M, (1)}}_{L^4(\gamma)}^2 \norm[0]{f_\ell^{M}}_{L^4(\gamma)}^2 \abs[0]{\rho(\biota_3 - \biota_1)}
	\end{equs}
	where the last step is due to Proposition~\ref{p:estimate_hermite_shift} and the RHS is finite for each fixed~$M \in \N$. 
	(In fact, we have only used that~$f_k \in \mathbb{D}^{1,4}(\gamma)$ for this argument.)
	
	Therefore, we find that for each fixed~$M \in \N$, we have the bound 
	\begin{equs}
		\abs[1]{\bar{\CG}_N^{M;1}(12 \sVert 34)} 
		& \lesssim_M \frac{1}{N^2} 
		\sum_{\substack{0 \leq i_{1} < i_2 \leq N-1, \\ 0 \leq i_{3} < i_4 \leq N-1}} \abs[0]{\rho(\biota_2 - \biota_1)}^{d} \abs[0]{\rho(\biota_4 - \biota_3)}^{d} \times \\
		& \quad \times \del[1]{
			\abs[0]{\rho(\biota_3 - \biota_1)}
			+
		 	\abs[0]{\rho(\biota_4 - \biota_1)}
			+
			\abs[0]{\rho(\biota_3 - \biota_2)}
			+
			\abs[0]{\rho(\biota_4 - \biota_2)}}
	\end{equs}
	By Lemma~\ref{lem:path_sum_covariances}, the RHS goes to~$0$ as~$N \to \infty$.
	Note that the previous argument works because, visually, each of the four covariances we produced from the integration-by-parts creates a \enquote{bridge} between the pairs $1 \leftrightarrow 2$ and~$3 \leftrightarrow 4$, for example~$2 \leftrightarrow 1 \boldsymbol{\leftrightarrow} 3 \leftrightarrow 4$ for the first summand.
	
	\vspace{0.5em}
	\noindent
	$\triangleright$ \textit{Step 4b: Analysis of the term~$\bar{\CG}_N^{M;2}(12 \sVert 34)$ involving a product of expectations.}
	For this term, we will show that it is uniformly bounded in~$N \in \N$ but goes to~$0$ as~$M \to \infty$. We have
	\begin{equs}
		\abs[1]{\bar{\CG}_N^{M;2}(12 \sVert 34)}
		& \leq \frac{1}{N^2} 
		\sum_{\substack{0 \leq i_{1} < i_2 \leq N-1, \\ 0 \leq i_{3} < i_4 \leq N-1}} \abs[0]{\rho(\biota_2 - \biota_1)}^{d} \abs[0]{\rho(\biota_4 - \biota_3)}^{d} \abs[0]{\E\sbr[0]{F_1F_2}}\abs[0]{\E\sbr[0]{F_3F_4}} \\
		& \leq \norm[0]{\rho_{\w_1,\w_2}}_{\ell^{d}(\Z)}\norm[0]{\rho_{\w_3,\w_4}}_{\ell^{d}(\Z)} \prod_{j=1}^4 \norm[0]{F_j}_{L^2(\Omega)} \\
		& = \norm[0]{\rho_{\w_1,\w_2}}_{\ell^{d}(\Z)}\norm[0]{\rho_{\w_3,\w_4}}_{\ell^{d}(\Z)} \prod_{j=1}^4 \norm[0]{[\CS_{d} h_j]^{(d)}}_{L^2(\gamma)}  \\
		& \leq \norm[0]{\rho_{\w_1,\w_2}}_{\ell^{d}(\Z)}\norm[0]{\rho_{\w_3,\w_4}}_{\ell^{d}(\Z)} \prod_{j=1}^4 \norm[0]{\CS_{d} h_j}_{d,2} \\
		& \lesssim \norm[0]{\rho_{\w_1,\w_2}}_{\ell^{d}(\Z)}\norm[0]{\rho_{\w_3,\w_4}}_{\ell^{d}(\Z)} \prod_{j=1}^4 \norm[0]{h_j}_{L^2(\gamma)} 
	\end{equs}
	where the penultimate estimate is trivial and the last one is due to Proposition~\ref{p:estimate_hermite_shift}, eq.~\eqref{e:estimate_shift} (with~$k=n=d$ and~$p=2$.)
	The RHS is uniformly bounded in~$N \in \N$ and goes to~$0$ as~$M \to \infty$ because at least two factors~$h_j$ contain projections onto chaoses~$> M$ (recall that~$h_j$ depends on~$M$ via~\eqref{e:assignment:h})and chaos tails converge in~$L^2(\gamma)$.
	
	\vspace{0.5em}
	\noindent
	$\triangleright$ \textit{Conclusion of Step 4.}
	In summary, we have written the finite sum~$S_{\text{\tiny R}}$ over terms like~$\bar{\CG}_N^{M}(12 \sVert 34)$ which we have then split into two parts again: 
	\begin{enumerate}[label=(\arabic*)]
		\item One which, for fixed~$M \in \N$, goes to~$0$ as~$N \to \infty$. These are the terms corresponding to~$\bar{\CG}_N^{M;1}(12 \sVert 34)$.
		\item One which is uniformly bounded in~$N \in \N$ where the bound goes to~$0$ as~$M \to \infty$. These are the terms corresponding to~$\bar{\CG}_N^{M;2}(12 \sVert 34)$.
	\end{enumerate}
	
	\vspace{0.5em}
	\noindent
	$\triangleright$ \textit{Summary.} 
	Recall that~$S_{\text{\tiny IR}}$ and~$S_{\text{\tiny R}}$ have been introduced in~\eqref{e:L2_norm_reduction:2} and, implicitly, depend both on~$N$ and~$M$.
	All the previous steps combined imply that  
	\begin{equs}
		\lim_{N \to \infty} S_{\text{\tiny IR}} = 0 \quad \text{for all} \ M \geq d, \quad 
		\lim_{M \to \infty} \limsup_{N \to \infty} S_{\text{\tiny R}} = 0 \,.		
	\end{equs}
	The proof is complete.
\end{proof}

\subsection{A law of large numbers on the diagonal}  \label{s:LLN_diagonal}

In this section, we verify Assumption~\ref{ass:D1} of Proposition~\ref{prop:abstract_convergence_statement} and characterise the limit~$\CD^M$. 
This corresponds to Part~\ref{strategy:2} of the strategy outlined on p.~\pageref{strategy:2}.

In order to further motivate the definition of~$\CD_N^M$ in~\eqref{e:def_DNM}, observe that it arises from the following decomposition:
\begin{equs}[][e:decomp_SymSSN]
	\del[1]{\operatorname{Sym} \S_N^M}
	& =
	\frac{1}{2N}\del[4]{
		\sum_{0 \leq i < j \leq N-1} \vec{f}^M(X_i) \otimes \vec f^M(X_j) + \vec f^M(X_j) \otimes \vec f^M(X_i)
	} \\
	& = 
	\frac{1}{2N} \del[4]{\sum_{i=0}^{N-1} \vec f^M(X_i)}^{\otimes 2} - \frac{1}{2N} \sum_{i=0}^{N-1} \vec f^M(X_i) \otimes \vec f^M(X_i) 
	=
	\frac{1}{2} \del[1]{S_N^M}^{\otimes 2} - \frac{1}{2} \DD_N^M
\end{equs}
where we have written~$\DD_N^M \coloneqq \DD_N^M(0,1)$ for 
\begin{equation} \label{d:D_N}
	\mathbb{D}^M_N(s,t) 
	\coloneqq  
	\frac{1}{N} \sum_{i=\lfloor Ns \rfloor}^{\lfloor Nt \rfloor-1} \vec{f}(X_i) \otimes  \vec{f}(X_i)
\end{equation}

The following proposition is the main result of this section and can be seen as a \emph{law of large numbers on the diagonal}.
Recall that~$ \Delta^M(0) = \E\sbr[0]{\vec{f}^M(X_1) \otimes \vec{f}^M(X_1)}$.

\begin{proposition}[Diagonal LLN] \label{p:LLN_diagonal}
	We have
	\begin{equation} \label{p:LLN_diagonal:e1}
		\lim_{N \to \infty} \mathbb{D}^M_N(s,t)
		=
		(t-s) \thinspace \Delta^M(0) \, \quad \text{in} \quad L^2(\P) \,.
	\end{equation}
	As a consequence, Proposition~\ref{prop:abstract_convergence_statement}, Assumption~\ref{ass:D1} is verified with 
	\begin{equation} \label{e:def_lim_DM}
		\CD^M \coloneqq (0,(t_1-s_1),0,(t_2-s_2), \ldots, 0,(t_l - s_l)) \Delta^M(0) \in \del[1]{\tilde{T}^{(2)}(\R^m)}^l \,.	
	\end{equation}
\end{proposition}
\begin{proof}
	This is an immediate consequence of Lemmas~\ref{lem:fdd_expectation} and~\ref{l:diagonal_variance} below.
\end{proof}

In order to close the proof of Proposition~\ref{p:LLN_diagonal}, we require the following two lemmas: 
The first shows that~$\DD_N^M$ converges to the right quantity \emph{in expectation} and the second shows that its \emph{variance vanishes} as~$N \to \infty$.
\begin{lemma} \label{lem:fdd_expectation}
	For fixed~$M$ and~$s,t \in [0,1]$ with~$s < t$, we have
	\begin{equation*}
		\lim_{N \to \infty} \E\sbr[1]{\DD^M_N(s,t)}
		=
		(t-s) \thinspace \Delta^M(0) \,.
	\end{equation*}
\end{lemma}

\begin{proof}
	By stationarity, we have 
	\begin{equation*}
		\E\sbr[1]{\DD^M_N(s,t)}
		= 
		\frac{1}{N} \sum_{i=\lfloor Ns \rfloor}^{\lfloor Nt \rfloor -1} \E\sbr[1]{\vec{f}(X_i) \otimes  \vec{f}(X_i)}
		= 
		\frac{\lfloor Nt \rfloor - \lfloor Ns \rfloor}{N} \Delta^M(0) \,.
	\end{equation*}
	The claim follows immediately. 
	% %
\end{proof}

\begin{lemma} \label{l:diagonal_variance}
	For each~$M \geq d$ and~$k,\ell \in \llbracket 1,m \rrbracket$, we have 
	\begin{equation}
		\lim_{N \to \infty} \V\del[1]{\DD_N^{M;k,\ell}}	= 0 \,.
	\end{equation}
\end{lemma}

\begin{proof}
	By definition of~$\DD_N^{M;k,\ell}$, we have 
	\begin{equs}
		\thinspace 
		& \V\del[1]{\DD_N^{M;k,\ell}} 
		= 
		\V\del[4]{\frac{1}{N}\sum_{i=0}^{N-1} f_k^M(X^{(k)}_i)f_\ell^M(X^{(\ell)}_i)} \\
		= \ &  
		\sum_{i,j=0}^{N-1} \E\sbr[1]{f_k^M(X^{(k)}_i)f_\ell^M(X^{(\ell)}_i)f_k^M(X^{(k)}_j)f_\ell^M(X^{(\ell)}_j)} \\
		%%%%
		& \qquad \qquad \qquad - 
		\E\sbr[1]{f_k^M(X^{(k)}_i)f_\ell^M(X^{(\ell)}_i)} \E\sbr[1]{f_k^M(X^{(k)}_j)f_\ell^M(X^{(\ell)}_j)}
	\end{equs}
	Recall that~$\square_{4,N} = \llbracket 0, N-1 \rrbracket^4$ and let~$A_{4,N}$ be the set given by 
	\begin{equation*}
		A_{4,N}
		 \coloneqq \cbr[1]{\boldsymbol{i} = i_{1:4} \in \square_{4,N}: i_1 = i_2, \ i_3 = i_4} \,,
	\end{equation*}
	and note that 
	\begin{equs}
		\thinspace & 
		\sum_{i,j=0}^{N-1} \E\sbr[1]{f_k^M(X^{(k)}_i)f_\ell^M(X^{(\ell)}_i)f_k^M(X^{(k)}_j)f_\ell^M(X^{(\ell)}_j)} \\
		= \ & 
		\sum_{i_{1:4} \in A_{4,N}} \E\sbr[1]{f_k^M(X^{(k)}_{i_1})f_\ell^M(X^{(\ell)}_{i_2})f_k^M(X^{(k)}_{i_3})f_\ell^M(X^{(\ell)}_{i_4})} \,.
	\end{equs}
	By Corollary~\ref{coro:BM83}, we therefore have, as~$N \to \infty$, 
	\begin{equs}
		\V\del[1]{\DD_N^{M;k,\ell}} 
		\asymp 
		\frac{1}{N^2} \sum_{i,j=0}^{N-1} 
			& \left( \E\sbr[1]{f_k^M(X^{(k)}_i)f_k^M(X^{(k)}_j)} \E\sbr[1]{f_\ell^M(X^{(\ell)}_i)f_\ell^M(X^{(\ell)}_j)} \right. \\[0.5em]
			& 
			+ \left.
			\E\sbr[1]{f_k^M(X^{(k)}_i)f_\ell^M(X^{(\ell)}_j)} \E\sbr[1]{f_k^M(X^{(k)}_j)f_\ell^M(X^{(\ell)}_i)}
			\right)
	\end{equs}
	where we emphasise that each of the expectations in the previous expression contain~$i \leftrightarrow j$ pairings that will produce covariances encoded in~$\Delta^M_{k,k}(j-i)$, $\Delta^M_{\ell,\ell}(j-i)$,~$\Delta^M_{k,\ell}(j-i)$, and~$\Delta^M_{\ell,k}(j-i) = \Delta^M_{k,\ell}(i-j)^\top$.
	Indeed: We then have 
	\begin{equs}[][e:var_comp]
		\V\del[1]{\DD_N^{M;k,\ell}}
		& \asymp 
		\frac{1}{N^2} 
			\sum_{i,j=0}^{N-1}\del[2]{	
				\Delta^{M}_{k,k}(j-i) \Delta^{M}_{\ell,\ell}(j-i)
				+
				\Delta^{M}_{k,\ell}(j-i)\Delta^{M}_{\ell,k}(j-i)
			} \\
		& = 	
		\frac{1}{N^2} 
			\sum_{r=-(N-1)}^{N-1} (N-r) \del[2]{	
				\Delta^{M}_{k,k}(r) \Delta^{M}_{\ell,\ell}(r)
				+
				\Delta^{M}_{k,\ell}(r) \Delta^{M}_{\ell,k}(r)
			} 
	\end{equs}
	By assumption, we know that 
	\begin{equation*}
		\Delta^M_{\v,\w} \in \ell^1(\Z) \quad \text{for all} \quad \v,\w \in \cbr[0]{k,\ell} \subseteq \llbracket 1, m \rrbracket	
	\end{equation*}
	and so the RHS of~\eqref{e:var_comp} is~$O(N^{-1})$ by Cauchy--Schwarz. 
	This completes the proof. 
\end{proof}

\subsection{Moment computations} \label{s:moment_computations}

By the results in the last section, we explicitly control the diagonal~$\DD_N^M$ in~$L^2(\P)$ as~$N \to \infty$, for each fixed~$M \in \N$. 
As discussed in Section~\ref{s:strategy}, specifically~\eqref{e:def_DNM}, this allows us to focus on 
\begin{equation} \label{e:def_YMN_2}
	\mathbb{Y}^M_N(s,t) := \S^M_N(s,t) + \frac{1}{2}\mathbb{D}^M_N(s,t)
\end{equation}
instead of~$\S_N^M(s,t)$, i.e. to switch from the piecewise constant to the piecewise linear approximation.
The following two lemmas makes that observation rigorous. 

\begin{lemma}
\label{lem:conv_cont_case}
	Fix~$N \in \N$ and for~$i \in \llbracket 0,N-1 \rrbracket$ define
	\begin{equation*}
		Z^M_0 \coloneqq 0, \quad \xi^M_\ell \coloneqq \vec{f}^M(X_\ell), \quad Z^M_i \coloneqq \sum_{\ell = 0}^{i-1} \xi^M_\ell	
	\end{equation*}
	Furthermore, for a fixed time~$T > 0$, define the grid points
	\begin{equation}
		t_i \coloneqq \frac{iT}{N}, \quad i \in \llbracket 0,N \rrbracket \,.
	\end{equation}
	as well as the \emph{piecewise linear approximation}
	\begin{equation} \label{e:pcw_linear_interpol}
		Y^M_N(t) \coloneqq 
		\frac{1}{\sqrt{N}} \sum_{i=0}^{N-1} \del[3]{Z^M_i + \theta(t) \xi^M_i} \mathbf{1}_{t \in [t_i,t_{i+1})}, \quad 
		\theta(t) \coloneq \frac{t - t_i}{t_{i+1} - t_i} = \frac{Nt}{T} - i
		\,.
	\end{equation}
	Then, we have
	\begin{equation*}
		\Y^M_N = \int_{0}^T Y^M_N(r) \otimes \dif Y^M_N(r) \,.
	\end{equation*}
\end{lemma}

\begin{proof}
	In this proof, we suppress the dependence on~$M$ and do not write the superscript for enhanced clarity.
	Since we have
	\begin{equation*}
		\dot{Y}_N(t) = \frac{N}{T\sqrt{N}} \sum_{j=0}^{N-1} \xi_j \, \mathbf{1}_{t \in [t_j,t_{j+1})} ,
	\end{equation*}
	and
	\begin{equation*}
		\sum_{j=0}^{N-1} Z_j \otimes \xi_j
		=
		\sum_{j=1}^{N-1} Z_j \otimes \xi_j
		=
		\sum_{j=1}^{N-1} \sum_{i=0}^{j-1} \xi_i  \otimes \xi_j 
		= \sum_{0 \leq i < j \leq N-1}  \xi_i \otimes \xi_j
		\,,
	\end{equation*}
	a simple computation shows that
	\begin{equs}
		\thinspace &
		\int_{0}^T Y_N(r) \otimes \dif Y_N(r)
		= 
		\int_{0}^T Y_N(r) \otimes \dot{Y}_N(r) \dif r \\[0.5em]
		= \ &
		\frac{1}{T} \sum_{i=0}^{N-1} \sum_{j=0}^{N-1} \int_0^T \del[1]{Z_i + \theta(r)\xi_i}  \otimes \xi_j \, \mathbf{1}_{r \in [t_i,t_{i+1})} \mathbf{1}_{r \in [t_j,t_{j+1}]} \dif r \\[0.5em]
		= \ &
		\frac{1}{T} \sum_{i=0}^{N-1}  \int_0^T \del[1]{Z_i + \theta(r) \xi_i}  \otimes \xi_i \, \mathbf{1}_{r \in [t_i,t_{i+1})} \dif r %\\[0.5em]
		= %\ &
		\frac{1}{N} \sum_{j=0}^{N-1}  Z_j \otimes \xi_j + \frac{1}{2N} \sum_{i=0}^{N-1}  \xi_i \otimes \xi_i  \\[0.5em]
		= \ &
		\frac{1}{N}  \sum_{0 \leq i < j \leq N-1}  \xi_i \otimes \xi_j  + \frac{1}{2N} \sum_{i=0}^{N-1}  \xi_i \otimes \xi_i 
		=
		\Y_N 
		\,.
	\end{equs}
	The proof is complete.
\end{proof}

\begin{lemma}\label{l:ptwise_L2_conv}
	For any~$M \in \N$ and~$t \in [0,1]$, we have
	\begin{equation}
		\lim_{N\to\infty} \norm[0]{Y^M_N(t) - S^M_N(t)}_{L^2(\P)} = 0 \,.
	\end{equation}
\end{lemma}

\begin{proof}
	Note that, for any~$t \in [0,T]$, we have
	\begin{equation}
		S^M_N(t) 
		= \frac{1}{\sqrt{N}} \sum_{i=0}^{N-1} Z^M_i \mathbf{1}_{t \in [t_i,t_{i+1})}, \quad
		Y^M_N(t) - S^M_N(t) = \frac{1}{\sqrt{N}} \sum_{i=0}^{N-1} \theta(t) \delta Z^M_i \mathbf{1}_{t \in [t_i,t_{i+1})}
	\end{equation} 
	which, for~$t \in [t_i,t_{i+1})$ and thus~$\theta(t) \leq 1$, implies that
	\begin{equs}
		\norm[0]{Y^M_N(t) - S^M_N(t)}_{L^2(\P)}
		& =  
		\frac{1}{\sqrt{N}} \theta(t) \norm[0]{\xi^M_i}_{L^2(\P)} \\
		& \leq  
		\frac{1}{\sqrt{N}} \norm[0]{\vec{f}(X_i)}_{L^2(\P)} 
		=
		\frac{1}{\sqrt{N}} \del[4]{\sum_{k=1}^m\norm[0]{f_k}^2_{L^2(\gamma)}}^{\frac{1}{2}}
		\,.
	\end{equs}
	The RHS goes to $0$ as~$N\to \infty$ because~$f_k \in L^2(\gamma)$ for any~$k \in \llbracket 1,m \rrbracket$.
\end{proof}

The benefit of switching viewpoint from~$\S_N^M$ to~$\Y_N^M$ is that the latter immediately allows to build the full rough paths signature in a canonical way:

\begin{definition} \label{d:lift_YMN}
	For~$M \in \N$ and~$k \geq 1$, we set 
	\begin{equation} \label{e:YNM}
	\YY^{M;(k)}_N(s,t)
	\coloneqq \int_{\triangle^{(k)}_{s,t}} \dif Y^M_N(r_1) \otimes \ldots \otimes \dif Y^M_N(r_k)
	\end{equation}
where
\begin{equation*}
	\triangle^{(k)}_{s,t} \coloneqq  \cbr[1]{r_{1:k} \in [s,t]^{k}: \ s \leq r_1 < r_2 < \ldots < r_k \leq t}
\end{equation*}
denotes the~$k$-dim. simplex over the interval~$[s,t]$.
\end{definition}

\begin{remark}
	Note that, by Lemma~\ref{lem:conv_cont_case}, we have~$\YY_N^{M;(2)}(s,t) = \mathbb{Y}_N^M(s,t)$ where the latter is given in~\eqref{e:def_YMN_2}.
\end{remark}

From here onwards, we will fix~$T \coloneqq 1$ and only write~$T$ explicitly if we wish to highlight the explicit dependence of certain expressions on the endpoint of the interval~$[0,T]$.

\paragraph{A detour into words and shuffles.}
In this paragraph we will briefly discuss some algebraic aspects of tensor series beyond level two, cf. Section~\ref{s:rough_paths_primer}, which are needed to implement Point~\ref{strategy:3} of the strategy outlined in Section~\ref{s:strategy} above.
Since we deal with fixed time points only, we will not deal with any analytical aspects relating to variation regularity of the involved paths.  
For details on the results presented in this paragraph, see~\cite[Sec.~2]{lyons_caruana_levy}.

The first definition introduces the full tensor algebra. 
\begin{definition}
	Let~$T((\R^m)) \coloneqq \prod_{n \geq 0} (\mathbb{R}^m)^{\otimes n}$ with~$(\R^m)^{\otimes 0} \coloneqq \R$ denote the \emph{full tensor algebra}.
	Then, the subset~$T_1((\R^m)) \coloneqq \cbr[0]{\boldsymbol{x} \in T((\R^m)): \, x_0 = 1}$ forms a group w.r.t. the product~$\otimes$ which, for~$\boldsymbol{x} = (1,x_1,x_2,\ldots)$ and~$\boldsymbol{y} = (1,y_1,y_2,\ldots)$ is given by
	\begin{equation*}
		\boldsymbol{x} \otimes \boldsymbol{y}
		\coloneqq (1,z_0,z_1, \ldots), \quad 
		z_n = \sum_{k=0}^n x_k \otimes y_{n-k} \,.
	\end{equation*}
\end{definition}

The second definition introduces \emph{words} and \emph{shuffles} (see also~\cite[Sec.~2.4 and Exercise~2.2]{friz_hairer_book}) which, as we will see below, are useful in encoding algebraic identities between tensor components.  

\begin{definition}[Words and shuffles] \label{d:words_shuffles}
	For~$n \in \N$, let~$\w = \w_1 \ldots \w_n$ be a \emph{word} of length~$\abs[0]{\w} \coloneqq n$ where each letter~$\w_j$ is an element in the \emph{alphabet}~$\llbracket 1,m \rrbracket$. 
	The \emph{empty word} is denoted by~$\emptyset$.
	Let~$\mathtt{W}$ be the linear span of such words~$\w$. 
	
	We write~$\v \w$ for the concatenation of two words~$\v$ and~$\w$, i.e.~$\v \mathtt{i}$ denotes the word obtained from attaching the letter~$\mathtt{i} \in \llbracket 1,m \rrbracket$ to the right of~$\v$. 
	Then, we define the (commutative) \emph{shuffle product}~$\shuffle$ on~$\mathtt{W}$ recursively by declaring~$\emptyset$ the neutral element and 
	\begin{equation} \label{e:def_shuffle}
		\v \mathtt{i} \shuffle \w \mathtt{j} \coloneqq (\v \shuffle \w\mathtt{j})\mathtt{i} + (\v\mathtt{i}  \shuffle \w)\mathtt{j} \,.
	\end{equation}
	for any  two letters $\mathtt{i},\mathtt{j} \in \llbracket 1,m \rrbracket$ and any two words $\v,\w$.
\end{definition}

Finally, we introduce \emph{Chen's} and the \emph{shuffle identity}; both will be paramount in passing from (possibly) overlapping intervals~$(s_i,t_i)_{i = 1}^l \subseteq [0,1]^2$ to ordered intervals~$t_i < s_{i+1}$.
We let~$(\tte_k)_{k=1}^m$ be the standard basis in~$\R^m$.

\begin{definition}[Chen's and shuffle identity] \label{d:chen_shuffle}
	We let~$\tte_\w \coloneqq \tte_{\w_1} \otimes \ldots \otimes \tte_{\w_n}$ be the induced basis vector in~$(\R^m)^{\otimes n}$ and, for~$\boldsymbol{Y} \in T_1((\R^m))$, set~$\scal{\boldsymbol{Y},\w} \coloneqq \scal{\boldsymbol{Y},\tte_{\w}}$, with the convention that~$\scal{\boldsymbol{Y},\emptyset} \coloneqq 1$.
	A path~$\boldsymbol{Y}: \triangle_{0,1}^{(2)} \to T_1((\R^m))$ is said to satisfy 
	\begin{itemize}
		\item \emph{Chen's identity} if, for all triples~$0 \leq s < u < t \leq 1$, we have the identity $\boldsymbol{Y}_{s,t} = \boldsymbol{Y}_{s,u} \otimes \boldsymbol{Y}_{u,t}$.
		In the language of Definition~\ref{d:words_shuffles}, this identity can equivalently be phrased as follows:  
		For all words~$\w \in \mathtt{W}$, we have
		\begin{equation} \label{e:chen}
			\scal{\boldsymbol{Y}_{s,t},\w} = \sum_{\w = \v_1 \v_2} \scal{\boldsymbol{Y}_{s,u},\v_1} \scal{\boldsymbol{Y}_{u,t},\v_2} \,,  
		\end{equation} 
		where the sum ranges over all ways of obtaining the word~$\w$ as the concatenation of words~$\v_1$ and~$\v_2$.
		\item the \emph{shuffle relation}, if for all~$(s,t) \in \triangle_{0,1}^{(2)}$, we have
			\begin{equation} \label{e:shuffle}
				\scal{\boldsymbol{Y}_{s,t},\v \shuffle \w} 
				= 
				\scal{\boldsymbol{Y}_{s,t},\v} \scal{\boldsymbol{Y}_{s,t},\w} \,.
			\end{equation}
	\end{itemize}
\end{definition}

The previous discussion culminates in the following lemma which, since the path~$Y_N^M$ is piecewise linear, is a basic fact in rough paths theory (see for example~\cite[Sec.~7.6]{friz_victoir}):
\begin{lemma} \label{l:cts_rough_path}
	For any~$N, M \in \N$, the path
	\begin{equation*}
		\YY_N^M: \triangle_{0,1}^{(2)} \to T_1((\R^m)), \quad \YY_N^M(s,t) = \del[0]{1, \YY_N^{M;(1)}(s,t), \YY_N^{M;(2)}(s,t), \ldots}
	\end{equation*}
	satisfies Chen's identity and the shuffle relation introduced in Definition~\ref{d:chen_shuffle}.
\end{lemma}

Now recall that~$\mathbb{\bar B}^M(s,t) = \mathbb{B}^{M,\text{\tiny Strat}}(s,t) + \AA^M(t-s)$.
The next definition tells us that the full rough path lift of~$(B^M,\mathbb{\bar B}^M)$ is still a geometric rough path, i.e. that the higher-level effects of the area correction~$\AA$ do not destroy that property.

\begin{definition} \label{d:brownian_rp}
	For any~$M \in \N$, let~$\BBbar: \triangle^{(2)}_{0,1} \to T_1((\R^m))$ denote the full rough path of~$(B^M,\mathbb{\bar B}^M)$. 
	This defines a geometric rough path which, in particular, satisfies Chen's identity and the shuffle relation.
\end{definition}  

The previous definition, in fact, contains various non-trivial statements which are an immediate consequence of~\cite[Def.~3 and comments below]{BCFP19}.
Since we do not need the precise form of~$\BBbar^{M;(k)}(s,t)$ for~$k \geq 3$, we will not go into further details here.

\begin{remark}[Iterated sums-signature] \label{rmk:iterated_sum_signature}
	Given Definition~\ref{d:lift_YMN}, it is an interesting question whether one can, in a natural way, work with a higher-order iterated sum without converting to the piecewise linear approximation.    
	In fact, the answer is affirmative and leads to the notion of an (generalised) iterated sums-signature developed in~\cite{kiraly_oberhauser_2019, Diehl_EF_Tapia_20, Diehl_EF_Tapia_20_proceedings, Diehl_EF_Tapia_23}.
	However, its algebraic structure is more involved because, in contrast to the signature, it is only a \emph{quasi-}shuffle algebra character~\cite[Thm.~3.4]{Diehl_EF_Tapia_20}, i.e. the identity~\eqref{e:def_shuffle} requires a correction term (encoding the \enquote{diagonals}).
	At this point, whichever choice one makes is merely a question of \emph{algebraic} bookkeeping which, if translated correctly, ultimately leads to the same result in~\eqref{e:YY_N_rewriting}.
	It would be interesting to investigate whether one can leverage the algebraic machinery for computing expectations of iterated sum-signatures (see~\cite[Sec.~3.1]{Diehl_EF_Tapia_20_proceedings}) to streamline the \emph{analytic} considerations in Sections~\ref{s:higher_diagonals}--\ref{s:conv_moments}; 
	however, we will not follow this line of thought further in this article. 
\end{remark}

The goal of Section~\ref{s:moment_computations} is to establish the following theorem which, as we will see in Proposition~\ref{p:conv_fdd_stratonovich:equiv} below, is implied by the convergence moments of the \emph{full} rough path stated in Theorem~\ref{t:product_case}.
\begin{theorem} \label{thm:conv_fdd_stratonovich}
	For any~$M \in \N$,~$l \in \N$, and any sequence of intervals~$(s_i,t_i)_{i = 1}^l \subseteq [0,1]^2$ with~$s_i < t_i$, observe that the vectors~$\CB^M_N$ and~$\CB^M$ from~\eqref{e:def_BNM} resp.~\eqref{e:conv_BNM} in our new notation now read
	\begin{align*}
		\CB_N^M
		& = 	
		\del[1]{\YY^{M;(1)}_N(s_1,t_1), \YY^{M;(2)}_N(s_1,t_1), \ldots, \YY^{M;(1)}_N(s_l,t_l), \YY^{M;(2)}_N(s_l,t_l)} \,, \\[0.5em]
		%%%%
		\CB^M & = 
		(\BBbar^{M;(1)}(s_1,t_1), \BBbar^{M;(2)}(s_1,t_1), \ldots, \BBbar^{M;(1)}(s_l,t_l), \BBbar^{M;(2)}(s_l,t_l)) \,.
	\end{align*}	
	We then have 
	\begin{equation}  \label{e:conv_fdd_stratonovich}
		\CB_N^M \Longrightarrow \CB^M \quad \text{as} \quad N \to \infty \,.
	\end{equation}
\end{theorem}
We present the proof of this statement in Section~\ref{s:conv_moments}.
The first auxiliary result we need is the following:
\begin{lemma}[Moment-determinancy] \label{l:moment_determined}
	The law of the vector~$\CB^M$ is determined by its moments.
\end{lemma}

\begin{proof}
	For notational simplicity, we will focus on the $\R$-valued case and show the following: 
	Let~$F = (F_1,\ldots,F_n) \in \R^n$ of length~$n \in \N$ where each~$F_i$ is either in the zero-th, first, or in the second~$\R$-valued Wiener--It\^{o} chaos. Then, the law of~$F$ is moment-determined.
	
	Denote by~$\norm[0]{F}_1 \coloneqq \sum_{i=1}^n \abs[0]{F_i}$. By (the generalised) H\"older's inequality with~$p_i = n$ for~$i = 1,\ldots,n$, i.e. $\sum_{i=1}^n p_i^{-1} = 1$, we then find
	\begin{equs}
		\E\sbr[1]{e^{\lambda \norm[0]{F}_1}}
		& 
		=
		\E\sbr{\prod_{i=1}^n e^{\lambda \abs[0]{F_i}}}
		=
		\norm{\prod_{i=1}^n e^{\lambda \abs[0]{F_i}}}_{L^1(\P)}
		\leq
		\prod_{i=1}^n  \norm[1]{e^{\lambda \abs[0]{F_i}}}_{L^n(\P)}
		=
		\prod_{i=1}^n \E\sbr[1]{e^{\lambda n \abs[0]{F_i}}}_{L^1(\P)}^{\nicefrac{1}{n}}
	\end{equs}
	For~$\lambda \ll 1$, the expression on the RHS is finite~(see~\cite[Coro.~6.13]{janson}) and thus the characteristic function of~$F$ is analytic in a neighbourhood of~$0$.
	The conclusion follows by Carleman's condition.
\end{proof}

\subsubsection{Relation to moment convergence} \label{s:limit_exp_signature} 

We need the following generalisation of Fawcett's theorem~\cite[Thm.~3.9]{friz_hairer_book}:\footnote{Recall that the superscript~$M$ means that all involved functions are projected onto chaos components~$M$ and lower.}
\begin{lemma} \label{l:fawcett}
	Recall that
	\begin{equs}
		%%%
		\Delta^M(u) & = \E\sbr[0]{\vec{f}^M(X_1) \otimes \vec{f}^M(X_{1+u})}, \quad \Gamma^M = \sum_{u=1}^\infty \Delta^M(u), \\
		%%%
		\Sigma^M & = \Delta^M(0) + 2 \Sym(\Gamma^M), \quad \AA^M = \ASym(\Gamma^M) \,. 
		%%%
	\end{equs}
	and let~$\BBbar^M$ be the full rough path lift of~$(B^M,\mathbb{\bar B}^M)$ introduced in Definition~\ref{d:brownian_rp}.
	Considered as a~$T((\R^m))$-valued random variable, its expected signature is given as follows:
	\begin{equation} \label{e:expected_signature}
		\operatorname{ESig}^M_{s,t} \coloneqq \E\sbr[1]{\BBbar^M(s,t)} = \exp\del[4]{\frac{t-s}{2} \del[4]{\sum_{k,\ell = 1}^m \Sigma^M_{k,\ell} \, \tte_k \otimes \tte_\ell + 2 \sum_{1 \leq k < \ell \leq m} \AA_{k,\ell}^M \, [\tte_k,\tte_\ell]}} 
	\end{equation}
	where~$(\tte_k)_{k=1}^m$ denotes the standard basis in~$\R^m$ and~$[\tte_k,\tte_\ell] = \tte_k \otimes \tte_\ell - \tte_\ell \otimes \tte_k$ the canonical Lie bracket.
\end{lemma}
\begin{proof}
	The proof is analogous to that of~\cite[Thm.~3.9]{friz_hairer_book}, except that 
	\begin{equs}[][e:computation_signature]
		\thinspace &
		\E\sbr[0]{1 + B^M_{0,1} + \mathbb{\bar B}^M_{0,1}} = 1 + \frac{1}{2} \sum_{i,j=1}^m \del[0]{\Sigma^M_{k,\ell} + 2 \AA^M_{k,\ell}}\tte_k \otimes \tte_\ell \\
		= \ & 
		1 + \frac{1}{2} \sum_{i,j=1}^m \Sigma^M_{k,\ell} \tte_k \otimes \tte_\ell + \sum_{1 \leq k < \ell \leq m} \AA_{k,\ell}^M \, [\tte_k,\tte_\ell]
	\end{equs}
	where the last step uses the antisymmetry of~$\AA^M$.
	We also refer to~\cite[Rmk.~9]{friz_gassiat_lyons_15} and the references therein for further links to the literature.
\end{proof}

For~$s,t \in [0,1]$ such that~$s < t$, let us compute the component of the expected signature~$\ESig^M_{s,t}$ in~$(\R^m)^{\otimes 2n}$; the components in odd tensor powers are~$0$, as is obvious from~\eqref{e:expected_signature}.
For 
\begin{equation*}
	\tilde{\Sigma}^M \coloneqq \Sigma^M + 2 \AA^M = \Delta^M(0) + 2\Gamma^M
\end{equation*}
we have
\begin{equs}
	\ESig^M_{s,t} 
	& = \exp\del[4]{\frac{t-s}{2} \sum_{k,\ell = 1}^m \tilde{\Sigma}^M_{k,\ell}\, \tte_k \otimes \tte_\ell}
	=
	\sum_{l=0}^\infty \frac{1}{l!} \del[4]{\frac{t-s}{2} \sum_{k,\ell = 1}^m \tilde{\Sigma}^M_{k,\ell} \, \tte_k \otimes \tte_\ell}^{\otimes l} \\
	& =
	\sum_{l=0}^\infty \frac{(t-s)^l}{2^l l!} \sum_{k_{1:l}, \ell_{1:l}=1}^m \prod_{j=1}^l \tilde{\Sigma}^M_{k_j,\ell_j} \, \bigotimes_{j=1}^l \tte_{k_j} \otimes \tte_{\ell_j} \\
	& =
	\sum_{l=0}^\infty \frac{(t-s)^l}{2^l l!} \sum_{\w_{1:2l} = 1}^m \prod_{j=1}^l \tilde\Sigma^M_{\w_{2j-1},\w_{2j}} \, \bigotimes_{i=1}^{2l} \tte_{\w_i} \,.
\end{equs}
where we have used the first line in~\eqref{e:computation_signature}, i.e. antisymmetry of~$\AA^M$, to express the full exponent in~\eqref{e:expected_signature} in terms of~$(\tte_k \otimes \tte_\ell)_{k,\ell = 1}^m$.  
For~$n \in \N$ and a fixed word $\w= \w_{1:2n} \in \llbracket 1,m \rrbracket^{2n}$, the corresponding component of~$\ESig_{s,t}^M$ is then given by:
\begin{equation} \label{e:expec_sig_w}
	\scal{\ESig^M_{s,t},\w}
	= \frac{(t-s)^n}{2^n n!} \prod_{j=1}^n \tilde{\Sigma}^M_{\w_{2j-1},\w_{2j}} 
	= 
	\frac{(t-s)^n}{2^n n!} \prod_{j=1}^n \del[1]{\Delta^M_{\w_{2j-1},\w_{2j}}(0) + 2 \Gamma^M_{\w_{2j-1},\w_{2j}}} 
\end{equation}

\begin{remark}[Ladder pairing] \label{rmk:signaure_letter_pairing}
	Note that the product over~$j \in \llbracket 1,n \rrbracket$ in~\eqref{e:expec_sig_w} is always such that
	\begin{equation*}
		\contraction[1ex]{}{\w}{_1\,}{\w}
		\contraction[1ex]{\w_1\,\w_2\,}{\w}{_3\,}{\w}
		\contraction[1ex]{\w_1\,\w_2\,\w_3\,\w_4\;\cdots\;}{\w}{_{2n-1}\,}{\w}
		\w_1\,\w_2\,\w_3\,\w_4\;\cdots\;\w_{2n-1}\,\w_{2n}
		\quad\longleftrightarrow\quad 
		P_\star = 
		\contraction[1ex]{}{1}{\,}{2}
		\contraction[1ex]{1\,2\,}{3}{\,}{4}
		\contraction[1ex]{1\,2\,3\,4\;\cdots\;}{(2n\!-\!1)}{\,}{(2n)}
		1\,2 \,3\,4\;\cdots\;(2n\!-\!1)\,(2n) \,.
	\end{equation*}
	We will call this (unordered) pairing~$P_\star$ the \emph{ladder pairing}. 
\end{remark}

The following lemma shows that we can restrict ourselves to the case without overlaps.
It is essentially a consequence of Chen's and the shuffle identity.
\begin{lemma}[No overlaps] \label{l:no_overlaps}
	For~$i = 1,2$, fix lengths~$k_i \in \N$ and consider words~$\w_i$ in letters~$\llbracket 1,m \rrbracket$ such that $\abs[0]{\w_i} = k_i$.
	Further, let~$\bar s_i, \bar t_i \in [0,1]$ such that~$\bar s_i < \bar t_i$ and~$[\bar s_1,\bar t_1] \cap [\bar s_2,\bar t_2] \neq \emptyset$.
	Then, the product~$\scal{\YY^M_N(\bar s_1,\bar t_1),\w_1}\scal{\YY^M_N(\bar s_2,\bar t_2),\w_2}$ can be expressed as as a finite sum of terms
	\begin{equation*}
		\prod_{i} \scal{\YY^M_N(I_i),\w_{I_i}}
	\end{equation*}
	where the product is over a finite number of indices~$i \in \N$, intervals~$I_i = [s_i,t_i)$ such that~$\bigcup_{i} I_i= [\bar s_1,\bar t_1] \cup [\bar s_2,\bar t_2]$ and~$I_i \cap I_j = \emptyset$ if~$i \neq j$, and suitable words~$\w_{I_i}$ built from~$\w_{1}$ and~$\w_{2}$.
	The analogous statement holds if we replace each instance of~$\YY^M_N$ by~$\BBbar^M$. 
\end{lemma}

\begin{proof}
	For the sake of concreteness, we will only deal with the case where~$\bar s_1 < \bar s_2 < \bar t_1 < \bar t_2$; one can deal with the other cases of possible overlaps in an analogous way.
	By Lemma~\ref{l:cts_rough_path}, we can apply Chen's identity~\eqref{e:chen} to find
	\begin{equs}
		\scal{\YY^M_N(\bar s_1,\bar t_1),\w_1}
		& =
		\sum_{\w_1 = w_1^a w_1^b} \scal{\YY^M_N(\bar s_1,\bar s_2),\w_1^a}\scal{\YY^M_N(\bar s_2,\bar t_1),\w_1^b} \\
		\scal{\YY^M_N(\bar s_2,\bar t_2),\w_2}
		& =
		\sum_{\w_2 = w_2^a w_2^b} \scal{\YY^M_N(\bar s_2,\bar t_1),\w_2^a} \scal{\YY^M_N(\bar t_1,\bar t_2),\w_2^b}
	\end{equs}
	and then, by the shuffle identity~\eqref{e:shuffle}, 
	\begin{equation*}
		\scal{\YY^M_N(\bar s_2,\bar t_1),\w_1^b} \scal{\YY^M_N(\bar s_2,\bar t_1),\w_2^a}
		= 
		\scal{\YY^M_N(\bar s_2,\bar t_1),\w_1^b \shuffle \w_2^a} \,.
	\end{equation*}
	As a consequence, we have 
	\begin{equs}
		\thinspace
		&
		\scal{\YY^M_N(\bar s_1,\bar t_1),\w_1}\scal{\YY^M_N(\bar s_2,\bar t_2),\w_2} \\
		= \ & 
		\sum_{\w_1 = w_1^a w_1^b} \sum_{\w_2 = w_2^a w_2^b}
		\scal{\YY^M_N(\bar s_1,\bar s_2),\w_1^a}
		\scal{\YY^M_N(\bar s_2,\bar t_1),\w_1^b \shuffle \w_2^a}
		\scal{\YY^M_N(\bar t_1,\bar t_2),\w_2^b} 
	\end{equs}
	and the first claim follows.
	Since Chen's and the shuffle product identity are true for any geometric  rough path, the claim also holds for~$\BBbar^M$.
\end{proof}

As alluded to earlier, the main result of Section~\ref{s:moment_computations}, Theorem~\ref{thm:conv_fdd_stratonovich}, will be obtained as a consequence of the following theorem. 

\begin{theorem} \label{t:product_case}
	Let~$l \in \N$ and consider the vector~$k_{1:l} \in \N^l$.
	Further, for~$i \in \llbracket 1,l \rrbracket$, consider the words~$\w_i \in \llbracket 1,m \rrbracket^{k_i}$, i.e.~$\abs[0]{\w_i} = k_i$, as well as the \emph{pairwise disjoint} intervals~$[s_i, t_i] \subseteq [0,1]$.  
	Then, we have
	\begin{equation} \label{e:product_case}
		\lim_{N \to \infty} \E\sbr[4]{\prod_{i=1}^l \scal{\YY^M_N(s_i,t_i), \w_i}} = \prod_{i=1}^l \scal{\ESig^M_{s_i,t_i},\w_i} \,.
	\end{equation}
\end{theorem}

\begin{remark} \label{rmk:ordering_intervals}
	By commutativity of the product in~\eqref{e:product_case}, going forward, we will without loss of generality assume that the intervals in Theorem~\ref{t:product_case} are \emph{ordered} in the following way:
	\begin{equation*}
		0 \leq s_1 < t_1 < s_2 < t_2 \ldots < s_l < t_l \leq 1\,.
	\end{equation*}
\end{remark}

\begin{proposition} \label{p:conv_fdd_stratonovich:equiv}
	Theorem~\ref{t:product_case} implies Theorem~\ref{thm:conv_fdd_stratonovich}. 
\end{proposition}

\begin{proof}
	Consider intervals~$[\bar{s}_i, \bar{t}_i] \subseteq [0,1]$ which, as in Theorem~\ref{thm:conv_fdd_stratonovich}, are not necessarily disjoint.
	
	For~$(r_\ell^1,r_\ell^2)_{\ell=1}^l \subseteq \N^2$, we define the corresponding \emph{tensor monomial}~$G = G\del[1]{(r_\ell^1,r_\ell^2)_{\ell=1}^l}$ as follows:
		\begin{equation} \label{e:tensor_monomial}
			G(Y_1,\mathbb{Y}_1, \ldots, Y_l,\mathbb{Y}_l) \coloneqq \bigotimes_{\ell=1}^l Y_\ell^{\otimes r_\ell^1} \otimes \mathbb{Y}_{\ell}^{\otimes r_{\ell}^2} \in \bigotimes_{\ell=1}^l (\R^m)^{\otimes r_\ell^1} \otimes (\R^m)^{\otimes (2 r_\ell^2)} \,.
		\end{equation}
	Since the law of corresponding vector~$\CB^M$ (w.r.t. the intervals~$[\bar{s}_i,\bar{t}_i]$ that are suppressed in the notation) is moment-determined by Lemma~\ref{l:moment_determined}, the convergence~\eqref{e:conv_fdd_stratonovich} claimed in Theorem~\ref{thm:conv_fdd_stratonovich} is equivalent to showing that, for each tensor monomial~$G$ as in~\eqref{e:tensor_monomial}, we have 
	\begin{equs}[][e:tensor_mononomial:conv]
			\lim_{N \to \infty} & \E\sbr[1]{G\del[1]{\YY^{M;(1)}_N(\bar s_1,\bar t_1),\YY^{M;(2)}_N(\bar s_1,\bar t_1), \ldots, \YY^{M;(1)}_N(\bar s_l,\bar t_l),\YY^{M;(2)}_N(\bar s_l,\bar t_l)}} \\
			= \ &
			\E\sbr[1]{G\del[1]{\BBbar^{M;(1)}(\bar s_1,\bar t_1),\BBbar^{M;(2)}(\bar s_1,\bar t_1), \ldots, \BBbar^{M;(1)}(\bar s_l,\bar t_l),\BBbar^{M;(2)}(\bar s_l,\bar t_l)}}
	\end{equs}
	
	Equivalently, we can expand both expectations in a basis and then show that the previous convergence holds for all the coefficients.
	To that end, let the word~$\w$ be given by
		\begin{equation*}
			\w \coloneqq \bigsqcup_{\ell=1}^l \bigsqcup_{a = 1}^{r_\ell^1} \w^{1,a}_\ell \sqcup \bigsqcup_{b = 1}^{r_\ell^2} \w^{2,b}_\ell
		\end{equation*}
		where~$\abs[0]{\w^{1,a}_\ell} = 1$  and~$\abs[0]{\w^{2,b}_\ell} = 2$ for each~$(\ell,a,b) \in  \llbracket 1,l \rrbracket \ltimes \del[1]{\llbracket 1,r_\ell^1 \rrbracket \times \llbracket 1,r_\ell^2 \rrbracket}$.
	We then have 
	\begin{align}
			\thinspace & 
			\left\langle\E\sbr[1]{G\del[1]{\YY^{M;(1)}_N(\bar s_1,\bar t_1),\YY^{M;(2)}_N(\bar s_1,\bar t_1), \ldots, \YY^{M;(1)}_N(\bar s_l,\bar t_l),\YY^{M;(2)}_N(\bar s_l,\bar t_l)}},\w \right\rangle \notag \\
			= \ &  
			\E\sbr{\prod_{\ell=1}^l \prod_{a=1}^{r_\ell^1}  \scal{\YY^M(\bar s_\ell,\bar t_\ell),\w_\ell^{1,a}} \prod_{b=1}^{r_\ell^2}\scal{\YY^M(\bar s_\ell,\bar t_\ell),\w_{\ell}^{2,b}}} \label{e:tensor_mononomial:basis} \\[1em]
			= \ &  
			\E\sbr{\prod_{\ell=1}^l  \scal{\YY^M(\bar s_\ell,\bar t_\ell),\w_\ell^{1}} \scal{\YY^M(\bar s_\ell,\bar t_\ell),\w_{\ell}^{2}}} \notag
	\end{align}
	where the last equality is due to the shuffle product identity~\eqref{e:shuffle} and
		\begin{equation*}
			\w_\ell^i \coloneqq \shuffle_{a=1}^{r_\ell^i} \w_{\ell}^{i,r_\ell^i}, \quad i = 1,2, \quad \ell = 1, \ldots, l \,.
		\end{equation*}
		Note that, by abuse of notation, the latter is actually not just one word but, in fact, a \emph{sum over words}.	
		By linearity of the expectation, we can then write 
		\begin{equs}[][e:tensor_mononomial:basis_2]
			\thinspace & 
			\left\langle\E\sbr[1]{G\del[1]{\YY^{M;(1)}_N(\bar s_1,\bar t_1),\YY^{M;(2)}_N(\bar s_1,\bar t_1), \ldots, \YY^{M;(1)}_N(\bar s_l,\bar t_l),\YY^{M;(2)}_N(\bar s_l,\bar t_l)}},\w \right\rangle \\
			= \ &  
			\E\sbr{\prod_{\ell=1}^{l}  \scal{\YY^M_N(\bar s_\ell,\bar t_\ell),\w_\ell}} 
		\end{equs}
	where~$\w_\ell \coloneqq \w_{\ell}^1 \shuffle \w_{\ell}^2$ is, again, a suitable sum of words. 	
		
	Finally, by a repeated application of Lemma~\ref{l:no_overlaps}, one can now find some~$k \in \N$ and \emph{disjoint} intervals $I_j = [s_j,t_j]$ for~$j \in \llbracket 1,k \rrbracket$ as well as corresponding words~$\w_{I_j}$ such that
		\begin{equs}[][e:tensor_mononomial:basis_3]
			\thinspace & 
			\left\langle\E\sbr[1]{G\del[1]{\YY^{M;(1)}_N(s_1,t_1),\YY^{M;(2)}_N(s_1,t_1), \ldots, \YY^{M;(1)}_N(s_l,t_l),\YY^{M;(2)}_N(s_l,t_l)}},\w \right\rangle \\
			= \ &  
			\E\sbr{\prod_{j=1}^{k}  \scal{\YY^M_N(s_j, t_j),\w_{I_j}}}
		\end{equs}
	The previous identity also holds when replacing each instance of~$\YY_N^M$ by~$\BBbar^M$. 	
	Therefore, the convergence in~\eqref{e:tensor_mononomial:conv} is exactly the statement of Theorem~\ref{t:product_case} and the proof of the proposition is complete.
\end{proof}

In light of Proposition~\ref{p:conv_fdd_stratonovich:equiv}, we will spend the remainder of this section proving Theorem~\ref{t:product_case}. 

\subsubsection{Higher-order diagonals vanish} \label{s:higher_diagonals}

The definition of~$\YY_N^M(s,t)$ in~\eqref{e:YNM} as an integral rather than a sum was convenient because, by Lemma~\ref{l:cts_rough_path}, we could then use Chen's identity and the shuffle product relation, both defining properties of a geometric rough path.---However, as we will now have to compute moments and show that they converge to the limit claimed in Theorem~\ref{t:product_case}, we need to convert back to the discrete setting so as to leverage the combinatorial simplification implied by Corollary~\ref{coro:BM83} (which is, itself, a consequence of the Breuer--Major result, Proposition~\ref{p:bm83_irregular}.)
In this subsection, we will only deal with the case~$s=0$ and~$t = 1$; the adaptation of our arguments to the general case is straightforward.

Before we proceed, we need to introduce some notation. 
Recall that, for~$l \in \N$ and~$u,v \in \R$ such that~$u < v$, the (closure of) the $l$-simplex between~$u$ and~$v$ is given by  
	\begin{equation*}
		\bar{\triangle}_{u,v}^{(l)} 
		=  
		\cbr[1]{ r_{1:l} \in [u,v]^{l}: \ u \leq r_1 \leq r_2 \leq \ldots \leq r_l \leq v} 
	\end{equation*}
where we emphasise that some (or all) of the index components~$r_{1:l}$ could agree.
The following definition makes that observation precise by introducing a block decomposition for the index components.
\begin{definition}[Block map, indices, simplex, and weights] \label{d:block_map_indices}
	For~$l \in \N$, assume that the set~$\llbracket 1,l \rrbracket$ can be decomposed into~$p \in \llbracket 1,l \rrbracket$ blocks, that is:
	\begin{equation} \label{e:block_decomposition}
		\llbracket 1,l \rrbracket = \sqcup_{\ell=1}^p B_\ell, 
		\quad B_\ell = \cbr[0]{b^\ell_{1}, \ldots, b^\ell_{a_\ell}}, \quad b_1^\ell < \ldots < b_{a_\ell}^\ell 
	\end{equation} 
	Note that we have ordered the block components ascendingly and that we have set~$a_\ell \coloneqq \abs[0]{B_\ell}$ for~$\ell \in \llbracket 1,p \rrbracket$. 
	We then define the \emph{block map}~$\beta$ by 
	\begin{equation*}
		\beta: \llbracket 1, l \rrbracket \to \llbracket 1,p \rrbracket, \quad \beta(r) \coloneqq \ell \quad \text{iff} \quad r \in B_\ell 
	\end{equation*}
	and the \emph{block indices} $j_1 < \ldots < j_p$ by\footnote{Note that for $r_1,r_2\in \llbracket 1, 2n \rrbracket$, the fact that $r_1,r_2\in B_\ell$ is equivalent to $i_{r_1}=i_{r_2}$  and therefore to~$\beta(r_1)=\beta(r_2)$, so the definition of the block indices is consistent.}
	\begin{equation*}
		i_r = j_{\beta(r)}, 
		\quad \text{i.e.} \quad  
		a_\ell = \#\cbr[0]{r \in \llbracket 1,l \rrbracket: \ \beta(r) = \ell} \,.
	\end{equation*}
	For~$u,v \in \R$ such that~$u < v$, we then introduce the \emph{block simplex}~$\bar{\triangle}_{u,v}^{(l)}(B_{1:p})$ as follows:
	\begin{equation} \label{e:simplex_blocks}
		\bar{\triangle}_{u,v}^{(l)}(B_{1:p}) 
		\coloneqq 
		\cbr[0]{\bi = i_{1:l} \in \bar{\triangle}_{u,v}^{(l)}: \ i_j = i_k \quad \text{iff} \quad \beta(j) = \beta(k)}
	\end{equation}
	Finally, for a multi-index~$i_{1:l} \in \bar{\triangle}_{u,v}$ with block decomposition given in~\eqref{e:block_decomposition}, we define its~\emph{weight}~$\fw(i_{1:k})$ to be the volume of the simplex formed by its blocks, i.e. 
	\begin{equation*}
		\fw(i_{1:l}) \coloneqq \prod_{\ell=1}^p \frac{1}{a_\ell!} \,.
	\end{equation*}
\end{definition}

Let us present an example that illustrates the previous definition.
\begin{example} \label{ex:blocks}
	Assume~$l = 10$ and~$\llbracket 1,l \rrbracket = B_1 \sqcup B_2 \sqcup B_3 \sqcup B_4$ with
	\begin{equation*}
		B_1 = \cbr[0]{1,6,8}, \quad
		B_2 = \cbr[0]{2,5}, \quad
		B_3 = \cbr[0]{3,4,9}, \quad
		B_4 = \cbr[0]{7,10} \,
	\end{equation*}
	We can visually represent the block decomposition as
	\begin{equation*}
			\cbox{mutedorange}{\upshape{1}}
			\,
			\cbox{mutedgreen}{\upshape{2}}
			\,
			\cbox{mutedcyan}{\upshape{3}}
			\,
			\cbox{mutedcyan}{\upshape{4}}
			\,
			\cbox{mutedgreen}{\upshape{5}}
			\,
			\cbox{mutedorange}{\upshape{6}}
			\,
			\cbox{mutedmagenta}{\upshape{7}}
			\,
			\cbox{mutedorange}{\upshape{8}}
			\,
			\cbox{mutedcyan}{\upshape{9}}
			\,
			\cbox{mutedmagenta}{\upshape{10}}
			\quad 
			\longleftrightarrow 
			\quad 
			i_{1:10} = (j_1,j_2,j_3,j_3,j_2,j_1,j_4,j_1,j_3,j_4) \,.
	\end{equation*}
	where indices in the same block are of the same colour. 
\end{example}

We can use the block notation to compute the quantity~$\YY_N^{M; (k)}(0,T)$, given in~\eqref{e:YNM} as an iterated integral, explicitly.
To that end, recall from Lemma~\ref{lem:conv_cont_case} and its proof that~$\xi^M_j = \vec{f}^M(X_j)$ and that
\begin{equation*}
	\dot{Y}^M_N(t) = \frac{N}{T\sqrt{N}} \sum_{j=0}^{N-1} \xi^M_j \, \mathbf{1}_{t \in [t_j,t_{j+1})} \,,
\end{equation*}
One then immediately obtains the following representation for~$\YY_N^{M; (k)}(0,T)$:   
\begin{equation*}
	\YY_N^{M; (k)}(0,T)
	=
	\frac{N^{k/2}}{T^k}
	\sum_{0 \leq i_1 \leq \ldots \leq i_k \leq N-1}
	\del[4]{\bigotimes_{\ell=1}^k \vec{f}^M(X_{i_\ell})}
	\int_{\triangle_{0,T}^{(k)}} \prod_{\ell=1}^k
	\1_{r_\ell \in [t_{i_\ell},t_{i_\ell+1})} \dif r_{1:k} \,.
\end{equation*}
Grouping equal indices into blocks, with $p$ be the number of blocks and $a_\ell$ the size of the $\ell$-th block
(with $a_1+\cdots+a_p=k$), the simplex integral computes to
\begin{equation*}
	\int_{\triangle_{0,T}^{(k)}} \prod_{\ell=1}^k \1_{\cbr[0]{r_\ell \in [t_{j_\ell},t_{j_\ell+1}) }} \dif r_{1:k}
	=
	\prod_{\ell=1}^p \frac{T^{a_\ell}}{N^{a_\ell} a_\ell!} 
	= 
	\frac{T^k}{N^k} 
	\prod_{\ell=1}^p \frac{1}{a_\ell!} 
	= 
	\frac{T^k}{N^k} 
	\fw(i_{1:l})
\end{equation*}
which gives
\begin{equs}[][e:YY_N_rewriting]
	\YY_N^{M; (k)}(0,1)
	& = \frac{1}{N^{k/2}}
	\sum_{0\le j_1<\cdots<j_p \leq N-1}
	\bigotimes_{\ell=1}^p \frac{(\vec{f}^M(X_{j_\ell}))^{\otimes a_\ell}}{a_\ell!} \\ 
	%%%
	& =
	\frac{1}{N^{k/2}}
	\sum_{0 \leq i_1 \leq \ldots \leq i_k \leq N-1}
	\fw(i_{1:k})
	\bigotimes_{\ell=1}^k \vec{f}^M(X_{i_\ell})
\end{equs}
We will use that representation to prove the following proposition.
It states that, index configurations~$i_{1:k}$ where~$k$ is odd or where there is at least one block of size~$a \geq 3$ (called a \enquote{higher-order diagonal}) give rise to asymptotically vanishing contributions.

\begin{proposition}[Higher-order diagonals vanish] \label{p:higher_order_diagonals}
	For any~$s,t \in [0,1]$ with~$s < t$ and~$l \in \N$, we have 
	\begin{align}
		\thinspace & 
		\E\sbr[1]{\scal{\YY_N^{M; (l)}(s,t),\w}} \label{e:higher_order_diagonals}\\
		\asymp & \ 
		\frac{\1_{l=2n}}{N^{n}}
		\sum_{p=1}^{2n}
		\sum_{\substack{B_{1:p}, \\ \abs[0]{B_\ell} \in \cbr[0]{1,2}}}
		\sum_{i_{1:2n}\in \bar{\triangle}(B_{1:p})} \fw(i_{1:2n})
		\sum_{P\in \CM(2n)}
		\prod_{j=1}^{n}
		\Delta^M_{\w_{P_{2j-1}},\w_{P_{2j}}}\!\del[1]{i_{P_{2j}}-i_{P_{2j-1}}}  \,. \notag
	\end{align}
	as~$N \to \infty$, where~$\bar{\triangle}(B_{1:p}) \coloneqq \bar{\triangle}^{(2n)}_{\lfloor Ns \rfloor, \lfloor Nt \rfloor}(B_{1:p})$ as introduced in Definition~\ref{d:block_map_indices}.
\end{proposition}

For the proof, we need the concept of a \emph{block graph} associated with a matching~$P$ and the block decomposition~$B_{1:p}$ presented in Definition~\ref{d:block_map_indices}.
Recall that~$\beta$ denotes the block map introduced therein. 
\begin{definition}[(External) Block graph] \label{d:block_graph}
	Let~$l = 2n$ for some~$n \in \N$ and assume that~$\llbracket 1,2n \rrbracket$ can be decomposed into blocks~$B_1,\ldots,B_p$ in the sense of Definition~\ref{d:block_map_indices}.
	Consider further a perfect matching~$P \in \CM(2n)$ in the sense of Definition~\ref{d:matchings}, that is,~$P = \cbr[0]{P_{2r-1},P_{2r}}_{r=1}^n$.
	Then, the \emph{block graph}~$\CG = (V(\CG),E(\CG))$ associated with~$\llbracket 1,2n \rrbracket$, the blocks~$B_{1:p}$, and the matching~$P$ is defined as follows:
	\begin{itemize}
		\item The vertex set~$V(\CG) \coloneqq \llbracket 1,p \rrbracket$ is formed by the labels of the blocks.
		\item The edge set~$E(\CG)$ is defined as follows: 
		For each~$r\in \llbracket 1,n \rrbracket$ and couple $(P_{2r-1},P_{2r})$, we draw an edge between the blocks $\beta(P_{2r-1})$ and $\beta(P_{2r})$.
	\end{itemize}
	The edge set~$E(\CG)$ can be decomposed into
	\begin{itemize}
		\item \emph{internal edges} (or \emph{loops})~$E_{\internal}(\CG)$ given by those edges formed by the intra-block pairings, i.e. the ones with~$\beta(P_{2r-1}) = \beta(P_{2r})$.
		\item \emph{external edges}~$E_\ext$ formed by inter-block pairings, i.e. those satisfying~$\beta(P_{2r-1}) \neq \beta(P_{2r})$.
	\end{itemize}
	In a natural way, the sets of internal (resp. external) edges induces the set~$V_{\internal}(\CG)$ (resp.~$V_{\ext}(\CG)$) of internal (resp. external) vertices.
	We call the sub-graph~$\CG_\ext = (E_\ext,V_\ext) \subseteq \CG$ the \emph{external block graph}.
\end{definition}

The previous definition becomes much clearer from the following illustrative example.
\begin{example}
	Assume the setting of Example~\ref{ex:blocks} and consider~$P \in \CM(10)$ given by~$P = \{\cbr[0]{1,9}, \cbr[0]{2,7}, \cbr[0]{3,4}, \cbr[0]{5,10}, \cbr[0]{6,8}\}$. 
	We can add~$P$ to the visualisation of Example~\ref{ex:blocks} in the following way:  
	\begin{equation*}
			%
			%P = 
			\bcontraction[2ex]{}{\cbox{mutedorange}{\upshape{1}}}{\,\cbox{mutedgreen}{\upshape{2}}\,\cbox{mutedcyan}{\upshape{3}}\,\cbox{mutedcyan}{\upshape{4}}\,\cbox{mutedgreen}{\upshape{5}}\,\cbox{mutedorange}{\upshape{6}}\,\cbox{mutedmagenta}{\upshape{7}}\,\cbox{mutedorange}{\upshape{8}}\,}{\cbox{mutedcyan}{\upshape{9}}}
			\bcontraction[1ex]{\cbox{mutedorange}{\upshape{1}}\,\cbox{mutedgreen}{\upshape{2}}\,}{\cbox{mutedcyan}{\upshape{3}}}{\,}{\cbox{mutedcyan}{\upshape{4}}}
			\contraction[3ex]{\cbox{mutedorange}{\upshape{1}}\,\cbox{mutedgreen}{\upshape{2}}\,\cbox{mutedcyan}{\upshape{3}}\,\cbox{mutedcyan}{\upshape{4}}\,}{\cbox{mutedgreen}{\upshape{5}}}{\,\cbox{mutedorange}{\upshape{6}}\,\cbox{mutedmagenta}{\upshape{7}}\,\cbox{mutedorange}{\upshape{8}}\,\cbox{mutedcyan}{\upshape{9}}\,}{\cbox{mutedmagenta}{\upshape{10}}}
			\contraction[2ex]{\cbox{mutedorange}{\upshape{1}}\,}{\cbox{mutedgreen}{\upshape{2}}}{\,\cbox{mutedcyan}{\upshape{3}}\,\cbox{mutedcyan}{\upshape{4}}\,\cbox{mutedgreen}{\upshape{5}}\,\cbox{mutedorange}{\upshape{6}}\,}{\cbox{mutedmagenta}{\upshape{7}}}
			\contraction[1ex]{\cbox{mutedorange}{\upshape{1}}\,\cbox{mutedgreen}{\upshape{2}}\,\cbox{mutedcyan}{\upshape{3}}\,\cbox{mutedcyan}{\upshape{4}}\,\cbox{mutedgreen}{\upshape{5}}\,}{\cbox{mutedorange}{\upshape{6}}}{\,\cbox{mutedmagenta}{\upshape{7}}\,}{\cbox{mutedorange}{\upshape{8}}}
			\cbox{mutedorange}{\upshape{1}}\,\cbox{mutedgreen}{\upshape{2}}\,\cbox{mutedcyan}{\upshape{3}}\,\cbox{mutedcyan}{\upshape{4}}\,\cbox{mutedgreen}{\upshape{5}}\,\cbox{mutedorange}{\upshape{6}}\,\cbox{mutedmagenta}{\upshape{7}}\,\cbox{mutedorange}{\upshape{8}}\,\cbox{mutedcyan}{\upshape{9}}\,\cbox{mutedmagenta}{\upshape{10}}
	\end{equation*}
	The corresbonding block graph~$\CG$ then looks as follows: 
  	
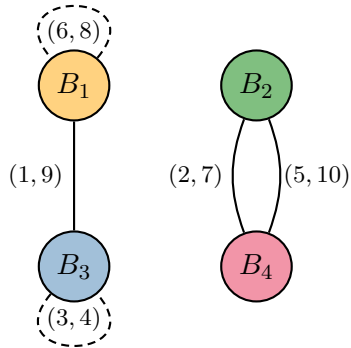
\begin{center}
		\begin{tikzpicture}[scale=0.8]
			\node[circle,draw,thick,fill=mutedorange!50, minimum size=6mm] (B1) at (0,3) {$B_1$};
			\node[circle,draw,thick,fill=mutedgreen!50,minimum size=6mm] (B2) at (3,3) {$B_2$};
			\node[circle,draw,thick,fill=mutedcyan!50,minimum size=6mm] (B3) at (0,0) {$B_3$};
			\node[circle,draw,thick,fill=mutedmagenta!50,minimum size=6mm] (B4) at (3,0) {$B_4$};
			\node (1) at (2,1.5) {{\footnotesize $(2,7)$}};
			\node (2) at (4,1.5) {{\footnotesize $(5,10)$}};
			
			\draw[thick,black] (B1) -- (B3) node[midway,left,font=\footnotesize] {$(1,9)$};
			
			\draw[thick,black] (B2) to[bend left=20] (B4) node[midway, left=2mm,font=\footnotesize] {};% {$(2,7)$};
			
			\draw[thick,black, densely dashed] (B3) to[loop below,looseness=5, out=230,in=310] (B3) node[below=4mm,font=\footnotesize] {$(3,4)$};
			
			\draw[thick,black] (B2) to[bend right=20] (B4) node[midway,below right,font=\footnotesize] {}; 
			
			\draw[thick,black, densely dashed] (B1) to[loop above,looseness=5,out=50,in=130] (B1) node[above=4mm,font=\footnotesize] {$(6,8)$};
			
		\end{tikzpicture}
	\end{center}
	The sub-graph~$\CG_{\ext}$ is obtained by removing the (dashed) self-edges corresponding to the pairs~$(3,4)$ and~$(6,8)$. 
	The number of connected components of~$\CG_{\ext}$ is~$2 = c \leq n = 5$, and likewise for~$\CG$.
\end{example}

As a final ingredient, we need the following elementary lemma.

\begin{lemma}[Tree convolution bound] \label{l:tree_convolution}
	If $T$ is a tree on $m$ vertices and each edge $e\in T$ carries a kernel $K_e\in \ell^1(\Z)$, then
	\begin{equation}  \label{e:tree_convolution}
		\sum_{x_1,\dots,x_m\in\{0,\dots,N-1\}}
		\prod_{e=(u,v)\in T} K_e(x_u-x_v)
		\;\le\;
		N\prod_{e\in T}\norm{K_e}_{\ell^1(\Z)}.
	\end{equation}
\end{lemma}

\begin{proof}
We root the tree $T$ at an arbitrary vertex, say vertex $1$. This induces a parent-child relationship: For each vertex $v \neq 1$, there is a unique parent $p(v)$, and the edge connecting $v$ to its parent is denoted $e_v = (v, p(v))$. Since $T$ has $m$ vertices, it has exactly $m-1$ edges, each of the form $e_v$ for $v \in \{2, \ldots, m\}$.

The left hand side of~\eqref{e:tree_convolution}, denoted by~$S$, can then be written as follows:
\begin{equation*}
	S = \sum_{x_1, \ldots, x_m \in \{0, \ldots, N-1\}} \prod_{v=2}^{m} K_{e_v}(x_v - x_{p(v)}) \,.	
\end{equation*}
We now sum over vertices in an order where children are processed before their parents. 
W.l.o.g, we may assume the vertices are labelled so that we sum in the order $x_m, x_{m-1}, \ldots, x_2, x_1$.

Consider summing over $x_v$ for any vertex $v$ different from the root. After all descendants of $v$ have been summed out, the only factor in the product that depends on $x_v$ is $K_{e_v}(x_v - x_{p(v)})$. For fixed $x_{p(v)} \in \{0, \ldots, N-1\}$, we substitute $y = x_v - x_{p(v)}$ to obtain
\begin{equation*}
	\sum_{x_v = 0}^{N-1} K_{e_v}(x_v - x_{p(v)}) 
	= \sum_{y = -x_{p(v)}}^{N-1 - x_{p(v)}} K_{e_v}(y) 
	\leq \sum_{y \in \mathbb{Z}} |K_{e_v}(y)| 
	= \|K_{e_v}\|_{\ell^1(\mathbb{Z})} \,.
\end{equation*}
Applying this bound iteratively for $v = m, m-1, \ldots, 2$, each summation over $x_v$ contributes a factor of at most $\|K_{e_v}\|_{\ell^1(\mathbb{Z})}$.

After summing over $x_2, \ldots, x_m$, it remains to sum over the root $x_1$. Since no edge kernel depends on $x_1$ alone (the root has no parent edge), this summation simply contributes a factor of $N$. Therefore,
\begin{equation*}
	S \leq N \cdot \prod_{v=2}^{m} \|K_{e_v}\|_{\ell^1(\mathbb{Z})} = N \prod_{e \in T} \|K_e\|_{\ell^1(\mathbb{Z})} \,,	
\end{equation*}
as claimed.
\end{proof}

We are now ready to prove Proposition~\ref{p:higher_order_diagonals}.

\begin{proof}[of Proposition~\ref{p:higher_order_diagonals}]
	For simplicity, we focus on the case~$s=0$ and~$t=1$; the general case follows analogously. 

	By~\eqref{e:YY_N_rewriting} and Corollary~\ref{coro:BM83} (with~$A_{l,N} = \triangle(B_{1:p}) \coloneqq \triangle_{0,N}^{(l)}(B_{1:p})$), 
	we have as~$N \to \infty$ that
	\begin{align}
		\thinspace
		&
		\E\sbr[1]{\scal{\YY_N^{M; (l)}(0,1),\w}} \notag\\
		= \ & 
		\frac{1}{N^{\nicefrac{l}{2}}}
		\sum_{p=1}^{l}\;
		\sum_{B_{1:p}}
		\;
		\sum_{i_{1:l} \in \triangle(B_{1:p})} \fw(i_{1:k})
		\E\sbr[4]{\prod_{r=1}^{2n} f_{\w_r}^{M}(X^{(\w_r)}_{i_r})} \label{e:aux_1}\\
		\asymp \ & 
		\frac{\1_{l=2n}}{N^{n}}
		\sum_{p=1}^{2n}
		\sum_{B_{1:p}}
		\sum_{i_{1:2n}\in \triangle(B_{1:p})} \fw(i_{1:k})
		\sum_{P\in \CM(2n)}
		\prod_{j=1}^{n}
		\E\sbr[2]{
		f^{M}_{\w_{P_{2j-1}}}\del[1]{X^{(\w_{P_{2j-1}})}_{i_{P_{2j-1}}}}\,
		f^{M}_{\w_{P_{2j}}}\del[1]{X^{(\w_{P_{2j}})}_{i_{P_{2j}}}}} \notag\\
		= \ & 
		\frac{\1_{l = 2n}}{N^{n}}
		\sum_{p=1}^{2n}
		\sum_{B_{1:p}}
		\sum_{i_{1:2n}\in \triangle(B_{1:p})} \fw(i_{1:k})
		\sum_{P\in \CM(2n)}
		\prod_{j=1}^{n}
		\Delta^M_{\w_{P_{2j-1}},\w_{P_{2j}}}\!\del[1]{i_{P_{2j}}-i_{P_{2j-1}}}. \notag
	\end{align}
	where the sum over blocks~$B_{1:p}$ is such that~$\bigsqcup_{\ell=1}^p B_\ell = \llbracket 1,2n \rrbracket$. 

	Recall the definition of the block map~$\beta$ and the block indices~$j_{1:p}$ in Definition~\ref{d:block_map_indices}.
	For a fixed pairing~$P \in \CM(2n)$ and a fixed number~$p$ of blocks~$B_{1:p}$, we are now interested in looking at
	\begin{equs}[][e:CN_PBw]
		\Lambda_N
		& =
		\Lambda_N(P,B_{1:p},\w) 
		\coloneqq
		\frac{1}{N^{n}}
		\sum_{i_{1:2n}\in \triangle(B_{1:p})} \fw(i_{1:2n})
		\prod_{r=1}^{n}
		\Delta^M_{\w_{P_{2r-1}},\w_{P_{2r}}}\del[1]{i_{P_{2r}}-i_{P_{2r-1}}} \\[0.5em]
		& =  
		\frac{C}{N^{n}}
		\sum_{0 \leq j_1 < \ldots < j_p \leq N-1} \fw(j_{\beta(1):\beta(2n)})
		\prod_{\substack{r=1, \\ \beta(P_{2r-1}) \neq \beta(P_{2r})}}^{n}
			\Delta^M_{\w_{P_{2r-1}},\w_{P_{2r}}}\del[1]{j_{\beta(P_{2r})}-j_{\beta(P_{2r-1})}}
	\end{equs}
	where the finite constant~$C$ in the previous equality is given by the \emph{intra-block pairings}
	\begin{equation} \label{e:intra_block_pairs}
		C 
		= C(P,B_{1:p},\w)
		\coloneqq 
		\prod_{\substack{r=1, \\ \beta(P_{2r-1}) = \beta(P_{2r})}}^{n}
		\Delta^M_{\w_{P_{2r-1}},\w_{P_{2r}}}(0) \,.
	\end{equation}
	Since, for fixed~$n \in \N$, there are only \emph{finitely many} block configurations~$B_{1:p}$ such that~$\llbracket 1,2n \rrbracket = \bigsqcup_{\ell=1}^p B_\ell$, by~\eqref{e:aux_1}, the claim follows if we can show the statement: 
	Suppose that, with~$a_\ell \coloneqq \abs[0]{B_\ell}$, the block configuration~$B_{1:p}$ is such that~$a_{\star} \coloneqq \max_{\ell = 1:p} a_\ell \geq 3$. 
	Then, for any pairing~$P \in \CM(2n)$ and word~$\w \in \llbracket 1,m \rrbracket^{2n}$, we have 
	\begin{equation} \label{e:higher_diag_new}
		\lim_{N\to\infty} \Lambda_N(P,B_{1:p},\w) = 0 \,.
	\end{equation}
	The rest of the proof deals with the claim in~\eqref{e:higher_diag_new}.
	
	Looking at the absolute value of~$\CC_N$, we can remove the simplex-constraints and analyse an iterated sum instead:
	\begin{equation} \label{e:bound:CN}
		\abs[0]{\Lambda_N}
		\leq  
		\frac{\abs[0]{C}}{N^{n}}
		\sum_{0 \leq j_{1:p} \leq N-1} 
		\prod_{\substack{r=1, \\ \beta(P_{2r-1}) \neq \beta(P_{2r})}}^{n}
		\abs[1]{\Delta^M_{\w_{P_{2r-1}},\w_{P_{2r-1}}}\del[1]{j_{\beta(P_{2r})}-j_{\beta(P_{2r-1})}}}
	\end{equation}
	
	Note that the assumption that~$a_\star \geq 3$ implies~$n \geq 2$.
	We now distinguish two cases: 
	First, beside~$a_\star \geq 3$, assume in addition that~$a_\ell \geq 2$ for all~$\ell \in \llbracket 1,p \rrbracket$. In this case, we have
	\begin{equation*}
		k = 2n = a_{[1:p]} \geq 2p + 1 \quad \Longleftrightarrow \quad p \leq \frac{2n-1}{2} = n - \frac{1}{2}
	\end{equation*}
	In this case, using the fact that~$\abs[0]{\rho(k)} \leq 1$  we can brutally bound, 
	\begin{equs}
		\abs[0]{\Lambda_N}
		& 
		\leq  
		\abs[0]{C} N^{p-n}
		\sum_{0 \leq j_{1:p} \leq N-1} 
		\prod_{\substack{r=1, \\ \beta(P_{2r-1}) \neq \beta(P_{2r})}}^{n}
		\abs[1]{\Delta^M_{\w_{P_{2r-1}},\w_{P_{2r}}}(0)}
		=
		O(N^{-\frac{1}{2}})
	\end{equs}
	and the claim follows. 
	\vspace{0.5em}
	
	The second case is much more complicated: In addition to~$a_\star \geq 3$, we assume that there exists one~$\ell \in \llbracket 1,p \rrbracket$ such that~$a_\ell = 1$.
	In this scenario, we cannot a priori rule out that~$p > n$; for example we have~$p = 2n - 2$ if there is one block of size~$3$ and the rest are singletons.
	We see that a much finer analysis is required.
	
	Let~$\CG_{\ext}$ denote the corresponding external block graph introduced in Definition~\ref{d:block_graph}.
	We can then rephrase the bound in~\eqref{e:bound:CN} as follows:
	\begin{equation} \label{e:tree_kernels}
		\abs[0]{\Lambda_N}
		\lesssim
		\Theta_N(\CG_{\ext})
		\coloneqq 
		\frac{1}{N^{n}}
		\sum_{0 \leq j_{1:p} \leq N-1} 
		\prod_{e=(u,v) \in E_\ext} K_e(j_v-j_u),
		\quad 
		K_e(\cdot) \coloneqq \abs[0]{\Delta^M_{\mathfrak{v}_1(e),\mathfrak{v}_2(e)}(\cdot)},
	\end{equation}
	where the pair of letters $(\mathfrak{v}_1(e),\mathfrak{v}_2(e))$ is induced by the edge $e$ (i.e. by the corresponding pair $\cbr[0]{P(2r-1),P(2r))}$; its precise form is not important for our argument.
	The only crucial fact is that, for each~$e \in E(\CG)$, we have~$K_e \in \ell^1(\Z)$ thanks to our assumption that~$\rho \in \ell^d(\Z)$.
	
	Denote by~$c$ the number of connected components of the graph $\CG_\ext$ which we label by $C_r = (V_r,E_r)$, $r = 1, \ldots, c$.
	Then
	\begin{equation*}
		\prod_{e=(u,v)\in E(\CG)} K_e(j_v-j_u)
		=
		\prod_{\ell=1}^{c}\;
		\prod_{e=(u,v) \in E_\ell} K_e(j_v-j_u)
	\end{equation*}
	and, denoting by~$T_C \subseteq E(C)$ a spanning tree for a connected component~$C$, we have
	\begin{equs}[][e:estimate_tree]
		\prod_{e=(u,v)\in E_\ext(C)} K_e(j_v-j_u)
		& \leq 
		\del[4]{\prod_{e\in T_C} K_e(j_v-j_u)}
		\del[4]{\prod_{e\in E_\ext(C)\setminus T_C} \norm{K_e}_{\ell^\infty(\Z)}} \\
		& \lesssim
		\prod_{e\in T_C} K_e(j_v-j_u) \,.
	\end{equs}
	We are now in a position to apply Lemma~\ref{l:tree_convolution} which, combined with~\eqref{e:estimate_tree}, implies that
	\begin{align}
		\Theta_N(\CG_\ext) & 
		\lesssim
		\frac{1}{N^{n}}
		\sum_{0 \leq j_{1:p} \leq N-1} 
		\,
		\prod_{\ell=1}^{c}
		\prod_{e=(u,v)\in T_{C_\ell}} K_e(j_v-j_u) \notag\\
		& =		
		\frac{1}{N^{n}}
		\prod_{\ell=1}^{c}
		\del[4]{
			\sum_{(j_v)_{v \in V_\ell} \subseteq \llbracket 0, N-1 \rrbracket^{\abs[0]{V_\ell}}} 
			\,
			\prod_{e=(u,v)\in T_{C_\ell}} K_e(j_v-j_u)} \label{e:bound_ANG}
		\\
		&
		\leq N^{c-n} \prod_{e \in F} \norm[0]{K_e}_{\ell^1(\Z)}, \quad F \coloneqq \bigcup_{\ell=1}^{c} T_{C_\ell}, \notag
	\end{align}
	where the factorisation into a product over connected components holds because, by definition, they are disjoint.
	Note that we have also introduced a spanning forest~$F$ corresponding to the spanning trees~$(T_{C_\ell})_{\ell = 1}^c$.
	
	In the final step, we now show that~$c \leq n-1$. 
	To that end, we first observe that any perfect matching~$P \in \CM(2n)$ has exactly~$n$ edges. Each connected component of the external block graph~$\CG_\ext$ contains at least one of the edges of~$P$, i.e.~$c \leq n$.
	Equality happens if and only if \emph{every} connected component of~$\CG_\ext$ contains \emph{exactly one} edge of the matching~$P$.
	Therefore, if there is at least one internal edge, we already know that~$c \leq n-1$. 

	Assume now that there are no internal edges, i.e. that there are~$n$ external edges and that~$c = n$. 
	Schematically, any graph of this type looks as follows: 
	
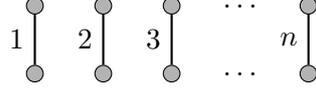
\begin{center}
		\begin{tikzpicture}[scale=0.3,baseline=0cm]
				%%%%%%%%
				\node at (6,3)  [xibig] [label={[label distance=0.05cm]90:}] (11) {};
				\node at (9,3)  [xibig] [label={[label distance=0.05cm]90:}] (12) {};
				\node at (12,3)  [xibig] [label={[label distance=0.05cm]90:}] (13) {};
				\node at (15,3)   [label={[label distance=0.05cm]90:}] (14) {$\ldots$};
				\node at (18,3)  [xibig] [label={[label distance=0.05cm]90:}] (15) {};
				%%%%%
				\node at (6,0)  [xibig] [label={[label distance=0.05cm]-90:}] (21) {};
				\node at (9,0)  [xibig] [label={[label distance=0.05cm]-90:}] (22) {};
				\node at (12,0)  [xibig] [label={[label distance=0.05cm]-90:}] (23) {};
				\node at (15,0)   [label={[label distance=0.05cm]90:}] (24) {$\ldots$};
				\node at (18,0)  [xibig] [label={[label distance=0.05cm]90:}] (25) {};
				%%%%%
				\draw[thick] (11) -- node[left] {$1$} (21);
				\draw[thick] (12) -- node[left] {$2$} (22);
				\draw[thick] (13) -- node[left] {$3$} (23);
				\draw[thick] (15) -- node[left] {$n$} (25);
				%%%%
				%%%%%
				%
			\end{tikzpicture}\; 
	\end{center}
	where we have labeled the connected components. In particular, any vertex~$v \in V_\ext(\CG)$ has degree~$1$ which, since there are no internal edges, means that each block is a singleton, i.e.~$a_\star = 1$.
	This is a contradiction to the assumption that~$a_\star \geq 3$.

	In summary, we always get~$c \leq n-1$ and, by~\eqref{e:bound_ANG}, that~$\Theta_N(\CG_\ext) = O(N^{-1})$. 
	The claim now follows from~\eqref{e:tree_kernels}.
\end{proof}

\subsubsection{Analysis of pairings} \label{s:analysis_pairings}

In the previous subsection, we have investigated which \enquote{diagonals} (or blocks) in~\eqref{e:higher_order_diagonals} contribute asymptotically; now, we will take a closer look at the pairings~$P \in \CM(2n)$. 

To this end, recall the definition of the discrete closed simplex
\begin{equation*}
	\bar{\triangle}_{\lfloor Ns \rfloor,\lfloor Nt \rfloor}^{(2n)}
	=
	\cbr[1]{i_{1:2n} \in \Z^{2n}: \ \lfloor Ns \rfloor \leq i_1 \leq i_2 \leq \ldots \leq i_{2n} \leq \lfloor Nt \rfloor -1} \,. 
\end{equation*}

\begin{definition}
	Fix~$n \in \N$ and a word~$\w = \w_{1:2n} \in \llbracket 1,m \rrbracket^{2n}$. 
	As before, let~$\CM(2n)$ be the set of pairwise matchings between vertices labelled by~$1, \ldots, 2n$. 
	Then, for any (unordered) matching~$P \in \CM(2n)$ given by 
	\begin{equation*}
		P = \cbr[1]{\cbr[0]{p_{2j-1},p_{2j}}: \ j \in \llbracket 1,n \rrbracket}
	\end{equation*}
	and~$s,t \in [0,1]$ such that~$s < t$, we define
	\begin{equation} \label{e:def:INst}
		\CI^{M;s,t}_N(P,\w) \coloneqq \frac{1}{N^n} \sum_{i_{1:2n} \in \bar{\triangle}_{\lfloor Ns \rfloor,\lfloor Nt \rfloor}^{(2n)}} \fw(i_{1:2n}) \, \prod_{j=1}^{n}  \Delta_{\w_{p_{2j-1}}, \w_{p_{2j}}}^M(i_{p_{2j}} - i_{p_{2j-1}}) \,.
	\end{equation}
\end{definition}

\begin{example}
	Let~$n = 3$ and~$P \in \CM(2n) \equiv \CM(6)$ be given by
	\begin{equation*}
		P = 
		\contraction[3ex]{}{1}{\,2\,3\,4\,5\,}{6}
		\contraction[2ex]{1\,}{2}{\,3\,}{4}
		\contraction[1ex]{1\,2\,}{3}{\,4\,}{5}
		1\,2\,3\,4\,5\,6, 
		\quad \text{i.e.} \quad
		P = \cbr[1]{\cbr[0]{1,6}, \cbr[0]{2,4}, \cbr[0]{3,5}}
	\end{equation*}
	Then, for any fixed~$\w = \w_{1:6} \in \llbracket 1,m \rrbracket^6$, we have
	\begin{equation*}
		\CI^{M;s,t}_N(P,\w) 
		= \frac{1}{N^3} \sum_{i_{1:6} \in \bar{\triangle}^{(6)}_{\lfloor Ns \rfloor,\lfloor Nt \rfloor}} \fw(i_{1:6}) \, \Delta_{\w_1,\w_6}^M(r_6 - r_1) \Delta_{\w_2,\w_4}^M(r_4 - r_2) \Delta_{\w_3,\w_5}^M(r_5 - r_3) \,.
	\end{equation*}
\end{example}

The first main result in this subsection states that non-ladder pairings vanish asymptotically.
\begin{proposition} \label{p:non_ladder_vanishes}
	Let~$n \in \N$ and~$\w = \w_{1:2n} \in \llbracket 1,m \rrbracket^{2n}$ and recall from Remark~\ref{rmk:signaure_letter_pairing} that
	\begin{equation} \label{e:ladder}
		P_\star = 
		\contraction[1ex]{}{1}{\,}{2}
		\contraction[1ex]{1\,2\,}{3}{\,}{4}
		\contraction[1ex]{1\,2\,3\,4\;\cdots\;}{(2n\!-\!1)}{\,}{(2n)}
		1\,2 \,3\,4\;\cdots\;(2n\!-\!1)\,(2n),
		\quad \text{i.e.} \quad
		P_\star = \cbr[1]{\cbr[0]{2j-1,2j}: j \in \llbracket 1,n \rrbracket}
	\end{equation}
	For any~$M \in \N$ and any~$s,t \in [0,1]$ such that~$s < t$ and~$P \neq P_\star$, we have that
	\begin{equation*}
		\lim_{N \to \infty} \CI^{M;s,t}_N(P,\w)
		=
		0
	\end{equation*}
\end{proposition}
The proof of this proposition is elementary but rather lengthy and technical; it is therefore postponed to Appendix~\ref{s:technical_proofs}.
In order to compute the asymptotic contribution of the ladder pairing, we need the following auxiliary result.

\begin{lemma} \label{l:simplex_nt}
	For~$n \in \N$ and~$s,t \in [0,1]$ with~$s < t$ as well as~$r_{1:n} \in 0, N-1 \rrbracket^{n}$, we define the set
	\begin{equation} \label{e:UNstr}
		U_N^{(n)}(s,t,r_{1:n}) \coloneqq 
		\cbr[1]{
			(u_1,v_1,\ldots,u_n,v_n) \in \bar{\triangle}_{\lfloor Ns 		\rfloor,\lfloor Nt \rfloor}^{(2n)}: \ 
			 v_j - u_j = r_j
			}	
	\end{equation}
	Then, we have 
	\begin{equation*}
		\frac{\# U_N^{(n)}(s,t,r_{1:n})}{N^n} \lesssim 1, \quad \lim_{N \to \infty} \frac{\# U_N^{(n)}(s,t,r_{1:n})}{N^n} = \frac{(t-s)^n}{n!} \,.
	\end{equation*}
\end{lemma}

\begin{proof}
	Set $a \coloneqq u_1$, write~$v_j = u_j + r_j$, and note that there exist~$b_{1:n-1} \in \N_0^{n-1}$ such that 
	\begin{equation*}
		u_{j+1} = v_j + 1 + b_j = u_j + r_j + 1 + b_j \,.
	\end{equation*}
	By a telescopic sum argument, for any~$k \geq 1$ we then find 
	\begin{equation*}
		u_k = a + r_{[1:k-1]} + b_{[1:k-1]} + k-1, 
		\quad 
		v_k = a + r_{[1:k]} + b_{[1:k-1]} + k-1 \,.
	\end{equation*} 
	For~$k = n$, the constraint~$v_n \leq \lfloor Nt \rfloor - 1$ combined with~$a \geq \lfloor Ns \rfloor$ then reads 
	\begin{equation*}
		\lfloor Ns \rfloor \leq a + b_{[1:n-1]} \leq \lfloor Nt \rfloor - n - r_{[1:n]} 
	\end{equation*}
	and, by a simple change of variables, this is equivalent to finding the number if integer solutions~$\bar{a}, \bar{b}_{1:n-1} \geq 0$ to the following equation: 
	\begin{equation*}
		0 \leq \bar a + \bar b_{[1:n-1]} \leq \lfloor Nt \rfloor - \lfloor Ns \rfloor - n - r_{[1:n]} 
	\end{equation*}
	By the well-known volume formula for the discrete simplex, we therefore find
	\begin{equation*}
		\# U_N^{(n)}(s,t,r_{1:n})
		=
		\binom{\sbr[1]{\lfloor Nt \rfloor - \lfloor Ns \rfloor - r_{[1:n]}}_+}{n} \,;
	\end{equation*}
	both claims now follow immediately.
\end{proof}

The second main result in this subsection shows that, as desired, the ladder pairing asymptotically gives rise to the expected signature.

\begin{proposition} \label{p:limit_ladder}
	Recall the definition of~$\CI_{N}^{M;s,t}(P,\w)$ in~\eqref{e:def:INst} and let~$P_\star$ be the ladder pairing given in~\eqref{e:ladder}.  
	For any~$M \in \N$, we then have
	\begin{equation*} 
		\lim_{N \to \infty} \CI_{N}^{M;s,t}(P_\star,\w)
		= \scal{\ESig^M_{s,t},\w} 
		=
		\frac{(t-s)^n}{2^n n!}  \prod_{j=1}^n (\Sigma^M_{\w_{2j-1},\w_{2j}} + 2 \AA^M_{\w_{2j-1},\w_{2j}}) \,.
	\end{equation*}
\end{proposition}

\begin{proof}
	For the sake of brevity, we only show the claim when~$s=0$ and~$t=1$; the general case is proved mutatis mutandis. Throughout this proof, we will write
	\begin{equation*}
		g_j \coloneqq \Delta^M_{\w_{2j-1},\w_{2j}} \in \ell^1(\Z) \,.
	\end{equation*} 
	By Proposition~\ref{p:higher_order_diagonals}, we know that the indices~$i_{2n} \in \bar{\triangle} \coloneqq \bar{\triangle}^{(2n)}_{0, N-1}$ can only form blocks of sizes~$1$ or~$2$. 
	We now call any such index~$i_{1:2n}$ \emph{bridging} if there exists a~$j \in \llbracket 1,n-1 \rrbracket$ such that~$i_{2j} = i_{2j+1}$, i.e. if a block of size~$2$ forms a bridge between two steps of the ladder pairing~$P_\star$.
	For example, we have
	\begin{equs}[][e:non_bridging_multiindex]
			%%%
			i_{1:6} \quad 
			& \triangleq \quad  
			\contraction[1ex]{}{\cbox{mutedorange}{1}}{\,}{\cbox{mutedorange}{2}}
			\contraction[1ex]{\cbox{mutedorange}{1}\,\cbox{mutedpurple}{2}\,}{\cbox{mutedpurple}{3}}{\,}{\cbox{mutedgreen}{4}}
			\contraction[1ex]{\cbox{mutedorange}{1}\,\cbox{mutedpurple}{2}\,\cbox{mutedpurple}{3}\,\cbox{mutedgreen}{4}\,}{\cbox{mutedcyan}{5}}{\,}{\cbox{mutedcyan}{6}}
			\cbox{mutedorange}{1}\,\cbox{mutedpurple}{2}\,\cbox{mutedpurple}{3}\,\cbox{mutedgreen}{4}\,\cbox{mutedcyan}{5}\,\cbox{mutedcyan}{6}
			%%%
			\quad \triangleq \quad \text{a bridging multi-index} \ (\text{through} \, B_2 = \cbr[0]{2,3})
			\\
			%%%%%
			%%%%%
			%%%%%
			i_{1:6} \quad 
			& \triangleq \quad 
			\contraction[1ex]{}{\cbox{mutedorange}{1}}{\,}{\cbox{mutedorange}{2}}
			\contraction[1ex]{\cbox{mutedorange}{1}\,\cbox{mutedpurple}{2}\,}{\cbox{mutedmagenta}{3}}{\,}{\cbox{mutedgreen}{4}}
			\contraction[1ex]{\cbox{mutedorange}{1}\,\cbox{mutedpurple}{2}\,\cbox{mutedmagenta}{3}\,\cbox{mutedgreen}{4}\,}{\cbox{mutedcyan}{5}}{\,}{\cbox{mutedcyan}{6}}
			\cbox{mutedorange}{1}\,\cbox{mutedpurple}{2}\,\cbox{mutedmagenta}{3}\,\cbox{mutedgreen}{4}\,\cbox{mutedcyan}{5}\,\cbox{mutedcyan}{6}
			%%%
			%%%	
			\quad \triangleq \quad \text{a non-bridging multi-index}
			%%%
	\end{equs}

	We then define\footnote{Note that, in line with our convention in~\eqref{e:multivariate_covariance}, it is important that we write~$i_{2j} - i_{2j-1}$ here (which corres\-ponds to~$g_j = \Delta^M_{\w_{2j-1},\w_{2j}}$), and not~$i_{2j-1} - i_{2j}$ (which would correspond to~$\Delta^M_{\w_{2j},\w_{2j-1}} = \del[0]{\Delta^M_{\w_{2j-1},\w_{2j}}}^{\top}$).} 
	\begin{equs}
		\CI_N^{M;0,1;\texttt{b}}(P_\star,\w)
		& \coloneqq 
		\frac{1}{N^{n}}
		\sum_{\substack{i_{1:2n} \in \bar{\triangle}, \\ i_{1:2n} \, \text{\tiny is bridging}}}
		\fw(i_{1:2n})
		\prod_{j=1}^{n}
		g_j\del[0]{i_{2j} - i_{2j-1}}, \\
		%%%
		\CI_N^{M;0,1; \texttt{n-b}}(P_\star,\w)
		& \coloneqq 
		\CI_{N}^{M;0,1}(P_\star,\w) - \CI_N^{M;0,1;\texttt{b}}(P_\star,\w)
	\end{equs}
	and divide the analysis into the corresponding cases: 

	\vspace{0.5em}
	\noindent
	$\triangleright$ \textit{Case~1: Bridging multi-indices.}
	Let~$j_\star$ be the minimal~$j \in \llbracket 1,n-1 \rrbracket$ for which~$i_{2j} = i_{2j+1}$.   
	Upon noting that~$\abs[0]{\fw(i_{1:2n})} \leq 1$, we brutally bound~$\CI_N^{s,t;\texttt{b}}(P_\star,\w)$ to get rid of the simplex constraints in the following way: 
	\begin{equs}[][e:touching_bound_1]
		\abs[0]{\CI_N^{M;0,1;\texttt{b}}(P_\star,\w)}
		& \leq
		\frac{1}{N^n} \sum_{\substack{0 \leq i_{1:2n} \leq N-1, \\ i_{2j_\star = i_{2j_\star+1}}}} \prod_{j=1}^n \abs[0]{g_j(i_{2j} - i_{2j-1})} \,.	
	\end{equs}
	Without loss of generality, we now assume that~$j_\star = 1$, otherwise we relabel the indices; this is permitted because, after the previous bound,  the indices are now independent of each other (apart from~$i_{2j_\star}$ and~$i_{2j_\star+1}$).
	Since~$j_\star = 1$, we know that~$i_2 = i_3$ and isolate the corresponding part of the sum in~\eqref{e:touching_bound_1} which involves those indices:
	\begin{equs}
		\sum_{i_1, i_4 = 0}^{N-1} \sum_{i_2 = i_3} \abs[0]{g_1(i_1-i_2)}\abs[0]{g_2(i_3-i_4)}
		\leq 
		\sum_{i_2 = 0}^{N-1} \norm[0]{g_1}_{\ell^1(\Z)} \norm[0]{g_2}_{\ell^1(\Z)}
		=
		N \norm[0]{g_1}_{\ell^1(\Z)} \norm[0]{g_2}_{\ell^1(\Z)}
	\end{equs}
	In combination with~\eqref{e:touching_bound_1}, this leads to the estimate
	\begin{align}
		\abs[0]{\CI_N^{M;0,1;\texttt{b}}(P_\star,\w)}
		& 
		\leq
		\frac{1}{N^{n-1}} \norm[0]{g_1}_{\ell^1(\Z)} \norm[0]{g_2}_{\ell^1(\Z)} \sum_{0 \leq i_{5:2n} \leq N-1} \prod_{j=3}^n \abs[0]{g_j(i_{2j} - i_{2j-1})} \label{e:touching_bound_2} \\[0.5em]
		%%%
		& \leq 
		\frac{1}{N^{n-1}} \norm[0]{g_1}_{\ell^1(\Z)} \norm[0]{g_2}_{\ell^1(\Z)} 
		\sum_{-(N-1) \leq r_{3:n} \leq N-1} \
		\sum_{\substack{0 \leq i_{5:2n} \leq N-1, \\ i_{2j} - i_{2j-1} = r_j}} \prod_{j=3}^n \abs[0]{g_j(r_j)} \notag \\[0.5em]
		%%%
		& \lesssim \frac{1}{N} \norm[0]{g_1}_{\ell^1(\Z)} \norm[0]{g_2}_{\ell^1(\Z)}  \prod_{j=3}^n \sum_{r=-(N-1)}^{N-1} \abs[0]{g_j(r)}
		\leq
		\frac{1}{N} \prod_{j=1}^n \norm[0]{g_j}_{\ell^1(\Z)} \notag \,.
	\end{align}
	where we have used that, for all~$r_{3:n} \in \Z^{n-2}$, we have the estimate 
	\begin{equation*}
		\sum_{\substack{0 \leq i_{5:2n} \leq N-1, \\ i_{2j} - i_{2j-1} = r_j}} 1 
		\lesssim N^{n-2} \,.
	\end{equation*} 
	From~\eqref{e:touching_bound_2}, we find that
	\begin{equation} \label{e:touching_vanishes}
		\lim_{N \to \infty}	\CI_N^{M;0,1;\texttt{b}}(P_\star,\w)
		=
		0 \,,
	\end{equation}
	i.e. bridging multi-indices do not contribute asymptotically.

	\vspace{0.5em}
	\noindent
	$\triangleright$ \textit{Case~2: Non-bridging multi-indices.}
	Since any index~$i_{1:2n} \in \bar{\triangle}$ whose components form blocks of size~$3$ and higher has an asymptotically vanishing contribution (see Proposition~\ref{p:higher_order_diagonals}), we need only consider the situation when there blocks of sizes~$1$ or~$2$. 
	If, additionally, the multi-index~$i_{1:2n}$ is non-bridging, the summation constraint becomes
	\begin{equation*}
		0 \leq u_1 \leq v_1 < u_2 \leq v_2 < \ldots < u_n \leq v_n \leq N-1, 
		\quad 
		\del[1]{u_j \coloneqq i_{2j-1}, \ v_j \coloneqq i_{2j}, \ j \in \llbracket 1,n \rrbracket}  
	\end{equation*}
	with corresponding index weight
	\begin{equation*}
		\fw(i_{1:2n}) = \prod_{j=1}^n \fw(u_j,v_j), \quad \fw(u_j,v_j) \coloneqq \1_{u_j < v_j} + \frac{1}{2} \1_{u_j = v_j} \,.
	\end{equation*}
	Accordingly, for~$\fw(r) \coloneqq \1_{r \neq 0} + \frac{1}{2} \1_{r = 0}$ we find
	\begin{equs}[][e:non-touching:1]
		\CI_N^{M;0,1; \texttt{n-b}}(P_\star,\w)
		& =
		\frac{1}{N^n} \sum_{u_1 \leq v_1 < \ldots < u_n \leq v_n} \prod_{j=1}^n \fw(u_j,v_j) g_j(v_j - u_j) \\[0.5em]
		%%%
		& = 
		\frac{1}{N^n} \sum_{0 \leq r_{1:n} \leq N-1} \, 
		\sum_{\substack{u_1 \leq v_1 < \ldots < u_n \leq v_n, \\ r_j = v_j - u_j}} \prod_{j=1}^n \fw(r_j) g_j(r_j) \\[0.5em]
		%%%
		& = 
		\sum_{0 \leq r_{1:n} \leq N-1} \, 
		\frac{\# U_N^{(n)}(0,1,r_{1:n})}{N^n} 
		\prod_{j=1}^n \fw(r_j) g_j(r_j) \,.
	\end{equs}
	where~$U_N^{(n)}(0,1,r_{1:n})$ is the set we have introduced in~\eqref{e:UNstr}.  
	By Lemma~\ref{l:simplex_nt}, we know that
	\begin{equation*}
		\frac{\# U_N^{(n)}(0,1,r_{1:n})}{N^n} \lesssim 1, 
		\quad 
		\lim_{N \to \infty} \frac{\# U_N^{(n)}(0,1,r_{1:n})}{N^n}  = \frac{1}{n!}
	\end{equation*}
	and, by the standing assumptions in this article, that  
	\begin{equation*}
		\prod_{j=1}^n \sum_{r \in \N_0} \abs[0]{\fw(r) g_j(r)} 
		\leq 
		\prod_{j=1}^n \norm[0]{g_j}_{\ell^1(\Z)} 
		< \infty \,.
	\end{equation*}
	Hence, by dominated convergence, eq.~\eqref{e:non-touching:1}, and the definition of~$g_j$, we find that 
	\begin{equs}
		\thinspace 
		& 
		\lim_{N \to \infty}
		\CI_N^{M;0,1; \texttt{n-b}}(P_\star,\w) \\
		= \ & 
		\frac{1}{n!} 
		\prod_{j=1}^n
		\sum_{r \in \N_0} \, 
		\fw(r) g_j(r)
		=
		\frac{1}{n!} 
		\prod_{j=1}^n
		\sum_{r \in \N_0} \, 
		\del[2]{\1_{r \neq 0} + \frac{1}{2} \1_{r = 0}} \Delta^M_{\w_{2j-1}, \w_{2j}}(r) \\[0.5em]
		%%%
		= \ & 
		\frac{1}{n!} 
		\prod_{j=1}^n
		\frac{1}{2}
		\del[3]{\Delta^M_{\w_{2j-1},\w_{2j}}(0)
			+
			2 \sum_{r \in \N} \Delta^M_{\w_{2j-1}, \w_{2j}}(r) 
		} \\[0.5em]
		%%%
		= \ & 
		\frac{1}{2^n n!} 
		\prod_{j=1}^n
		\del[0]{\Delta^M_{\w_{2j-1}, \w_{2j}}(0) + 2 \Gamma^M_{\w_{2j-1},\w_{2j}}}
		%%%
		= 
		\frac{1}{2^n n!} 
		\prod_{j=1}^n
		\del[0]{\Sigma^M_{\w_{2j-1},\w_{2j}}(0) + 2 \AA^M_{\w_{2j-1}, \w_{2j}}}
	\end{equs}
	which agrees with~\eqref{e:expec_sig_w} (when~$s = 0$ and~$t = 1$).
	The proof is complete.
\end{proof}

\subsubsection{Convergence of moments} \label{s:conv_moments}
Recall that the remaining goal of Section~\ref{s:moment_computations} is to show Theorem~\ref{t:product_case} (which, by Proposition~\ref{p:conv_fdd_stratonovich:equiv}, then implies Theorem~\ref{thm:conv_fdd_stratonovich}). 

The following proposition establishes Theorem~\ref{t:product_case} in the specific case when~$l = 1$, i.e. the situation when there is just one single interval~$[s,t]$. 
Below, we will then see how to reduce the general case~$l > 1$ to this scenario.
\begin{proposition} \label{p:exp_sig}
	Fix a length~$l \in \N$, a word~$\w = \w_{1:l} \in \llbracket 1,m \rrbracket^l$, and let~$\YY_N^M(s,t)$ be given as in~\eqref{e:YNM} resp.~\eqref{e:YY_N_rewriting}. 
	Then, we have
	\begin{equation*}
		\lim_{N \to \infty} \E\sbr[1]{\scal{\YY_N^M(s,t),\w}}
		=
		\begin{cases}
			0 & \quad \text{if} \quad l \ \text{odd} \\
			\scal{\operatorname{ESig}^M_{s,t},\w} & \quad \text{if} \quad l=2n 
		\end{cases}
	\end{equation*}
	where~$\operatorname{ESig}$ is the expected signature of the Stratonovich Brownian motion~$B^M$ as given in~\eqref{e:expec_sig_w}.
\end{proposition}

\begin{proof}
	The proof is a combination of the statements in the previous subsections. 
	First, we note that, by Corollary~\ref{coro:BM83} (with~$A_{l,N}(s,t) = \bar{\triangle}^{(l)}_{\lfloor Ns \rfloor, \lfloor Nt \rfloor}$, cf. the comments at the end of Remark~\ref{rmk:breuer_major}), we have  
	\begin{align}
		\thinspace &
		\E\sbr[1]{\scal{\YY_N^{M}(s,t),\w}}  \label{e:conv_exp_sig_0}
		= 
		\frac{1}{N^{\nicefrac{l}{2}}}
		\sum_{i_{1:l} \in \bar{\triangle}^{(l)}_{\lfloor Ns \rfloor, \lfloor Nt \rfloor}}
		\E\sbr[4]{\prod_{r=1}^{l} f_{\w_r}^{M}(X^{(\w_r)}_{i_r})} \\
		%%%
		\asymp \ & 
		\frac{\1_{l = 2n}}{N^{n}}
		\sum_{i_{1:2n} \in \bar{\triangle}^{(2n)}_{\lfloor Ns \rfloor, \lfloor Nt \rfloor}}
		\sum_{P\in \CM(2n)}
		\prod_{j=1}^{n}
		\E\sbr[1]{
			f^{M}_{\w_{P(2j-1)}}\del[1]{X^{(\w_{P(2j-1)})}_{i_{P(2j-1)}}}\,
			f^{M}_{\w_{P(2j)}}\del[1]{X^{(\w_{P(2j)})}_{i_{P(2j)}}}} \notag \\
		%%%
		= \ & 
		\frac{\1_{l = 2n}}{N^{n}}
		\sum_{i_{1:2n} \in \bar{\triangle}^{(2n)}_{\lfloor Ns \rfloor, \lfloor Nt \rfloor}}
		\sum_{P\in \CM(2n)}
		\prod_{j=1}^{n}
		\Delta^M_{\w_{P(2j-1)},\w_{P(2j)}}\!\del[1]{i_{P(2j)}-i_{P(2j-1)}}  \,. \notag
	\end{align}
	which already settles the case if~$l$ is odd. 
	
	It remains to deal with the case when~$l$ is even, i.e.~$l = 2n$ for some~$n \in \N$. 
	In that case, we rewrite expression in~\eqref{e:conv_exp_sig_0} with~$\CI_N^{s,t}(P,\w)$ as given in~\eqref{e:def:INst} and apply Proposition~\ref{p:non_ladder_vanishes} to conclude that
	\begin{equation} \label{e:conv_exp_sig_1}
		\E\sbr[1]{\scal{\YY_N^{M}(s,t),\w}} 
		=
		\sum_{P\in \CM(2n)} \CI_N^{M;s,t}(P,\w) 
		\asymp
		\CI_N^{M;s,t}(P_\star,\w)  
	\end{equation}
	as~$N \to \infty$, where~$P_\star$ is the ladder partition recalled in the afore-mentioned proposition. 

	Finally, Proposition~\ref{p:limit_ladder}, combined with~\eqref{e:conv_exp_sig_1} shows that 
	\begin{equation} \label{e:conv_exp_sig_2}
		\lim_{N \to \infty} \E\sbr[1]{\scal{\YY_N^{M}(s,t),\w}} 
		= \lim_{N \to \infty}
		\CI_N^{M;s,t}(P_\star,\w)  
		=
		\scal{\ESig^M_{s,t},\w} 
	\end{equation}
	as claimed. The proof is complete.
\end{proof}

The following lemma is the analogue to Proposition~\ref{p:non_ladder_vanishes} in case of multiple (disjoint) simplices. 
It will allow us to reduce the genuine product case ($l > 1$) to the case of a single interval ($l = 1$) that we have just dealt with. 

\begin{lemma}[Cross-simplex pairings vanish] \label{l:cross_simplex_pairings}
	Let~$l \in \N$ and consider the vector~$k_{1:l} \in \N^l$.
	Further, for~$\ell \in \llbracket 1,l \rrbracket$, consider the words~$\w_\ell \in \llbracket 1,m \rrbracket^{k_\ell}$, i.e.~$\abs[0]{\w_\ell} = k_\ell$, as well as the intervals~$[s_\ell, t_\ell) \subseteq [0,1]$ such that~$t_\ell < s_{\ell+1}$ for any index $\ell \in \llbracket 1,l-1 \rrbracket$.

	Assume that~$k_{[1:l]}$ is even, set~$2n \coloneqq k_{[1:l]}$, and consider the simplices
	\begin{equation}\label{e:cross_simplex_def}
		\triangle_\ell \coloneqq \triangle^{(k_\ell)}_{\lfloor N s_\ell \rfloor, \lfloor N t_\ell \rfloor}, \quad \ell = \llbracket 1,l \rrbracket, \quad
		\triangle \coloneqq \bigcup_{\ell=1}^l \triangle_\ell \,, 
	\end{equation}
	and label the elements in~$\triangle$ by~$i_{1:2n}$.
	Further assume that the a matching~$P \in \CM(2n)$ has \emph{cross-simplicial connections}, i.e. there there is at least one pair~$\{a,b\} \in P$ such that~$a \leq k_{\ell}$ and~$b \geq k_{\ell}+ 1$ for some index~$i \in \llbracket 1,l \rrbracket$.
	Then, for any~$M \in \N$ and 
	\begin{equation*}
		\CI^{M;\triangle}_N(P,\w) \coloneqq \frac{1}{N^n} \sum_{i_{1:2n} \in \triangle} \prod_{j=1}^{n} \Delta^{M}_{\w_{p_{2j-1}},\w_{p_{2j}}}(i_{p_{2j}} - i_{p_{2j}})  \,,
	\end{equation*}
	we have
	\begin{equation*}
		\lim_{N \to \infty} \CI^{M;\triangle}_N(P,\w) = 0 \,.
	\end{equation*}
	In other words: Pairings with cross-simplicial connections are asymptotically negligible.
	
	We denote by~$\CM_{\times}(2n)$ all the pairings that do have cross-simplicial connections, and by $\CM_-(2n) \coloneqq \CM(2n) \setminus \CM_{\times}(2n)$ those that do not.
\end{lemma}

The proof of the previous lemma is deferred to Appendix~\ref{s:technical_proofs}.
Finally, we are ready to give the proof of Theorem~\ref{thm:conv_fdd_stratonovich}.

\begin{proof}[of Theorem~\ref{thm:conv_fdd_stratonovich}]
	By Proposition~\ref{p:conv_fdd_stratonovich:equiv}, it suffices to show Theorem~\ref{t:product_case} which, by Proposition~\ref{p:exp_sig}, holds for~$l = 1$.
	
	In light of Lemma~\ref{l:no_overlaps} and Remark~\ref{rmk:ordering_intervals}, recall that we may assume without loss of generality that~$t_\ell < s_{\ell+1}$ for any index $\ell \in \llbracket 1,l-1 \rrbracket$, that is: The intervals are disjoint and ordered but there may be gaps between them. 
	
	Let~$\triangle$,~$\w = \w_1 \ldots \w_l$, and~$\CI_N^\triangle(P,\w)$ as in Lemma~\ref{l:cross_simplex_pairings} and recall that~$k_\ell = \abs[0]{\w_\ell}$ denotes the length of the word~$\w_\ell$. Then, for~$k \coloneqq k_{[1:l]}$, we clearly have~$\abs[0]{\w} = k$.
	By Corollary~\ref{coro:BM83} (with~$A_{l,N} = \triangle$ given in~\eqref{e:cross_simplex_def}), we then find\footnote{Note that~$\scal{\YY^M_N(s_i,t_i), \w_i}$ contains~$f^M_{\w_{i,1}}, \ldots, f^M_{\w_{i,k_i}}$ for~$\w_i = \w_{i,1} \ldots \w_{i,k_i}$, i.e.~$k_i$ instances of the component functions of~$\vec{f}^M$. Accordingly, the product over~$i=1$ to~$l$ contains~$k = k_{[1:l]}$ such factors, which is why the indicator~$\1_{k=2n}$ appears, and not~$\1_{l = 2n}$.} 
	\begin{align}
		\E\sbr[4]{\prod_{i=1}^l \scal{\YY^M_N(s_i,t_i), \w_i}}
		& \asymp
		\frac{\1_{k=2n}}{N^n} \sum_{P \in \CM(2n)} \sum_{i_{1:2n} \in \triangle} \prod_{j=1}^n \E\sbr[1]{f^M_{\w_{p_{2j-1}}}\del[1]{X^{(\w_{p_{2j-1}})}_{r_{p_{2j-1}}}} f^M_{\w_{p_{2j}}}\del[1]{X^{(\w_{p_{2j}})}_{r_{p_{2j}}}}} \notag \\[0.5em]
		& =
		\sum_{P \in \CM(2n)} \CI^{M;\triangle}_N(P,\w) \1_{k=2n} 
		%
		%& 
		\asymp
		\sum_{P \in \CM_-(2n)} \CI^{M;\triangle}_N(P,\w) \1_{k=2n} \,.
		\label{e:prod_case_1}
	\end{align}
	as~$N \to \infty$, where the last step uses the conclusion of Lemma~\ref{l:cross_simplex_pairings} to restrict the sum to pairings without cross-simplicial connections.
	
	Note that this already settles the case if~$k$ is odd. Indeed: If~$k$ is odd, then there exists at least one~$\ell \in \llbracket 1,l \rrbracket$ such that~$k_\ell = \abs[0]{\w_\ell}$ is odd. Accordingly, since~$\scal{\ESig^M_{s_\ell,t_\ell},\w_{\ell}} = 0$ by Lemma~\ref{l:fawcett}, by Proposition~\ref{p:exp_sig} we then find  
	\begin{equation} \label{e:product_odd_case}
		\lim_{N \to \infty}\E\sbr[4]{\prod_{i=1}^l \scal{\YY^M_N(s_i,t_i), \w_i}}
		= 0 
		= \prod_{i=1}^l \scal{\ESig^M_{s_i,t_i},\w_i} \,.
	\end{equation}
	
	It remains to prove the claim when~$k$ is even, i.e.~$k = 2n$ for some~$n \in \N$. 
	Since we are looking at \emph{perfect} matchings, note that~$P \in \CM_-(2n)$ automatically implies that~$k_\ell = 2n_\ell$ for any~$\ell \in \llbracket 1,l \rrbracket$, for otherwise there would be at least one pairing connecting different simplices.
	In particular, we can \emph{uniquely} decompose~$P \in \CM_-(2n)$ as 
	\begin{equation*}
		P = \bigsqcup_{\ell=1}^{2n} P^{\ell}, \quad P^\ell \in \CM(2n_\ell) \,,
	\end{equation*}
	where~$P^\ell = \cbr[1]{\{p^\ell_{2j-1},p^\ell_{2j}\}}_{j=1}^{n_\ell}$ pairs the variables cor\-res\-pon\-ding to the simplex~$\triangle_\ell$, i.e. $p^\ell_{2j-1}, p^\ell_{2j} \in \llbracket 1,2n_\ell \rrbracket$ for any~$j \in \llbracket 1,n_\ell \rrbracket$ and $\ell \in \llbracket 1,l \rrbracket$.
	As a consequence, we have
	\begin{equation*}
		\CI^{M;\triangle}_N(P,\w) 
		= 
		\prod_{\ell=1}^l \CI^{M;\triangle_\ell}_N(P^\ell,\w_\ell)
	\end{equation*}
	and then 
	\begin{equation} \label{e:prod_case_2}
		\sum_{P \in \CM_-(2n)} \CI^{M;\triangle}_N(P,\w) 
		= 
		\sum_{\substack{P^1, \ldots, P^{2n}, \\ P^\ell \in \CM(2n_\ell)}} \prod_{\ell=1}^l \CI^{M;\triangle_\ell}_N(P^\ell,\w_\ell)
		=
		\prod_{\ell=1}^l \sum_{P \in \CM(2n_\ell)}  \CI^{M;\triangle_\ell}_N(P,\w_\ell) 
	\end{equation}
	Note that, by definition, $\CI^{M;\triangle_\ell}_N(P,\w_\ell) = \CI^{M;s_\ell,t_\ell}_N(P,\w_\ell)$ where the latter is given in~\eqref{e:def:INst}.
	The combination of~\eqref{e:prod_case_1} and~\eqref{e:prod_case_2} then gives, for~$k = 2n$, that  
	\begin{equation} \label{e:prod_case_3}
		\E\sbr[4]{\prod_{i=1}^l \scal{\YY^M_N(s_i,t_i), \w_i}}
		\asymp
		\prod_{\ell=1}^l \sum_{P \in \CM(2n_\ell)}  \CI^{M;s_\ell,t_\ell}_N(P,\w_\ell)  
		= 
		\prod_{\ell=1}^l \E\sbr[1]{\scal{\YY_N^{M}(s_\ell,t_\ell),\w_\ell}}
	\end{equation}
	where the last equality follows from the definitions.
	Recall that~$\abs[0]{\w_\ell} = k_\ell = 2n_\ell$ for any~$	\ell \in \llbracket 1,l \rrbracket$; therefore, the claim now follows from Proposition~\ref{p:exp_sig} in the case where~$k$ is even. 
	The proof is complete.	
\end{proof}

\subsection{Removing the chaos cut-off} \label{s:removing_chaos_cutoff}

In a final step, we need to remove the chaos cut-off on the limiting vectors~$\CD^M$ and~$\CB^M$, corresponding to the assumptions~\ref{ass:D2} and~\ref{ass:B2} in Proposition~\ref{prop:abstract_convergence_statement}, respectively. 

\medskip
The following lemma deals with assumption~\ref{ass:D2}.

\begin{lemma} \label{lem:conv_DM}
	Recall that~$\CD^M$ has been introduced in~\eqref{e:def_lim_DM} and defined~$\CD$ in the same way, but with~$\Delta^M(0)$ replaced by~$\Delta(0)$, such that~$\CD^M$ and~$\CD$ are deterministic (tensor-valued) vectors.
	Then, we have~$\CD^M \to \CD$ as~$M \to \infty$.
\end{lemma}

\begin{proof}
	By definition of~$\CD^M$ and~$\CD$, the claim is equivalent to showing that~$\Delta^M(0) \to \Delta(0)$ as~$M \to \infty$. 
	However, this is immediate: For each~$k, \ell \in \llbracket 1,m \rrbracket$ and each~$M \geq d$, by Cauchy--Schwarz and the fact that~$\abs[0]{\rho_{k,\ell}(u)} \leq 1$, we have 
	\begin{equs}[][e:estimate_conv_DM]
		\abs[0]{\Delta^M_{k,\ell}(u) - \Delta_{k,\ell}(u)} 
		& =
		\abs[3]{\sum_{q \geq M} q! c_q^{(k)} c_q^{(\ell)} \rho_{k,\ell}(u)^q}
		\leq 
		\sum_{q \geq M} q! \del[1]{c_q^{(k)}}^2 \sum_{q \geq M} q! \del[1]{c_q^{(\ell)}}^2 \abs[0]{\rho_{k,\ell}(u)}^d \\[0.5em]
		& =
		\norm[0]{f_k^M}_{L^2(\gamma)} \norm[0]{f_\ell^M}_{L^2(\gamma)} \abs[0]{\rho_{k,\ell}(u)}^d \,.
	\end{equs}
	for any~$u \in \Z$, i.e. in particular for~$u = 0$.
	The quantity on the RHS goes to~$0$ as~$M \to \infty$ because~$f_k, f_\ell \in L^2(\gamma)$ and chaos tails convergence in~$L^2(\gamma)$.
\end{proof}

We record the following corollary.
\begin{corollary} \label{coro:convergence_AA}
	For the deterministic matrix~$\AA$ defined as~\eqref{e:AA} and~~$\AA^M$ defined in the same way, but with~$\vec{f}$ replaced by the projection~$\vec{f}^M$, we have~$\AA^M \to \AA$ as~$M \to \infty$.
\end{corollary}

\begin{proof}
	Since~$\Delta_{k,\ell}(u)^{\top} = \Delta_{\ell,k}(u) = \Delta_{k,\ell}(-u)$ and analogously for~$\Delta^M$, by~\eqref{e:estimate_conv_DM} we have 
	\begin{equs}
		\abs[0]{\AA^M_{k,\ell} - \AA_{k,\ell}}
		& \leq
		\frac{1}{2} \sum_{u = 1}^\infty \del[1]{\abs[0]{\Delta^M_{k,\ell}(u) - \Delta_{k,\ell}(u)} + \abs[0]{\Delta^M_{k,\ell}(-u) - \Delta_{k,\ell}(-u)}} \\
		%%%
		& \leq 
		\norm[0]{f_k^M}_{L^2(\gamma)} \norm[0]{f_\ell^M}_{L^2(\gamma)} \norm[0]{\rho_{k,\ell}}_{\ell^d(\Z)}^d \,.
	\end{equs}
	The conclusion follows like in Lemma~\ref{lem:conv_DM}.
\end{proof}

We need another lemma to show that assumption~\ref{ass:B2} in Proposition~\ref{prop:abstract_convergence_statement} is satisfied.

\begin{lemma} \label{lem:cutoff_removal_RP}
	Suppose that~$\boldsymbol{\tilde B} = (B,\mathbb{B}^{\text{\tiny Strat}})$ where~$B$ is an $m$-dimensional Brownian motion with covariance matrix~$\Sigma$ and $\mathbb{B}^{\text{\tiny Strat}}$ its Stratonovich lift, and analogously for~$\boldsymbol{\tilde B}^M = (B^M,\mathbb{\bar B}^{\text{\tiny Strat},M})$ with covariance matrix~$\Sigma^M$.
	If~$\Sigma^M \to \Sigma$ as~$M \to \infty$, then or any~$s,t \in [0,1]$ such that~$s < t$, we have
	\begin{equation*}
		\lim_{M \to \infty} B^M(t) = B(t), \quad 
		\lim_{M \to \infty} \mathbb{B}^{\text{\tiny Strat}, M}(s,t) = \mathbb{ B}^{\text{\tiny Strat}}(s,t) \quad \text{in} \quad L^2(\P) \,. 
	\end{equation*}
\end{lemma}

\begin{proof}
	Let~$W$ be a standard Brownian motion in~$\R^m$, since both~$\Sigma^M$ and~$\Sigma$ are non-negative, we may write
	\begin{equation*}
		B = A W, \quad B^M = A^M W, \quad AA^{\top} = \Sigma, \quad A^M (A^M)^{\top} = \Sigma^M \,.
	\end{equation*} 
	
	For the Brownian motion, i.e. the first order component, we have
	\begin{equs}
		\thinspace &
		\norm[0]{B^M(t) - B(t)}_{L^2(\P)}^2 
		=
		\E\sbr[0]{\abs[0]{B^M(t) - B(t)}^2} \\
		= \ & 
		\E\sbr[1]{\operatorname{Tr}\sbr[0]{(A^M - A) W(t) W(t)^\top (A^M - A)^\top}} \\[0.5em]
		= \ & 
		\operatorname{Tr}\sbr[1]{(A^M - A) \E\sbr[0]{W(t) W(t)^\top} (A^M - A)^\top}
		=
		t \norm[0]{A^M - A}_{\operatorname{HS}}^2 
		\leq t \norm[0]{\Sigma^M - \Sigma}_{1}
	\end{equs}
	where the last estimate is a due to the classical Powers-St\o{}rmer inequality. Finally, the RHS goes to~$0$ as~$M \to \infty$ by assumption.
	
	For the second order component, we note that
	\begin{equs}
		\mathbb{B}^{\text{\tiny Strat}}(s,t) 
		& = \int_{s}^t B(r) \otimes \circ \dif B(r)
		= A \mathbb{W}(s,t) A^{\top}, \quad \mathbb{W}(s,t) \coloneqq \del[3]{\int_{s}^t W(r) \dif W(r)}
	\end{equs}
	and analogously for~$\mathbb{B}^{\text{\tiny Strat}, M}(s,t)$. Therefore, we find
	\begin{equs}
		\E\sbr[1]{\norm[0]{\mathbb{B}^{\text{\tiny Strat}, M}(s,t) - \mathbb{B}^{\text{\tiny Strat}}(s,t)}^2_{\operatorname{HS}}}
		& \lesssim
		\norm[0]{A^M - A}_{\operatorname{HS}}^2 \E\sbr[1]{\norm[0]{\mathbb{W}(s,t)}_{\operatorname{HS}}^2} \del[1]{\norm[0]{A^M}_{\operatorname{HS}}^2 + \norm[0]{A}_{\operatorname{HS}}^2} \\
		& \leq
		\norm[0]{\Sigma^M - \Sigma}_{1} \E\sbr[1]{\norm[0]{\mathbb{W}(s,t)}_{\operatorname{HS}}^2} \del[1]{\norm[0]{\Sigma^M}_{1} + \norm[0]{\Sigma}_{1}}
	\end{equs}
	where we have again used the Powers-St\o{}rmer inequality. 
	The RHS goes to~$0$ as~$M \to \infty$ by assumption because it also implies that~$\sup_{M \in \N}\norm[0]{\Sigma^M}_1 < \infty$.
\end{proof}

Finally, we immediately obtain the following corollary.

\begin{corollary} \label{coro:L2conv_strat_area_correction}
	Let~$B$ a Brownian motion with covariance matrix~$\Sigma$ and consider $\mathbb{\bar B}(s,t) = \mathbb{B}^{\text{\tiny Strat}}(s,t) + \AA (t-s)$, and analogously for~$(B^M, \mathbb{\bar B}^M)$.
	Then, 
	\begin{equation*}
		\lim_{M \to \infty} \mathbb{\bar{B}}^M(s,t) = \mathbb{\bar B}(s,t) \quad \text{in} \quad L^2(\P) \,. 
	\end{equation*}
	As a consequence, Assumption~\ref{ass:B2} in Proposition~\ref{prop:abstract_convergence_statement} is satisfied.
\end{corollary}

\begin{proof}
	This is a direct consequence of Lemma~\ref{lem:cutoff_removal_RP} and Corollary~\ref{coro:convergence_AA}.
\end{proof}

\subsection{Proof of Theorem~\ref{thm:conv_fdd}} \label{s:proof_fdd_convergence}

Finally, we have gathered all the tools to prove the main result of this section, the convergence of the finite-dimensional distributions, as outlined in the strategy in Section~\ref{s:strategy}.
\begin{proof}[of Theorem~\ref{thm:conv_fdd}]
	Let~$l \in \N$, consider the intervals~$(s_i,t_i)_{i = 1}^l \subseteq [0,1]^2$ with $s_i < t_i$, and define the vectors~$\CA_N$ and~$\CL$ as in~\eqref{def:AN} and~\ref{def:L}, respectively. 
	Then, the assertion of the theorem is identical to the convergence statement in~\eqref{e:goal_fdd_conv}.
	
	To this end, for~$M \geq d$, recall the decomposition
	\begin{equation*}
		\CA_N = \CR_N^M + \CB_N^M + \CC_N^M - \frac{1}{2} \CD_N^M
	\end{equation*}
	from~\eqref{e:decomp_AN} where 
	\begin{itemize}
		\item $\CR_N^M$ is given in~\eqref{e:def_RNM},~\eqref{e:1st_order_tail}, and~\eqref{e:2nd_order_tail:1},
		\item $\CB_N^M$ is given in~\eqref{e:def_BNM}, \eqref{e:pcw_linear_interpol}, and~\eqref{e:def_YMN_2}, 
		\item $\CC_N^M$ is given in~\eqref{e:def_CNM} and~\eqref{e:pcw_linear_interpol}, and
		\item $\CD_N^M$ is given in~\eqref{e:def_DNM} and~\eqref{d:D_N}.
	\end{itemize}
	We now collect all the results from Section~\ref{s:convergence_fdd} which show that the prerequisites of Proposition~\ref{prop:abstract_convergence_statement} are met.
	We proceed in chronological order:  
	\begin{itemize}
		\item Re~$\CR_N^M$: The assumptions~\ref{ass:R1} and~\ref{ass:R2} on~$\CR_N^M$ are verified in Section~\ref{s:reduction}, specifically Proposition~\ref{p:reduction}.
		\item Re~$\CD_N^M$: The assumption~\ref{ass:D1} is verified in Proposition~\ref{p:LLN_diagonal}. 
		Assumption~\ref{ass:D2} is shown to hold true in Lemma~\ref{lem:conv_DM}.
		\item Re~$\CC_N^M$: Assumption~\ref{ass:C1} is verified in Lemma~\ref{l:ptwise_L2_conv}. 
		\item Re~$\CB_N^M$: Assumption~\ref{ass:B1} holds true because of Theorem~\ref{thm:conv_fdd_stratonovich}. 
		Assumption~\ref{ass:B2} is verified in Corollary~\ref{coro:L2conv_strat_area_correction}.
	\end{itemize}
	The claim now follows by Proposition~\ref{prop:abstract_convergence_statement}.
\end{proof}

\appendix

\section{Modifications to the proof by Breuer and Major} \label{a:proof_BM83}

In this section, we present the changes in the proof of~\cite[Proposition on~p.~433]{breuer_major_83} that are required to establish our Proposition~\ref{p:bm83_irregular}, cf.~Remark~\ref{rmk:breuer_major}.
After two steps, we will reach a point after which one can follow Breuer and Major's strategy verbatim; we will refrain from repeating those arguments. 

\begin{proof}
	Our proof mildly modifies two steps in the proof by Breuer and Major.
	
	\vspace{0.5em}
	\noindent
	$\triangleright$ \textit{Step 1: Vertical reordering the levels.}
	Recall that~$l$ is the number of levels and~$\bq = q_{q:l}$.
	For a diagram $G \in \Gamma[\bq]$, and a permutation $\pi$ of $\llbracket 1, l \rrbracket$, $\pi G$ is defined as the diagram with levels $q_{\pi^{-1}(1)}, \ldots, q_{\pi^{-1}(l)} $ (so that the $\pi(j)$-th level of $\pi G$ has cardinality $q_j$), and such that
	$$ w \in E(G) \iff \pi(w) \in \pi E(G)$$
	where, for an edge $w := \{(\ell_1,k_1), (\ell_2,k_2) \}$ we write $\pi(w) := \{(\pi(j_1),k_1), (\pi(j_2),k_2)\}$. Note now that, for an edge $w \in E(G)$, we have
	$$d_i(\pi(w)) = \pi(d_i(w)), \qquad i =1,2. $$
	Moreover, have that $\mathring{T}_{\pi G}(\pi(\w),\pi(\bq),N) = \mathring{T}_G(\w,\bq,N)$ where $\pi(\w) = (\w_{\pi(1)},\ldots,\w_{\pi(l)})$ and~$\pi(\bq) = (q_{\pi(1)},\ldots,q_{\pi(l)})$. 
	Indeed, we have
	\begin{equs}
		\mathring{T}_{\pi G}(\pi(\w),\pi(\bq),N) 
		&=  \frac{1}{N^{\nicefrac{l}{2}}} \sum_{1 \leq i_{1:l} \leq N} \prod_{w \in \pi E(G)} \abs[0]{\rho_{\w_{d_1(w)},\w_{d_2(w)}}(i_{d_2(w)} - i_{d_1(w)})} \\ 
		&= \frac{1}{N^{\nicefrac{l}{2}}} \sum_{1 \leq i_{1:l} \leq N} \prod_{w' \in E(G)} \abs[0]{\rho_{\w_{d_1(\pi(w'))},\w_{d_2(\pi(w'))}}(i_{d_2(\pi(w'))} - i_{d_1(\pi(w'))})}\\
		&= \frac{1}{N^{\nicefrac{l}{2}}} \sum_{1 \leq i_{1:l} \leq N} \prod_{w' \in E(G)} \abs[0]{\rho_{\w_{\pi(d_1(w'))},\w_{\pi(d_2(w'))}}(i_{\pi(d_2(w'))} - i_{\pi(d_1(w'))})} \\
		& = \frac{1}{N^{\nicefrac{l}{2}}} \sum_{1 \leq i'_{1:l} \leq N} \prod_{w' \in E(G)} \abs[0]{\rho_{\w'_{d_1(w')},\w'_{d_2(w')}}(i'_{d_2(w')} - i'_{d_1(w')})} \\
		& = \mathring{T}_{G}(\w,\bq,N) \,.
	\end{equs}
	In the previous calculation, we have changed the variables $w' = \pi^{-1}(w)$ in the second line and $i'_r \coloneqq i_{\pi(r)}$ as well as~$\w'_r \coloneqq \w_{\pi(r)}$ for $r \in \llbracket 1, l \rrbracket$ in the penultimate step.

\vspace{0.5em}
\noindent
$\triangleright$ \textit{Step 2: Rewriting~$\mathring{T}_G(\w,\bq,N)$.}
For~$\bq = q_{1:l}$ and~$G \in \Gamma(\bq)$, we define
\begin{equation*}
k_G: \llbracket 1,l \rrbracket \to \N, \quad 
k_G(r) \coloneqq \abs[0]{\cbr[0]{w \in E(G): \thinspace d_1(w) = r}}
\end{equation*}
i.e.~$k_G(r)$ is the number of all edges~$w \in E(G)$ starting in row~$r$.
(Note that, since for an edge $w = ((i,q_i),(j,q_j)) \in G(V)$, we have imposed~$i < j$, this definition is indeed meaningful.)

By the inequality between the arithmetic and the geometric mean, for each~$r \in \llbracket 1,l \rrbracket$ for which~$k_G(r) \geq 1$, we can now simply estimate
\begin{equation*}
\prod_{\substack{w \in E(G), \\ d_1(w) = r}} \abs[0]{\rho_{\w_r,\w_{d_2(w)}}(i_{d_2(w)} - i_r)}
\leq
\frac{1}{k_G(r)} \sum_{\substack{w \in E(G), \\ d_1(w) = r}} \abs[0]{\rho_{\w_r,\w_{d_2(w)}}(i_{d_2(w)} - i_r)}^{k_G(r)} \,. 
\end{equation*}
In case~$k_G(r) = 0$, the product on the LHS of the previous display is the empty product, i.e. equal to~$1$. 
Summing in~$\mathring{T}_G(\w,\bq,N)$ over~$0 \leq i_1 \leq N-1$ for fixed~$i_{2:l}$, we have
\begin{equs}
\mathring{T}_G(\w, \bq,N)
& \leq
\frac{1}{N^{\nicefrac{l}{2}}} 
\sup_{b \in \llbracket 1,m \rrbracket} \sup_{0 \leq j \leq N-1} \del[4]{\sum_{i_1 = 0}^{N-1} \abs[0]{\rho_{\w_1,b}(j- i_1)}^{k_G(1)}} \times \\
& \qquad \times \sum_{0 \leq i_{2:l} \leq N-1} \prod_{r=2}^{l} \prod_{\substack{w \in E(G), \\ d_1(w) = r}} \abs[0]{\rho_{\w_r,\w_{d_2(w)}}(i_{d_2(w)} - i_r)}^{k_G(r)} \,.
\end{equs} 
We now iterate this estimate for~$i_{2:l}$ to find that
\begin{equs}[][e:estimate_TG]
\mathring{T}_G(\w,\bq,N)
& \leq
\frac{1}{N^{\nicefrac{l}{2}}} 
\prod_{r=1}^l
\sup_{b_r \in \llbracket 1,m \rrbracket}
\sup_{0 \leq j_r \leq N-1} \del[4]{\sum_{i_r = 0}^{N-1} \abs[0]{\rho_{\w_r,b_r}(j_r - i_r)}^{k_G(r)}} \\
%%%
& \leq
\frac{1}{N^{\nicefrac{l}{2}}} 
\prod_{r=1}^l
\sup_{a_r, b_r \in \llbracket 1,m \rrbracket}
\sup_{0 \leq j_r \leq N-1} \del[4]{\sum_{i_r = 0}^{N-1} \abs[0]{\rho_{a_r,b_r}(j_r - i_r)}^{k_G(r)}} \\
%%%
& \leq
\frac{1}{N^{\nicefrac{l}{2}}} 
\prod_{r=1}^l
%\sup_{a, b \in \llbracket 1,m \rrbracket}
\sum_{\abs[0]{i} \leq N} \abs[0]{\rho_{a_r^*,b_r^*}(i)}^{k_G(r)}
\end{equs}
where, in the last estimate, we have used that for any~$w \in E(G)$ and~$0 \leq \bi \leq N-1$, we know that $j_{r} - i_{r} \in \llbracket -(N-1),(N-1) \rrbracket$, i.e.~$\abs[0]{j_{r} - i_{r}} \leq N-1$. 
In addition, we have written~$a_r^*$ and~$b_r^*$ for the respective indices where the supremum (in fact, it is a maximum) is attained.

Apart from notational differences, the right hand side in~\eqref{e:estimate_TG} is identical to~\cite[eq.~(2.17)]{breuer_major_83}, except for the fact that each covariance factor~$\rho_{a_r^*,b_r^*}$ depends on~$r$. 
However, it is immediate to see that the only property of~$\rho_{a_r^*,b_r^*}$ that is required is that it is an element in~$\ell^d(\Z)$.
Therefore, Proposition~\ref{p:bm83_irregular} is established by following the remaining arguments in the proof of~\cite[Proposition on~p.~433]{breuer_major_83}.
\end{proof}

\section{Technical proofs} \label{s:technical_proofs}

In this appendix, we prove the technical results that involve pairings, i.e. Proposition~\ref{p:non_ladder_vanishes} and Lemma~\ref{l:cross_simplex_pairings}.

\subsection{Proof that non-ladder pairings vanish}

In this subsection, we establish Proposition~~\ref{p:non_ladder_vanishes}, i.e. the statement that non-ladder pairings do not contribute asymptotically.

\begin{proof}[of Proposition~\ref{p:non_ladder_vanishes}]
	We first assume that~$s = 0$ and~$t = 1$. The case of general~$s$ and~$t$ will be reduced to this scenario at the end of this proof.
	For~$n \in \N$, we will  denote any pairing~$P \in \CM(2n)$ by~$P = \cbr[1]{\cbr[0]{a_k,b_k}}_{k=1}^n$ where we assume that~$a_k < b_k$. For~$k \in \llbracket 1,n \rrbracket$, we will then set 
	\begin{equation*}
		g_k \coloneqq \Delta^M_{\w_{a_k},\w_{b_k}} 
	\end{equation*}
	which satisfies~$g_k \in \ell^1(\Z)$ by the integrability assumption that~$\cbr[0]{\rho_{k,\ell}}_{k,\ell=1}^m \subseteq \ell^d(\Z)$.
	With the shorthand~$\triangle_N = \triangle^{(2n)}_{0,N-1}$, we may then rewrite~$\CI^{M;0,1}_N(P,\w)$ (defined in~\eqref{e:def:INst}) as follows: 
	\begin{equation} \label{e:I_N(P)}
		\CI_N(P) \coloneqq \CI^{M;0,1}_N(P,\w) = \frac{1}{N^n} \sum_{i_{1:2n} \in \triangle_N} \prod_{k=1}^{n} g_k(i_{b_k} - i_{a_k}) \,.
	\end{equation}
	We then~$i_0 \coloneqq 0$ and, for~$\ell \in \llbracket 1,2n \rrbracket$, introduce the new \emph{gap variables}~$v_\ell \coloneqq i_\ell - i_{\ell-1}$ and we say that the gap~$v_\ell$ is \emph{even} (\emph{odd}) if~$\ell$ is even (odd).
	The new variables satisfy the following constraints:
	\begin{equation*}
		v_\ell \geq 0, \quad \sum_{\ell=1}^{2n} v_\ell = i_{2n} \leq N-1, \quad i_j = \sum_{\ell=0}^j v_\ell, \quad i_b - i_a = \sum_{\ell=a+1}^b v_\ell \,.
	\end{equation*}
	For~$r \in \llbracket 1, 2n\rrbracket$, we define a new set of constraints 
	\begin{equation*}
		\mathring{\triangle}_N^{(r)} \coloneqq \cbr[1]{v_{1:r} \coloneqq (v_1,\ldots,v_{r}) \in \N^{r}: \ v_\ell \geq 0, \ v_{[1:r]} \leq N-1}, \quad v_{[1:r]} \coloneqq \sum_{\ell=1}^{r} v_\ell \,,
	\end{equation*}
	and then obtain
	\begin{equation} \label{e:I_N_trafo}
		\CI_{N}(P) 
		= \frac{1}{N^n} \sum_{v_{1:2n} \in \mathring{\triangle}_{N}^{(2n)}} \prod_{k=1}^{n} g_k\del[4]{\sum_{\ell = a_k+1}^{b_k} v_\ell} %\\
		%
		%& 
		= \frac{1}{N^n} \sum_{v_{1:2n} \in \mathring{\triangle}_{N}^{(2n)}} \prod_{k=1}^{n} g_k\del[0]{v_{[I_k]}} 
	\end{equation}
	where, in the last equality, we have set~$I_k \coloneqq \llbracket a_k+1, b_k \rrbracket$ and~$v_{[I_k]} \coloneqq \sum_{\ell \in I_k} v_\ell$ to simplify the notation. 

	\vspace{-1em}
	\paragraph{Observations about the pairings.}
	Before we proceed, we record the following observations about the pairings~$P \in \CM(2n)$. 
	\begin{enumerate}[label=(\roman*)]
		\item \label{i_gap} For any~$P \in \CM(2n)$, every even gap appears in~$\CI_{N}(P)$, that is: For every~$j \in \llbracket 1,n \rrbracket$,  there exists a~$k \in \llbracket 1,n \rrbracket$ such that~$a_k + 1 \leq 2j \leq b_k$. 
		
		Indeed, since~$2j$ is even, the set~$\llbracket 1, 2j-1 \rrbracket$ has odd cardinality and therefore cannot be perfectly matched within itself. As a consequence, at least one element~$a$ in that set needs to be matched with another element~$b \geq 2j$, i.e.~$a + 1 \leq 2j \leq b$.
		
		As a consequence, the gap~$v_{2j}$ appears in the sum within~$g_k$ in~\eqref{e:I_N_trafo}.
		\item \label{ii_gap} An odd gap appears in~$P \in \CM(2n)$ if and only if~$P \neq P_\star$.
		
		Indeed, if~$P = P_\star$, i.e.~$\cbr[0]{a_k,b_k} = \cbr[0]{2k-1,2k}$ for every~$k \in \llbracket 1,n \rrbracket$, then $a_k + 1 = b_k = 2k$. Therefore, we have~$\sum_{\ell = a_k +1}^{b_k} v_\ell = v_{b_k} = v_{2k}$ and only the even gaps appear in~\eqref{e:I_N_trafo}.
		
		In contrast, if~$P \neq P_\star$, there exists at least one~$k \in \llbracket 1,n \rrbracket$ such that~$b_k \geq a_k + 2$.
		Therefore, the set~$I_k$ contains at least one even and one odd number.
	\end{enumerate}
	Recall that we have assumed~$P \neq P_\star$ in the claim of Proposition~\ref{p:non_ladder_vanishes}.
	We now let 
	\begin{equation*}
		\CG \coloneqq \cbr[1]{\ell \in \llbracket 1,2n \rrbracket: \ \exists k \in \llbracket 1,n \rrbracket \ \text{s.t.} \ \ell \in I_k} \subseteq \llbracket 1,2n \rrbracket
	\end{equation*}
	denote the set of \emph{all} gap variables appearing in the product in~\ref{e:I_N_trafo}.
	By observations~\ref{i_gap} and~\ref{ii_gap} above, we know that all the even indices in~$ \llbracket 1,2n \rrbracket$ and at least one odd index belongs to~$\CG$.
	Therefore, denoting by~$M \coloneqq \llbracket 1,2n \rrbracket \setminus \CG$ the set of~\emph{free gaps}, we have
	\begin{equation} \label{e:gaps_cardinalities}
		r \coloneqq \abs[0]{\CG} \geq n+1, \quad m \coloneqq \abs[0]{M} \leq n-1, \quad m+r = 2n \,.
	\end{equation}
	In a first step, we now relabel the variables~$v_{1:2n} = (h_{1:m}, y_{1:r})$ with~$h_{1:m} = (v_\ell)_{\ell \in M}$ and $y_{i:r} = (v_\ell)_{\ell \in \CG}$ and \enquote{sum out} the free gaps variables~$h_{1:m}$ in~\eqref{e:I_N_trafo}. 
	It is a classical combinatorial fact that, for fixed~$y_{1:r}$, we have 
	\begin{equs}[][e:simplex_volume]
		\#\cbr[1]{h_{1:m} \geq 0: \ h_{[1:m]} \leq N - 1- y_{[1:r]}} 
		& = 
		\binom{N - 1 - y_{[1:r]} + m}{m} \1_{y_{[1:r]} \leq N-1} \\[0.5em]
		& \lesssim N^m \1_{y_{[1:r]} \leq N-1} 
	\end{equs}
	and thus 
	\begin{equation} \label{e:I_N_trafo:other}
		\abs[0]{\CI_{N}(P)}
		\lesssim 
		N^{m-n} 
		\sum_{y_{1:r} \in \mathring{\triangle}_{N}^{(r)}} \prod_{k=1}^{n} \abs[0]{g_k\del[0]{y_{[I_k]}}} \dif y_{1:r}
		\eqqcolon
		J_N(P) \,.
	\end{equation}
	Fix~$\eta \in (0,1)$. We now split the domain of summation into the subsets
	\begin{equs}
		\CA_N(\eta) & \coloneqq \cbr[1]{y_{1:r} \in \mathring{\triangle}_{N}^{(r)}, \ \max_{k=1,\ldots,n} y_{[I_k]} \geq \eta (N-1)} \\
		\CB_N(\eta) & \coloneqq \cbr[1]{y_{1:r} \in \mathring{\triangle}_{N}^{(r)}, \ \max_{k=1,\ldots,n} y_{[I_k]} < \eta (N-1)} 
	\end{equs}
	and, suppressing the dependence on~$P$, write 
	\begin{equation} \label{e:decomp_J_N}
		J_N(P) 
		=
		J_N^\CA(\eta) + J_N^\CB(\eta)
	\end{equation}
	for
	\begin{equation} \label{e:decomp_J_N_parts}
		J_N^\CC(\eta) \coloneqq 
	 	N^{m-n} \sum_{y_{1:r} \in \CC_N(\eta)} \prod_{k=1}^{n} \abs[0]{g_k\del[0]{y_{[I_k]}}}, \quad \CC \in \cbr[0]{\CA,\CB} \,.
	\end{equation}
	Before we separately bound the terms~$J_N^\CA(\eta)$ and~$J_N^\CB(\eta)$, we analyse the change of variables that will be relevant in both cases.
	\paragraph{Change of variables.}
	For~$k \in \llbracket 1,n \rrbracket$ and~$\ell \in \llbracket 1,r \rrbracket$, let~$u_k \coloneqq y_{[I_k]}$ and observe that
	\begin{equation*}
		u_{1:n} = C y_{1:r}, \quad C \in \cbr[0]{0,1}^{n \times r}, \quad C_{k,\ell} \coloneqq \1_{\ell \in I_k} \,. 
	\end{equation*}
	Recall that~$P = (a_k,b_k)_{k=1}^n$ and, w.l.o.g., we can order the pairings in such a way that
	\begin{equation*}
		b_1 < b_2 < \ldots < b_n \,.
	\end{equation*}
	Since~$I_k = \llbracket a_k+1,b_k \rrbracket$, we definitely have that~$b_k \in I_k$ and hence~$k \in \CG$ for any~$k \in \llbracket 1,n \rrbracket$.
	The matrix~$C$ is an~$(n \times r)$-matrix and~$n$ of its~$r$ columns---recall from~\eqref{e:gaps_cardinalities} that $r \geq n+1$---correspond to the indices~$b_{1:n}$: We extract them in the submatrix~$B \in \cbr[0]{0,1}^{n \times n}$ of~$C$. 
 	More precisely, for~$k, j \in \llbracket 1,n \rrbracket$, we set~$B_{k,j} \coloneqq \1_{b_j \in I_k}$, i.e.
	\begin{equation*}
		B_{k,j} = 1  \quad \Longleftrightarrow \quad b_j \in I_k \quad \Longleftrightarrow \quad a_k + 1 \leq b_j \leq b_k \,.
	\end{equation*}  
	Since we have ordered the~$b_j$'s ascendingly, this means that~$B_{k,j} = 0$ for~$j > k$ and~$B_{k,k} = 1$ for any~$k \in \llbracket 1,n \rrbracket$, i.e.~$B$ is lower triangular and its diagonal only has unit entries.
	As a consequence, $\det B = 1$ and~$B$ is invertible.
	
	We have extracted~$n$ of the~$r$ columns of~$C$ to form the matrix~$B$. W.l.o.g., we now reorder the coordinates of~$y = y_{1:r}$ in such a way that
	\begin{equation} \label{e:decomp_y}
		y_{1:r} = 
		\begin{pmatrix}
			y_B \\
			y_R
		\end{pmatrix}, \quad 
		C = (B, R) \quad 
		\text{for}
		\quad
		y_B \in \N_0^n, \ y_R \in \N_0^{r-n}, \ B \in \cbr[0]{0,1}^{n \times n}, R \in \cbr[0]{0,1}^{n \times (r-n)} \,.
	\end{equation}
	In this way, we have
	\begin{equation*}
		u_{1:n} = C y_{1:r} = B y_B + R y_R \in \N_0^n 
	\end{equation*}
	and the linear transformation
	\begin{equation} \label{e:change_of_var}
		\begin{cases}
		\Phi: & \N_0^n \times \N_0^{r-n} \to \N_0^n \times \N_0^{r-n}, \\
		& (y_B, y_R) \mapsto (u_{1:n}, y_R) = (B y_B + R y_R, y_R)
		\end{cases}
	\end{equation}
	can be represented via the block-matrix
	\begin{equation*}
		\begin{pmatrix}
			B & R \\
			0 & I
		\end{pmatrix}, \quad
		0 \in \N_0^{(r-n) \times n}, \quad I \in \N_0^{(r-n) \times (r-n)} \,.
	\end{equation*}
	As a consequence, we have~$\det \Phi = \det B \det I = \det B = 1$ and~$\Phi$ is an invertible, linear change of variables.% with Jacobian~$1$, i.e. $\dif y_{1:r} = \dif u_{1:n} \dif y_R$.
	
	\medskip
	With these preparations at hand, we can now bound the term~$J_N^\CA(\eta)$ and~$J_N^\CB(\eta)$ given in~\eqref{e:decomp_J_N_parts}, starting with the former.
	
	\vspace{-0.5em}
	\paragraph{(a) The bound on~$\boldsymbol{J_N^\CA(\eta)}$.} 
	By a trivial union bound, we have
	\begin{equation*}
		\1_{\max_{k=1,\ldots,n} y_{[I_k]} \geq \eta (N-1)} \leq \sum_{k=1}^n \1_{y_{[I_k]} \geq \eta (N-1)}
	\end{equation*}
	and therefore
	\begin{equation} \label{e:rhs_JA}
		J_N^\CA(\eta) 
		\leq 
		N^{m-n} \sum_{k=1}^n \sum_{\substack{y_{1:r} \in \mathring{\triangle}_{N}^{(r)}, \\ y_{[I_k]} \geq \eta (N-1)}} \prod_{k=1}^{n} \abs[0]{g_k\del[0]{y_{[I_k]}}} \,.
	\end{equation}
	Now we perform the change of variables given by~$\Phi$ in~\eqref{e:change_of_var}. Recall that this entails that $u_{1:n} \coloneqq (y_{[I_1]}, \ldots, y_{[I_n]})$ and thus we have for all~$j \in \llbracket 1,n \rrbracket$ that~$u_j = y_{[I_j]} \leq y_{[1:r]} \leq (N-1)$ where the last bound comes from the definition of the domain of summation~$\mathring{\triangle}^{(r)}_N$.
	Under this change of variables, the constraint~$y_{1:r} = (y_B,y_R) \geq 0$ becomes a constraint on~$(u_{1:n}, y_R)$. 
	More precisely, for fixed~$u_{1:n} \geq 0$, we can define the fibre
	\begin{equation*}
		F(u_{1:n}) \coloneqq \cbr[1]{y_R \in \N_0^{r-n}: \ y_B = B^{-1}\del[1]{u_{1:n} - R y_R} \geq 0} \,.
	\end{equation*}
	We can obtain a relatively crude, but sufficient bound on the cardinality of this fibre in the following way: Note that
	\begin{equation*}
		u_{[1:n]} = \sum_{k=1}^n y_{[I_k]} = \sum_{k=1}^n \sum_{\ell \in I_k} y_\ell = \sum_{\ell \in \CG} c_\ell y_\ell
	\end{equation*}
	where~$c_\ell \coloneqq \abs[0]{\cbr[0]{k: \ell \in I_k}} \geq 1$ counts the number of occurences of the gap~$\ell \in \CG$. As a consequence, recalling the decomposition of~$y_{1:r}$ from~\eqref{e:decomp_y}, we find
	\begin{equation*}
		y_{[R]} \leq y_{[1:r]} \equiv \sum_{\ell \in \CG} y_\ell \leq u_{[1:n]} 
	\end{equation*}
	where we have used the notation~$y_{[R]} = \sum_{\ell \in R} y_\ell$ to sum up all the coordinates of the subvector~$y_R$.
	As a consequence, we have
	\begin{equs}
		F(u)
		& \subseteq 
		\cbr[1]{y_R \in \N_0^{r-n}: y_{[R]} \leq u_{[1:n]} , \ y_B = B^{-1}\del[1]{u_{1:n} - R y_R} \geq 0}  \\[0.3em]
		& \subseteq 
		\cbr[1]{y_R \in \N_0^{r-n}: y_{[R]} \leq u_{[1:n]}} \,. 
	\end{equs}
	Like in~\eqref{e:simplex_volume}, we then get the bound
	\begin{equation} \label{e:volume_fibre}
		\#(F(u_{1:n})) 
		\lesssim \del[0]{u_{[1:n]}}^{r-n}
		\leq
		\del[0]{nN}^{r-n}
	\end{equation}
	where the last inequality is due to the fact that~$u_j \leq N-1$ for any~$j \in \llbracket 1,n \rrbracket$, as we have argued above.
	In combination with~\eqref{e:rhs_JA}, we then have
	\begin{align}
		J_N^\CA(\eta) 
		& \leq 
		N^{m-n} \sum_{k=1}^n \sum_{\substack{y_{1:r} \in \mathring{\triangle}_{N}^{(r)}, \\ y_{[I_k]} \geq \eta (N-1)}} \prod_{k=1}^{n} \abs[0]{g_k\del[0]{y_{[I_k]}}} \notag \\
		%%%
		& \leq N^{m-n} 
		\sum_{k=1}^n \sum_{\substack{u_k \in [\eta (N-1),N-1], \\ u_k \in \N_0}} 
		\sum_{u_{1:n\setminus\cbr[0]{k}} \in \mathring{\triangle}_{N}^{(n-1)}} \#(F(u_{1:n})) \prod_{j=1}^n \abs[0]{g_j(u_j)} \label{e:calc_JA}
		\\
		%%%
		& \leq n^{r-n} N^{m+r-2n} \sum_{k=1}^n 
		\sum_{\substack{u \in [\eta (N-1),N-1], \\ u \in \N_0}}  
		\abs[0]{g_k(u)} 
		\sum_{u_{1:n\setminus\cbr[0]{k}} \in \mathring{\triangle}_{N}^{(n-1)}} \prod_{j=1}^n \abs[0]{g_j(u_j)} \notag
		\\
		%%%
		& \leq n^{r-n} N^{m+r-2n} 
		\prod_{j=1}^n \norm[0]{g_j}_{\ell^1(\Z)}
		\sum_{k=1}^n \frac{1}{\norm[0]{g_k}_{\ell^1(\Z)}} \sum_{\substack{u \in [\eta (N-1),N-1], \\ u \in \N_0}}  
		\abs[0]{g_k(u)} \notag
	\end{align}
	Recall from~\eqref{e:gaps_cardinalities} that~$m+r-2n = 0$, so the factor~$N^{m+r-2n}$ in the previous expression is actually~$1$.
	However, for any fixed~$\eta \in (0,1)$, we have 
	\begin{equation} \label{e:calc_JA:2}
		\sum_{\substack{u \in [\eta (N-1),N-1], \\ u \in \N_0}}  
		\abs[0]{g_k(u)}  
		\leq 
		\sum_{u = \lceil \eta (N-1) \rceil}^{\infty}  
		\abs[0]{g_k(u)}  
		\to 0 \quad \text{as} \quad N \to \infty
	\end{equation}
	because~$g_k \in \ell^1(\Z)$. 
	This establishes the claim for~$J_N^\CA(\eta)$ and we are left to deal with~$J_N^\CB(\eta)$.
	\vspace{-0.5em}
	\paragraph{(b) The bound on~$\boldsymbol{J_N^\CB(\eta)}$.} 
	We apply the change of variables~$\Phi$ from~\eqref{e:change_of_var} again, i.e. in particular we set~$u_{1:n} \coloneqq (y_{[I_1]}, \ldots, y_{[I_n]})$.
	This time however, the definition of~$\CB_N(\eta)$ implies that~$u_k < \eta (N-1)$ for any~$k \in \llbracket 1,n \rrbracket$; as a consequence, the cardinality in~\eqref{e:volume_fibre} now reads 
	\begin{equation} \label{}
		\#(F(u_{1:n})) 
		\lesssim
		\del[0]{u_{[1:n]}}^{r-n}
		\leq 
		\del[0]{\eta nN}^{r-n}
	\end{equation}
	In complete analogy to the computations for~$J_N^\CA(\eta)$ in~\eqref{e:calc_JA}, we then find
	\begin{equs}[][e:calc_JB]
		J_N^\CB(\eta) 
		& 
		\leq N^{m-n} \sum_{\substack{u_{1:n} \in [0,\eta (N-1)]^n, \\ u_j \in \N_0}} \#(F(u_{1:n})) \prod_{j=1}^n \abs[0]{g_j(u_j)} \\
		& \leq (\eta n)^{r-n} N^{m+r-2n} \sum_{\substack{u_{1:n} \in [0,\eta (N-1)]^n, \\ u_j \in \N_0}}  \prod_{j=1}^n \abs[0]{g_j(u_j)} \\
		& \leq \eta^{r-n}  n^{r-n}  \prod_{j=1}^n \norm[0]{g_j}_{\ell^1(\Z)}
	\end{equs}
	where we have again used that~$m+r-2n = 0$ from~\eqref{e:gaps_cardinalities}. 
	From the same equation, we also obtain~$r - n \geq 1$ such that we can make the previous expression arbitrarily small by choosing~$\eta$ small enough.
	\vspace{-0.5em}
	\paragraph{(c) Summary of the computations.} 
	Let~$\eps > 0$. From~\eqref{e:I_N_trafo:other}, \eqref{e:decomp_J_N}, and~\eqref{e:decomp_J_N_parts}, for fixed~$\eta \in (0,1)$ and~$N \in \N$, we have that
	\begin{equation*}
		\abs[0]{\CI_{N}(P)} \leq J_N^\CA(\eta) + J_N^\CB(\eta) \,.
	\end{equation*}
	For the two terms on the right hand side, we have shown: 
	\begin{itemize}
		\item By~\eqref{e:calc_JA}, uniformly over~$\eta \in (0,1)$, we can choose~$N \in \N$ large enough such that~$J_N^\CA(\eta) < \eps/ 2$.
		\item By~\eqref{e:calc_JB}, uniformly over~$N \in \N$, we can choose~$\eta \in (0,1)$ small enough such that~$J_N^\CB(\eta) < \eps/ 2$.
	\end{itemize}
	Altogether, for any~$\eps > 0$, we can therefore find~$N \in \N$ large and~$\eta \in (0,1)$ small enough such that
	\begin{equation*}
		\abs[0]{\CI_{N}(P)} \leq J_N^\CA(\eta) + J_N^\CB(\eta) < \eps \,,
	\end{equation*}
	i.e.~$\CI_N(P) \to 0$ as~$N \to \infty$.
	This finishes the proof in case~$s = 0$ and~$t = 1$.	

	\medskip
	In case of general~$s,t \in [0,1]$, recall the definition of~$\CI^{M;s,t}_N(P,\w)$ from~\eqref{e:def:INst}. 
	By translation invariance of the summand therein and the simple change of va\-ria\-bles~$\tilde{i}_k \coloneqq i_k - \lfloor Ns \rfloor$, we can switch to the simplex~$\triangle^{2n}_{0,\lfloor Nt \rfloor - \lfloor Ns \rfloor - 1}$ as the domain of summation.
	The above proof (for~$s=0$ and~$t=1$) then works mutatis mutandis.
\end{proof}

\subsection{Proof that cross-simplex pairings vanish}
In this subsection, we present the proof of Lemma~\ref{l:cross_simplex_pairings} which states that cross-simplicial pairings give rise to asymptotically negligible contributions.

\begin{proof}[of Lemma~\ref{l:cross_simplex_pairings}]
	The proof is basically the same as that for Proposition~\ref{p:non_ladder_vanishes}, specifically the analysis of the term~$J_N^\CA(\eta)$, so we only provide a sketch of the argument. 
	With the same notation as in that proof, the expression we get after changing to gap variables analogous to~\eqref{e:I_N_trafo} reads
	\begin{equation} \label{e:pf_cross_pairings_1}
		\abs[0]{\CI^{M;\triangle}_N(P,\w)} 
		\leq \frac{1}{N^n} \sum_{v_{1:2n} \in \tilde{\triangle}_N^{(2n)} } \prod_{k=1}^{n} \abs[4]{g_k\del[4]{\sum_{\ell = a_k+1}^{b_k} v_\ell}} 
	\end{equation}
	where~$v_{[1:2n]} \coloneqq \sum_{\ell=1}^{2n} v_\ell$ and 
	\begin{equs}
		\thinspace & 
		\tilde{\triangle}_N^{(2n)} \\
		\coloneqq \ & \cbr[1]{v_{1:2n} \in \N_0^{2n}: \ v_\ell \geq 0, \ v_{[1:2n]} \leq C'(N-1), \ \exists \, u: \ \text{s.t.} \ v_u \gtrsim \lfloor N s_{u+1} \rfloor - \lfloor N t_u \rfloor},
	\end{equs}
	for some constant~$C' > 0$, for example~$C' \coloneqq (t_l - s_1) + 1$ will do. Note that we have potentially increased the domain of integration by choosing~$C'$ in this way, hence the inequality in~\eqref{e:pf_cross_pairings_1}.
	
	Most importantly, the the fact that there exists an~$u \in \llbracket 1,2n-1 \rrbracket$ such that
	\begin{equation}
		v_u > \lfloor N s_{u+1} \rfloor - \lfloor N t_u \rfloor \gtrsim N(s_{u+1} - t_{u}) \eqqcolon \eta N, \quad \eta \coloneqq s_{u+1} - t_{u} > 0
	\end{equation}
	is due to the assumption that there is at least one cross-simplicial pair~$(a,b) \in P$.
	One can now re-do the steps that led to~\eqref{e:calc_JA} and, thanks to the fact that~$\eta > 0$, conclude by choosing~$N$ large enough as in~\eqref{e:calc_JA:2}.
\end{proof}

%%%%%
%%%%%
%%%%%

\section{Counterexample to the conditional decay condition} \label{a:counterexample}

In this Appendix, we present a counterexample to the conditional decay condition of Gehringer, Li, and Sieber~\cite[Def.~3.11]{Gehringer_Li_Sieber_2022}, cf. item~\ref{intro:conditional_decay} on p.~\pageref{intro:conditional_decay}.

\begin{definition} \label{d:conditional_decay_condition}
	Let~$X = (X_k)_{k \in \Z}$ a stationary, centred Gaussian sequence with covariance function~$\rho$ such that~$\rho(0) = 1$, i.e.~$X_0 \sim \gamma \equiv \CN(0,1)$.
	Further, for~$k \in \Z$, let
	\begin{equation*}
		\CF_k \coloneqq \sigma\del[0]{X_i: i \leq k} \,.
	\end{equation*} 
	A function~$f \in L^2(\gamma)$ is said to satisfy the \emph{conditional decay condition} (w.r.t.~$X$) if 
	\begin{equation*}
		\sum_{\ell > 0} \norm[0]{\E\sbr[0]{f(X_\ell) \sVert[0] \CF_0}}_{L^2(\Omega)} < \infty \,.
	\end{equation*}
\end{definition}

Let us now introduce the\footnote{Conventionally, the parameter~$r \in (0,1/2)$ is denoted by~$d$. We changed the notation to avoid conflicts with the (lower bound of the) Hermite rank~$d \in \N$.} \emph{FARIMA($0,r,0$) model} for~$r \in (0,\nicefrac{1}{2})$, see the monograph by Pipiras and Taqqu~\cite[Sec.~2.4]{Pipiras_Taqqu_2017} or~\cite{taqqu_et_al_95} for a concise introduction; see also Remark~\ref{rmk:fBM} for its relation to fractional Brownian increments.

\begin{definition}[FARIMA($0,r,0$) model] \label{d:farima_model}
	Let~$r \in (0,\nicefrac{1}{2})$ and consider a sequence~$(\eps_k)_{k \in \Z}$ of i.i.d. Gaussian random variables such that~$\eps_0 \sim \CN(0,1)$.
	We define
	\begin{equation} \label{e:farima_model}
		X_i \coloneqq \sum_{j=0}^\infty c_j \eps_{i-j}, \quad c_j \coloneqq \frac{\Gamma(j+r)}{\Gamma(j+1)\Gamma(r)}, \quad i,j \in \Z \,.
	\end{equation}
\end{definition}

Using Stirling's formula, one can show that 
\begin{equation} \label{e:farima_coeff_asymp}
	c_j \asymp \frac{j^{r-1}}{\Gamma(r)} \quad \text{as} \quad j \to \infty \,,
\end{equation}
see~\cite[Eq.~(2.4.5)]{Pipiras_Taqqu_2017}. The following lemma is the content of~\cite[Coro.~2.4.4]{Pipiras_Taqqu_2017}.

\begin{lemma} \label{l:cov_farima}
	For~$r \in (0,\nicefrac{1}{2})$ and~$k \in \Z$, we have 
	\begin{equation} \label{e:cov_farima}
		\rho(k) = \frac{\Gamma(1-2r)}{\Gamma(1-r)\Gamma(r)} \frac{\Gamma(\abs[0]{k}+r)}{\Gamma(\abs[0]{k}-r+1)} 
		\sim 
		\abs[0]{k}^{2r-1} \frac{\Gamma(1-2r)\sin(r\pi)}{\pi} \quad \text{as} \quad \abs[0]{k} \to \infty \,.
	\end{equation}
\end{lemma}

\begin{corollary}
	For~$r \in (0,\nicefrac{1}{2})$, the FARIMA($0,r,0$) model is \emph{long-range dependent} in the sense of~\cite[Condition~II, p.~$17$]{Pipiras_Taqqu_2017}, i.e.~$\abs[0]{\rho}_{\ell^1(\Z)} = \infty$.
\end{corollary}

\begin{remark}[Relation to fractional Brownian differences] \label{rmk:fBM}
	Consider a two-sided fractional Brownian motion~(fBM)~$B_H$ with Hurst parameter~$H \in (1/2,1)$. 
	Another natural choice of a long-range dependent sequence is given by the \emph{fractional Gaussian noise}~$\tilde{X}_i \coloneqq B_H(i+1) - B_H(i)$, the coveriance function of which behaves like $\tilde{\rho}(k) \asymp \abs[0]{k}^{2H-2}$ as~$\abs[0]{k} \to \infty$, see~\cite[Sec.~2.8]{Pipiras_Taqqu_2017}.   
	In particular, for~$H = r + \nicefrac{1}{2}$, the asymptotics of~$\tilde{\rho}$ and~$\rho$ are the same: \emph{Morally}, this justifies to think of the variables~$X_i$ as given by fractional Brownian differences.
	One should be able to replicate the arguments in this appendix for~$\tilde{X}_i$ as well, using the locally independent decomposition of fBM~\cite[Sec.~3.1]{hairer_05} in a similar way as Gehringer and Li~\cite[Sec.~3.5]{gehringer_li_20}.
	We opted to present the FARIMA($0,r,0$) model because its concrete form in~\eqref{e:farima_model} allows for less technical arguments. 
\end{remark}

Finally, we can present the counterexample we alluded to.

\begin{proposition} \label{p:farima_counterexample}
	Let~$d \geq 2$. Then, for any~$r \in \del[1]{\frac{1}{2}-\frac{1}{d}, \frac{1}{2}-\frac{1}{2d}}$, the following statements hold true:
	\begin{enumerate}[label=(\roman*)]
		\item \label{farima_counterexample:i} The covariance function~$\rho$ defined in~\eqref{e:cov_farima} satisfies~$\rho \in \ell^d(\Z)$
		\item \label{farima_counterexample:ii} 
		Then,~$f = H_d$ \emph{violates} the conditional chaos decay condition with respect to the FARIMA($0,r,0$) model introduced in Definition~\ref{d:farima_model}.
	\end{enumerate} 
\end{proposition}

\begin{proof}
	Regarding the claim in~\ref{farima_counterexample:i}: By Lemma~\ref{l:cov_farima}, observe that~$\rho \in \ell^d(\Z)$ is equivalent to
	\begin{equation*}
		(2r-1)d < -1 \quad \Longleftrightarrow \quad r < \frac{1}{2} - \frac{1}{2d} \,.
	\end{equation*} 
	Let us turn to the claim in~\ref{farima_counterexample:ii}:
	For fixed $\ell > 0$, we decompose
	\begin{equation*}
		X_\ell = \E\sbr[0]{X_\ell \sVert[0] \CF_0} + R_{\ell}	
	\end{equation*}
	where $R_{\ell} \coloneqq X_\ell - \E\sbr[0]{X_\ell \sVert[0] \CF_0}$ is uncorrelated with, hence independent from $\CF_k$, as our process is Gaussian. 
	As a consequence
	\begin{equation*}
		\E\sbr[0]{H_d(X_\ell) \sVert[0] \CF_0} 
		=  
		\E_{R_\ell}[H_d(\E\sbr[0]{X_\ell \sVert[0] \CF_0} + R_{\ell})]
	\end{equation*}
	where, on the RHS, we take the expectation w.r.t. $R_{\ell}$, the random variable $X_{\ell}$ still fixed. 
	Recall the following binomial-type expansion formula for Hermite polynomials and~$a,b \in \R$ such that~$a^2 + b^2 = 1$ and~$x,y \in \R$:
	\begin{equation*}
		H_d(ax + by) = \sum_{j=0}^d \binom{d}{j} a^j b^{d-j} H_j(x) H_{d-j}(y)\,.
	\end{equation*}
	Let now~$a = a_\ell \coloneqq \norm[0]{\E\sbr[0]{X_\ell \sVert[0] \CF_0}}_{L^2(\Omega)}$ and~$b = b_\ell \coloneqq \norm[0]{R_\ell}_{L^2(\Omega)}$. 
	Recall that~$f= H_d$ and observe that 
	\begin{equs} 
		\E\sbr[0]{H_d(X_\ell)\sVert[0] \CF_0} 
		%%%
		&=  \E_{R_{\ell}}\sbr[3]{H_d\del[3]{a_\ell \frac{\E[X_\ell|\CF_0]}{a_{\ell}} + b_{\ell} \frac{R_{\ell}}{b_\ell}}} \\
		%%%
		&= \sum_{j=0}^d a_{\ell}^j b_{\ell}^{d-j} H_j \del[3]{\frac{\E[X_\ell|\CF_0]}{a_{\ell}}} \E \sbr[3]{H_{d-j} \left(\frac{R_{\ell}}{b_{\ell}}\right)} \\
		&= a_{\ell}^d \, H_d \del[3]{\frac{\E[X_\ell|\CF_0]}{a_{\ell}}}
	\end{equs}
	where we used that all the expectations in the second line that are evaluated in elements of chaoses of order at least $1$ vanish.
	Using~\eqref{e:farima_model}, we can then compute
	\begin{equs}
		\E\sbr[0]{X_\ell \sVert[0] \CF_0}
		& = 
		\sum_{j=0}^\infty c_j \E\sbr[0]{\eps_{\ell-j} \sVert [0] \CF_0}
		=
		\sum_{j \geq \ell} c_j \eps_{\ell-j} 
	\end{equs}
	and thus, since sequence~$(\eps_k)_{k \in \Z}$ consists of independent random variables, we find  
	\begin{equs}
		a_\ell^2 
		& = 
		\V\del[1]{\E\sbr[0]{X_\ell \sVert[0] \CF_0}}
		= 
		\sum_{j \geq \ell} c_j^2 \,.
	\end{equs}
	In particular, we see that~$\frac{\E[X_\ell|\CF_0]}{a_\ell} \sim \gamma$ which, combined with~\eqref{e:farima_coeff_asymp}, implies that 
	\begin{equs}
		\thinspace &
		\sum_{\ell > 0} \norm[0]{\E\sbr[0]{H_d(X_\ell) \sVert[0] \CF_0}}_{L^2(\Omega)}
		= \  
		\sqrt{d!} \sum_{\ell > 0} a_\ell^d   
		=
		\sqrt{d!} \sum_{\ell > 0} \del[3]{\sum_{j \geq \ell} c_j^2}^{\nicefrac{d}{2}} \\
		\asymp \ &  
		\sqrt{d!} \sum_{\ell > 0} \del[3]{\sum_{j \geq \ell} j^{2r-2}}^{\nicefrac{d}{2}}
		\asymp \   
		\sqrt{d!} \sum_{\ell > 0} \ell^{d(r-\nicefrac{1}{2})} \,.
	\end{equs} 
	Note that, in the last step, we have used that the inner sum converges which is equivalent to~$r < 1/2$. 
	Finally, we observe that the last sum diverges because 
	\begin{equation*}
		d(r-\nicefrac{1}{2}) > - 1 \quad \Longleftrightarrow \quad r > \frac{1}{2} - \frac{1}{d} \,.
	\end{equation*}
	The proof is complete.
\end{proof}

{\small
	\bibstyle{alpha}
	\bibliography{AKP26_References.bib}
}

\vspace{2em}
\noindent

\noindent
\begin{minipage}[t]{0.45\textwidth}
	\raggedright
	\noindent
	\small
	\textbf{Henri Elad Altman} \\
	Université Sorbonne Paris Nord \\
	LAGA, CNRS UMR 7539, Institut Galilée \\
	99 Av. J.-B. Clément, 93430 Villetaneuse, France \\
	{\it E-mail address:}
	{\tt elad-altman@math.univ-paris13.fr}
	
	\vspace{2em}

	\noindent
	\small
	\textbf{Tom Klose} \\
	University of Oxford \\ 
	Mathematical Institute \\
	Woodstock Road \\
	Oxford, OX2 6GG, United Kingdom \\
	{\it E-mail address:}
	{\tt tom.klose@maths.ox.ac.uk}
\end{minipage}
\hfill
\begin{minipage}[t]{0.48\textwidth}
	\raggedright
	\noindent
	\small
	\textbf{Nicolas Perkowski} \\
	Freie Universit\"at Berlin \\
	Arnimallee 7 \\
	14195 Berlin, Germany \\ 
	{\it E-mail address:}
	{\tt perkowski@math.fu-berlin.de}

	\medskip

	\noindent
	Max-Planck Institute for Mathematics \\ in the Sciences \\
	Inselstraße 22 \\
	04103 Leipzig, Germany \\
	{\it E-mail address:}
	{\tt nicolas.perkowski@mis.mpg.de}	
\end{minipage}

\end{document}